\documentclass{informs3}
\OneAndAHalfSpacedXI

\usepackage{etoolbox}
\usepackage{shortcuts}
\usepackage{amsmath,amssymb,mathtools,bm}
\usepackage{outlines}
\usepackage{hyperref}
\usepackage[capitalise]{cleveref}
\usepackage{booktabs}
\usepackage{nicefrac}
\usepackage{url,graphicx}
\usepackage{algorithm,subcaption}
\usepackage[noend]{algpseudocode}
\usepackage{array,multirow}
\usepackage[dvipsnames]{xcolor}
\usepackage{tikz}
\usetikzlibrary{positioning,arrows,decorations,decorations.markings,shadows,positioning,arrows.meta,matrix,fit}

\newcommand{\Z}{\mathcal Z}

\newcommand{\Y}{\mathcal Y}
\newcommand{\D}{\mathcal D}
\newcommand{\F}{\mathcal F}
\newcommand{\crit}{\mathcal C}
\usepackage{qtree}

\newcommand{\edit}{}
\newcommand{\blockedit}{}

\newcommand{\PL}{\mathrm{PL}}
\newcommand{\DL}{\mathrm{DL}}
\newcommand{\PQ}{\mathrm{PQ}}
\newcommand{\DQ}{\mathrm{DQ}}
\newenvironment{tightcenter}{%
  \setlength\topsep{0pt}%
  \setlength\parskip{0pt}%
  \par\centering}{\par\noindent\ignorespacesafterend}

\usepackage[sort]{natbib}
\bibpunct[, ]{(}{)}{,}{a}{}{,}%
 \def\bibfont{\small}%

\Crefname{assumption}{Assumption}{Assumptions}

\AtBeginEnvironment{APPENDICES}{\crefalias{section}{appendix}}
\AtBeginEnvironment{APPENDICES}{\crefalias{subsection}{appendix}}

 \newenvironment{continuance}[2][Example]
  {\par\textsc{Example #2, Cont'd (#1).}}
  {\par}
 \newenvironment{continuances}[2][Example]
  {\par\textsc{Examples #2, Cont'd (#1).}}
  {\par}

 \TheoremsNumberedThrough     
\ECRepeatTheorems

\EquationsNumberedThrough    

\MANUSCRIPTNO{MS-0001-2020.1}

\newcommand{\myfulltitle}{Stochastic Optimization Forests}

\begin{document}

\newcommand*\samethanks[1][\value{footnote}]{\footnotemark[#1]}
\RUNTITLE{\myfulltitle}
\TITLE{\myfulltitle}
\RUNAUTHOR{Kallus and Mao}
\ARTICLEAUTHORS{
\AUTHOR{Nathan Kallus\thanks{Alphabetical order.}$^1$\qquad Xiaojie Mao\samethanks$^2$}
\AFF{$^1$Cornell University, New York, NY 10044, \EMAIL{kallus@cornell.edu}\\$^2$School of Economics and Management, Tsinghua University, Beijing 100084, China, \EMAIL{maoxj@sem.tsinghua.edu.cn}}
}

\ABSTRACT{We study contextual stochastic optimization problems, where we leverage rich auxiliary observations (\eg, product characteristics) to improve decision making with uncertain variables (\eg, demand). We show how to train forest decision policies for this problem by growing trees that choose splits to directly optimize the downstream decision quality, rather than split to improve prediction accuracy as in the standard random forest algorithm. We realize this seemingly computationally intractable problem by developing approximate splitting criteria that utilize optimization perturbation analysis to eschew burdensome re-optimization for every candidate split, so that our method scales to large-scale problems. We prove that our splitting criteria consistently approximate the true risk and that our method achieves asymptotic optimality. We extensively validate our method empirically, demonstrating the value of optimization-aware construction of forests and the success of our efficient approximations. We show that our approximate splitting criteria can reduce running time hundredfold, while achieving performance close to forest algorithms that exactly re-optimize for every candidate split.}
\HISTORY{First posted version: July, 2020. This version: January, 2022.}

\maketitle

\section{Introduction}\label{sec: intro}

In this paper we consider the contextual stochastic optimization (CSO) problem,
\begin{align}\label{eq:cso}
&z^*(x)\in\argmin_{z\in\Z}
\Eb{c(z;Y)\mid X=x}
,\\\label{eq:constraints}
&\Z=\braces{z\in\R d~~:~~
\begin{array}{ll}
h_{k}(z) = 0,~~&~~k=1,\dots,s,\\ h_{k}(z) \le 0,~~&~~k=s + 1,\dots, m
\end{array}},
\end{align}
wherein, having observed contextual features $X=x\in \mathcal X \subseteq \R p$, we seek a decision $z\in\Z$ to minimize average costs, which are impacted by a yet-unrealized uncertain variable $Y\in\Y$.
\Cref{eq:cso} is essentially a stochastic optimization problem \citep{shapiro2014lectures} where the distribution of the uncertain variable is given by the \emph{conditional} distribution of $Y\mid X=x$. Crucially, this corresponds to using the observations of features $X=x$ to best possibly control total average costs over new realizations of pairs $(X,Y)$; that is,
$$\ts
\E[c(z^*(X);Y)]=\min_{
z(x):\R p\to\Z
}\E[c(z(X);Y)].
$$
Stochastic optimization can model many managerial decision-making problems in inventory management \citep{simchi2005logic}, revenue management \citep{talluri2006theory}, finance \citep{cornuejols2006optimization}, and other application domains \citep{shapiro2014lectures,kleywegt2001stochastic}. And, CSO in particular captures the interplay of such decision models with the availability of rich side observations  of other variables (\ie, covariates $X$) often present in modern datasets, which can help significantly reduce uncertainty and improve performance compared to \emph{un}conditional stochastic optimization \citep{bertsimas2014predictive}.

Since the exact joint distribution of $(X,Y)$, which specifies the CSO in \cref{eq:cso}, is generally unavailable,
we are in particular interested in learning a well-performing policy $\hat z(x)$ based on $n$ independent and identically distributed (i.i.d.) draws from the joint distribution of $(X,Y)$:
$$\text{Data}:~~\D=\{(X_1,Y_1),\dots,(X_n,Y_n)\},~~(X_i,Y_i)\sim (X,Y)~\text{i.i.d.}$$
The covariates $X$ may be any that can help predict the value of the uncertain variable $Y$ affecting costs so that we can reduce uncertainty and improve performance. 
\edit{A common approach is to first make predictions using models that are trained without consideration of the downstream decision-making problem and then solve optimization given their plugged-in predictions. However, this approach completely \emph{separates} prediction and optimization. Since all predictive models make errors, especially when learning a complex object such as the conditional distribution of $Y$ given $X$, the error trade-offs of this approach may be undesirable for the end task of decision-making.
In this paper we aim to }
\begin{tightcenter}
\edit{\bf \textit{learn effective forest-based CSO policies that \underline{integrate} prediction and optimization}.}
\end{tightcenter}

\begin{figure}[t!]%
\begin{minipage}[b]{0.3\textwidth}\centering%
\begin{subfigure}[b]{\textwidth}{\qtreecenterfalse\Tree [.{$x_1\leq4$} [.{$x_2\leq 6$} {$\tau_1(x)=1$} {$\tau_1(x)=2$} ] {$\tau_1(x)=3$} ]}\vspace{3.5em}\caption{A depth-3 tree. When the condition in a branching node holds, we take the left branch.}\label{fig: treefig tree}\end{subfigure}%
\end{minipage}\hspace{0.05\textwidth}\begin{minipage}[b]{0.3\textwidth}\centering%
\begin{subfigure}[b]{\textwidth}\includegraphics[width=\textwidth]{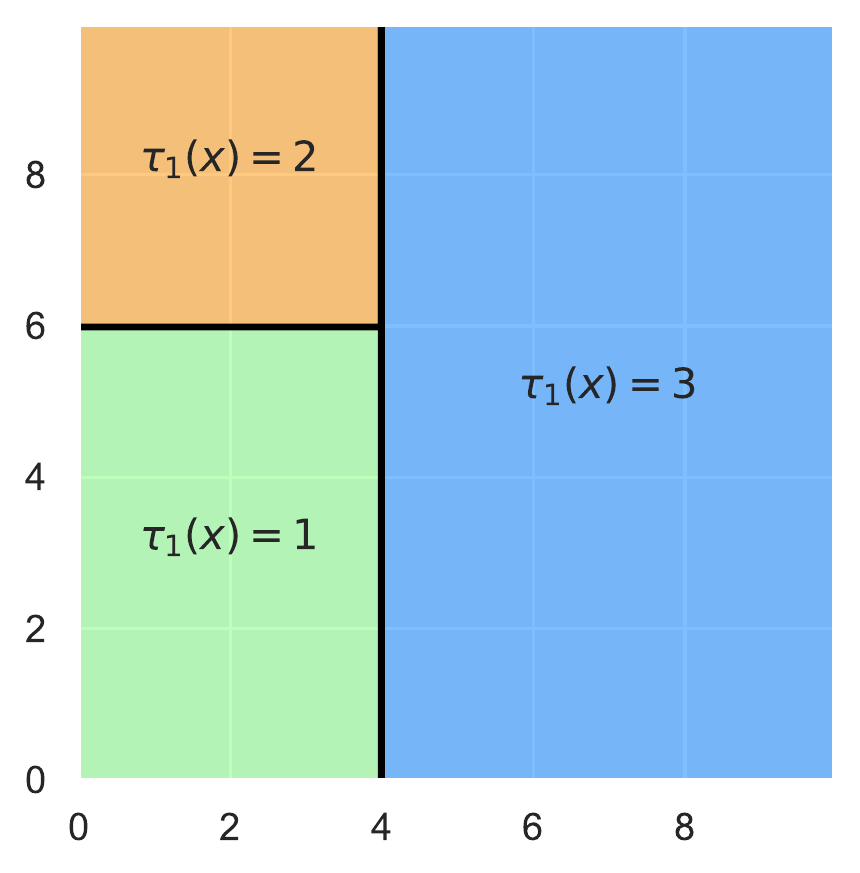}\caption{Each tree gives a partition of $\R d$, where each region corresponds to a leaf of the tree.}\label{fig: treefig regions}\end{subfigure}
\end{minipage}\hspace{0.05\textwidth}\begin{minipage}[b]{0.3\textwidth}\centering%
\begin{subfigure}[b]{\textwidth}\includegraphics[width=\textwidth]{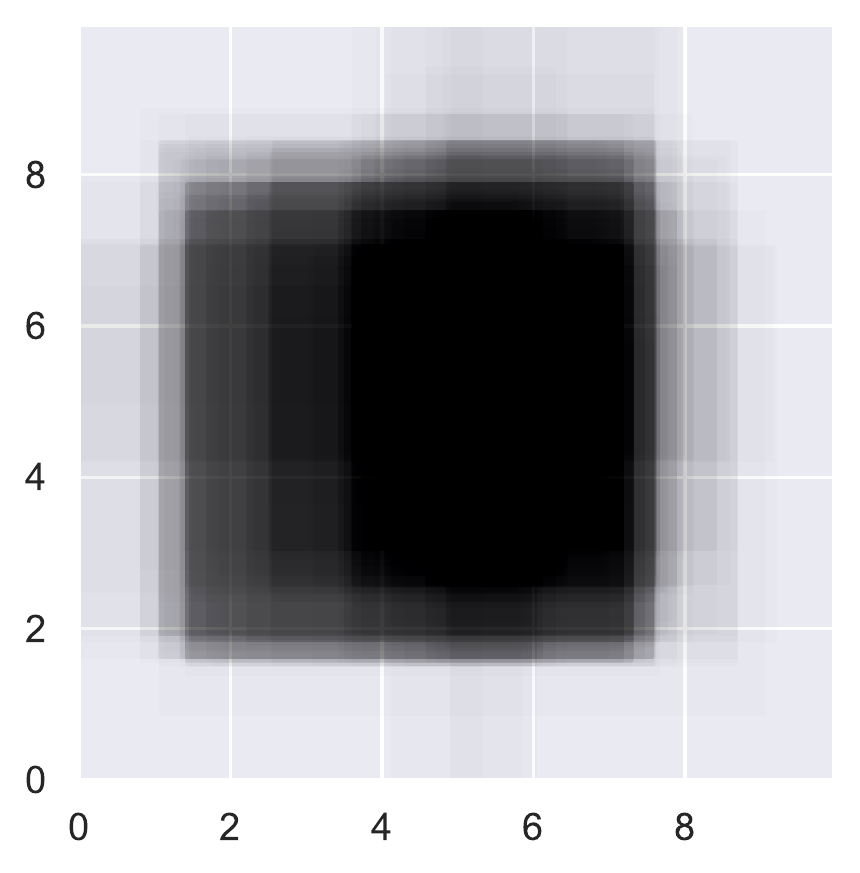}\caption{Darker regions fall into the same region as $x=(0,0)$ for more trees in a forest.%
}\label{fig: treefig dist}\end{subfigure}
\end{minipage}%
\caption{A forest of trees, $\F=\{\tau_1,\dots,\tau_T\}$, parameterizes a forest policy $\hat z(x)$ for CSO as in \cref{eq: forest policy}.}\label{fig: treefig}%
\end{figure}

To make a decision at a new query point $x$, a forest policy uses a \emph{forest} $\F=\{\tau_1,\dots,\tau_T\}$ of \emph{trees} $\tau_j$ to reweight the sample to emphasize data points $i$ with covariates $X_i$ ``close'' to $x$. Each tree, $\tau_j:\R p\to\{1,\dots,L_j\}$, is a partition of $\R p$ into $L_j$ regions, 
where the function $\tau_j$ takes the form of a binary tree with internal nodes splitting on the value of a component of $x$ (see \cref{fig: treefig tree,fig: treefig regions}). We then reweight each data point $i$ in the sample by the frequency $w_i(x)$ with which $X_i$ ends up in the same region (tree leaf) as $x$, over trees in the forest (see \cref{fig: treefig dist}). Using these weights, we solve a weighted sample analogue of \cref{eq:cso}. That is, a forest policy has the following form, where the forest $\F$ constitutes the parameters of the policy $\hat z(x)$:
\begin{align}\label{eq: forest policy}
&\hat z(x)\in\argmin_{z\in{\Z}}
\sum_{i=1}^nw_{i}(x)c(z;Y_i), ~~~ w_{i}(x) \coloneqq \frac{1}{T}\sum_{j=1}^T\frac{\indic{\tau_j(X_i)=\tau_j(x)}}{\sum_{i'=1}^n\indic{\tau_j(X_{i'})=\tau_j(x)}}.
\end{align}
\citet{bertsimas2014predictive} considered using a forest policy where the forest $\F$ is given by running the random forest \edit{(RandForest)} algorithm \citep{breiman2001random}. \edit{The RandForest algorithm, however, builds trees that target the prediction problem of learning $\Eb{Y\mid X=x}$, rather than the CSO problem in \cref{eq:cso}.}
Namely, it builds each tree $\tau_j$ by, starting with all of $\R p$, recursively subpartitioning each region $R_0\subseteq \R p$ into the two subregions $R_0=R_1\cup R_2$ that minimize the sum of squared distance to the mean of data in each subregion (\ie, $\sum_{j=1,2}\min_{z\in\R d}\sum_{i:X_i\in R_j}{\magd{z-Y_i}^2_2}$).
For prediction, random forests are notable for adeptly handling high-dimensional feature data non-parametrically as they only split on variables relevant to prediction, especially compared to other methods for generating localized weights $w_i(x)$ like $k$-nearest neighbors and Nadaraya–Watson kernel regression. 
\edit{However, for CSO they might miss signals more relevant to the particular optimization structure in \cref{eq:cso}, deteriorating downstream policy performance in the actual decision-making problem.}
\edit{\label{para: grf}\cite{athey2019generalized} proposed a Generalized Random Forest (GenRandForest) algorithm to estimate roots of conditional estimating equations, which can be repurposed for \emph{unconstrained} CSO problems by solving their first order optimality conditions. Their splitting criteria are based on approximating the \emph{mean squared errors} of equation root estimates, which again may fail to capture signals more important for the particular cost function in \cref{eq:cso} when optimization is one's aim.}

\edit{In this paper, we design new algorithms to construct decision trees and forests that directly target the CSO problem in \cref{eq:cso}. Specifically, we choose tree splits to optimize the cost of resulting decisions instead of standard impurity measures (\eg, sum of squared errors), thereby incorporating the  general cost function $c(z;Y)$ and constraints $\Z$ into the tree construction. A similar idea was suggested in endnote 2 of 
\citet{bertsimas2014predictive} but is dismissed because it would be too computationally cumbersome to use this to evaluate many candidate splits in each node of each tree in a forest. In this paper, we solve this task in a computationally efficient manner by leveraging a second-order perturbation analysis of stochastic optimization, resulting in efficient and effective large-scale forests tailored to the decision-making problem of interest that lead to strong performance gains in practice.}

Our contributions are as follows. We formalize the oracle splitting criterion for recursively partitioning trees to target the CSO problem and then use second-order perturbation analysis to show how to approximate the intractable oracle splitting criterion by extrapolating from the given region, $R_0$, to the candidate subregions, $R_1,R_2$, \edit{provided that the CSO problem is sufficiently smooth.} 
We do this in \cref{sec: unconstr} for the unconstrained setting and in \cref{sec: constr} for the constrained setting.
Specifically, we consider both an approach that extrapolates the optimal value and an approach that extrapolates the optimal solution.
Crucially, our perturbation approach means that we only have to solve a stochastic optimization problem at the root region, $R_0$, and then we can efficiently extrapolate to what will happen to average costs for any candidate subpartition of the root, allowing us to efficiently consider many candidate splits.
Using these new efficient approximate splitting criteria, we develop the stochastic optimization tree (StochOptTree) algorithm, which we then use to develop the stochastic optimization forest (StochOptForest) algorithm by running the former many times. The StochOptForest algorithm fits forests to \emph{directly} target the downstream decision-making problem of interest, and then uses these forests to construct effective forest policies for CSO. 
In \cref{sec: empirical}, we empirically demonstrate the success of our StochOptForest algorithm and the value of forests constructed to directly consider the downstream decision-making problem.
In \cref{sec: asympt-opt} we provide asymptotic optimality results for StochOptForest.
In \cref{sec:discussion} we offer a discussion of and comparison to related literature and in \cref{sec:conclusion} we offer some concluding remarks.
We extend our results to \emph{stochastically}-constrained CSO problems in \cref{sec:cso-general}, \edit{develop variable-importance measures in \cref{app-sec: var-importance}}, and provide additional empirical results  in \cref{app-sec: more-empirical}. 
We defer all proofs to \cref{sec: proofs}.

\subsection{Running Examples of CSOs}

We will have a few running examples of CSOs.

\begin{example}[Multi-Item Newsvendor]\label{ex: mnv} 
In the multi-item newsvendor problem we must choose the order quantities for $d$ products, $z = (z_1, \dots, z_d)$, before we observe the random demand for each of these, $Y = (Y_1, \dots, Y_d)$, in order to control holding and backorder costs. Whenever the order quantity for product $l$ exceeds the demand for the product we pay a holding cost of $\alpha_l$ per unit. And, whenever the demand exceeds the order quantity, we pay a backorder cost of $\beta_l$ per unit. The total cost is 
\begin{equation}\label{eq:newsvendorcost}
\ts c(z;y)=\sum_{l=1}^d\max\{\alpha_l(z_l-y_l),\,\beta_l(y_l-z_l)\}.
\end{equation}
Negating and adding a constant we can also consider this equivalently as the sale revenue up to the smaller of $z_l$ and $y_l$, minus ordering costs for $z_l$ units.
The order quantities may be unrestricted (in which case the $d$ problems decouple). They may be restricted by a  capacity constraint,
\begin{align*}
\Z = \braces{z\in \mathbb R^d:  \sum_{l = 1}^d z_l \le C, ~ z_l \ge 0, ~ l = 1, \dots, d},
\end{align*}
where $C$ is a constant that stands for the inventory capacity limit. 

Covariates $X$ in this problem may be any that can help predict future demand. For example, for predicting demand for home video products, \citet{bertsimas2014predictive} use data from Google search trends, data from online ratings, and past sales data.
\end{example}

\begin{example}[Variance-based Portfolio Optimization]\label{ex: portfolio-var}
Consider $d$ assets with random future returns $Y = (Y_1, \dots, Y_d)$, and decision variables $z = (z_1, \dots, z_d)$ that represent the fraction of investment in each asset in a portfolio of investments, constrained to be in the simplex $\Delta^d=\{z\in\R{d}: \sum_{l = 1}^d z_l = 1,\,z_l\geq0,\,l=1,\dots,d\}$.
Then the return of the portfolio is $Y^\top z$. 
We want the portfolio $z(x)$ to minimize the variance of the return given $X = x$. This can be formulated as a CSO by introducing an additional unconstrained auxiliary optimization variable $z_{d+1}\in\Rl$ and letting
\begin{align}\label{eq:portfolio-var}
c(z; y) = \prns{y^\top z_{1:d} - z_{d+1}}^2.
\end{align}
We can either let $\mathcal Z=\Delta^d\times\Rl$  or relax nonnegativity constraints to allow short selling.

More generally we may consider optimizing a linear combination of the conditional mean and variance of the return, which corresponds to a CSO with the following cost function:
\begin{align}\label{eq:portfolio-var2}
c(z; y) = \prns{y^\top z_{1:d} - z_{d+1}}^2 - \rho y^\top z_{1:d}, ~~ \rho > 0.
\end{align}

Covariates $X$ in this problem may be any that can help predict future returns. Examples include past returns, stock fundamentals, economic fundamentals, news stories, \emph{etc.}
\end{example}

\begin{example}[CVaR-based Portfolio Optimization]\label{ex: portfolio}
When the asset return distributions are not elliptically symmetric, Conditional Value-at-Risk (CVaR) may be a more suitable risk measure than variance \citep{rockafellar2000optimization}. 
We may therefore prefer to consider minimizing the CVaR at level $\alpha$ given $X = x$, defined as
\begin{align*}
\op{CVaR}_\alpha(Y^\top z\mid X=x)=
\min_{w \in \mathbb R} \Eb{\frac{1}{\alpha}\max\braces{w - Y^\top z, 0} -  w \mid X = x}.\end{align*}
This again can be formulated as a CSO by introducing an additional unconstrained auxiliary optimization variable $z_{d+1}\in\Rl$ and letting 
\begin{align}\label{eq:portfolio}
c(z; y) = \frac{1}{\alpha}\max\braces{z_{d+1} - y^\top z_{1:d},\, 0}-  z_{d+1}.
\end{align}
We can analogously incorporate the simplex constraint or relax the nonnegativity constraint as in \cref{ex: portfolio-var}. 
We can also optimize a weighted combination of the different criteria (mean, variance, CVaR at any level); we need only introduce a separate auxiliary variable for variance and for CVaR at each level considered.
\end{example}

\begin{example}[Prediction of Conditional Expectation]\label{ex: prediction}
While the above provides examples of actual decision-making problems, the problem of prediction also fits into the CSO framework as a special case. Namely, if $Y\in\R d$, $c(z;y)=\frac{1}{2}\magd{z-y}_2^2$, and $\Z =\R d$ is unconstrained, then we can see that $z^*(x)=\Eb{Y\mid X=x}$. This can be understood as the best-possible (in squared error) prediction of $Y$ in a draw of $(X,Y)$ where only $X$ is revealed. Fitting forest models to predict $\Eb{Y\mid X=x}$ is precisely the target task of random forests, which use squared error as a splitting criterion. We further compare to other literature on \emph{estimation} using random forests in \cref{sec: comp to est}. A key aspect of handling general CSOs, as we do, is dealing with general cost functions and constraints and targeting the expected cost of our decision rather than the error in estimating $z^*(x)$.
\end{example}

\section{The Unconstrained Case}\label{sec: unconstr}

We begin by studying the unconstrained case as it is simpler and therefore more instructive.
Throughout this section, we let $\Z=\R d$.
We extend to the more general constrained case in \cref{sec: constr}.
To develop our StochOptForest algorithm, we start by considering  the StochOptTree algorithm, which we will then run many times to create our forest. To motivate our StochOptTree algorithm, we will first consider an idealized splitting rule for an idealized policy, then consider approximating it using perturbation analysis, and then consider estimating the approximation using data. Each of these steps constitutes one of the next subsections.

\subsection{The Oracle Splitting Rule}\label{sec: unconstr oracle}

Given a partition, $\tau:\R p \to\{1,\dots,L\}$, of $\R p$ into $L$ regions, consider the policy $z_\tau(x)\in\argmin_{z\in\Z}\Eb{c(z;Y)\indic{\tau(X)=\tau(x)}}$ that, for each $x$, optimizes costs only for $(X,Y)$ where $X$ falls in the same region as $x$.
Note that this policy is hypothetical and not implementable in practice given just the data as it involves the \emph{true} joint distribution of $(X,Y)$.
We wish to learn a partition $\tau$ described by a binary decision tree with nodes of the form ``$x_j \leq\theta$?'' such that it leads to a well-performing policy $z_\tau(x)$, that is, has small risk $\Eb{c(z_\tau(X);Y)}$.
Finding the best $\tau$ over all trees of a given depth is generally a very hard problem, even if we knew the distributions involved. To simplify it, suppose we fix a partition $\tau$ and we wish only to refine it slightly by taking one of its regions, say $R_0=\tau^{-1}(L)$, and choosing some $j \in \{1,\dots,p\},\,\theta\in\Rl$ to construct a new partition $\tau'$ with $\tau'(x)=\tau(x)$ for $x\notin R_0$, $\tau'(x)=L$ for $x\in R_1=R_0\cap\{x\in\R p:x_j\leq\theta\}$, and $\tau'(x)=L+1$ for $x\in R_2=R_0\cap\{x\in\R p:x_j>\theta\}$.
That is, we further subpartition the region $R_0$ into the subregions $R_1,R_2$.
We would then be interested in finding the choice of $(j,\theta)$ leading to minimal risk, 
$\Eb{c(z_{\tau'}(X);Y)}=\Eb{c(z_{\tau'}(X);Y)\indic{X\notin R_0}}+\Eb{c(z_{\tau'}(X);Y)\indic{X\in R_1}}+\Eb{c(z_{\tau'}(X);Y)\indic{X\in R_2}}$.
Notice that the first term is constant in the choice of the subpartition and only the second and third terms matter in choosing the subpartition. We should therefore seek the subpartition that leads to the minimal value of
\begin{equation}\label{eq:oraclecrit}\ts
\crit^\text{oracle}(R_1,R_2)=\sum_{j=1,2}\Eb{c(z_{\tau'}(X);Y)\indic{X\in R_j}}=\sum_{j=1,2}\min_{z\in\Z}\Eb{c(z;Y)\indic{X\in R_j}},
\end{equation}
where the last equality holds because the tree policy $z_{\tau'}$ makes the best decision within each region of the new partition.
We call this the \emph{oracle splitting criterion}. Searching over choices of $(j,\theta)$ in some given set of possible options, the best refinement of $\tau$ is given by the choice minimizing this criterion.
If we start with the trivial partition, $\tau(x)=1\;\forall x$, then we can recursively refine it using this procedure in order to grow a tree of any desired depth. When $c(z;y)={\frac12\magd{z-y}_2^2}$ and the criterion is estimated by replacing expectations with empirical averages, this is precisely the regression tree algorithm of \citet{breiman1984classification}, in which case the estimated criterion is easy to compute as it is simply given by rescaled within-region variances of $Y_i$. For general $c(z;y)$, however, computing the criterion involves solving a general stochastic optimization problem that may have no easy analytical solution (even if we approximate expectations with empirical averages) and it is therefore hard to do quickly for many, many possible candidates for $(j,\theta)$, and correspondingly it would be hard to grow large forests of many of these trees.

\subsection{Perturbation Analysis of the Oracle Splitting Criterion}\label{sec: perturbation-unconstr}
Consider a region $R_0\subseteq\R p$ and its candidate subpartition $R_0=R_1\cup R_2$, $R_1\cap R_2=\varnothing$. 
Let 
\begin{align}\label{eq:v}
&\edit{v_j(t)=\min_{z\in\Z}\;f_{0}(z)+t\prns{f_{j}(z) - f_{0}(z)},~z_j(t)\in\argmin_{z\in\Z}\;f_{0}(z)+t\prns{f_{j}(z) - f_{0}(z)}},
\\\notag
&\text{where}~f_j(z)=\Eb{c(z;Y)\mid{X\in R_j}},\quad j=0,1,2, ~ t \in [0, 1].
\end{align}
\edit{The optimization objective function  in \cref{eq:v} is obtained from perturbing the objective function $f_0\prns{z}$ in the region $R_0$ towards the objective function $f_j\prns{z}$ in a subregion $R_j$ for $j = 1, 2$. The perturbation magnitude is quantified by the  parameter $t \in [0, 1]$. Note that the optimal values of fully perturbed problems (\ie, $t = 1$) in two subregions determine the oracle splitting criterion:}
\begin{align}\label{eq: oracle2}
\crit^\text{oracle}(R_1,R_2)=p_1v_1(1)+p_2v_2(1),\quad\text{where}\quad p_j=\Prb{X\in R_j},
\end{align}
and ideally we would use \edit{these values}
to evaluate the quality of the subpartition. But we would rather not have to solve the stochastic optimization problem involved in \cref{eq:v} \edit{at $t = 1$ \emph{repeatedly}} for every candidate subpartition. Instead, we would rather solve the \emph{single} problem $v_1(0)=v_2(0)$, \ie, the problem corresponding to the region $R_0$, and try to extrapolate from there what happens as we take $t\to1$, \edit{\ie, the limiting problem corresponding to each subregion $R_j$ for each candidate split}. To solve this, we consider the perturbation of the problem $v_j(t)$ at $t=0$ as we increase it infinitesimally and use this to approximate $v_j(1)$. As long as the distribution of $Y\mid X\in R_0$ is not too different from that of $Y\mid X\in R_j$, this would be a reasonable approximation.

First, we note that a first-order perturbation analysis would be insufficient. We can show that under appropriate continuity conditions and if $\argmin_{z\in\Z}f_0(z)=\{z_0\}$ is a singleton, we would have $v_j(t)=(1-t)f_0(z_0)+tf_j(z_0)+o(t)$.
\footnote{This is, for example, a corollary of \cref{thm:secondorder}, although weaker continuity conditions would be needed for this first-order statement. We omit the details as the first-order analysis is ultimately not useful.}
 We could use this to approximate $v_j(1)\approx f_j(z_0)$ by plugging in $t=1$ and ignoring the higher-order terms, which makes intuitive sense: \edit{if we only perturb the objective slightly, the optimal solution is approximately unchanged  and we only need to evaluate its new objective value}.
This would lead to the approximate splitting criterion $p_1v_1(1)+p_2v_2(1)\approx p_1f_1(z_0)+p_2f_2(z_0)$. However, since $p_1f_1(z_0)+p_2f_2(z_0)=p_0f_0(z_0)$, this is ultimately unhelpful as it does not at all depend on the choice of subpartition.

Instead, we must conduct a finer, second-order perturbation analysis in order to understand the effect of the choice of subpartition on risk. The next result does this for the unconstrained case.

\begin{theorem}[Second-Order Perturbation Analysis: Unconstrained]\label{thm:secondorder}
Fix $j=1,2$.
Suppose the following conditions hold:
\begin{enumerate}
\item  $f_0(z)$ and $f_j(z)$ are twice continuously differentiable; \label{cond: unconst-smooth} 
\item The inf-compactness condition: \edit{there exist} constants $\alpha$ and $t_0 \in (0, 1]$ such that the sublevel sets 
\edit{$\left\{
z \in \mathbb R^d: ~ f_{0}(z)+t\prns{f_{j}(z)-f_0\prns{z}} \le \alpha
\right\}$ }
are nonempty and uniformly bounded for $t \in [0, t_0)$; \label{cond: unconst-comp} 
\item  $f_0(z)$ has a unique minimizer $z_0$ over $\mathbb R^d$, and $\nabla^2 f_0(z_0)$ is positive definite; \label{cond: unconst-pd}  
\end{enumerate}
Then 
\begin{align}v_j(t)&=(1-t)f_0(z_0)+
tf_{j}(z_0)
- \frac12t^2{\nabla f_{j}(z_0)}^\top \prns{\nabla^2 f_{0}(z_0)}^{-1} {\nabla f_{j}(z_0)}
+o(t^2),
\label{eq:apxrisk}
\\
z_j(t)&=z_0-t\prns{\nabla^2 f_{0}(z_0)}^{-1} {\nabla f_{j}(z_0)}+o(t).
\label{eq:apxsol}
\end{align}
\end{theorem}

\cref{thm:secondorder} gives the second order expansion of the optimal value $v_j(t)$ and the first order expansion of \emph{any} choice of $z_j(t)$ that attains $v_j(t)$ \edit{around $t = 0$}.
\edit{These expansions quantify the impact on the optimal value and optimal solution when infinitesimally perturbing the objective function in region $R_0$ towards that in a subregion $R_j$.}
\edit{One crucial condition of \cref{thm:secondorder} is that the objective functions are sufficiently smooth (condition \ref{cond: unconst-smooth}).}
This condition holds for any subpartition if we assume that $\Eb{c(z;Y)\mid X}$ is almost surely twice continuously \edit{differentiable}, which is trivially satisfied if the cost function $c(z; Y)$ is a almost surely twice continuously differentiable function of $z$.
\edit{However,} even if $c(z; Y)$ is nonsmooth (\eg, \cref{ex: mnv,ex: portfolio}), $\Eb{c(z;Y)\mid X}$ may still be sufficiently smooth if the distribution of $Y \mid X$ is continuous (see examples in \cref{sec: approx-crit} below). In particular, one reason we defined the oracle splitting criterion using the population expectation rather than empirical averages is that for many relevant examples such as newsvendor and CVaR only the population objective may be smooth while the sample objective may be nonsmooth and therefore not amenable to perturbation analysis.

Condition \ref{cond: unconst-comp} ensures that if we only slightly perturb the objective function $f_0$, optimal solutions of the resulting perturbed problem are always bounded, and never escape to infinity. This means that without loss of generality we can restrict our attention to a compact subset of $\mathbb R^d$. This compactness condition and the smoothness condition (condition \ref{cond: unconst-smooth}) together ensure the existence of optimal solutions for any optimization problem corresponding to $t\in[0, t_0)$. In addition, this condition 
is crucial for ensuring $z(t) \to z_0$ as $t \to 0$ \citep[Proposition 4.4]{perturbation2000}. 
\edit{One sufficient condition for this is that any optimal solution $z^*(X)$ in \cref{eq:cso} is almost surely bounded, \eg, when conditional quantiles of all item demands in \cref{ex: mnv} are almost surely bounded.}
Finally, the regularity condition (condition \ref{cond: unconst-pd})
is obviously satisfied if $f_0(z)$ is strictly convex, which is implied if either $\Eb{c(z;Y)\mid X}$ or $c(z;Y)$ is almost surely strictly convex. Condition \ref{cond: unconst-pd} may be satisfied even if the cost function $c(z;Y)$ is not strictly convex: \eg, it holds for the newsvendor problem (\cref{ex: mnv}) when the density of $Y_l\mid X\in R_0$ is positive at $z_0$ for all $l=1,\dots,d$.

\subsection{Approximate Splitting Criteria}\label{sec: approx-crit}

\Cref{thm:secondorder} suggests two possible approximations of the oracle splitting criterion.

\subsubsection*{Approximate Risk Criterion.}
If we use \cref{eq:apxrisk} to extrapolate to $t=1$, ignoring the higher-order terms, we arrive at
$$
v_j(1)\approx
f_{j}(z_0)
-\frac12{\nabla f_{j}(z_0)}^\top \prns{\nabla^2 f_{0}(z_0)}^{-1} {\nabla f_{j}(z_0)}.
$$
\edit{Taking} a weighted average of this over $j=1,2$, we arrive at an approximation of the oracle splitting criterion $\crit^\text{oracle}$ in \cref{eq:oraclecrit}.
Since $p_1f_1(z_0)+p_2f_2(z_0)=p_0f_0(z_0)$ is constant in the subpartition, we may ignore these terms, leading to the following criterion:
\begin{equation}\label{eq:apxriskcrit}
\ts\crit^\text{apx-risk}(R_1,R_2)=-\frac{1}{2}\sum_{j=1,2}p_j{\nabla f_{j}(z_0)}^\top \prns{\nabla^2 f_{0}(z_0)}^{-1} {\nabla f_{j}(z_0)}.
\end{equation}

By strengthening the conditions in \cref{thm:secondorder}, we can in fact show that this approximation becomes arbitrarily accurate as the partition becomes finer.

\begin{theorem}\label{thm:apxriskapx}
Suppose the following conditions hold for both $j = 1, 2$:
\begin{enumerate} 
\item Condition \ref{cond: unconst-smooth} of \cref{thm:secondorder}.
\item Condition \ref{cond: unconst-comp} of \cref{thm:secondorder} holds for all $t \in [0, 1]$.
\item 
\edit{$f_0(z)+t\prns{f_j(z)-f_0\prns{z}}$} has a unique minimizer $z_0$ and \edit{$\nabla^2 \prns{f_0(z)+t\prns{f_j(z)-f_0\prns{z}}}$} is positive definite at this unique minimizer for all $t \in [0, 1]$.
\item $\Eb{c(z; Y) \mid X = x}$ is twice Lipschitz-continuously differentiable in $x$. \label{cond: thm-apxriskapx-lipschitz}
\end{enumerate}
Then
$$
\abs{\crit^\text{oracle}(R_1,R_2)-p_0f_0(z_0)-\crit^\text{apx-risk}(R_1,R_2)} = o(\mathcal D_0^2),
$$
where $\mathcal D_0=\sup_{x,x'\in R_0}\edit{\magd{x-x'}_2}$ is the diameter of $R_0$.
\end{theorem}
Again, note that $p_0f_0(z_0)$ is constant in the choice of subpartition.

\subsubsection*{Approximate Solution Criterion.}
Since $v_j(1)=f_j(z_j(1))$, we can also approximate $v_j(1)$ by approximating $z_j(1)$ and plugging it in. Using \cref{eq:apxsol} to extrapolate $z_j(t)$ to $t=1$ and ignoring the higher-order terms, we arrive at the following approximate criterion:
\begin{equation}\label{eq:apxsolcrit}
\ts\crit^\text{apx-soln}(R_1,R_2)=\sum_{j=1,2}p_jf_j\prns{z_0-\prns{\nabla^2 f_{0}(z_0)}^{-1} {\nabla f_{j}(z_0)}}.
\end{equation}
Notice this almost looks like applying a Newton update to $z_0$ in the $\min_z f_j(z)$ problem, \edit{\label{para: naive-newton}namely, the solution that optimizes the second order expansion of $f_j(z)$ at $z_0$. 
However, a naive Newton update will require to invert the Hessian for $f_j$, which varies across different candidate splits. 
In contrast, the criterion $\crit^\text{apx-soln}$ requires the Hessian for $f_0$, meaning we only have to invert a Hessian once for all candidate subpartitions.}

\edit{\label{para: grf2}For unconstrained CSO problems in this section, we may also apply the GenRandForest algorithm in \cite{athey2019generalized} to solve their first order optimality condition. The GenRandForest algorithm uses a similar way to approximate optimal solutions in split subregions. It chooses splits to maximize the difference between approximate solutions in two subregions induced by each candidate split, as their proposition 1 shows that this approximately minimizes the total mean squared errors of the resulting estimated optimal solutions. In contrast, by using $\crit^\text{apx-soln}$, we choose splits to minimize the expected cost of the approximate optimal solutions, thereby directly targeting the ultimate objective in CSO problems. More importantly, we tackle the constrained case (\cref{sec: constr}) while the GenRandForest algorithm cannot. In \cref{sec: empirical,sec: empirical nv}, we show the impact of both of these differences can be significant in practice when optimization is the aim.} 

\edit{In the following theorem, we show that the approximate solution criterion also becomes arbitrarily accurate as the partition becomes finer.}
\begin{theorem}\label{thm:apxsolapx}
Suppose the assumptions of \cref{thm:apxriskapx} hold.
Then
$$
\abs{\crit^\text{oracle}(R_1,R_2)-\crit^\text{apx-soln}(R_1,R_2)} =  o(\mathcal D_0^2).
$$
\end{theorem}

\subsubsection*{Revisiting the Running Examples.} The approximate criteria above crucially depend on the gradients $\nabla f_1(z_0), \nabla f_2(z_0)$ and Hessian $\nabla^2 f_0(z_0)$. We next study these quantities for some examples. 

\begin{example}[Derivatives with Smooth cost]\label{ex: smooth}
If $c(z;y)$ is itself twice continuously  differentiable for every $y$, then under regularity conditions that enable the exchange of derivative and expectation (\eg, $\abs{(\nabla c\prns{z;Y})_\ell}\leq W$ for all $z$ in a neighborhood of $z_0$ with $\Eb{W\mid{X_i\in R_j}}<\infty$), we have  $\nabla f_j(z_0)=\Eb{\nabla c\prns{z_0;Y}\mid{X_i\in R_j}}$ and $\nabla^2 f_0(z_0) = \Eb{\nabla^2 c\prns{z_0;Y}\mid{X_i\in R_0}}$. 
\end{example}

\begin{continuance}[Derivatives in Multi-Item Newsvendor]{\ref{ex: mnv}}
In many cases, $c(z;y)$ is not smooth, as in the example of the multi-item newsvendor cost in \cref{eq:newsvendorcost}. In this case, it suffices that the distribution of $Y\mid X\in R_j$ is continuous for gradients and Hessians to exist. Then,
we can show that $(\nabla f_j(z_0))_l=(\alpha_l +\beta_l)\Prb{Y_l\leq z_{0,l}\mid X\in R_j}-\beta_l$ and $(\nabla^2 f_{0}(z_0))_{ll}=(\alpha_l +\beta_l)\mu_{0,l}(z_0)$ for $l=1,\dots,d,\,j=1,2$, and $(\nabla^2 f_{0}(z_0))_{ll'}=0$ for $l \neq l'$, where $\mu_{0,l}$ is the density function of $Y_l\mid X\in R_0$. So, $\nabla^2 f_0(z_0)$ is invertible as long as $\mu_{0, l}(z_0) > 0$ for $l = 1, \dots, d$.
\end{continuance}

\begin{continuance}[Derivatives in Variance-based Portfolio Optimization]{\ref{ex: portfolio-var}}
The cost function in \cref{eq:portfolio-var} is an instance of \cref{ex: smooth} (smooth costs). 
Using block notation to separate the first $d$ decision variables from the final single auxiliary variable,
we  verify in \cref{prop: gradient-var-port} that  
\begin{align}
\nabla f_j(z_0) &= 
2\begin{bmatrix}
\Eb{YY^\top \mid {X \in R_j}}z_{0, 1:d} - \Eb{Y \mid X \in R_j}z_{0, d+1} \\ 
z^\top_{0, 1:d}\prns{\Eb{Y \mid X \in R_0} - \Eb{Y \mid X \in R_j}}
\end{bmatrix}\label{eq: grad-var}, \\
\nabla^2 f_0(z_0) &=
2\begin{bmatrix}
\Eb{YY^\top \mid {X \in R_0}} & -\Eb{Y \mid X \in R_0} \\
-\Eb{Y^\top \mid X \in R_0} & 1
\end{bmatrix}. \label{eq: hessian-var}
\end{align}
Notice $\nabla^2 f_0(z_0)$ is invertible if and only if \edit{the covariance matrix} $\var\prns{Y \mid X \in R_0}$ is invertible. 
\end{continuance}

\begin{continuance}[Derivatives in CVaR-based Portfolio Optimization]{\ref{ex: portfolio}}\label{ex: cvar-continue}
Like the newsvendor cost in \cref{eq:newsvendorcost}, the CVaR  cost in \cref{eq:portfolio} is not smooth either. Again we assume that the distribution of $Y\mid X\in R_j$ is continuous. \edit{Then, when $z_0 \ne 0$} (\cref{prop: gradient-portfolio} in \cref{sec: supplement}),
\begin{align}
&\nabla f_j(z_0) = 
\frac{1}{\alpha}
\begin{bmatrix}
-\Eb{Y\indic{\overline Y_0 \le q^{\alpha}_0(\overline Y_0)} \mid X \in R_j} \\
 \Prb{q^{\alpha}_0(\overline Y_0) - \overline Y_0 \ge 0 \mid X \in R_j} - \alpha
\end{bmatrix}, ~~~ \overline Y_0 \coloneqq Y^\top z_{0, 1:d}
\label{eq: portfolio-grad}  \\
&\nabla^2 f_0(z_0) = 
\frac{\mu_{0}\prns{q^{\alpha}_0(\overline Y_0)}}{\alpha}
\begin{bmatrix}
\Eb{YY^\top \mid \overline Y_0 = q^{\alpha}_0(\overline Y_0), X \in R_0} &  -\Eb{Y \mid \overline Y_0 = q^{\alpha}_0(\overline Y_0), X \in R_0}  \\
-\Eb{Y^\top \mid \overline Y_0 = q^{\alpha}_0(\overline Y_0), X \in R_0} & 1
\end{bmatrix}.
\label{eq: portfolio-hess}   
\end{align} 
where $\mu_{0}$ is the density function of $\overline Y_0$ given $X \in R_0$, and $q^{\alpha}_0(\overline Y_0)$ as the $\alpha$-level quantile of $\overline Y_0$ given $X \in R_0$. Notice that the Hessian matrix $\nabla^2 f_0(z_0)$ may not necessarily be invertible. This arises due to the homogeneity of returns in scaling the portfolio, so that second derivatives in this direction may vanish. This issue is corrected when we consider the constrained case where we fix the scale of the portfolio (see \cref{sec: constr}).\footnote{Indeed the unconstrained case for the portfolio problem is in fact uninteresting: the zero portfolio gives minimal variance, and CVaR may be sent to infinity in either direction by infinite scaling.}
\end{continuance}

\subsubsection*{Re-optimizing auxiliary variables.}
In \cref{ex: portfolio-var,ex: portfolio}, $z$ contains both auxiliary variables and decision variables, and we  construct the approximate criteria based on gradients and Hessian matrix with respect to both sets of variables. 
A natural alternative is to re-optimize the auxiliary variables first so that the objective only depends on decision variables, and then  evaluate the corresponding gradients and Hessian matrix. 
That is, if we partition $z=(z^{\text{dec}},z^{\text{aux}})$, then we can re-define $f_j(z^{\text{dec}})=\min_{z^\text{aux}}\Eb{c((z^{\text{dec}},z^{\text{aux}});Y)\mid{X\in R_j}}$ and $z^{\text{dec}}_0 = \argmin_{z^{\text{dec}}} f_0(z^{\text{dec}})$. 
The perturbation analysis remains largely the same by simply using the gradients and Hessian matrix for the redefined $f_j(z^{\text{dec}})$ at $z^{\text{dec}}_0$.
This leads to an alternative approximate splitting criterion. 
However, evaluating the gradients $\nabla f_j(z^{\text{dec}}_0)$ for $j = 1, 2$  
would now involve repeatedly finding the optimal solution $\argmin_{z^\text{aux}}\Eb{c((z^{\text{dec}}_0,z^{\text{aux}});Y)\mid{X\in R_j}}$ for all candidate splits. 
See \cref{sec: auxiliary-var} for details.

Since the point of our approximate criteria is to avoid re-optimization for every candidate split, this alternative is practically relevant only when re-optimizing the auxiliary variables is very computationally easy. For example, in \cref{ex: portfolio-var}, 
$z^\text{dec}$ corresponds to the first $d$ variables and
re-optimizing the auxiliary $(d+1)\thh$ variable amounts to computing the mean of $Y^\top z^{\text{dec}}_0$ in each subregion, which can be done quite efficiently as we vary the candidate splits.

\subsection{Estimating the Approximate Splitting Criteria}\label{sec: est-approx-crit}

The benefit of our approximate splitting criteria, $\crit^\text{apx-soln}(R_1,R_2),\,\crit^\text{apx-risk}(R_1,R_2)$ in \cref{eq:apxriskcrit,eq:apxsolcrit}, is that they only involve the solution of $z_0$. Thus, if we want to evaluate many different subpartitions of $R_0$, we need only solve for $z_0$ once, compute $(\nabla^2 f_{0}(z_0))^{-1}$ once, and then only re-compute $\nabla f_{j}(z_0)$ for each new subpartition. Still, this involves quantities we do not actually know since all of these depend on the joint distribution of $(X,Y)$. We therefore next consider the estimation of these approximate splitting criteria from data.

Given estimators $\hat H_0,\,\hat h_1,\,\hat h_0$ of $\nabla^2 f_{0}(z_0),\,\nabla f_1(z_0),\,\nabla f_2(z_0)$, respectively (see examples below), we can construct the estimated approximate splitting criteria as
\begin{align}
\ts\hat{\crit}^\text{apx-risk}(R_1,R_2)&\ts= -\sum_{j=1,2}\frac{n_j}{n}\hat h_j^\top \hat H^{-1}_0 \hat h_j,\label{eq: est-approx-risk-unconstr}\\
\ts\hat{\crit}^\text{apx-soln}(R_1,R_2)&\ts=\sum_{j=1,2}\frac1n\sum_{i=1}^n\indic{X_i\in R_j}c\prns{\hat z_0-\hat H^{-1}_0 \hat h_j;\;Y_i},\label{eq: est-approx-sol-unconstr}
\end{align}
where $n_j=\sum_{i=1}^n\indic{X_i\in R_j}$.

Under appropriate convergence of $\hat H_0,\,\hat h_1,\,\hat h_2$, these estimated criteria respectively converge to the population approximate criteria ${\crit}^\text{apx-risk}(R_1,R_2)$ and ${\crit}^\text{apx-soln}(R_1,R_2)$ in \cref{sec: approx-crit}, as summarized by the following self-evident proposition.
\begin{proposition}\label{thm:critconverge}
If 
$\edit{\|\hat H^{-1}_0 - \prns{\nabla^2 f_{0}(z_0)}^{-1}\|_{\op{F}}} = o_p(1)$, $\edit{\|\hat h_j - \nabla f_j(z_0)\|_2} = O_p(n^{-1/2})$ for $j = 1, 2$, then
\begin{align*}
\ts\hat{\crit}^\text{apx-risk}(R_1,R_2) \ts= {\crit}^\text{apx-risk}(R_1,R_2) + O_p(n^{-1/2}).
\end{align*}
If also $\left|\frac{1}{n}\sum_{i = 1}^n \indic{X_i \in R_j}c\prns{\hat z_0-\hat H^{-1}_0 \hat h_j;Y_i} - p_jf_j\prns{z_0-\prns{\nabla^2 f_{0}(z_0)}^{-1} {\nabla f_{j}(z_0)}}\right| = O_p(n^{-1/2})$ for $j = 1, 2$, then 
\begin{align*}
\ts\hat{\crit}^\text{apx-soln}(R_1,R_2) = {\crit}^\text{apx-soln}(R_1,R_2) + O_p(n^{-1/2}).
\end{align*}
\end{proposition}
If we can find estimators that satisfy the conditions of \cref{thm:critconverge}, then together with \cref{thm:apxriskapx,thm:apxsolapx}, we will have shown that the estimated approximate splitting criteria can well approximate the oracle splitting criterion when samples are large and the partition is fine. 
It remains to find appropriate estimators.

\subsubsection*{General Estimation Strategy.} 
Since the gradients and Hessian to be estimated are evaluated at a point $z_0$ that is itself unknown, a general strategy is to first estimate $z_0$ and then estimate the gradients and Hessian at this estimate. This is the strategy we follow in the examples below.

Specifically,
we can first estimate $z_0$ by its sample analogue:
\begin{align}
\hat z_0 \in \argmin_{z \in \mathbb \Z} 
\widehat{p_0f_0}(z), \text{ where } \widehat{p_0f_0}(z) \coloneqq  
\frac{1}{n}\sum_{i = 1}^{n}\indic{X_i \in R_0}c(z; Y_i). \label{eq: z0hat}
\end{align}
Under standard regularity conditions, the estimated optimal solution $\hat z_0$ above is consistent (see \cref{lemma: consistency-est-sol} in \cref{sec: supplement}).
Then, given generic estimators $\hat H_0(z)$ of $\nabla^2 f_0(z)$ at any one $z$ and similarly estimators $\hat h_j(z)$ of $\nabla f_j(z)$ for $j=1,2$, we let $\hat H_0=\hat H_0(\hat z_0)$ and $\hat h_j=\hat h_j(\hat z_0)$. Examples of this follow.

\subsubsection*{Revisiting the Running Examples.} We next discuss examples of possible estimates $\hat H_0,\,\hat h_j$ that can be proved to satisfy the conditions in \cref{thm:critconverge} (see \cref{prop: ex-portfolio,prop: ex-smooth,prop: ex-mnv,prop: ex-portfolio-var} in \cref{sec: supplement}
for details). All of our examples use the general estimation strategy above.

\begin{continuance}[Estimation with Smooth Cost]{\ref{ex: smooth}}
If $c(z;y)$ is itself twice continuously differentiable in $z$ for every $y$, we can simply use $\hat H_0(\hat z_0)=\frac1{n_0}\sum_{i=1}^n\indic{X_i\in R_0}\nabla^2 c\prns{\hat z_0;Y_i}$ and $\hat h_j(\hat z_0)={\frac1{n_j}\sum_{i=1}^n\indic{X_i\in R_j}\nabla c\prns{\hat z_0;Y_i}}$. In \cref{prop: ex-smooth} in \cref{sec: supplement}, we show that these satisfy the conditions of \cref{thm:critconverge} thanks to the smoothness of $c(z; y)$. 
\Cref{ex: portfolio-var} is one example of this case, which we discuss below.
\Cref{ex: prediction} is another example.

In particular for the squared error cost function in \cref{ex: prediction} ($c(z;y)=\frac{1}{2}\magd{z-y}^2_2$), we show in \cref{prop: square-cost} in \cref{sec: supplement} that, using the above $\hat H_0,\hat h_j$, we have
$$
\frac1{n}\sum_{i=1}^n\indic{X_i\in R_0}c(\hat z_0;Y_i)+\hat{\crit}^\text{apx-risk}(R_1,R_2)=\hat{\crit}^\text{apx-soln}(R_1,R_2)=\sum_{j=1,2}\frac{n_j}{2n}\sum_{l = 1}^d \op{Var}(\{Y_{i, l}:X_i\in R_j\}),
$$
which is exactly the splitting criterion used for regression by random forests, namely the sum of squared errors to the mean within each subregion. Notice the very first term is constant in $R_1,R_2$.
\end{continuance}

\begin{continuance}[Estimation in Multi-Item Newsvendor]{\ref{ex: mnv}}
In the previous section we saw that the gradient and Hessian depend on the cumulative distribution and density functions, respectively.
We can therefore estimate the gradients by ${\hat h_{j,\ell}(\hat z_0)}={\frac{\alpha_l + \beta_l}{n_j}\sum_{i=1}^n\indic{X_i\in R_j,\,Y_l\leq \hat z_{0, l}}} - \beta_l$,
and the Hessian using, for example, kernel density estimation: ${\hat H_{0,ll}(\hat z_0)}=\frac{\alpha_l+\beta_l}{n_j b}\sum_{i=1}^n\indic{X_i\in R_0}\mathcal K((Y_{i, l}-\hat z_{0, l})/b)$,
where $\mathcal K$ is a kernel such as $\mathcal K(u)=\indic{\abs{u}\leq\frac12}$ and $b$ is the bandwidth, and $\hat H_{0,ll'}(\hat z_0)=0$ for $l\neq l'$. 
We show the validity of these estimates in \cref{prop: ex-mnv} in \cref{sec: supplement}.
\end{continuance}

\begin{continuance}[Estimation in Variance-based Portfolio Optimization]{\ref{ex: portfolio-var}} 
With $\hat z_0 = \{\hat z_{0, 1}, \dots, \hat z_{0, d}, \hat z_{0, d+1}\}$ given by solving the problem \cref{eq: z0hat}, 
the gradient and Hessian in \cref{eq: grad-var,eq: hessian-var} can be estimated by their sample analogues: 
\begin{align*}
&\hat h_j(\hat z_0) = 
2\begin{bmatrix}
\frac{1}{n_j}\sum_{i = 1}^n \indic{X_i \in R_j} Y_iY_i^\top \hat z_{0, 1:d} - \frac{1}{n_j} \sum_{i = 1}^n \indic{X_i \in R_j} Y_i \hat z_{0, d+1}  \\ 
\hat z^\top_{0, 1:d}\prns{\frac{1}{n_0}\sum_{i = 1}^n \indic{X_i \in R_0} Y_i - \frac{1}{n_j}\sum_{i = 1}^n \indic{X_i \in R_j} Y_i}
\end{bmatrix},  \\
&\hat H_0(\hat z_0) = 
2\begin{bmatrix}
\frac{1}{n_0}\sum_{i = 1}^n \indic{X_i \in R_0}Y_iY_i^\top & -\frac{1}{n_0}\sum_{i = 1}^n \indic{X_i \in R_0}Y_i \\
-\frac{1}{n_0}\sum_{i = 1}^n \indic{X_i \in R_0}Y_i^\top & 1
\end{bmatrix}.
\end{align*}
These estimators are in fact specific examples of the general smooth case in \cref{ex: smooth}, so they too can be analyzed by \cref{prop: ex-smooth} in \cref{sec: supplement}. 
\end{continuance}

\begin{continuance}[Estimation in CVaR-based Portfolio Optimization]{\ref{ex: portfolio}}
It is straightforward to estimate the gradient in \cref{eq: portfolio-grad}: 
\begin{align}\label{eq: port-grad-est}
\hat h_j(\hat z_0) = 
\frac{1}{\alpha}\begin{bmatrix}
-\frac{1}{n_j}\sum_{i = 1}^n \indic{Y_i^\top \hat z_{0,1:d} \le \hat q^{\alpha}_0(Y^\top \hat z_{0,1:d}), X_i \in R_j}Y_i \\
\frac{1}{n_j}\sum_{i = 1}^n \indic{Y^\top_i \hat z_{0, 1:d} \le \hat q^{\alpha}_0(Y^\top \hat z_{0,1:d}), X_i \in R_j} - \alpha
\end{bmatrix}
\end{align}
where $\hat q^{\alpha}_0(Y^\top \hat z_{0,1:d})$ is the empirical $\alpha$-level quantile of $Y^\top \hat z_{0, 1:d}$ based on data in $R_0$.
The Hessian matrix in \cref{eq: portfolio-hess} is more challenging to estimate, since it involves many conditional expectations given the event $Y^\top z_{0,1:d}=q^{\alpha}_0(Y^\top z_{0,1:d})$. In principle, we could estimate these nonparametrically by, for example, kernel smoothing estimators \citep[\eg, ][]{Fan1998,Yin2010,Loubes2019,chen2015local}. 
For simplicity and since this is only used as an approximate splitting criterion anyway, in our empirics we can consider a parametric approach instead, which we will use in our empirics in \cref{sec: cvar-empirical}:  if $Y \mid X \in R_0$ has a Gaussian distribution $\mathcal N(m_0, \Sigma_0)$, then
\begin{align}
&\Eb{Y \mid Y^\top z_{0,1:d} = q^{\alpha}_0(Y^\top z_{0,1:d}), X \in R_0} = m_0 + \Sigma_0 z_{0, 1:d} \prns{z_{0,1:d}^\top \Sigma_0 z_{0,1:d}}^{-1}(q^{\alpha}_0(Y^\top z_{0,1:d}) - m_0^\top z_{0,1:d}), \label{eq: hessian-normal1}\\
&\var\prns{Y \mid Y^\top z_{0,1:d} = q^{\alpha}_0(Y^\top z_{0,1:d}), X \in R_0} = \Sigma_0 - \Sigma_0 z_{0, 1:d} \prns{z_{0,1:d}^\top \Sigma_0 z_{0,1:d}}^{-1}z_{0,1:d}^\top \Sigma_0,\label{eq: hessian-normal2}
\end{align}
and $\Eb{YY^\top \mid Y^\top z_{0,1:d} = q^{\alpha}_0(Y^\top z_{0,1:d}), X \in R_0}$ can be directly derived from these two quantities. 
We can then estimate these quantities by plugging in $\hat z_0$ for $z_0$, the empirical mean estimator of $Y$ for $m_0$, the empirical variance estimator of $Y$ for $\Sigma_0$, and the empirical $\alpha$-level quantile of $Y^\top {\hat z_{0,1:d}}$ for $q^{\alpha}_0(Y^\top z_{0,1:d})$, all based only on the data in $R_0$.
Finally, we can estimate 
$\mu_{0}\prns{q^{\alpha}_0(Y^\top z_{0,1:d})}$ by a kernel density estimator $\frac{1}{n_0b}\sum_{i = 1}^n \indic{X_i \in R_0}\mathcal K\prns{\prns{Y_i^\top\hat z_{0,1:d} - \hat q^{\alpha}_0(Y^\top \hat z_{0,1:d})}/b}$.
Although the Gaussian distribution may be misspecified, the resulting estimator is more stable than and easier to implement than nonparametric estimators (especially considering that it will be used repeatedly in tree construction) and it can still approximate the relative scale of entries in the Hessian matrix reasonably well. 
In \cref{sec: cvar-empirical}, we empirically show that our method  based on these approximate estimates works well even if the Gaussian model is misspecified. If it happens to be correctly specified, we can also theoretically validate that the estimator satisfies the conditions of \cref{thm:critconverge}  (see \cref{prop: ex-portfolio} in \cref{sec: supplement}).
\end{continuance}

\subsection{The Stochastic Optimization Tree and Forest Algorithms}\label{sec: tree-procedure}

With the estimated approximate splitting criteria in hand, we can now describe our StochOptTree and StochOptForest algorithms. Specifically, we will first describe how we use our estimate approximate splitting criteria to build trees, which we will then combine to make a forest that leads to a CSO decision policy $\hat z(x)$ as in \cref{eq: forest policy}.

\begin{algorithm}[t!]\OneAndAHalfSpacedXI
    \caption{\textsc{Recursive procedure to grow a StochOptTree (unconstrained case)}}
    \label{alg: tree}
    \begin{algorithmic}[1]
    \Procedure{StochOptTree.Fit}{region $R_0$, data $\mathcal D$, $\mathtt{depth}$, $\mathtt{id}$}
    \State $\hat z_0 \gets$ \Call{Minimize}{$\sum_{(X_i,Y_i) \in \mathcal D}\indic{X_i \in R_0}c(z; Y_i)$, $z\in \Z$} \Comment{Solve \cref{eq: z0hat}}\label{alg: tree z0 step}
    \State $\hat H_0 \gets$ \textproc{Estimate} $\nabla^2 f_0(z)$ \textproc{at} $z=\hat z_0$\label{alg: tree estim 0 step}
    \State $\mathtt{CandSplit} \gets$ \Call{GenerateCandidateSplits}{${R_0}$, $\mathcal D$} \Comment{Create the set of possible splits}
    \State $\hat C\gets\infty$
    \For{$(j, \theta) \in \mathtt{CandSplit}$} \Comment{Optimize the estimated approximate criterion}
        \State $(R_1,\,R_2)\gets (R_0\cap\{x\in\R p:x_j\leq\theta\},\,R_0\cap\{x\in\R p:x_j>\theta\})$
        \State $(\hat h_1,\hat h_2) \gets$ \textproc{Estimate} $\nabla f_1(z),\,\nabla f_2(z)$ \textproc{at} $z=\hat z_0$\label{alg: tree estim j step}
        \State $C\gets\hat{\mathcal C}^{\text{apx-risk/apx-soln}}(R_1,\,R_2)$\Comment{Compute the criterion using $\hat H_0,\hat h_1,\hat h_2,\mathcal D$}\label{alg: tree crit step}
        \If{$C<\hat C$}\;$(\hat C,\hat j,\hat\theta)\gets(C,j,\theta)$
        \EndIf
    \EndFor
    \If{\Call{Stop?}{$(\hat j, \hat \theta)$, $R_0$, $\mathcal D$, $\mathtt{depth}$}} 
        \State \textbf{return} $(x\;\mapsto\;\mathtt{id})$
    \Else
    	\State $\mathtt{LeftSubtree}~\gets$ \Call{StochOptTree.Fit}{$R_0\cap\{x\in\R p:x_j\leq\theta\}$, $\mathcal D$, $\mathtt{depth}+1$, $2\mathtt{id}$} 
        \State $\mathtt{RightSubtree}~\gets$ \Call{StochOptTree.Fit}{$R_0\cap\{x\in\R p:x_j>\theta\}$, $\mathcal D$, $\mathtt{depth}+1$, $2\mathtt{id}+1$} 
        \State \textbf{return} $(x\;\mapsto\;{x_j\leq\theta}{\;}{?}{\;}{\mathtt{LeftSubtree}(x)}{\;}{:}{\;}{\mathtt{RightSubtree}(x)})$
    \EndIf
    \EndProcedure 
    \end{algorithmic}
\end{algorithm}

\subsubsection*{StochOptTree Algorithm.}
We summarize the tree construction procedure in \cref{alg: tree}. 
We will extend \cref{alg: tree} to the constrained case in \cref{sec: constr}.
This procedure partitions a generic region, $R_0$, into two children subregions, $R_1,R_2$, by an axis-aligned cut along a certain coordinate of covariates. 
It starts with solving the optimization problem within $R_0$ according to \cref{eq: z0hat}, and then finds the best split coordinate $\hat j$ and cutoff value $\hat \theta$ over a set of candidate splits by minimizing\footnote{Ties can be broken arbitrarily.} the estimated approximate risk criterion in \cref{eq: est-approx-risk-unconstr} or the estimated approximate solution criterion in \cref{eq: est-approx-sol-unconstr}.
\edit{Once the best split $(\hat j, \hat \theta)$ is found, $R_0$ is partitioned into the two subregions accordingly, and the whole procedure continues on recursively until a stopping criterion.}

\edit{%
There are a few subroutines to be specified. First, there is the optimization of $\hat z_0$. Depending on the structure of the problem, different algorithms may be appropriate. For example, 
if $c(z;y)$ is the maximum of several linear functions, a linear programming solver may be used. 
More generally, 
since the objective has the form of a sum of functions, methods such as stochastic gradient descent \citep[aka stochastic approximation;][]{nemirovski2009robust} may be used. 
Second, there is the estimation of $\hat H_0,\hat h_1,\hat h_2$, which was discussed in \cref{sec: est-approx-crit}.
Third, we need to generate a set of candidate splits, which can be done in different ways.
The original RandForest algorithm \citep{breiman2001random} randomly selects a pre-specified number of distinct coordinates $j$ from $\{1,\dots,p\}$ without replacement, and considers $\theta$ to be 
all 
midpoints in the $X_{i,j}$ data, which exhausts all possible subpartitions along each selected coordinate.
Another option is to consider a random subset of cutoff values, possibly enforcing that the sample sizes of the corresponding two children nodes are balanced, as in \citet{pmlr-v32-denil14}. 
This approach not only enforces balanced splits, which is important for statistical guarantees (see \cref{thm: asymp-opt}),
but it also reduces the computation time.
Finally, we need to decide when to stop the tree construction. A typical stopping criterion is when each child region reaches a pre-specified number of data points \citep[\eg,][]{breiman2001random}. 
Depth may also additionally be restricted. 
Note that in an actual implementation if the stopping criterion would have stopped regardless of the split chosen, we can short circuit the call and skip the split optimization.
}

Notice that $\hat z_0$ and $\hat H_0$ need only be computed once for each recursive call to \textproc{StochOptTree.Fit}, while $\hat h_1,\hat h_2$ need to be computed for each candidate split. All estimators $\hat h_j$ discussed in \cref{sec: est-approx-crit} take the form of a sample average over $i\in R_j$, for $j=1,2$, and therefore can be easily and quickly computed for each candidate split. Moreover, when candidate cutoff values consist of all midpoints of the sample values in the $j\thh$ coordinate, such sample averages can be efficiently updated by proceeding in sorted order, where only one datapoint changes from one side of the split to the other at a time, similarly to how the original random forest algorithm maintains within-subpartition averages of outcomes and their squares for each candidate split.

Notably, the tree construction computation is typically dominated by the step of searching best splits. 
This step can be implemented very efficiently with our approximate criteria, since they only involve estimation of gradients and simple linear algebra operations (\cref{sec: approx-crit}). Only one optimization and Hessian computation is needed at the beginning of each recursive call.
In particular, we do not  need to solve optimization problems repeatedly for each candidate split, which is the central aspect of our approach and which enables the construction of large-scale forests.

\begin{algorithm}[t!]\OneAndAHalfSpacedXI
    \caption{\textsc{Procedure to fit a StochOptForest}}
    \label{alg: forest}
    \begin{algorithmic}[1]
    \Procedure{StochOptForest.Fit}{data $\mathcal{D}$, number of trees $T$} 
    \For{$j = 1$ to $T$}
        \State $\mathcal I^\text{tree}_j,~ \mathcal I^\text{dec}_j \gets $ \Call{Subsample}{$\{1,\dots,\abs{\mathcal D}\}$}
        \State $\tau_j \gets$ \Call{StochOptTree.Fit}{$\mathcal X$, $\{(X_i,Y_i)\in\mathcal D:i\in\mathcal I^\text{tree}_j\}$, $1$, $1$} \Comment Fit tree using the sub-dataset $\mathcal I^\text{tree}_j$ 
    \EndFor
    \State \textbf{return} $\{(\tau_j,\,\mathcal I^\text{dec}_j):j=1,\dots,T\}$
    \EndProcedure 
    \end{algorithmic}
\end{algorithm}

\begin{algorithm}[t!]\OneAndAHalfSpacedXI
    \caption{\textsc{Procedure to make a decision using StochOptForest}}
    \label{alg: forest pred}
    \begin{algorithmic}[1]
    \Procedure{StochOptForest.Decide}{data $\mathcal{D}$, forest $\{(\tau_j,\,\mathcal I^\text{dec}_j):j=1,\dots,T\}$, target $x$}
    \State $w(x) \gets$ \Call{Zeros}{$|\mathcal D|$} \Comment{Create an all-zero vector of length $|\mathcal D|$}
    \For{$j = 1,\dots,T$}
    \State $\mathcal N(x)\gets\{i\in\mathcal I^\text{dec}_j:\tau_j(X_i)=\tau_j(x)\}$\Comment{Find the $\tau_j$-neighbors of $x$ among the data in $\mathcal I^\text{dec}_j$}
    \For{$i \in \mathcal N(x)$}\;$w_i(x) \gets w_i(x) + \frac{1}{\abs{\mathcal N(x)}T}$\Comment{Update the sample weights}
        \EndFor
    \EndFor
    \State \textbf{return} \Call{Minimize}{$\sum_{(X_i,Y_i) \in \mathcal D}w_i(x)c(z; Y_i)$, $z\in {\Z}$}\Comment{Compute the forest policy \cref{eq: forest policy}}
    \EndProcedure 
    \end{algorithmic}
\end{algorithm}

\subsubsection*{StochOptForest Algorithm.} In \cref{alg: forest}, we  summarize the algorithm of building forests using trees constructed by \cref{alg: tree}. It involves an unspecified subsampling subroutine. For each $j=1,\dots,T$, we consider possibly subsampling the data on which we will fit the $j\thh$ tree ($\mathcal I^\text{tree}$) as well as the data which we will later use to generate localized weights for decision-making ($\mathcal I^\text{dec}$). There are different possible ways to construct these subsamples. Following the original random forest algorithm, we may set $\mathcal I^\text{tree}_j=\mathcal I^\text{dec}_j$ equal to a bootstrap sample (a sample of size $n$ with replacement).
Alternatively, we may set $\mathcal I^\text{tree}_j=\mathcal I^\text{dec}_j$ to be sampled as a fraction of $n$ \emph{without} replacement,
which is an approach adopted in more recent random forest literature as it is more amenable to theoretical analysis and has similar empirical performance \citep[eg.,][]{mentch2016quantifying,scornet2015}.
Alternatively, we may also \emph{sequentially} sample $\mathcal I^\text{tree}_j$, $\mathcal I^\text{dec}_j$ without replacement so the two are disjoint (\eg, take a random half of the data, then further split it at random into two).
The property that the two sets are disjoint, $\mathcal I^\text{tree}_j\cap\mathcal I^\text{dec}_j=\varnothing$, is known as \emph{honesty} and it is helpful in proving statistical consistency of random forests \citep{athey2019generalized,wager2018estimation,pmlr-v32-denil14}.\footnote{We may similarly use the $\approx1/e$ fraction of the data not selected by the bootstrap sample to construct $\mathcal I^\text{dec}_j$ to achieve honesty, but this is again uncommon as it is difficult to analyze.}

\subsubsection*{Final Decision.} 
In \cref{alg: forest pred}, we  summarize the algorithm of making a decision at new query points $x$ once we have fit a forest, that is, compute the forest policy, \cref{eq: forest policy}.
Although the tree algorithm we developed so far, \cref{alg: tree}, is for the unconstrained case, we present \cref{alg: forest pred} in the general constrained case.
In a slight generalization of \cref{eq: forest policy}, we actually allow the data weighted by each tree to be a subset of the whole dataset (\ie, $\mathcal I^\text{dec}_j$), as described above.
Namely, the weights $w_i(x)$ computed by \cref{alg: forest pred} are given by
\begin{align}\label{eq: weights}
w_{i}(x) = \frac{1}{T}\sum_{j=1}^T\frac{\indic{i\in\mathcal I^\text{dec}_j,\tau_j(X_i)=\tau_j(x)}}{\sum_{i'=1}^n\indic{i\in\mathcal I^\text{dec}_j,\tau_j(X_{i'})=\tau_j(x)}},
\end{align}
which is slightly more general than \cref{eq: forest policy}.
\Cref{alg: forest pred} then optimizes the average cost over the data with sample weights given by $w_i(x)$.
Note that under honest splitting, for each single tree, each data point is used in either placing splits or constructing weights, but not both. However, since each tree uses an independent random subsample, 
every data point 
will participate in the construction of some trees and also the computation of weights by other trees.
Therefore, all observations contribute to both forest construction and the weights in the final decision making. In this sense, despite appearances, honest splitting is not ``wasting'' data. 

The weights $\{w_i(x)\}_{i = 1}^n$ generated by \cref{alg: forest pred} represent the average frequency with which each data point falls into the same terminal node as $x$.
The measure given by the sum over $i$ of $w_i(x)$ times the Dirac measure at $Y_i$ can be understood as an estimate for the conditional distribution of $Y\mid X=x$.
However, in contrast to non-adaptive weights such as given by $k$-nearest neighbors or Nadaraya–Watson kernel regression \citep{bertsimas2014predictive}, which non-parametrically estimate this conditional distributional generically, our weights \emph{directly} target the optimization problem of interest, focusing on the aspect of the data that is relevant to the optimization problem, which makes our weights much more efficient.
Moreover, in contrast to using weights given by standard random forests, which targets prediction with minimal squared error, our weights target the \emph{right} downstream optimization problem.

\section{The Constrained Case}\label{sec: constr}
In this section, we develop approximate splitting criteria for training forests for general CSO problems with constraints as described at the onset in \cref{eq:cso}.
Namely, in this section we let $\Z =\braces{z\in\R d:
h_{k}(z) = 0,\, k =1, \dots, s, ~ h_{k}(z) \le 0,\, k =s+1, \dots, m}$ be as in \cref{eq:constraints}.
The oracle criterion we target remains $\crit^\text{oracle}(R_1,R_2)$ as in \cref{eq:oraclecrit} with the crucial difference that now $\Z$ need not be $\R d$ and may be constrained as above.
We then proceed as in \cref{sec: unconstr}: we approximate the oracle criterion in two ways, then we estimate the approximations, and then we use these estimated splitting criteria to construct trees.
\edit{Since the perturbation analysis in the presence of constraints is somewhat more cumbersome, this section will be more technical. But the high level idea remains the same as the simpler unconstrained case in \cref{sec: unconstr}.}

\subsection{Perturbation Analysis of the Oracle Splitting Criterion}
Again, consider a region $R_0\subseteq\R d$ and its candidate subpartition $R_0=R_1\cup R_2$, $R_1\cap R_2=\varnothing$. We define $v_j(t),\,z_j(t),\,f_j(t)$ as in \cref{eq:v} with the crucial difference that now $\Z$ is constrained.
The oracle criterion is given by $\crit^\text{oracle}(R_1,R_2)=p_1v_1(1)+p_2v_2(1)$, as before. We again approximate $v_1(1), v_2(1)$ by computing $v_1(t), v_2(t)$ at $t=0$ (where they are equal and do not depend on the subpartition) and then extrapolating from there by leveraging second order perturbation analysis. We present our key perturbation result for this below.

\begin{theorem}[Second-Order Perturbation Analysis: Constrained]\label{thm:secondorder-const}
Fix $j=1,2$.
Suppose the following conditions hold:
\begin{enumerate}
\item $f_0(z),f_j(z)$ are twice continuously differentiable. \label{cond: constr-smooth} 
\item The problem corresponding to $f_0(z)$ has a unique minimizer $z_0$ over $\Z$. \label{cond: unique-sol}
\item The inf-compactness condition: \edit{there exist} constants $\alpha$ and $t_0\in (0, 1]$ such that the sublevel set 
\edit{$\left\{
z \in \mathcal Z: ~ f_{0}(z)+t\prns{f_{j}(z)-f_0\prns{z}} \le \alpha
\right\}$ 
}
is nonempty and uniformly bounded over $t \in [0, t_0)$.  \label{cond: constr-compact}
\item $z_0$ is associated with a unique Lagrangian multiplier $\nu_0$ that also satisfies the strict complementarity condition:  $\nu_{0, k} > 0$ if $k \in K_h(z_0)$, where $K_h(z_0) = \{k: h_{k}(z_0) = 0, k = s+1, \cdots, m\}$ is the index set of active at $z_0$ inequality constraints. \label{cond: constr-uniquenss}
\item \label{cond: constr-MF}  The Mangasarian-Fromovitz constraint qualification condition at $z_0$:
\begin{align*}
&\ts \nabla h_{1}(z_0),\,\dots,\,\nabla h_{s}(z_0) \text{ are linearly independent},~\text{and}  \\
&\ts\exists d_z \text{ s.t. } \nabla h_{k}(z_0)d_z = 0, ~ k = 1, \dots, s, ~ \nabla h_{k}(z_0)d_z < 0, ~ k \in K_h(z_0).
\end{align*}  
\item \label{cond: constr-2nd} Second order sufficient condition:
\begin{align*}
d_z^\top \prns{\nabla^2 f_0(z_0) +  \sum_{k = 1}^{m} \nu_{0,k}  \nabla^2  h_{k}(z_0)} d_z > 0\quad \forall d_z \in C(z_0)\setminus \{0\},  
\end{align*}
where 
$C(z_0)$ is the critical cone defined as follows:
\begin{align*}
C(z_0) = \left\{d_z: d_z^\top\nabla h_{k}(z_0) = 0, ~ \text{for } k \in \{1, \dots, s\} \cup K_h(z_0)\right\}.
\end{align*}
\end{enumerate}
Let $d_z^{j*}$ be the first part of the (unique) solution of the following linear system of equations:
\begin{align}\label{eq: linear-equation}
&\begin{bmatrix}
~\edit{\prns{\nabla^2 f_0(z_0) +  \sum_{k = 1}^{m} \nu_{0,k}  \nabla^2  h_{k}(z_0)}}  ~&~ \nabla {\mathcal H^{K_h}}^\top(z_0)~ \\
~\nabla^\top\mathcal H^{K_h}(z_0) ~&~ 0 ~
\end{bmatrix}
\begin{bmatrix}
d_z^j \\ \xi 
\end{bmatrix}
= \begin{bmatrix}
-\edit{\prns{\nabla  f_j(z_0) - \nabla f_0(z_0)}}  \\
0 
\end{bmatrix},
\end{align}
\edit{where $\nabla^\top \mathcal H(z_0) \in \mathbb \R^{m\times d}$ is the matrix whose $k^\text{th}$ row is $\prns{\nabla h_{k}(z_0)}^\top$, and $\nabla^\top{\mathcal H}^{K_h}(z_0) \in \mathbb R^{s + |K_h(z_0)|}$ consists only of the rows corresponding to equality and active inequality constraints.}

Then 
\begin{align}
v_j(t) \label{eq: approx-risk}
	&= (1-t)f_{0}(z_0)+tf_{j}(z_0)  
	\\\notag&\phantom{=}+ \frac{1}{2}t^2\braces{d_{z}^{j*\top} \prns{\nabla^2 f_0(z_0) +  \sum_{k = 1}^{m} \nu_{0,k}  \nabla^2  h_{k}(z_0)} d_{z}^{j*} + 2d_{z}^{j*\top} \prns{{\nabla  f_j(z_0) - \nabla f_0(z_0)}}} + o(t^2) , \\
z_j(t) &= z_0 + t d_{z}^{j*} + o(t) \label{eq: approx-sol}.
\end{align}
\end{theorem}

Due to the presence of constraints, the approximations of optimal value $v_j(t)$ and optimal solution $z_j(t)$ in \cref{thm:secondorder-const} require more complicated conditions than those in \cref{thm:secondorder}. In particular, we need to incorporate constraints in the inf-compactness conditions (condition \ref{cond: constr-compact}) and second order sufficient condition (condition \ref{cond: constr-2nd}), impose uniquenss and strict complementarity regularity conditions for the Lagrangian multiplier (condition \ref{cond: constr-uniquenss}), and assume a constraint qualification condition (condition \ref{cond: constr-MF}).
The coefficient matrix on the left hand side of the linear system of equations in \cref{eq: linear-equation} is invertible due to the second order sufficient condition in condition \ref{cond: constr-2nd} 
\citep[see][Proposition 4.2.2]{Bertsekas1995}, which ensures that $d_z^{j*}$ uniquely exists. 
\edit{These regularity conditions guarantee that the optimal value $v_j(t)$ and optimal solution $z_j(t)$ vary smoothly with  perturbations to the optimization objective, and they rule out problems whose optimal solution may change non-smoothly. For example, optimal solutions to linear programming problems may change to completely different vertices under even tiny perturbations to linear objectives.
\footnote{\edit{In the context of such linear problems, \cite{elmachtoub2020decision} propose to optimize the oracle criterion by exhaustive search. As noted before, this is computationally burdensome, and indeed their focus is on smaller-scale models, with particular benefits to interpretability. In \cref{prop: equi-spo-stochopt} in \cref{sec: supplement}, we formally argue that their criterion coincides with what we called the oracle criterion in \cref{eq:oraclecrit} in the case of linear costs.}}
Nevertheless, \cref{thm:secondorder-const} may still apply to some problems with linear costs and nonlinear constraints such as the quadratically constrained problems in Section 6.2 of \cite{elmachtoub2017smart}.}

\edit{\label{para: thm4-interp} Concretely, the constraints in \cref{ex: mnv,ex: portfolio-var,ex: portfolio} all ensure that the decision variables are bounded. So the inf-compactness condition (condition \ref{cond: constr-compact}) is satisfied when there is no additional auxiliary variable (\eg, the newsvendor problem), or when the auxiliary variables at CSO optimal solutions are almost surely bounded. (\eg, conditional expectation or conditional quantiles of optimal portfolio returns in \cref{ex: portfolio-var} or \ref{ex: portfolio}, respectively) 
Moreover, since these constraints are all simple affine constraints, in \cref{sec: supplement} \cref{prop: LICQ}, we  verify that they satisfy a stronger linear independence constraint qualification condition than condition \ref{cond: constr-MF}, which ensures the unique existence of Lagrangian multiplier $v_0$ for  the solution $z_0$ (condition \ref{cond: constr-uniquenss}). 
Since our problems in \cref{ex: mnv,ex: portfolio-var,ex: portfolio} are all convex, the second order sufficient condition (condition \ref{cond: constr-2nd}) can ensure $z_0$ to be the unique optimal solution (condition \ref{cond: unique-sol}). 
This second order sufficient condition trivially holds when the Hessian matrix is positive definite and the constraints are affine, \eg, under the conditions we discuss in \cref{sec: approx-crit} for \cref{ex: mnv,ex: portfolio-var}.
In contrast to these conditions, the strict complementary slackness in condition \ref{cond: constr-uniquenss} is generally more difficult to verify exactly. 
However, even if it does not hold exactly, splitting criteria based on the approximations in \cref{eq: approx-risk,eq: approx-sol} may still capture signals relevant to CSO problems, especially compared to RandForest, which completely ignores the optimization problem structure.}

Note that \cref{thm:secondorder} is a special case of \cref{thm:secondorder-const} without constraints ($m=0$).
Indeed, without the constraints, the regularity conditions for the Lagrangian multiplier and constraint qualification condition are vacuous, and 
conditions \ref{cond: constr-compact} and \ref{cond: constr-2nd} reduce to the inf-compactness and positive definite Hessian matrix conditions in \cref{thm:secondorder} respectively.  And, without constraints, the linear equation system in \cref{eq: linear-equation} consists only of the part corresponding to $d^j_z$ and the solution exactly coincides with the linear term in \cref{eq:apxsolcrit}.
\Cref{thm:secondorder-const} can itself be viewed as a special case of our \cref{thm:secondorder-const2} in \cref{sec:cso-general}, where we tackle CSO problems with both deterministic and stochastic constraints.

\subsection{Approximate Splitting Criteria}\label{eq: criteria-constr}
Analogous to \cref{thm:secondorder} for unconstrained problems, \cref{thm:secondorder-const} for constrained problems also motivates two different approximations of the oracle splitting criterion $\crit^\text{oracle}(R_1,R_2)=p_1v_1(1)+p_2v_2(1)$. Extrapolating \cref{eq: approx-risk} and 
\cref{eq: approx-sol} to $t = 1$ and ignoring the high order terms gives an approximate risk and approximate solution criterion, respectively: 
\begin{align}\label{eq:apxriskcrit-cons}
\ts\crit^\text{apx-risk}(R_1,R_2)
	&= \frac{1}{2}\sum_{j = 1, 2}p_jd_z^{j*\top} \prns{\nabla^2 f_0(z_0) +  \sum_{k = 1}^{m} \nu_{0,k}  \nabla^2  h_{k}(z_0)} d_z^{j*}  
  + \sum_{j = 1, 2}p_jd_z^{j*\top} \prns{{\nabla  f_j(z_0) - \nabla f_0(z_0)} }, \\
\label{eq:apxsolcrit-cons}\crit^\text{apx-soln}(R_1,R_2)
	&=\sum_{j=1,2}p_jf_j\prns{z_0 + d_z^{j*}},
\end{align}
where in the approximate risk criterion, $\crit^\text{apx-risk}(R_1,R_2)$, we \edit{again} omit from the extrapolation the constant term $\sum_{j = 1, 2}p_j\prns{f_j(z_0)} = p_0f_0(z_0)$, \edit{as it} does not depend on the choice of subpartition.

\subsubsection*{Estimating the Approximate Splitting Criteria.}
\edit{We next discuss a general strategy}
to estimate our more general approximate splitting criteria in \cref{eq:apxriskcrit-cons,eq:apxsolcrit-cons} that handle constraints. 
First, we start by estimating $z_0$ by its sample analogue as in \cref{eq: z0hat}, where crucially now $\Z$ is constrained.
Then we can estimate the gradients of $f_j$ at $z_0$ for $j=0,1,2$ and the Hessians of $f_0$ at $z_0$ in the very same way that gradients of $f_j$ and Hessians of $f_0$ were estimated in \cref{sec: approx-crit}, namely, estimating them at $\hat z_0$, which is now simply solved \emph{with} constraints. 
Gradients and Hessians of $h_{k}$ at $z_0$ can be estimated by simply plugging in $\hat z_0$, since the functions $h_{k}$ are known deterministic functions.
We can estimate $K_h(z_0)$ by $K_h(\hat z_0)$, \ie, the index set of the inequality constraints that are active at $\hat z_0$.
Next, we can estimate $\nu_0$ by solving $\widehat{\nabla  f}_0(\hat z_0) +  \sum_{k = 1}^{m} \nu_{k}  \nabla  h_{k}(\hat z_0)=0$ subject to $\nu_k \ge 0$ for $k\in K_h(\hat z_0)$ and $\nu_k=0$ for $k\in \{s+1,\dots,m\}\setminus K_h(\hat z_0)$, or alternatively by using a solver for \cref{eq: z0hat} that provides associated dual solutions. Finally, we can estimate $d^{j*}_z$ by solving \cref{eq: linear-equation} with estimates plugged in for unknowns. With all of these pieces in hand, we can estimate our approximate criteria in \cref{eq:apxriskcrit-cons,eq:apxsolcrit-cons}.

\subsubsection*{Revisiting the Running Examples.}
In \cref{sec: est-approx-crit}, we discussed how to estimate gradients and Hessians of the objectives of our running examples. Now we revisit the examples and discuss their constraints.
The nonnegativity and capacity constraints in \cref{ex: mnv} can be written as $h_{1}(z') = \sum_{l = 1}^d z_l \le C, ~ h_{l+1}(z') = - z_l \le 0, ~ l = 1, \dots, d$,
and the simplex constraint in \cref{ex: portfolio,ex: portfolio-var} as $h_{1}(z) = \sum_{l = 1}^d z_l = 1, ~ h_{l + 1}(z) = -z_l\leq0,\,l=1,\dots,d$.
These are all deterministic linear constraints: their gradients are known constants 
and their Hessians are zero.

\subsection{Construction of Trees and Forests}\label{sec: construction-constr}
It is straightforward to now extend the tree fitting algorithm, \cref{alg: tree}, to the constrained case. First, we note that in line \ref{alg: tree z0 step} that solves for $\hat z_0$, we use a constrained feasible set ${\Z}$. Then, we update line \ref{alg: tree estim 0 step} to estimate $\nabla f_0(z_0),\nabla^2 f_0(z_0),\nabla h_{k}(z_0),\nabla^2 h_{k}(z_0),K_h(z_0),\nu_0$. Next, we update line \ref{alg: tree estim j step} to estimate $\nabla f_j(z_0),d_z^{j*}$. And, finally, we update line \ref{alg: tree crit step} to use the general splitting criteria in \cref{eq:apxriskcrit-cons,eq:apxsolcrit-cons} where we plug in these estimates for the unknowns.

A crucial point that is key to the tractability of our method even in the presence of constraints is that the only step that requires any re-computation for each candidate split is the estimation of $\nabla f_j(z_0),d_z^{j*}$. As in the unconstrained case, \edit{estimators for} $\nabla f_j(z_0)$ usually consist of very simple sample averages over the data in the region $R_j$ \edit{so they can also be very quickly computed}. Moreover, only the right-hand side defining $d_z^{j*}$ in \cref{eq: linear-equation} varies with each candidate split, so the equation can be presolved using 
an $LU$ 
decomposition or a similar approach. 
Therefore, we can easily and quickly consider many candidate splits, and correspondingly grow large-scale forests.

\Cref{alg: forest} for fitting the forest remains the same, since the only change in fitting is in the consideration of tree splits. And, \cref{alg: forest pred} was already written in the general constrained setting and so also remains the same. In particular, after growing a forest where tree splits take the constraints into consideration and given this forest, we impose the constraints in $\Z$ when computing the final forest-policy decision, $\hat z(x)$.

\section{Empirical Study}\label{sec: empirical}

In this section we study our algorithm and baselines empirically to investigate the value of optimization-aware construction of forest policies and the success of our algorithm in doing so.
\edit{We focus on constrained CSO problems with CVaR objectives, including one simulated portfolio optimization problem and one real-data shortest path problem. 
In \cref{sec: empirical nv,sec: empirical-min-var,sec: more-mean-var}, we show additional experimental results for unconstrained multi-item newsvendor problems (\cref{ex: mnv}) and constrained variance-based portfolio optimization problems (\cref{ex: portfolio-var}).}

\subsection{CVaR Portfolio Optimization}\label{sec: cvar-empirical}
\begin{figure*}[t!]%
\centering%
\begin{subfigure}{\textwidth}\centering%
    \includegraphics[width=0.8\textwidth]{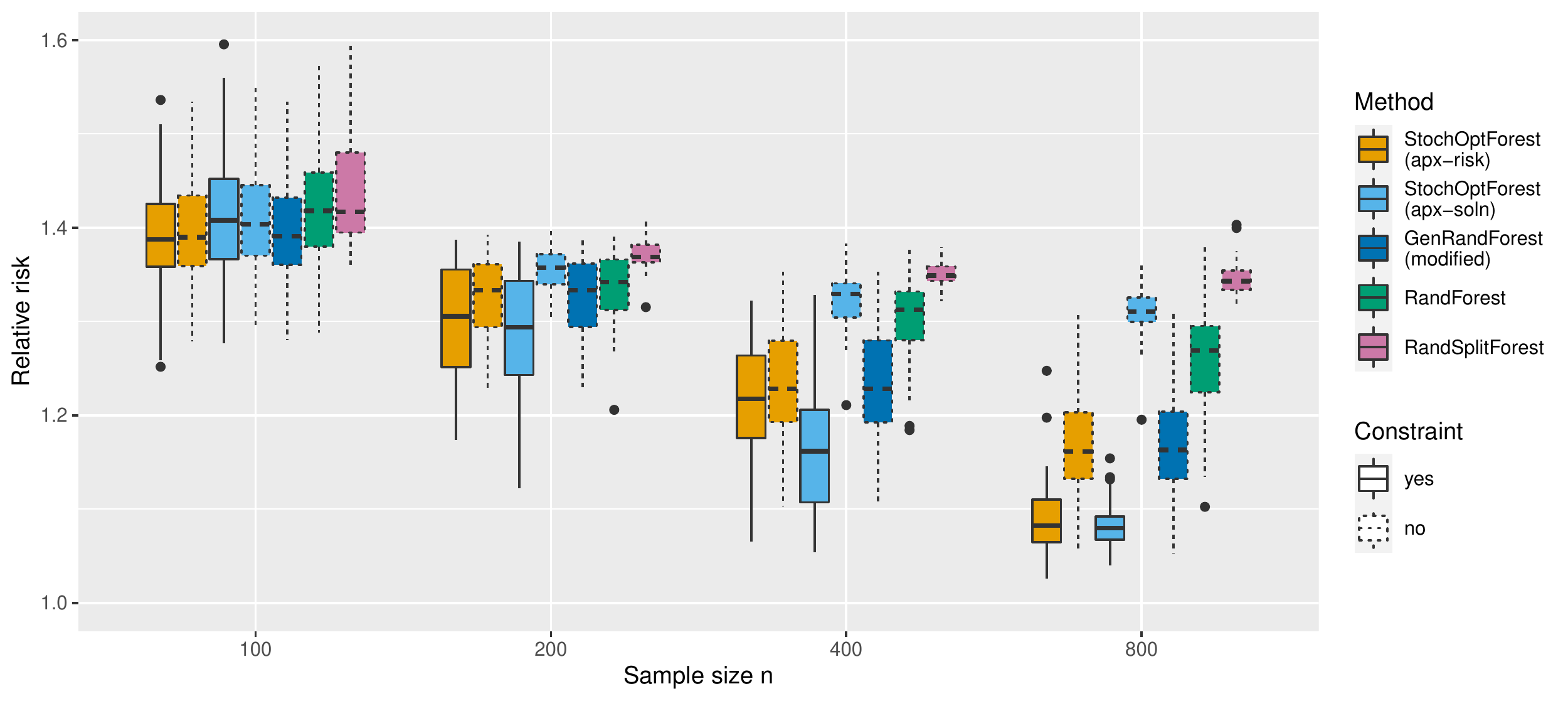}
    \caption{Relative risks of different forest policies \edit{(lower relative risk means better performance)}.}\label{fig: cvar forests risk}
\end{subfigure}%
\\
\begin{subfigure}{\textwidth}\centering%
    \includegraphics[width=0.8\textwidth]{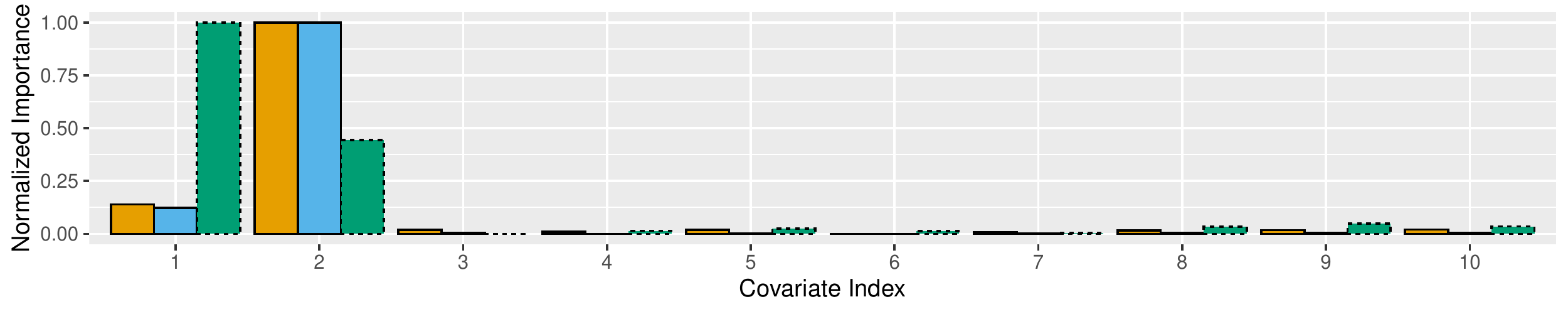}
    \caption{Feature Importance.}\label{fig: cvar forests imp}
\end{subfigure}
\caption{Results for the CVaR portfolio optimization problem.}%
\vspace{-7pt}\end{figure*}

\begin{table}[t!]\footnotesize
\centering 
\begin{tabular}{|c|c|c|c|}
\hline
Method & $n = 100$ & $n = 200$ & $n = 400$ \\
\hline
StochOptTree (oracle) & $41.41$ ($5.43$) &  $165.03$ ($15.77$) & $695.88$ ($83.91$) \\
\hline 
StochOptTree (apx-risk) & $0.26$ ($0.08$) & $0.68$ ($0.36$) & $1.68$ ($0.54$) \\
\hline 
StochOptTree (apx-soln) & $0.22$ ($0.05$)  & $0.70$ ($0.20$) & $2.24$ ($0.33$)  \\
 \hline 
\end{tabular}
\SingleSpacedXI\caption{Average running time in seconds over $10$ repetitions (and standard deviations) of constructing one tree for different algorithms in the CVaR optimization problem.}
\label{table: time-cvar}
\vspace{-7pt}\end{table}

We first apply our method to the CVaR portfolio optimization problem (see \cref{ex: portfolio}). We consider $d = 3$ assets and $p = 10$ covariates. The covariates $X$ are drawn from a standard Gaussian distribution, and the asset returns are independent and are drawn from the conditional distributions $Y_1 \mid X \sim 1 + 0.2\exp(X_1) - \text{LogNormal}\prns{0, 1 - 0.5\indic{-3 \le X_2 \le -1}}$, $Y_2 \mid X \sim 1 - 0.2 X_1  - \text{LogNormal}\prns{0, 1 - 0.5\indic{-1 \le X_2 \le 1}}$, and $Y_3 \mid X \sim 1 + 0.2|X_1| - \text{LogNormal}\prns{0, 1 - 0.5\indic{1 \le X_2 \le 3}}$.
\edit{We seek an investment policy $z(\cdot) \in \mathbb R^d$ that for each $x$ aims to achieve smallest risk $\text{CVaR}_{0.2}\prns{ Y^\top z(x) \mid X=x}$, or equivalently the $0.8$-CVaR of the portfolio loss $-Y^\top z\prns{x}$, while satisfying the simplex constraint, \ie, $\mathcal Z  = \braces{z \in \mathbb R^d: \sum_{l = 1}^d z_l = 1, z_l \ge 0}$.}

\edit{We compare our StochOptForest algorithm using either the apx-risk or apx-soln approximate criterion for constrained problems (\cref{eq:apxriskcrit-cons,eq:apxsolcrit-cons}) to five benchmarks, where all algorithms are identical except for their \emph{splitting criterion}. 
    The first two benchmarks are our StochOptForest algorithm using apx-risk and apx-soln criteria that (mistakenly) \emph{ignore} the constraints (\ie, \cref{eq:apxriskcrit,eq:apxsolcrit}). 
    The third benchmark is a modified\footnote{\edit{Note we cannot apply the original GenRandForest algorithm to solve \emph{unconstrained} CVaR optimization: every step of tree construction requires computing the optimal \emph{unconstrained} solution in the region $R_0$ to be partitioned, which however does not exist because without constraints the CVaR objectives can be made arbitrarily small. We thus have to slightly modify the GenRandForest algorithm to compute the optimal \emph{constrained} solution in every region to be partitioned, from which we then compute the GenRandForest splitting criterion for the first order optimality condition of unconstrained CVaR optimization. We furthermore regularize the Hessian matrix as it is not generally invertible, as discussed after \cref{eq: portfolio-hess}, which would make the GenRandForest splitting criterion undefined. See \cref{sec: more-cvar}.}} GenRandForest algorithm \citep{athey2019generalized} applied to the first order optimality condition for the CVaR optimization problem \emph{without} the simplex constraint, as GenRandForest is designed for \emph{unconstrained} problems. 
    The fourth benchmark is the regular RandForest, which uses the squared error splitting criterion in \cref{ex: prediction} and targets the predictions of asset mean returns, and the fifth is the RandSplitForest algorithm, which chooses splits uniformly at random (without using the portfolio return data). 
    For our approximate criteria (both constrained and unconstrained) and the GenRanForest criterion, we use the parametric Hessian estimator in \cref{eq: hessian-normal1,eq: hessian-normal2} (which is \emph{misspecified} in this example). 
    We do not compare to StochOptForest with the oracle splitting criterion since it is too computationally intensive as we investigate further below (see \cref{table: time-cvar}).
    In all forest algorithms, we use 
    an ensemble of $500$ trees.
    To compute $\hat z_0$ in our StochOptTree algorithm (\cref{alg: tree z0 step} line \ref{alg: tree z0 step}) as well as to compute the final forest policy for any forest, we formulate the constrained CVaR optimization problem as a linear programming problem \citep{rockafellar2000optimization} and solve it using Gurobi 9.0.2.
    We evaluate each forest policy $\hat z(\cdot)$ by its relative risk compared to the optimal $z^*(\cdot)$, namely the raio of $\Eb{\text{CVaR}_{0.2}\prns{ Y^\top \hat z(x) \mid X}\mid \mathcal{D}}$ over $\Eb{\text{CVaR}_{0.2}\prns{ Y^\top  z^*(x) \mid X}}$, which we approximate using a very large testing dataset.
    See \cref{sec: more-cvar} for more details.
}

\Cref{fig: cvar forests risk}
 shows the distribution of the relative risk over $50$ replications for each forest algorithm across different training set size $n \in \{100, 200, 400, 800\}$. The dashed boxes corresponding to ``Constraint = no'' indicate that the associated method does not take constraints into account when choosing the splits, which applies to all four benchmarks. (Note that \emph{all} methods consider constraints in computing a the final forest-policy decision, $\hat z(x)$.)
We can observe that our StochOptForest algorithms with approximate criteria that incorporate constraints achieve the best relative risk over all sample sizes, and their relative risks decrease considerably when the training set size $n$ increases.  
In contrast, the relative risks of all benchmark methods \edit{decrease much more slowly. 
Therefore, both failing to target the cost function structure (\edit{GenRandForest},\footnote{\edit{GenRandForest criterion partly captures the cost function structure as it incorporates the corresponding first order optimality condition information, but it chooses splits to maximize the discrepancy of approximate solutions in the induced subregions, rather than optimize their decision costs directly.}} RandForest, and RandSplitForest) \emph{and} failing to take constraints into account (all five benchmark methods) can significantly undermine the ultimate decision-making quality.} 
In contrast, our StochOptForest algorithms based on the approximate criteria effectively account for both so they perform much better.    
Moreover, our results show that even though the normal distribution assumption used to derive our Hessian estimator (\cref{eq: hessian-normal1,eq: hessian-normal2}) is wrong in our experiment,
our proposed forest policies still achieve superior performance, which illustrates the robustness of our methods. 

\edit{To further understand these results, we also consider feature importance measures based on each forest algorithm. 
In \cref{app-sec: var-importance}, we extend the impurity-based feature importance measures \citep{hastie2001the} to our StochOptForest method.
Recall there are $p=10$ covariates, and the first two determine the distributions of asset returns. 
The first covariate influences the conditional mean of return distributions more, while the second one influences more the distribution tails. 
In \cref{fig: cvar forests imp}, we visualize the feature importance measures for our proposed method and RandForest when $n = 800$. 
The importance measures are normalized for each method so that the most important feature has an importance value equal to $1$. 
We can observe that our StochOptForest methods (incorporating constraints) value the second covariate  more than the first one, which shows the importance of signals in the return distribution tails for CVaR optimization. 
In contrast, the RandForest algorithm puts more importance on the first covariate, validating that it is designed to target the prediction of asset mean returns. 
There do not exist feature importance measures for the GenRandForest algorithm. Instead, we show its average frequency of splitting on each covariate in \cref{sec: more-cvar} \cref{fig: cvar-lognormal-split-freq}. 
We observe that the GenRandForest method splits on noise covariates (\ie, the $3$rd to $10$th covariate) more frequently than our proposals, which may partly explain its inferior performance.}

\edit{We also consider the average running time of our proposed algorithm in \cref{table: time-cvar}.}
We compare our StochOptTree algorithm with approximate criteria \emph{incorporating} constraints to the oracle splitting criterion (using empirical expectations). We consider 10 repetitions, in each of which we apply each tree algorithm with the same specifications to construct a single tree on the same training data with varying size $n \in \{100, 200, 400\}$. 
We run this experiment on a MacBook with 
2.7 GHz Intel Core i5 processor. 
We can see that the running time of our StochOptTree algorithm with apx-risk criterion is hundreds of times faster than the StochOptTree algorithm with the oracle criterion that must solve the constrained CVaR optimization problems for each candidate split. 
The computational gains of our approximate criteria relative to the oracle criterion also grow with larger sample size $n$ (from around $200$ times faster at $n=100$ to more than $400$ times faster at $n = 400$), as the CVaR optimization problem becomes slower to solve. 

Since the StochOptForest algorithm with the oracle criterion is extremely slow, we can only evaluate its performance in a small-scale experiment in \cref{fig: cvar-lognormal-oracle} in \cref{sec: more-cvar}. Focusing on constructing small forests of only $50$ trees with $n$ up to $400$, we find the performance of the oracle criterion is marginally better than our approximate criteria.  
However, our approximate criteria are much more computationally efficient, which enables us to leverage larger datasets for better performance.
In \cref{sec: more-cvar}, we also show that  similar results hold for portfolio optimization with a linear combination of CVaR and mean return as the objective (\cref{fig: cvar-lognormal-combined}) and for CVaR optimization with asset returns drawn from normal distributions (\cref{fig: cvar-normal}). 
We include additional empirical results on minimizing the variance of investment portfolios (see \cref{ex: portfolio-var}) in \cref{sec: empirical-min-var}, and show that the performance of our approximate criteria is close to the oracle criterion.  

\subsection{\edit{CVaR Shortest Path Problem Using Uber Movement Data}}
\label{sec: shortest-path}
\begin{figure}[t!]%
\begin{minipage}[b]{0.23\textwidth}\centering%
\begin{subfigure}[b]{\textwidth}\includegraphics[width=\textwidth]{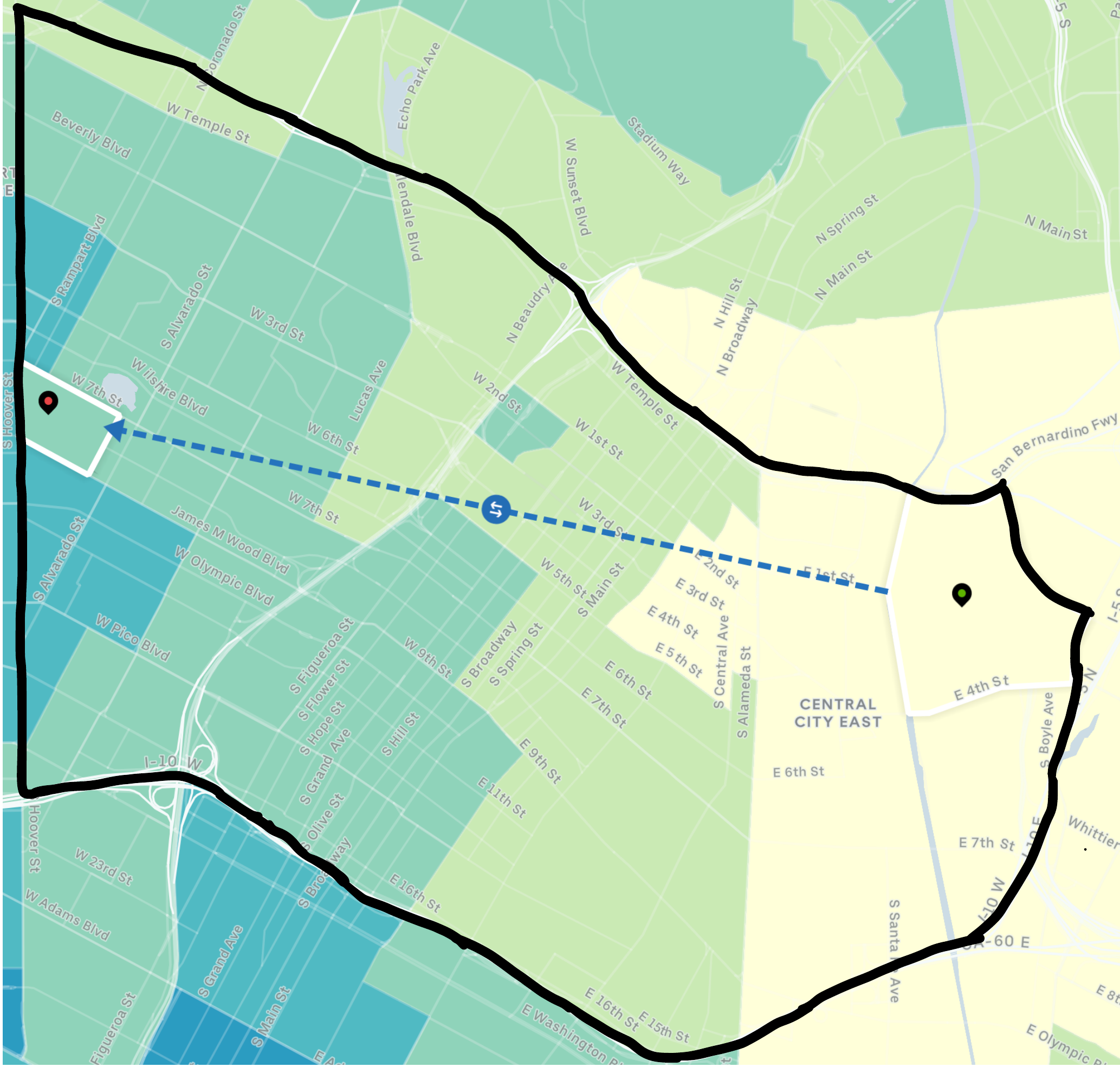}\caption{The downtown Los Angeles region (enclosed by the black lines) for the shortest path problem.%
}\label{fig: uber}\end{subfigure}
\end{minipage}%
\hspace{0.01\textwidth}%
\begin{minipage}[b]{0.76\textwidth}\centering%
\begin{subfigure}[b]{\textwidth}\includegraphics[width=\textwidth]{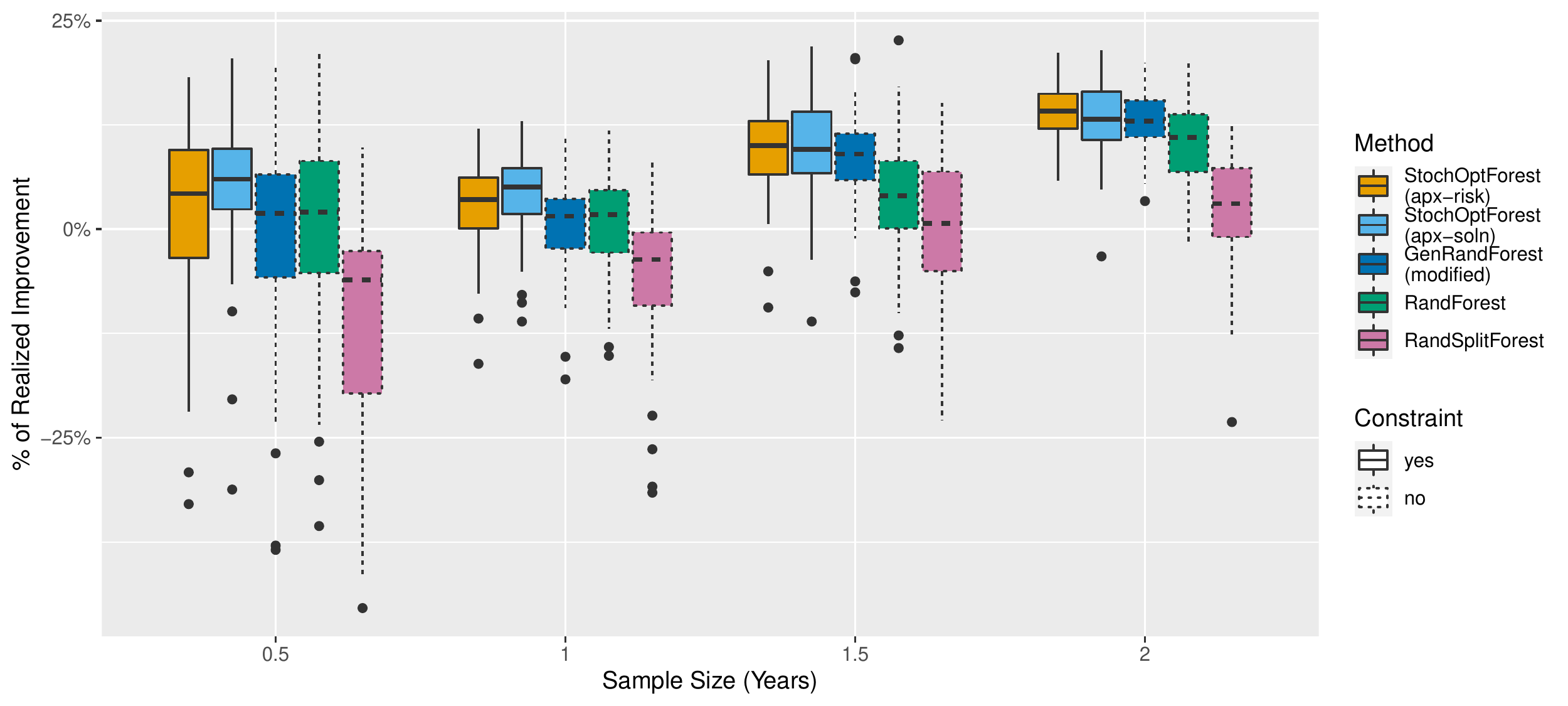}\caption{Percentages of realized improvement by different forest policies for the CVaR shortest-path problem. A higher percentage is better.}\label{fig: cvar sp data}\end{subfigure}
\end{minipage}%
\caption{Set up and results for the CVaR shortest path problem using Uber Movement data.}\label{fig: uberexperiment}%
\end{figure}

\edit{
We next demonstrate our methods in a shortest path problem, using traveling times data in Los Angeles (LA) collected from Uber Movement (\url{https://movement.uber.com}). 
We focus on $45$ census tracts in downtown LA (see \cref{fig: uber}), collecting historical data of average traveling times from each of these census tracts to its neighbors during five periods in each day (AM Peak, Midday, PM Peak, Evening, Early Morning) in $2018$ and $2019$. 
This results in $3650$ observations of traveling times $Y_j$ for $j = 1, \dots, 93$ edges on a graph with $45$ nodes. 
We consider $p = 197$ covariates $X$ including weather, period of day and other calendar features, and lagged traveling times.
We aim to go from an eastmost census tract (Aliso Village) to a westmost census tract (MacArthur Park) in this region (green and and red marks in \cref{fig: uber}, receptively), through a path between them, encoded by $z \in \braces{0, 1}^{d}$ with $d = 93$, where $z_j$ indicates whether we travel on edge $j$. 
In particular, we consider the CSO problem $z^*(x) \in \argmin_{z(\cdot) \in\Z}\text{CVaR}_{0.8}\prns{Y^\top z\prns{x} \mid X = x}$ where $\Z$ is given by standard flow preservations constraints, with a source of $+1$ at Aliso Village and a sink of $-1$ at MacArthur Park.
See \cref{apx-sec: short-path} for more details about data collection, optimization formulation, and other experiment specifications.}

\edit{We again compare different forest algorithms as we do in \cref{sec: cvar-empirical}, but to reduce computation we only train them up to  $100$ trees. 
We consider four different sample sizes ranging from $0.5$-year to the whole $2$-year data. 
For each sample size, we randomly split the corresponding dataset into two halves as training data $\mathcal{D}_{\text{train}}$ and testing data $\mathcal{D}_{\text{test}}$ respectively. Note that the distribution of $Y \mid X$ is unknown, so we can no longer benchmark the performance of each forest policy $\hat z\prns{\cdot}$ trained on $\mathcal{D}_{\text{train}}$ against the CSO optimal policy $z^*(\cdot)$ as in \cref{sec: cvar-empirical}. 
Instead, we compare their percentages of realized improvement, termed the coefficient of prescriptiveness in \citet{bertsimas2014predictive}. Namely, we consider the ratio between each method's improvement over the context-free sample average approximation (SAA), which finds a single solution $\hat z_{\text{SAA}}$ to optimize the 
average cost 
on the whole training data, over the improvement over SAA of the (infeasible) perfect-information shortest path, which in each test sample computes the shortest path for the observed travel time $Y$. 
Notice that more effective forest policies have higher percentages of realized improvement, but even known-distributions optimal policy $z^*(\cdot)$ to \cref{eq:cso} cannot generally achieve $100\%$ realized improvement as the covariates do not perfectly predict travel times.}

\edit{In \cref{fig: cvar sp data}, we show the results across $50$ realizations of random train-test splits. 
We observe that as the sample sizes increase, all methods tend to perform better. 
In particular, our StochOptForest algorithm with either the apx-risk or apx-soln criterion (incorporating constraints) outperforms all benchmarks across all sample sizes, with the clearest improvement seen using the apx-risk criterion and in smaller datasets. 
Overall the results show that incorporating the optimization problem structure in the tree construction can lead to improvements, when optimization is the aim.}

\section{Asymptotic Optimality}\label{sec: asympt-opt}
In this section, we prove that under some regularity conditions, 
our forest policy asymptotically attains the optimal risk, namely, $\Eb{c(\hat z_n(x); Y) \mid X = x}$ converges in probability to $\min_{z \in \mathcal Z}\Eb{c(z; Y) \mid X = x}$ as $n \to \infty$ for any $x \in \mathcal X$. 

It is well known that forests algorithms with adaptively constructed trees  are extremely difficult to analyze, so  some simplifying regularity conditions are often needed to make the theoretical analysis tractable \citep{biau2016random}. In this section, we assume the tree regularity conditions introduced by \citet{wager2018estimation,athey2019generalized}.
\begin{assumption}[Regular Trees]\label{assump: tree}
The trees constructed satisfy the following regularity conditions for constants $\omega \in (0, 0.2], \pi \in [0, 1)$, and an integer $k_n > 0$:
\begin{enumerate}
\item Every tree split puts at least a fraction $\omega$ of observations in the parent node into each child node. Every leaf node in every tree contains between $k_n$ and $2k_n - 1$ observations. 
\item For an index set $\mathcal J\subseteq\{1,\dots,p\}$ such that $\Eb{c(z;Y)\mid X}=\Eb{c(z;Y)\mid X_{\mathcal J}}$ for all $z\in\Z$, for each leaf of each tree and for each $j\in\mathcal J$, the average probability of splitting along feature $x_j$ is bounded below by $\pi/p$, averaging over nodes on the path from the root to the leaf and marginalizing over any randomization of candidate splits (and conditioning on the data).
\item Each tree grows on a subsample of size $s_n$ drawn randomly without replacement from the whole training data, and it is honest, \ie, $\mathcal I_j^{\text{tree}} \cap \mathcal I_j^{\text{dec}} \ne \emptyset$ with $|\mathcal I_j^{\text{tree}}| + |\mathcal I_j^{\text{dec}}| = s_n$
for $j = 1, \dots, T$. 
\end{enumerate} 
\end{assumption}
Condition 1 in \cref{assump: tree} specifies that the stopping criterion must ensure a minimal leaf size and that all candidate splits be balanced in that they put at least a constant fraction of observations in each child node. Without this condition, even when sample size $n$ is large, some imbalanced splits may run out of data so quickly that some leaves are not sufficiently partitioned and thus too large. As a result, the estimation bias of the objective function may fail to vanish even when $n \to \infty$. 
Condition 2 requires the trees to split along every relevant direction at sufficient frequency, which ensures that the leaves of the trees become small in all relevant dimensions of the feature space as $n$ gets large. Relevant features are described by those such that the random cost of any decision is mean-independent of $X$ given only these relevant features, which is trivially satisfied for $\mathcal J=\{1,\dots,p\}$. 
Condition 3 specifies that we use subsample splitting, \ie, the data used to construct each tree ($\mathcal I_j^\text{tree}$) and the data used to construct localized weights from this tree for final decision-making ($\mathcal I_j^\text{dec}$) are disjoint. 
This so-called honesty property plays a critical role in the theoretical analysis of forest algorithms but it may be largely technical. 
In \cref{sec: empirical}, we empirically show that our forest policies appear to achieve asymptotic optimality even without using honest subsample splitting. 
In \cref{sec: honesty}, we further illustrate in \cref{fig: honsty} that StochOptForest with no subsample splitting (\ie, $\mathcal I_j^\text{dec} = \mathcal I_j^\text{tree}$) performs better than the honest version with splitting, which can be explained as honest trees using fewer data for tree construction and decision-making. 

In the following assumption, we further impose some regularity conditions on the cost function $c(z; y)$ and the distribution of $Y \mid X$. 
\begin{assumption}[Distribution Regularity]\label{assump: distr} 
Fix $x\in\mathcal X$ and assume the following conditions:
\begin{enumerate}
\item The marginal distribution of $X$ has a density, its support $\mathcal X$ is compact, and the density is bounded away from $0$ and $\infty$ on $\mathcal X$.
\item \label{assump: distr inf-compact} There exist a constant $\alpha$ and a compact set $\mathcal C \subseteq \mathcal Z$ such that, $\{z \in \mathcal Z: \Eb{c(z; Y) \mid X = x} \le \alpha\} \subseteq \mathcal C$ and $\{z \in \mathcal Z:  \sum_{i = 1}^n w_i(x)c(z; Y_i) \le \alpha\} \subseteq \mathcal C$ for $w_i(x)$ in \cref{eq: weights} almost surely eventually.
\item There exists a function $b(y)$ such that for any $z, z' \in \mathcal C, y \in \mathcal Y$, $\abs{c(z; y) - c(z'; y)} \le b(y)\edit{\|z - z'\|_2}$. Moreover, there exists a positive constant $\tilde C$ such that $\Eb{b(Y) \mid X = x} \le \tilde C < \infty$.
\item There exist constants $L_c,L_b$ such that $\sup_{z\in\mathcal C}\sup_{x'\in\mathcal X}\abs{\Eb{c(z; Y) \mid X_{\mathcal J} = x_{\mathcal J}}-\Eb{c(z; Y) \mid X_{\mathcal J} = x'_{\mathcal J}}}\leq L_c\edit{\magd{x_{\mathcal J}-x'_{\mathcal J}}_2}$ and $\sup_{x'\in\mathcal X}\abs{\Eb{b(Y) \mid X_{\mathcal J} = x_{\mathcal J}}-\Eb{b(Y) \mid X_{\mathcal J} = x'_{\mathcal J}}}\leq L_b\edit{\magd{x_{\mathcal J}-x'_{\mathcal J}}_2}$.
\item There exist positive constants $\eta, \eta', C$ such that 
\[
\sup_{z \in \mathcal C}\Eb{e^{\eta\abs{c(z; Y) - \Eb{c(z; Y) \mid X = x}}} \mid X = x} \le C < \infty, ~~ \Eb{e^{\eta'\abs{b(Y) - \Eb{b(Y)\mid X = x}}} \mid X = x} \le C < \infty.
\]
\end{enumerate}
\end{assumption} 
One important condition in \cref{assump: distr}
 is that the cost function $c(z; y)$ is Lipschitz-continuous in $z$ on the compact set $\mathcal C$. In the following proposition, we validate that \cref{ex: mnv,ex: portfolio-var,ex: portfolio} all satisfy this condition. 
\begin{proposition}\label{prop: lipschitz}
For any $z, z' \in \mathcal C$:
\begin{enumerate}
\item The cost function $c(z;y)=\sum_{l=1}^d\max\{\alpha_l(z_l-y_l),\,\beta_l (y_l-z_l)\}$ for the newsvendor problem in  \cref{ex: mnv}  satisfies that $\abs{c(z; y) - c(z'; y)} \le \sqrt{d}\max\{\alpha_l, \beta_l\}\edit{\|z - z'\|_2}$.
\item The cost function $c(z; y) = \prns{y^\top z_{1:d} - z_{d+1}}^2$ for the variance-based portfolio optimization problem in \cref{ex: portfolio-var} satisfies that $\abs{c(z; y) - c(z'; y)} \le 4\sqrt 2 (\sup_{\tilde z\in\mathcal C}\edit{\edit{\|\tilde z\|_2}})\max\{1, \edit{\|y\|^2_2}\}\edit{\|z - z'\|_2}$.
\item The cost function $c(z; y) = \frac{1}{\alpha}\max\braces{z_{d+1} - y^\top z_{1:d},\, 0}-  z_{d+1}$ for the CVaR \edit{optimization} problem in  \cref{ex: portfolio}  satisfies that $\abs{c(z; y) - c(z'; y)} \le \prns{\edit{\|y\|_2} + 1 + \frac{1}{\alpha}}\edit{\|z - z'\|_2}$.
\end{enumerate}
\end{proposition}

Under the assumptions above, we can prove that the forest policy is asymptotically optimal. 
\begin{theorem}\label{thm: asymp-opt}
Let $x \in \mathcal X$ be fixed. 
If \cref{assump: tree,assump: distr} hold at the given $x$ and if $k_n \to \infty$, $s_n/k_n \to \infty$, $\log T/k_n \to 0$, and $Tk_n/s_n \to 0$, then 
\begin{align}
\sup_{z \in \mathcal C}\abs{\sum_{i = 1}^n w_i(x)c(z; Y_i) - \Eb{c(z; Y) \mid X = x}} \overset{p}{\to} 0.\label{eq: unif-converg}
\end{align}
It follows that any choice $\hat z_n(x) \in \argmin_{z \in \mathcal Z}\sum_{i = 1}^n w_i(x)c(z; Y_i)$ satisfies that as $n \to \infty$,  
\begin{align}\label{eq: asymp-opt}
\abs{\Eb{c(\hat z_n(x); Y) \mid X = x} -  \min_{z \in \mathcal Z}{\Eb{c(z; Y) \mid X = x}}} \overset{p}{\to} 0.
\end{align}
\end{theorem} 

\Cref{thm: asymp-opt} provides asymptotic optimality of $\hat z_n(x)$ point-wise in $x$. The result can straightforwardly be extended to be uniform in $x$ if we simply assume the conditions in \cref{assump: distr} hold for all $x\in\mathcal X$ with common constants.

\section{Discussion}\label{sec:discussion}

In this section we offer some discussions.
First, we discuss how our work is related to and differs from work on \emph{estimation} using localized  weights and forests in particular. 
Then we discuss other related work on CSO and on integrating prediction and optimization. \edit{We discuss additional related literature about tree models and perturbation analysis in \cref{sec: literature}.}

\subsection{Comparison to Estimation}\label{sec: comp to est}
The idea of using localized weights to estimate parameters given covariate values has a long history in statistics and econometrics, including applications in local maximum likelihood \citep{tibshirani1987local,fan1998local},  local generalized method of moments \citep{lewbel2007local}, local estimating equation \citep{carroll1998local} and so on.  
These early works typically use non-adaptive localized weights like nearest-neighbor weights or Nadaraya-Watson kernel weights, which only use the information of covariates.
Recently, some literature propose to use forest-based weights for local parameter estimation \citep[\eg,][]{meinshausen2006quantile,scornet2015random,athey2019generalized,pmlr-v97-oprescu19a}, 
which generalizes the original random forest algorithm for regression and classification problems \citep{breiman2001random} to other estimation problems where the estimand depends on the $X$-conditional distribution.
These forest-based weights are derived from the proportion of  trees in which each observation falls in the same terminal node as the target covariate value.
Since those trees are adaptively constructed using label data as well, random forest weights are shown to be more effective in modeling complex heterogeneity in high dimensions than non-adaptive weights.
Recent literature has studied the statistical guarantees of random forests in estimating conditional expectation functions \citep[see reviews in][]{biau2016random,wager2018estimation}, or more general parameters defined by local estimating equations \citep{athey2019generalized,pmlr-v97-oprescu19a}.

Among the statistical estimation literature above, 
\edit{closest to our work is \cite{athey2019generalized}, who propose the GenRandForest algorithm to} estimate roots of conditional estimating equations. 
This is closely related to our decision making problem, because the optimal solution of \emph{unconstrained} CSO is also the root of a conditional estimating equation given by the first order optimality condition. 
For example, 
the optimal solutions of conditional newsvendor problem in \cref{ex: mnv} without constraints are conditional quantiles, which are also considered by \cite{athey2019generalized} under the conditional estimating equation framework. 
For computational efficiency, \cite{athey2019generalized} also propose a gradient-based approximation for roots in candidate subpartitions (\edit{see discussions below \cref{eq:apxsolcrit}}), and then find the best split that maximizes the discrepancy of the approximate roots in the subregions, thereby approximately minimizing the total \emph{mean squared error} of the estimated roots \citep[][Proposition 1]{athey2019generalized}.

In contrast, our paper has a fundamentally different goal: we target \textit{decision-making risk} (expected cost) rather than \textit{estimation risk} 
(accuracy). 
In our apx-risk and apx-soln criteria, we directly approximate the optimal average cost itself and use this to choose a split, rather than estimation error of the solution. 
In \cref{sec: empirical nv}, we provide one empirical example of unconstrained newsvendor problem where the heterogeneity of optimal solution estimation is drastically different from the heterogeneity of the optimal decision-making, which illustrates the benefit of targeting decision quality when the decision problem, rather than the estimation problem, is of interest. 
Moreover, our methods uniquely accommodate constraints, which are prevalent in decision-making problems but rare in statistical estimation problems. 
\edit{For \emph{constrained} CSO, the optimal solution cannot be characterized by local estimating equations so the GenRandForest algorithm  is not applicable.}
In \cref{sec: empirical}, we provided empirical examples
\edit{of constrained CVaR optimization problems where the taking into account constraints is key to constructing good trees.}

\subsection{CSO and Integrating Prediction and Optimization}\label{sec: CSO lit}
Our paper builds on the CSO framework, and the general local learning approach, \ie, estimating the objective (and stochastic constraints in \cref{sec:cso-general}) by weighted averages with weights reflecting the proximity of each covariate observation to the target value. 
\citet{hannah2010nonparametric,hanasusanto2013robust,bertsimas2014predictive} propose the use of nonparametric weights that use only the covariate observations $X$ and do not depend on observations of the uncertain variable $Y$, such as Nadaraya-Watson weights. 
\citet{bertsimas2014predictive} formally set up the CSO framework, propose a wide variety of machine learning methods for local weights construction, and provide rigorous asymptotic optimality guarantees.  
In particular, they additionally propose weights based on decision trees and random forests that incorporate the uncertain variable information, and show their superiority when the covariate dimension is high. 
However, their tree and forest weights are constructed from standard regression algorithms that target prediction accuracy instead of downstream decision quality, primarily because targeting the latter would be too computationally expensive. Our paper resolves this computational challenge by leveraging approximate criteria that can be efficiently computed.  

Optimization problems that have unknown parameters, such as an unknown distribution or a conditional expectation, are often solved by a two-stage approach: the unknown parameters are estimated or predicted, then these are plugged in, and then the approximated optimization problem is solved. 
The estimation or prediction step is often done independently of the optimization step, targeting standard accuracy measures such as mean squared error without taking the downstream optimization problem into account. 
However, all predictive models make errors and 
when prediction and optimization are completely divorced, the error tradeoffs may be undesirable for the end task of decision-making. To deal with this problem, recent literature propose various ways to tailor the predictions to the optimization problems. 

\cite{elmachtoub2017smart} study a special CSO problem where $c(z;y)=y^\top z$ is linear and constraints are deterministic and known. In this special case, the parameter of interest is the conditional expectation $\Eb{Y\mid X=x}$, which forms the linear objective's coefficients.
They propose to fit a parametric model to predict the coefficients by minimizing a convex surrogate loss of the suboptimality of the decisions induced by predicted coefficients. 
\cite{elmachtoub2020decision} study the same linear CSO problem 
and instead predict the coefficients nonparametrically by decision trees and random forests with suboptimality as the splitting criterion. \edit{In \cref{sec: supplement} \cref{prop: equi-spo-stochopt}, we show this criterion is equivalent to what we termed the oracle splitting criterion in \cref{eq:oraclecrit} in the case of linear costs.}
Since this involves full re-optimization for each candidate split, they are limited to very few candidate splits, suggesting using one per candidate feature, and they 
consider a relatively small number of trees in their forests. 
In contrast, we consider the general CSO problem and use efficient approximate criteria, which is crucial for large-scale problems and training large tree ensembles. \edit{\cite{hu2021fast} also study linear CSO problems and they show both theoretically and empirically that with correctly specified models, integrated approaches may perform worse than the simpler predict-then-optimize approach. Our paper demonstrates the benefit of a forest-based integrated approach in \emph{nonlinear} CSO problems, where a predict-then-optimize approach would have to learn the whole conditional distribution, not just the conditional expectation.}

\citet{donti2017task} study smooth convex optimization problems with 
a parametric model for the conditional distribution of the uncertain variables (in both objective and constraints) given covariates, and fit the parametric models by minimizing the decision objective directly using gradient descent methods on the optimization risk instead of the log-likelihood. 
\citet{wilder2019melding} further extend this approach to nonsmooth problems by leveraging differentiable surrogate problems.
However, unless the cost function depends on the uncertain variables linearly, the stochastic optimization problem may involve complicated integrals with respect to the conditional distribution model. 
In contrast, our paper focuses on nonparametric forest models that cannot be trained by gradient-based methods, and we can straightforwardly target the CSO using localized weights.
\edit{\citet{Notz2020} consider convex optimization problems with nondifferentiable cost functions, and propose a subgradient boosting algorithm to directly learn a decision policy. 
While this approach can handle complex objectives, 
it can only accommodate very simple constraints like box constraints, as it is difficult to impose complex constraints on boosting decision policies. In contrast, our approach based on the CSO framework can readily handle general constraints.}

\section{Concluding Remarks}\label{sec:conclusion}

In CSO problems, covariates $X$ are used to reduce the uncertainty in the variable $Y$ that affects costs in a decision-making problem. The remaining uncertainty is characterized by the conditional distribution of $Y\mid X=x$. A crucial element of effective algorithms for learning policies for CSO from data is the integration of prediction and optimization. One can try to fit generic models that predict the distribution of $Y\mid X=x$ for every $x$ and then plug this in place of the true conditional distribution, but fitting such a model to minimize prediction errors without consideration of the downstream decision-making problem may lead to ill-performing policies. In view of this, we studied how to fit forest policies for CSO (which use a forest to predict the conditional distribution) in a way that directly targets the optimization costs. The na\"ive direct implementation of this is hopelessly intractable for many important managerial decision-making problems in inventory and revenue management, finance, \emph{etc.} Therefore, we instead developed efficient approximations based on second-order perturbation analysis of stochastic optimization. The resulting algorithm, StochOptForest, is able to grow large-scale forests that directly target the decision-making problem of interest, which empirically leads to significant improvements in decision quality over baselines.

\bibliographystyle{informs2014}
\bibliography{literature}

\newpage
\ECHead{
\begin{center}
$ $\\
Supplemental Material for\\[8pt]
\myfulltitle
\end{center}
}\vspace{8pt}

\begin{APPENDICES}
\section{Contextual Stochastic Optimization with Stochastic Constraints}\label{sec:cso-general}
In \cref{sec: constr}, we analyzed CSO problems with only deterministic constraints. In this section, we further extend our results and methods to CSO problems with both deterministic and stochastic constraints. Specifically, we consider CSO problems given by
\begin{align}\label{eq: cso-general-2}
z^*(x)&\in\argmin_{z\in\Z(x)}
\Eb{c(z;Y)\mid X=x}
,\\
\Z(x)&=\braces{z\in\R d~:~\notag
\begin{array}{ll}
g_k(z;x)=\Eb{G_k(z;Y)\mid X=x}=0, ~k=1,\dots,s,\\
g_k(z;x)=\Eb{G_k(z;Y)\mid X=x}\le 0,~k=s+1,\dots,m,\\
h_{k'}(z) = 0, ~k'=1,\dots,s', ~ h_{k'}(z) \le 0, ~k'=s' + 1,\dots, m'
\end{array}},
\end{align}
where the stochastic constraints (those given by $\{g_k(z; x)\}_{k = 1}^m$) depend on the unknown distribution of $Y\mid X=x$ and need to be learned from data as well. Note that the constraint set $\mathcal Z(x)$ now varies with $x$ due to the stochastic constraints. 

Analogously, we consider forest policies of the following form:  
\begin{align}\label{eq: forest policy2}
\hat z(x)&\in\argmin_{z\in\hat{\Z}(x)}
\sum_{i=1}^nw_{i}(x)c(z;Y_i), ~~~ w_{i}(x) \coloneqq \frac{1}{T}\sum_{j=1}^T\frac{\indic{\tau_j(X_i)=\tau_j(x)}}{\sum_{i'=1}^n\indic{\tau_j(X_{i'})=\tau_j(x)}}\\*\notag
\hat{\Z}(x)&=\braces{z\in\R d~:~
\begin{array}{ll}
\sum_{i=1}^nw_{i}(x)G_k(z;Y_i)
=0,~k=1,\dots,s,\\
\sum_{i=1}^nw_{i}(x)G_k(z;Y_i)
\leq0,~k=s+1,\dots,m,\\
h_{k'}(z) = 0, ~k'=1,\dots,s', ~ h_{k'}(z) \le 0, ~k'=s' + 1,\dots, m'
\end{array}}.
\end{align}

Notice that $\hat{\Z}(x)\neq \Z(x)$ so, unlike the deterministic case, $\hat z(x)$ may violate the constraints of the CSO problem, \ie, $\hat z\prns{x} \not\in \hat\Z\prns{x}$. In \cref{sec: disc expectation} we further discuss the nuances and challenges of handling stochastic constraints and the benefits of our approach as well as possible robust variants.

\subsubsection*{Examples of stochastic constraints.}
\begin{continuance}[Stochastic Constraints in Multi-Item Newsvendor]{\ref{ex: mnv}}
A typical example for stochastic constraints in the multi-item newsvendor problem is the following 
stochastic aggregate service level constraint:
\begin{align}\label{eq: service-level}
\mathcal Z(x) = \braces{z\in \mathbb R^d:  
\Eb{\sum_{l = 1}^d \max\{Y_l-z_l,0\}\mid X=x} \le C', ~ z_l \ge 0, ~ l = 1, \dots, d},
\end{align}
where $C'$ is a constant that stands for the maximal allowable average number of customers experiencing a stock out across items.
\end{continuance}

\begin{continuances}[Stochastic Constraints in Portfolio Optimization]{\ref{ex: portfolio-var} and \ref{ex: portfolio}}
For another example, we may impose the following mean return constraint  with a minimum return of $R$ in the portfolio optimization:
\begin{align}\label{eq: mean-return}
\mathcal Z(x) = \braces{z\in\R{d+1}: \Eb{Y^\top z_{1:d} \mid X = x} \ge R,\,z_{1:d}\in\Delta^d}.
\end{align}
More generally we can also include in the constraints any number of criteria or weighted combinations of criteria (mean, variance, CVaR at any level); we need only introduce a separate auxiliary variable for variance and for CVaR at each level considered. These would all constitute stochastic constraints.
\end{continuances}

\subsection{Perturbation Analysis}
In this section, we develop approximate splitting criteria for training forests for general CSO problems  described in \cref{eq: cso-general-2}. 
We extend the oracle splitting criterion in \cref{eq:oraclecrit} to accommodate additional stochastic constraints:
\begin{align}\label{eq:oraclecrit-general}
\crit^\text{oracle}(R_1,R_2)&\ts=\sum_{j=1,2}\min_{z\in \mathcal Z_j}\Eb{c(z;Y)\indic{X\in R_j}},\\\notag
\mathcal Z_j &=\left\{z:
\begin{array}{l}
g_{j, k}(z) = \Eb{G_k(z;Y)\mid{X \in R_j}} = 0, ~ k = 1, \dots, s, \\
 g_{j, k}(z) = \Eb{G_k(z;Y)\mid{X \in R_j}} \le 0, ~ k = s+1, \dots, m,  \\
h_{k'}(z) = 0, ~ k' =1, \dots, s', ~ h_{k'}(z) \le 0, ~ k' =s'+1, \dots, m'
\end{array}
\right\}.
\end{align}

Again, consider a region $R_0\subseteq\R d$ and its candidate subpartition $R_0=R_1\cup R_2$, $R_1\cap R_2=\varnothing$.
We define the following family of optimization problems for $t \in [0, 1]$: 
\begin{align}\label{eq:v-const}
&v_j(t)=\min_{z\in\mathcal Z_j(t)}\;f_{0}(z)+\edit{t\prns{f_{j}(z) - f_0\prns{z}}},~z_j(t)\in\argmin_{z\in\mathcal Z_j(t)}\;f_{0}(z)+\edit{t\prns{f_{j}(z) - f_0\prns{z}}},\quad j=1,2,
\\\notag
&{\text{where}~}\mathcal Z_j(t) =\left\{z:
\begin{array}{l}
g_{0, k}(z) + \edit{t\prns{g_{j, k}(z)-g_{0, k}(z)}} = 0, ~ k = 1, \dots, s, \\
g_{0, k}(z) + \edit{t\prns{g_{j, k}(z)-g_{0, k}(z)}} \le 0, ~ k = s+1, \dots, m,  \\
h_{k'}(z) = 0, ~ k' =1, \dots, s', ~ h_{k'}(z) \le 0, ~ k' =s'+1, \dots, m'
\end{array}
\right\},\quad j=1,2,\\[0.2em]\notag
&\phantom{\text{where}}~f_j(z)=\Eb{c(z;Y) \mid{X\in R_j}},\,g_{j, k}(z) = \Eb{G_k(z;Y)\mid{X \in R_j}},\quad j=0,1,2,\,k=1,\dots,m. 
\end{align}
Note that here only the stochastic constraints (and not the deterministic constraints) are interpolated by $t$, since only stochastic constraints vary from $R_0$ to $R_1,\, R_2$. 

The oracle criterion is again given by $\crit^\text{oracle}(R_1,R_2)=p_1v_1(1)+p_2v_2(1)$. 
In the following theorem, we present a  general perturbation analysis that enables us to approximate $v_1(1), v_2(1)$ in presence of both deterministic and stochastic constraints.
\begin{theorem}[Second-Order Perturbation Analysis: Stochastic and Deterministic Constraints]\label{thm:secondorder-const2}
Fix $j=1,2$.
Suppose the following conditions hold:
\begin{enumerate}
\item $f_0(z),f_j(z),g_{0, k}(z),g_{j, k}(z)$ for $k = 1, \dots, m$ are twice continuously differentiable. \label{cond: constr-smooth} 
\item The problem corresponding to $f_0(z)$ has a unique minimizer $z_0$ over $\Z_j(0)$. \label{cond: unique-sol}
\item The inf-compactness condition: \edit{there exist} constants $\alpha$ and $t_0>0$ such that the constrained level set $\left\{
z \in \mathcal Z_j(t): ~ f_{0}(z)+\edit{t\prns{f_{j}(z)-f_0\prns{z}}} \le \alpha
\right\}$ is nonempty and uniformly bounded over $t \in [0, t_0)$.  \label{cond: constr-compact}
\item $z_0$ is associated with a unique Lagrangian multiplier $(\lambda_0, \nu_0)$ that also satisfies the strict complementarity condition:  $\lambda_{0, k} > 0$ if $k \in K_g(z_0)$ and $\nu_{0, k'} > 0$ if $k' \in K_h(z_0)$, where $K_g(z_0) = \{k: g_{0, k}(z_0) = 0, k = s+1, \cdots, m\}$ and $K_h(z_0) = \{k': h_{k'}(z_0) = 0, k' = s'+1, \cdots, m'\}$ are the index sets of active at $z_0$ inequality constraints corresponding to $t = 0$. \label{cond: constr-uniquenss}
\item \label{cond: constr-MF}  The Mangasarian-Fromovitz constraint qualification condition at $z_0$:
\begin{align*}
&\ts\nabla_z g_{0, k}(z_0), ~ k = 1, \dots, s \text{ are linearly independent},\\&\ts \nabla_z h_{k'}(z_0), ~ k' = 1, \dots, s' \text{ are linearly independent},~\text{and}  \\
&\ts\exists d_z \text{ s.t. } \nabla_z g_{0, k}(z_0)d_z = 0, ~ k = 1, \dots, s, ~ \nabla_z g_{0, k}(z_0)d_z < 0, ~ k \in K_g(z_0), \\
&\ts\phantom{\exists d_z, \text{ s.t. }} \nabla_z h_{k'}(z_0)d_z = 0, ~ k' = 1, \dots, s', ~ \nabla_z h_{k'}(z_0)d_z < 0, ~ k' \in K_h(z_0).
\end{align*}  
\item \label{cond: constr-2nd} Second order sufficient condition:
\begin{align*}
d_z^\top \mathcal L(z_0; \lambda_0, \nu_0) d_z > 0~ \forall d_z \in C(z_0)\setminus \{0\},  
\end{align*}
where $\mathcal L(z; \lambda, \nu)$ is the Lagrangian for  the problem corresponding to $t = 0$, \ie, 
$\mathcal L(z; \lambda, \nu) =  f_0(z) + \sum_{k = 1}^m \lambda_{k} g_{0, k}(z) + \sum_{k' = 1}^{m'} \nu_{k'}   h_{k'}(z)$ and the critical cone $C(z_0)$ is defined as follows:
\begin{align*}
C(z_0) = \left\{d_z: 
\begin{array}{l}
d_z^\top\nabla g_{0, k}(z_0)  = 0, ~ \text{for } k \in \{1, \dots, s\} \cup K_g(z_0) \\
d_z^\top\nabla h_{k'}(z_0) = 0, ~ \text{for } k' \in \{1, \dots, s'\} \cup K_h(z_0)
\end{array}
\right\}.
\end{align*}  
\end{enumerate}
Define $\mathcal G_j(z_0) \in \mathbb R^m$ as a column  vector whose $k^\text{th}$ element is $g_{j, k}(z_0)$, and ${\mathcal G}^{K_g}_j(z_0) \in \mathbb R^{s + |K_g(z_0)|}$ as only elements corresponding to equality and active inequality constraints. We analogously define Define $\mathcal H(z_0) \in \mathbb R^m$ as a column  vector whose $k^\text{th}$ element is $h_{k}(z_0)$, and ${\mathcal H}^{K_h}(z_0) \in \mathbb R^{s + |K_h(z_0)|}$ as only elements corresponding to equality and active at $z_0$ inequality constraints.
 
Then 
\begin{align}
v_j(t) 
	&= (1-t)f_{0}(z_0)+tf_{j}(z_0)  + t\lambda_0^\top \prns{\mathcal G_j(z_0) - \mathcal G_0(z_0)} + o(t^2) \nonumber \\
	+& \frac{1}{2}t^2\braces{d_{z}^{j*\top} \nabla^2_{zz}\mathcal L(z_0; \lambda_0, \nu_0) d_{z}^{j*} + 2d_{z}^{j*\top} \prns{{\nabla  f_j(z_0) - \nabla f_0(z_0)} +  \prns{\nabla\mathcal G_j^\top(z_0) - \nabla\mathcal G_0^\top(z_0)}\lambda_0}} , \label{eq: approx-risk2}\\
z(t) &= z_0 + t d_{z}^{j*} + o(t) \label{eq: approx-sol2},
\end{align}

where $d_z^{j*}$ is the first part of the unique solution of the following linear system of equations:
\begin{align}\label{eq: linear-equation2}
&\begin{bmatrix}
\edit{\nabla^2_{zz} \mathcal L(z_0; \lambda_0, \nu_0)} & \nabla {\mathcal G^{K_g}_0}^\top(z_0) & \nabla {\mathcal H^{K_h}}^\top(z_0) \\
\nabla^\top \mathcal G^{K_g}_0(z_0) & 0 & 0 \\
\nabla^\top\mathcal H^{K_h}(z_0) & 0 & 0
\end{bmatrix}
\begin{bmatrix}
d_z^j \\ \xi \\ \eta 
\end{bmatrix}
 \\
&\qquad\qquad\qquad\qquad\qquad = \begin{bmatrix}
\edit{- \prns{\nabla  f_j(z_0) - \nabla f_0(z_0)} - \prns{\nabla\mathcal G_j^\top(z_0) - \nabla\mathcal G_0^\top(z_0)}\lambda_0}  \\
- \prns{\mathcal G^{K_g}_j(z_0) - \mathcal G^{K_g}_0(z_0)} \\
0 
\end{bmatrix}. \nonumber 
\end{align}
\end{theorem}
\cref{thm:secondorder-const2} looks very similar to \cref{thm:secondorder-const} except that we need to account for the presence of stochastic constraints in all conditions, and also in the final perturbation result. Note that if we remove the requirement on stochastic constraints in the conditions, and set $g_{j, k}(z) = 0$ for $j = 1, 2$, $k = 1, \dots, m$ in \cref{eq: approx-risk2,eq: approx-sol2,eq: linear-equation2}, then we recover the conclusion in \cref{thm:secondorder-const}.

\subsection{Approximate Splitting Criteria}\label{sec: criteria-stoch-constr}

\cref{eq: approx-risk2,eq: approx-sol2} in \cref{thm:secondorder-const2} motivate the following two different approximate splitting critera: 
\begin{align}\label{eq:apxriskcrit-cons2}
\ts\crit^\text{apx-risk}(R_1,R_2)
	&= \frac{1}{2}\sum_{j = 1, 2}p_jd_z^{j*\top} \nabla^2_{zz}\mathcal L(z_0; \lambda_0, \nu_0) d_z^{j*}  \\\nonumber
	&\phantom{=}+ \sum_{j = 1, 2}p_jd_z^{j*\top} \prns{{\nabla  f_j(z_0) - \nabla f_0(z_0)} +  \prns{\nabla\mathcal G_j^\top(z_0) - \nabla\mathcal G_0^\top(z_0)}\lambda_0}, \\
\label{eq:apxsolcrit-cons2}\crit^\text{apx-soln}(R_1,R_2)
	&=\sum_{j=1,2}p_jf_j\prns{z_0 + d_z^{j*}},
\end{align}
where in the approximate risk criterion, $\crit^\text{apx-risk}(R_1,R_2)$, we omit from the extrapolation the term $\sum_{j = 1, 2}p_j\prns{f_j(z_0) + \lambda_0^\top \prns{\mathcal G_j(z_0) - \mathcal G_0(z_0)}} = p_0\prns{f_0(z_0) - \lambda_0^\top \mathcal G_0(z_0)}$, which does not depend on the choice of subpartition.

\subsection{Estimating the Approximate Splitting Criteria}
To estimate the approximate splitting criteria in \cref{eq:apxriskcrit-cons2,eq:apxsolcrit-cons2}, we still  estimate $z_0$ by its sample analogue first:
\begin{align}
&\hat z_0 \in \argmin_{z \in \hat{\mathcal Z}_0} 
\widehat{p_0f_0}(z), \text{ where } \widehat{p_0f_0}(z) \coloneqq  
\frac{1}{n}\sum_{i = 1}^{n}\indic{X_i \in R_0}c(z; Y_i), \label{eq: z0hat-const2} \\
&\hat{\mathcal Z}_0 =\left\{z:
\begin{array}{l}
 \frac{1}{n}\sum_{i = 1}^n G_k(z;Y)\indic{X_i \in R_0} = 0, ~ k = 1, \dots, s, \\
 \frac{1}{n}\sum_{i = 1}^n G_k(z;Y)\indic{X_i \in R_0} \le 0, ~ k = s+1, \dots, m,  \\
h_{k'}(z) = 0, ~ k' =1, \dots, s', ~ h_{k'}(z) \le 0, ~ k' =s'+1, \dots, m'
\end{array}
\right\}. \nonumber 
\end{align}
Then we can estimate the gradients of $f_j,g_{j,k}, h_{k'}$ at $z_0$,  Hessians of $f_0,g_{0,k}, h_{k'}$ at $z_0$, the Lagrangian multipliers $\lambda_0,\nu_0$, and the index sets $K_g(z_0), K_h(z_0)$ of active  inequality constraints, and $d^{j*}_z$ as we do in \cref{sec: est-approx-crit}, namely, by estimating all of them at $\hat z_0$. 
 With all of these pieces in hand, we can finally estimate our approximate criteria in \cref{eq:apxriskcrit-cons2,eq:apxsolcrit-cons2}.

\subsubsection*{Revisiting the Running Examples.}
We now illustrate the estimation of gradients and Hessians for stochastic constraints using \cref{eq: mean-return,eq: service-level} as examples. 
The aggregate service level constraint in \cref{eq: service-level}
has the same structure as the objective function in \cref{ex: mnv}  and so estimating the corresponding gradients and Hessians can be done in the same way as estimating the objective gradients and Hessians as in \cref{sec: est-approx-crit}. 
The minimum mean return constraint in \cref{eq: mean-return} corresponds to $G_1(z; Y) = R-Y^\top z\leq 0$. Then $\nabla^2 g_{j,1}(z_0)$ is zero and we can estimate $g_{j,1}(z_0)$ and $\nabla g_{j,1}(z_0)$ using simple sample averages, as in \cref{ex: smooth}, Cont'd in \cref{sec: est-approx-crit}.

\subsection{Construction of Trees and Forests}
\begin{algorithm}[t!]\OneAndAHalfSpacedXI
    \caption{\textsc{Procedure to make a decision using StochOptForest}}
    \label{alg: forest pred2}
    \begin{algorithmic}[1]
    \Procedure{StochOptForest.Decide}{data $\mathcal{D}$, forest $\{(\tau_j,\,\mathcal I^\text{dec}_j):j=1,\dots,T\}$, target $x$}
    \State $w(x) \gets$ \Call{Zeros}{$|\mathcal D|$} \Comment{Create an all-zero vector of length $|\mathcal D|$}
    \For{$j = 1,\dots,T$}
    \State $\mathcal N(x)\gets\{i\in\mathcal I^\text{dec}_j:\tau_j(X_i)=\tau_j(x)\}$\Comment{Find the $\tau_j$-neighbors of $x$ among the data in $\mathcal I^\text{dec}_j$}
    \For{$i \in \mathcal N(x)$}\;$w_i(x) \gets w_i(x) + \frac{1}{\abs{\mathcal N(x)}T}$\Comment{Update the sample weights}
        \EndFor
    \EndFor
    \State $\hat{\Z}(x)\gets\braces{z\in\R d~:~
\begin{array}{ll}
\sum_{(X_i,Y_i) \in \mathcal D}w_{i}(x)G_k(z;Y_i)
=0,~k=1,\dots,s,\\
\sum_{(X_i,Y_i) \in \mathcal D}w_{i}(x)G_k(z;Y_i)
\leq0,~k=s+1,\dots,m,\\
h_{k'}(z) = 0, ~k'=1,\dots,s', ~ h_{k'}(z) \le 0, ~k'=s' + 1,\dots, m'
\end{array}}$ \label{line: constr}
    \State \textbf{return} \Call{Minimize}{$\sum_{(X_i,Y_i) \in \mathcal D}w_i(x)c(z; Y_i)$, $z\in \hat{\Z}(x)$}\Comment{Compute the forest policy \cref{eq: forest policy}}
    \EndProcedure 
    \end{algorithmic}
\end{algorithm}

It is now possible to extend the tree fitting algorithm, \cref{alg: tree}, to the general CSO problem in \cref{eq: cso-general-2}. We now solve 
$\hat z_0$ in 
line \ref{alg: tree z0 step} using \cref{eq: z0hat-const2} instead, \ie, using the estimated constraint set $\hat{\Z}_0$. Then we update line \ref{alg: tree estim 0 step} to estimate $\lambda_0,\nu_0,K_g(z_0),K_h(z_0),\nabla f_0(z_0),\nabla^2 f_0(z_0),\nabla g_{0,k}(z_0),\nabla^2 g_{0,k}(z_0),\nabla h_{k}(z_0),\nabla^2 h_{k}(z_0)$, and update line \ref{alg: tree estim j step} to estimate $\nabla f_j(z_0),\nabla g_{j,k}(z_0),d_z^{j*}$. And, finally, we update line \ref{alg: tree crit step} to use the  splitting criteria \cref{eq:apxriskcrit-cons2,eq:apxsolcrit-cons2} with these estimates. 
Again, \Cref{alg: forest} for fitting the forest remains the same, since changing the optimization problem only involves how to choose  tree splits but not how to combine the tree. 
Finally, with the extra stochastic constraints, we need to use \cref{alg: forest pred2} instead of the previous \cref{alg: forest pred} for the final decision making. The only difference is that we use the forest weights to approximate the constraint set $\hat{\mathcal Z}(x)$ to solve for the final forest-policy decision.

\subsection{Challenges with Stochastic Constraints in CSO}\label{sec: disc expectation}
\subsubsection*{Infeasibility of Stochastic Constraints.} In presence of stochastic constraints, we may run into infeasible problems. Consider the portfolio optimization problem with the constraint $\mathcal Z(x) = \braces{z\in\R{d+1}: \Eb{Y^\top z_{1:d} \mid X = x} \ge R,\,z_{1:d}\in\Delta^d}$ as an example. 
Note that if the return requirement is positive, $R > 0$, and the conditional mean return for every asset given $X = x$ is negative, \ie, $\Eb{Y_l \mid X = x} < 0, l = 1, \dots, d$, then  the constraint set $\mathcal Z(x)$ is empty since we constrain the decisions $z_1, \dots, z_d$ to be all nonnegative.
Thus the \emph{conditional} portfolio optimization problem with this constraint set can  become infeasible for some point $x$, even if the \emph{unconditional} mean return for every asset is positive so the \emph{unconditional} stochastic optimization counterpart is still feasible. 
This appears as an intrinsic challenge with conditional stochastic constraints.

However, in some cases, infeasibility may not be an issue, and our forest algorithm can still provide quality decision rules. 
For example, in \cref{sec: empirical-mean-var}, we show that our forest policies still perform well for mean-variance portfolio optimization that allows shortselling, \ie, $\mathcal Z(x) = \braces{z\in\R{d+1}: \Eb{Y^\top z_{1:d} \mid X = x} \ge R,\,\sum_{l = 1}^d z_{l} = 1}$. 
This problem is often feasible, since we no longer enforce the nonnegativity constraints that may be at odd with the conditional mean return constraint. 

\subsubsection*{Violations of Stochastic Constraints.}

Because the stochastic constraints are not known, we need to estimate conditional expectations of $G_k(z; Y)$ at $X=x$ to approximate the constraint set $\Z(x)$ for every query point $x$. 
This is much harder than estimating $z_0$, $f_j(z_0)$, $\nabla^2 f_0(z_0)$, $\nabla f_j(z_0)$, \emph{etc.}, for a {given} fixed $R_0,R_1,R_2$. 
It is akin to the difference between estimating a marginal expectation and estimating a whole regression function.
This means that even when given a fixed forest, we may still need to solve nontrivial estimation subproblems first  for final decision-making. If the constraint set is not approximated accurately, then the resulting decisions may violate the stochastic constraints very often. 

In this setting,  our approach in constructing a policy was to use the forest weights to also approximate the stochastic constraints (see \cref{eq: forest policy2} or line \ref{line: constr} in \cref{alg: forest pred2}), and our approach in constructing the forest was to consider an oracle splitting criterion that enforces only the approximate constraints (\cref{eq:oraclecrit-general}).
Note that for this reason, the oracle splitting criterion might not necessarily encourage splitting on constraint-determining covariates. 
Instead, our focus is on considering stochastic constraints in the splitting criterion for the purpose of approximately  assessing the change in risk at constrained solutions.
Therefore, we may be concerned that using forest weights to approximate stochastic constraints may not estimate the constraints well, and the resulting forest policy may often violate the stochastic constraints. 
For this reason, our approach may be most relevant when violation of the stochastic constraints can be tolerated.
Despite the potential weakness of constraint violation, 
our approach seems to be a reasonable proxy that still works well in practice (provided that infeasibility is tolerable).  
See \cref{sec: empirical-mean-var} for experiments where the constraints and objective even involve completely {different} covariates. 

Considering more robust variations on our approach in the presence of stochastic constraints to reduce constraint violation may constitute fruitful future research.
Indeed, an inherent issue is that the \emph{risk} of constraint violation is not clearly defined -- were it infinite making decisions from data is hopeless, and were it well-defined we may be able to directly address it in the objective.
A possible future direction is the enforcement of stochastic constraints with high probability with respect to the sampling process by using distributionally robust constraints, as done for example by \cite{bertsimas2018data,bertsimas2018robust} in \emph{non-conditional} problems. This may be considered both in the construction of a forest policy given a forest as well as in the construction of the forest itself.
A crucial difference with non-conditional problems is that \emph{in addition} to the \emph{variance} of estimating expectations from a finite sample, which the referenced works tackle, we would also need to consider the inevitable \emph{bias} of estimating a conditional expectation at $X=x$ from a sample where the event $X=x$ is never observed.
While the finite-sample variation may be easier to characterize and introduce robustness for, characterizing the latter bias may involve substantive structural assumptions on how the distribution of $Y\mid X=x$ changes with small perturbations to $x$. And, controlling for such perturbations non-adaptively (\eg, by bounding bias using a Lipschitz assumption) may be very susceptible to the curse of covariate dimensionality.

\subsection{Experiments: Mean-Variance Portfolio Optimization}\label{sec: empirical-mean-var}
\begin{figure*}[t!]\centering%
\begin{subfigure}{0.8\textwidth}\centering%
     \includegraphics[width=\textwidth]{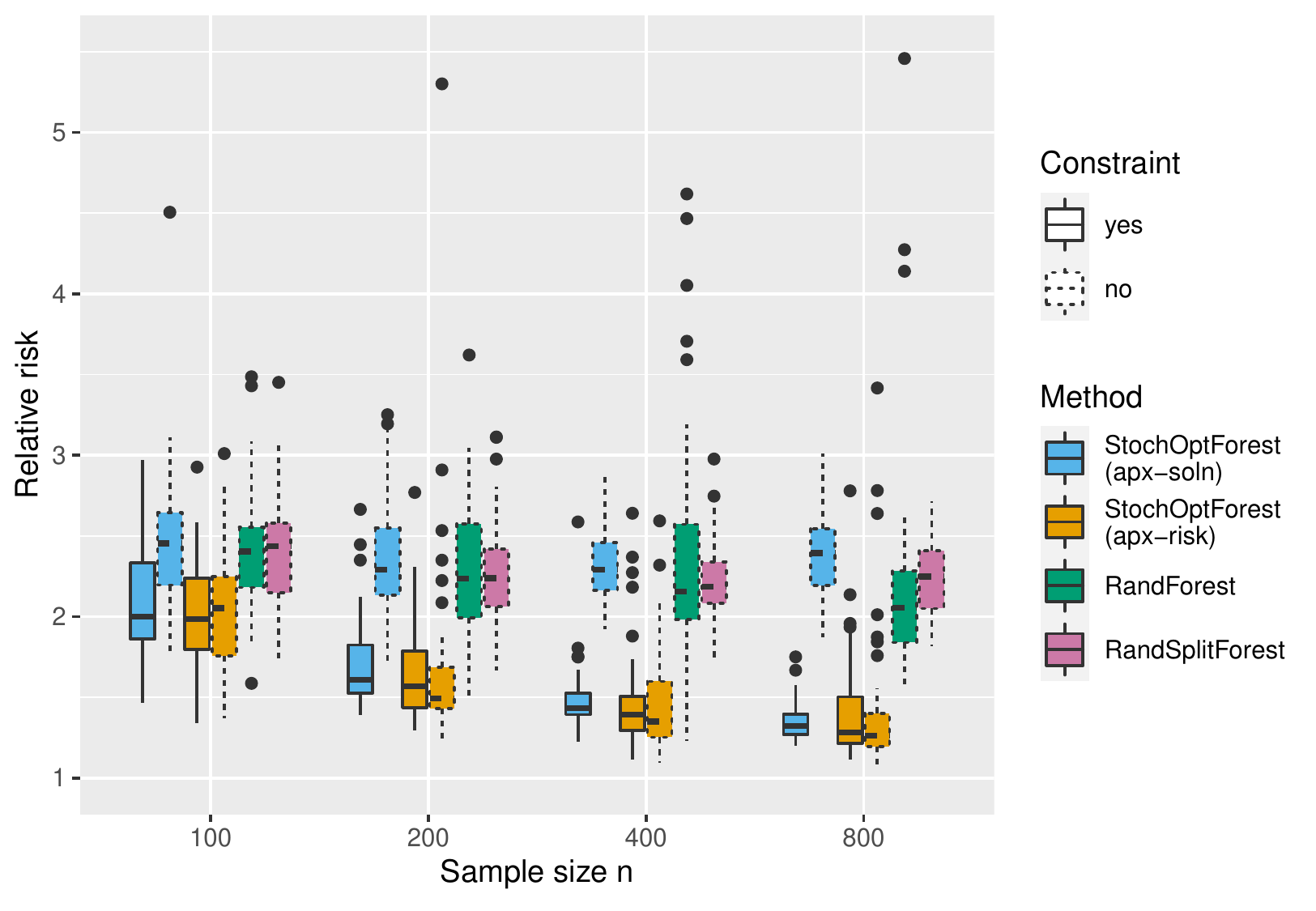}
    \caption{\parbox{0.8\textwidth}{Relative risk of different forest policies (relative to optimal risk with similar mean return).}}
    \label{fig: mean-var-rel_risk-full}
\end{subfigure}%
\begin{subfigure}{0.2\textwidth}\centering%
    \includegraphics[width=\textwidth]{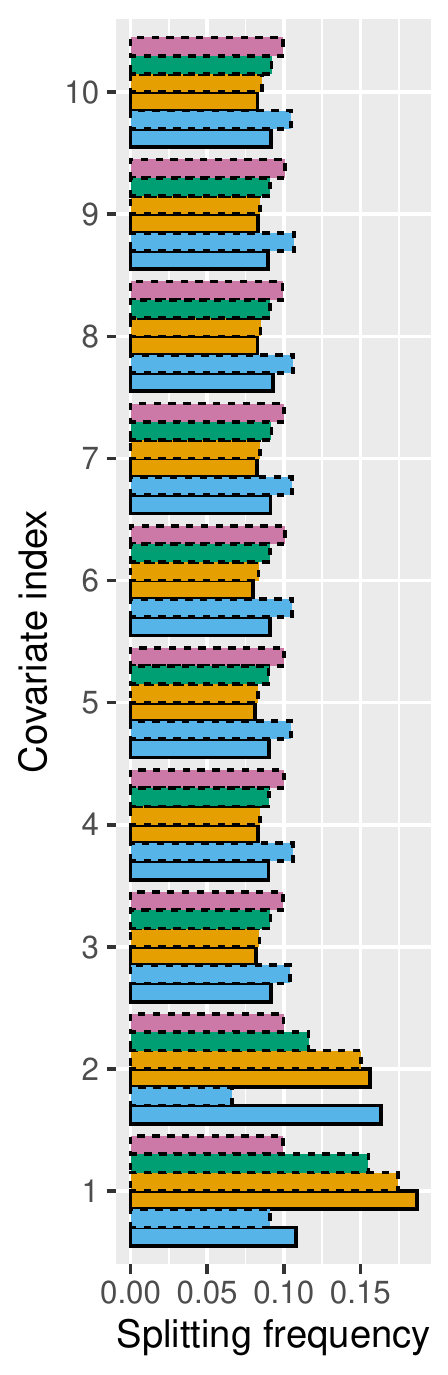}
    \caption{Splitting frequency.}
    \label{fig: mean-var-feature-full}
\end{subfigure}\\
\begin{subfigure}{0.49\textwidth}\centering%
\includegraphics[width=\textwidth]{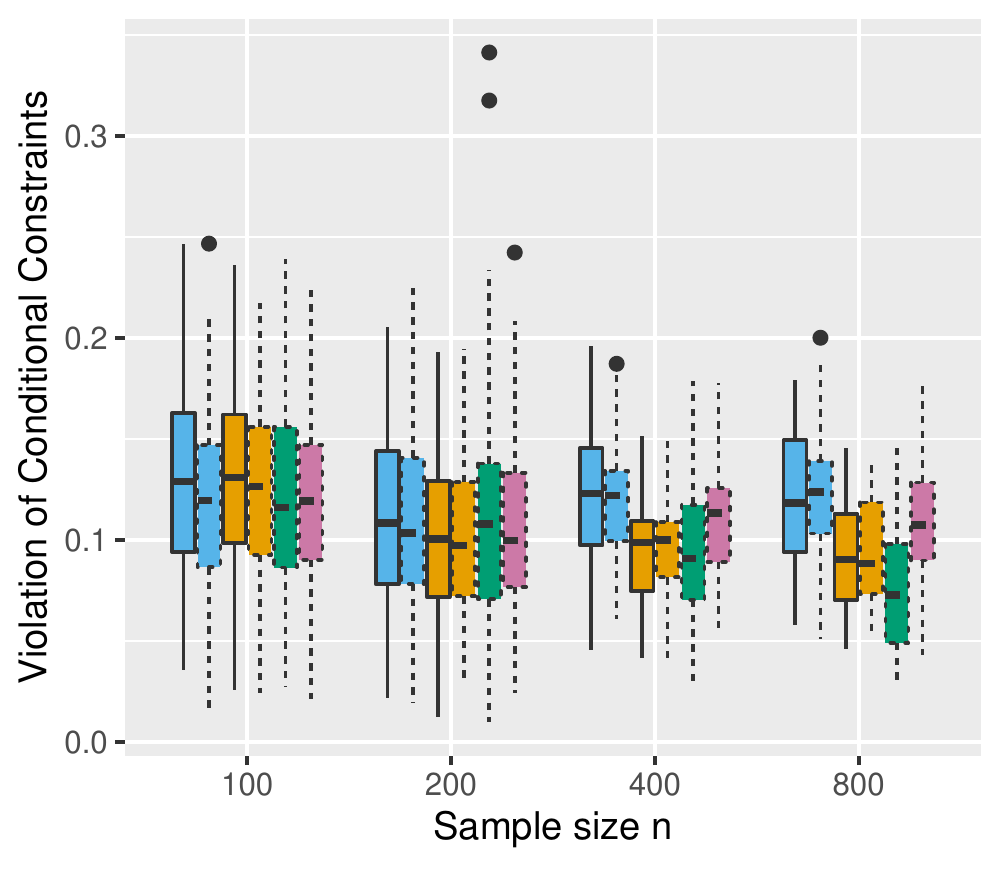}
    \caption{Conditional constraint violation.}
    \label{fig: mean-var-cond-violation}
\end{subfigure}
\begin{subfigure}{0.49\textwidth}\centering%
\includegraphics[width=\textwidth]{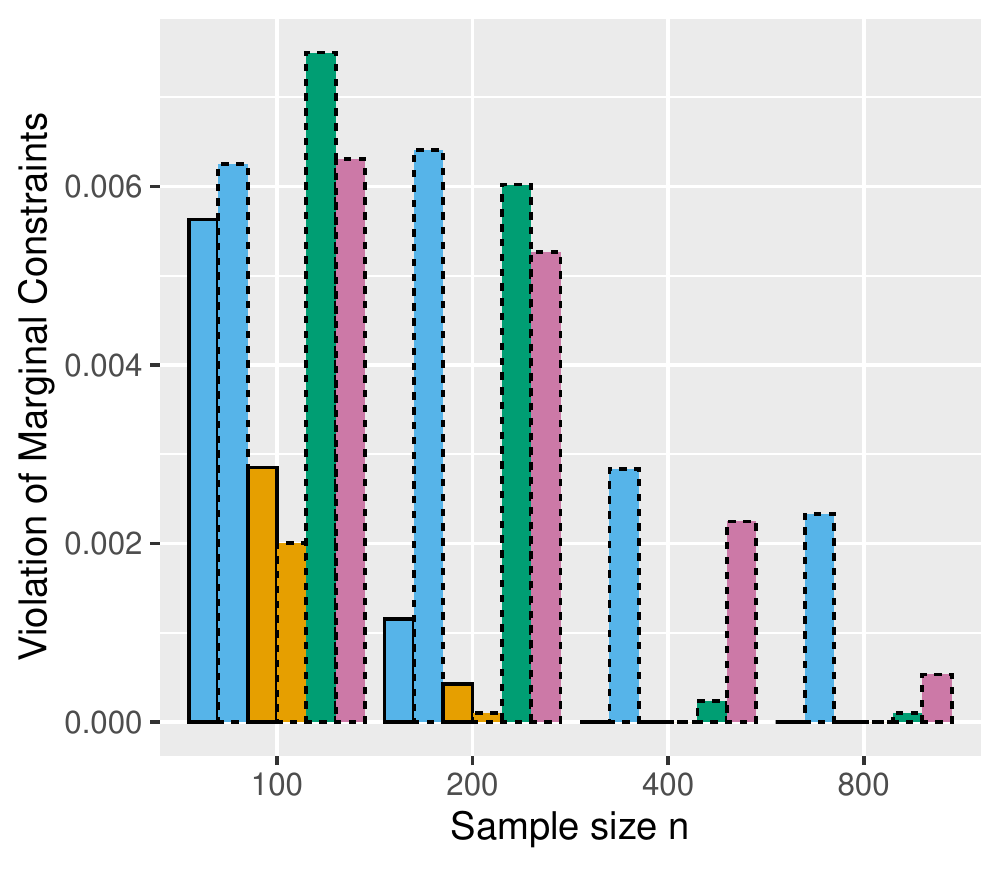}
    \caption{Marginal constraint violation.}
    \label{fig: mean-var-marg-violation}
\end{subfigure}
\caption{Results for mean-variance portfolio optimization CSO problem.}
\end{figure*}

In this section, we apply our methods to the mean-variance portfolio  optimization problem (see also \cref{ex: portfolio-var}): 
we seek an investment policy $z(\cdot) \in \mathbb R^d$ that for each $x$ aims to achieve small risk $\text{Var}\prns{ Y^\top z(x) \mid X=x}$ while satisfying a budget and mean return constraint, \ie, $\mathcal Z(x) = \mathcal Z(x;R) = \braces{z \in \mathbb R^d: \sum_{l = 1}^d z_l(x) = 1, \Eb{Y^\top z(x) \mid X = x} \ge R}$. 
We consider $d = 3$ assets and $p = 10$ covariates. The covariates $X$ are drawn from a standard Gaussian distribution, and the asset returns are independent and are drawn from the conditional distributions $Y_1 \mid X \sim \text{Normal}\prns{\exp(X_1), 5 - 4\indic{-3 \le X_2 \le -1}}$, $Y_2 \mid X \sim \text{Normal}\prns{-X_1, 5 - 4\indic{-1 \le X_2 \le 1}}$, and $Y_3 \mid X \sim \text{Normal}\prns{|X_1|, 5 - 4\indic{1 \le X_2 \le 3}}$.

We compare our StochOptForest algorithm with either the apx-risk and apx-soln approximate criterion for  problems with both deterministic constraints and stochastic constraints (\cref{sec: criteria-stoch-constr}) to four benchmarks: our StochOptForest algorithm with apx-risk and apx-soln criteria that \emph{ignore} the constraints in the forest construction (\cref{sec: approx-crit}), 
the regular random forest algorithm RandForest (which targets the predictions of asset mean returns), and the RandSplitForest algorithm that chooses splits uniformly at random.
We do not compare to StochOptForest with the oracle splitting criterion as it is too computationally intensive, as we investigate further below.
In all forest algorithms, we use 
an ensemble of $500$ trees where the tree specifications are the same as those in \cref{sec: empirical nv}.
To compute $\hat z_0$ in our StochOptTree algorithm (\cref{alg: tree z0 step} line \ref{alg: tree z0 step}) as well as to compute the final forest policy for any forest, we use Gurobi 9.0.2 to solve the linearly-constrained quadratic optimization problem. 
For each $n \in \{100, 200, 400, 800\}$, we repeat the following experiment $50$ times. We first draw a training set $\mathcal D$ of $n$ to fit a forest policy $\hat z(\cdot)$ using each algorithm. Then we sample $200$ query points $x_0$. For each query point $x_0$, we evaluate the conditional risk $\text{Var}\prns{Y^\top \hat z(x_0) \mid X = x_0, \mathcal D}$ using the true conditional covariance matrix $\text{Var}\prns{Y \mid X = x_0}$. 
Note that the forest policy $\hat z(\cdot)$ may not perfectly satisfy the stochastic constraint for conditional mean return, \ie, $\hat R(x_0) = \Eb{Y^\top \hat z(x_0) \mid X = x_0, \mathcal D}$ may be smaller than the pre-specified threshold $R$.
We therefore benchmark its performance against the minimum conditional risk with mean return equal to that of $\hat z(\cdot)$, namely, $z^*(x_0;\hat R(x_0))=\argmin_{z \in \mathcal Z(x_0; \hat R(x_0))}\text{Var}\prns{Y^\top z \mid X = x_0}$, which we compute by Gurobi using the true conditional mean and covariance as input.
We then average these conditional risks over the $200$ query points $x_0$ to estimate $\Eb{\text{Var}\prns{Y^\top \hat z(X) \mid X, \mathcal D} \mid \mathcal D}$ and $\Eb{\text{Var}\prns{Y^\top z^*(X;\hat R(X)) \mid X,\mathcal D}\mid \mathcal D}$. We define the \textit{relative risk} of each forest algorithm for each replication as the ratio of these two quantities. 

\Cref{fig: mean-var-rel_risk-full} shows the distribution of the relative risk over replications for each forest algorithm across different $n$. The dashed boxes corresponding to ``Constraint = no'' indicate that the associated method does not take constraints into account when choosing the splits, which applies to all four benchmarks. 
We can observe that our StochOptForest algorithms with approximate criteria that incorporate constraints achieve the best relative risk over all sample sizes, and their relative risks decrease considerably when the training set size $n$ increases.  
In contrast, the relative risks of RandForest, RandSplitForest, and StochOptForest with the constraint-ignoring apx-soln criterion \edit{decrease very slowly when $n$ increases}. Interestingly, in this example, the performance of StochOptForest algorithm with the constraint-ignoring apx-risk criterion performs similarly to our proposed algorithms that do take constraints into account when choosing splits. 

\Cref{fig: mean-var-feature-full} shows the average frequency of each covariate being selected to be split on in all nodes of all trees constructed by each algorithm over all replications when $n = 800$. We note that RandForest, RandSplitForest, and StochOptForest with the constraint-ignoring apx-soln criterion split much less often on the covariate $X_2$ that governs the conditional variances of asset returns. Since the conditional variances directly determine the objective function in the mean-variance problem, this roughly explains the inferior performance of these methods.

We further evaluate how well the estimated policy $\hat z(\cdot)$ from each forest algorithm  satisfies the mean return constraint. 
In \cref{fig: mean-var-cond-violation}, we present the distribution of the average magnitude of  violation for the \textit{conditional} mean return constraint, \ie, $\Efb{\max\{R - \hat R(X), 0\}\mid \mathcal D}$, over replications for each forest algorithm.
We can observe that for small $n$, all methods have similar average violations, while for large $n$ $(n \ge 400)$,
RandForest appears to achieve the smallest average violation, closely followed by   
StochOptForest with apx-risk criteria (incorporating constraints or not). 
This is consistent with the fact that 
these methods split more often on the covariate $X_1$ that governs the conditional mean returns, as seen in \cref{fig: mean-var-feature-full}. 
However, this relative advantage of RandForest in terms of conditional constraint violation is greatly overshadowed by its bad risk, even relative to its more constrained mean return (\cref{fig: mean-var-rel_risk-full}). More generally, this seeming advantage in constraint satisfaction is largely due to the fact that RandForest is specialized to predict the conditional mean function well, which fully determines the constraint. For stochastic constraints involving a nonlinear function of $Y$, we expect RandForest will not satisfy the constraints well just as it fails to do well in the objective here or in \cref{sec: empirical nv}. 
(See also \cref{sec: disc expectation}.) 
In \cref{fig: mean-var-marg-violation}, we further evaluate the violation magnitude for the \textit{marginal} mean return constraint, \ie, $\max\fbraces{\Efb{R - \hat R(X)}, 0}$,  where the expectation inside is averaged over all $50$ replications. 
We note that the violations for all algorithms are extremely small. This means that the marginal mean return constraint implied by the conditional constraint, \ie, $\Eb{Y^\top \hat z(X) \mid \mathcal D} \ge R$, is almost satisfied for all algorithms. 

In \cref{fig: mean-var-oracle} in \cref{sec: more-mean-var}, we also evaluate the performance of 
StochOptForest algorithm with the oracle criterion in a small-scale experiment, and show that the performance of either apx-soln or apx-risk criterion for constrained problems is close to the oracle criterion, despite the fact they are \textit{much} faster to compute.
In \cref{fig: mean-var-extra} in \cref{sec: more-mean-var}, we additionally show that similar results also hold for different mean return constraint thresholds $R$.

\section{Variable Importance Measures}\label{app-sec: var-importance}

\edit{
In \cref{sec: empirical,sec: asympt-opt}, we show both empirically and theoretically that our StochOptForest algorithm can achieve good decision-making performance for CSO problems. 
However, sometimes we may not only seek quality decisions, but also hope to identify which covariates are important in determining these decisions. 
In this case, measuring the importance of each covariate is very useful. 
}

\edit{
In prediction tasks, the standard random forest algorithm provides two common ways to measure variable importance: impurity-based importance measure and permutation-based importance measure.
Below we first describe these two impurity measures for the regular random forest algorithm, and then based on this we motivate variable importance measures for our StochOptForest algorithm. 
}

\edit{
The impurity-based importance measure is also called the Mean Decrease in Impurity (MDI; see Section 6.1.2, \citealp{louppe2015understanding}), which is based on the impurity measure used in tree splitting, \eg, entropy or Gini index for classification trees and variance for regression trees.  
The MDI of each covariate is a weighted sum of impurity decreases for all tree nodes that split on this covariate, averaged over all trees in a forest. To formalize it, fix a forest consisting of trees $\tau_1, \dots, \tau_T$ and for each internal node $t$ in each tree $\tau_i$ (denoted as $t \in \tau_i$ with slight abuse of notation), denote its splitting covariate as $\hat j_t$, the number of data points reaching the node as $n_t$, and the impurity decrease due to this split as $\Delta I({\hat j_t, t})$. Then the MDI importance measure for a covariate $X_j$ can be written as 
\begin{align}\label{eq: MDI}
\operatorname{MDI}\prns{j} = \frac{1}{T}\sum_{i=1}^{T} \sum_{t \in \tau_i}\indic{\hat j_t = j}{\hat p_t \Delta I({\hat j_t, t})}.
\end{align}
where $\hat p_t =\frac{n_t}{n}$ estimates the probability of an observation reaching the node $t$.
}

\edit{
Another importance measure is based on a permute-and-predict procedure using out-of-bag samples \cite[Section 15.3.2]{hastie2001the}. Suppose we hope to measure the importance of a covariate $X_j$ based on a given random forest. Then for each tree in this random forest, we first record its prediction accuracy on  the out-of-bag samples (\ie, samples that were not used to build this tree), and then compute its prediction accuracy again after randomly permuting the $X_j$ observations in the out-of-bag samples. 
Then we measure the importance of $X_j$ by the decrease in accuracy due to permuting this covariate, averaged over all trees in the random forest.
}

\edit{
It is natural to consider extending these two types of variable importance measures to our StochOptForest algorithm. 
First, consider a direct analogue of the permutation-based importance measure in the decision-making setting: we evaluate the increase in decision cost due to permuting each covariate in each tree, and average them over all trees. 
However, to compute the decisions for out-of-bag samples and evaluate their costs, 
we need to solve optimization problems in all leaf regions of each decision tree.
This can be very time consuming when the trees are deep (so they have many leaf regions) and when there are a large number of trees. 
Therefore, permutation-based importance measures may often be too computationally intensive for our proposed algorithm. 
}

\edit{
Instead, we focus on impurity-based variable importance measures for our proposed algorithm, as it only requires quantities that are already computed in the tree construction process. 
Recall that the impurity-based variable importance measures for the random forest algorithm uses the same impurity measure as that in the tree splitting criterion (\eg, Gini index, entropy, or variance). 
This motivates us to view our proposed tree splitting criteria as the impurity measures.
We first consider the oracle splitting criterion in \cref{eq:oraclecrit}. 
To formalize its variable importance measure, fix an internal node $t_0$ of a tree $\tau_i$ that splits on the covariate $\hat j_{t_0}$, denote its two children nodes as $t_1$ and $t_2$, and denote the probability of an observation reaching these nodes as $p_{t_0}, p_{t_1}, p_{t_2}$ respectively (which can be easily estimated by the fractions of samples reaching these nodes). 
Viewing these three nodes as regions $R_0, R_1, R_2$ respectively, a natural way to measure the impurity decrease due to the split $\hat j_{t_0}$ is 
\begin{align}\label{eq: var-imp-oracle}
\Delta I^{\text{oracle}}({\hat j_{t_0}, t_0}) = v_0\prns{0} - \prns{\frac{p_{t_1}}{p_{t_0}}v_1\prns{1} + \frac{p_{t_2}}{p_{t_0}}v_2\prns{1}} = v_0\prns{0} - \frac{1}{p_{t_0}}\crit^{{oracle}}\prns{R_1, R_2}
\end{align}
When using the approximate risk criterion, we note that $p_{t_0}f_0\prns{z_0} + \crit^\text{apx-risk}(R_1,R_2)$ approximates $\crit^{{oracle}}\prns{R_1, R_2}$ (see \cref{thm:apxriskapx}), so naturally the impurity decrease under the the apx-risk criterion is
\begin{align}\label{eq: var-imp-apxrisk}
\Delta I^{\text{apx-risk}}({\hat j_{t_0}, t_0}) =  -\frac{1}{p_{t_0}} \crit^{{apx-risk}}\prns{R_1, R_2}.
\end{align}
When using the apx-sol criterion, we note that $\crit^\text{apx-soln}(R_1,R_2)$ approximates $\crit^{{oracle}}\prns{R_1, R_2}$ (see \cref{thm:apxsolapx}), so naturally 
\begin{align}\label{eq: var-imp-apxsoln}
\Delta I^{\text{apx-soln}}({\hat j_{t_0}, t_0}) = v_0\prns{0} - \frac{1}{p_{t_0}}\crit^{{apx-soln}}\prns{R_1, R_2}.
\end{align}
Depending on which criterion is used in the StochOptForest, we can estimate the corresponding impurity decrease measure in \cref{eq: var-imp-oracle,eq: var-imp-apxrisk,eq: var-imp-apxsoln} and plug it into \cref{eq: MDI} to quantify the variable importance of each covariate. 
Finally, since the importance measures are relative, we normalize them by assigning the largest a value of $1$ and scaling the others accordingly. 
}

\section{Additional Experimental Details}\label{app-sec: more-empirical}

\subsection{Multi-item Newsvendor}\label{sec: empirical nv}
\begin{figure*}[t!]%
\centering%
\begin{subfigure}{0.78\textwidth}\centering%
    \includegraphics[width=\textwidth]{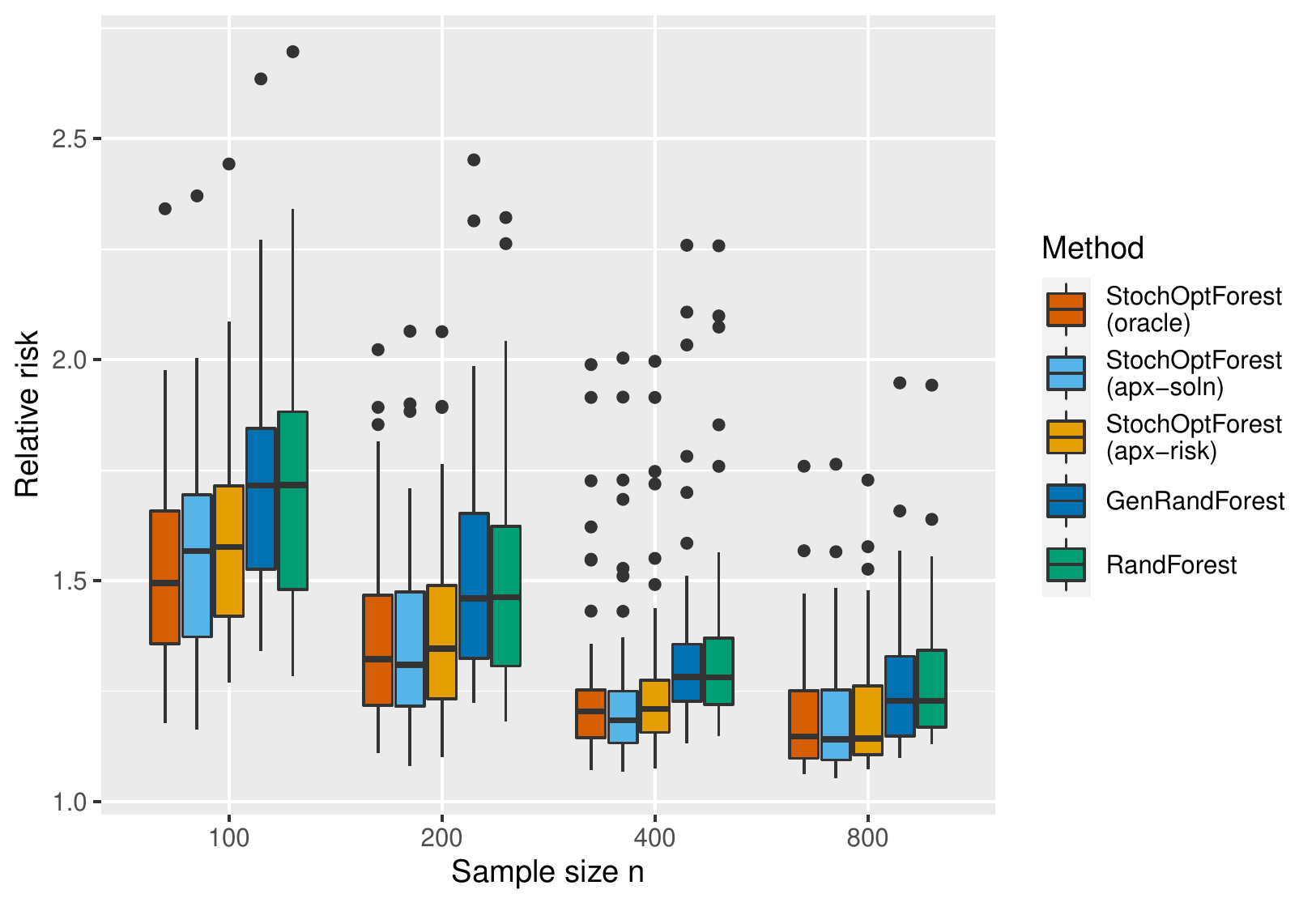}
    \caption{Relative risk of different forest policies.}\label{fig: nv forests risk}
\end{subfigure}%
\begin{subfigure}{0.2\textwidth}\centering%
    \includegraphics[width=\textwidth]{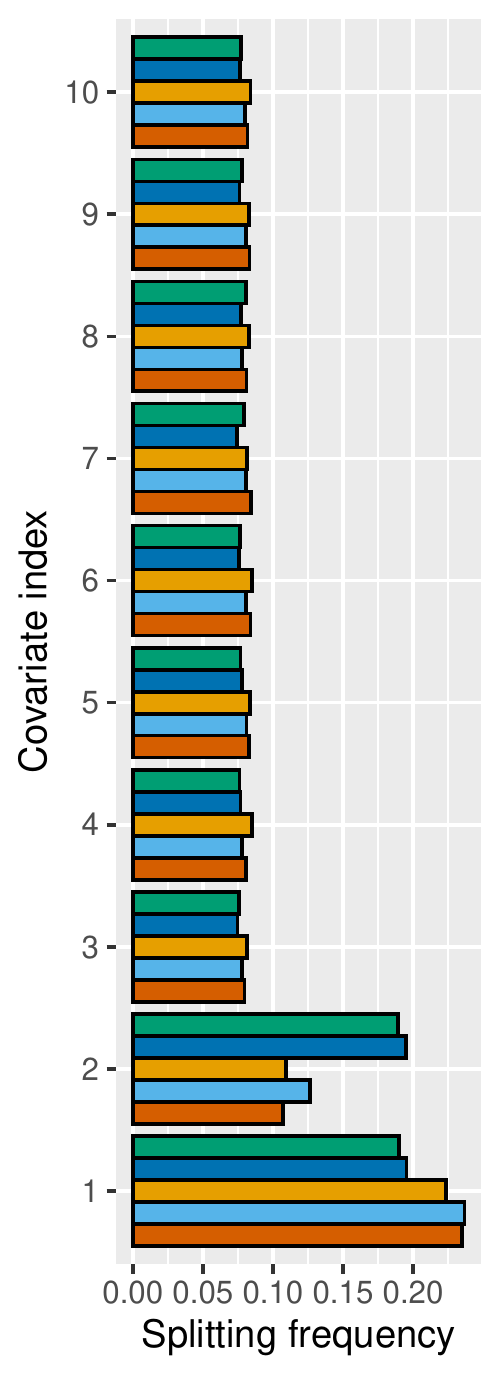}
    \caption{Splitting frequency.}\label{fig: nv forests splits}
\end{subfigure}\\
\begin{subfigure}{\textwidth}\centering%
    \includegraphics[width=\textwidth]{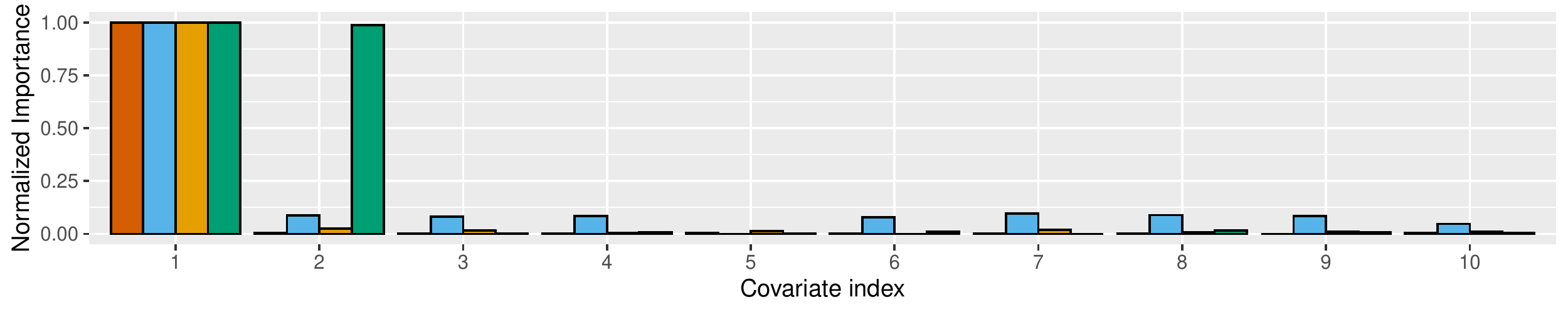}
    \caption{Feature Importance.}\label{fig: nv forests importance}
\end{subfigure}
\caption{Results for the multi-item newsvedor CSO problem with varying $n$ and fixed $p = 10$.}%
\end{figure*}

\begin{figure*}[t!]
\begin{subfigure}{\textwidth}\centering%
   \includegraphics[width=\textwidth]{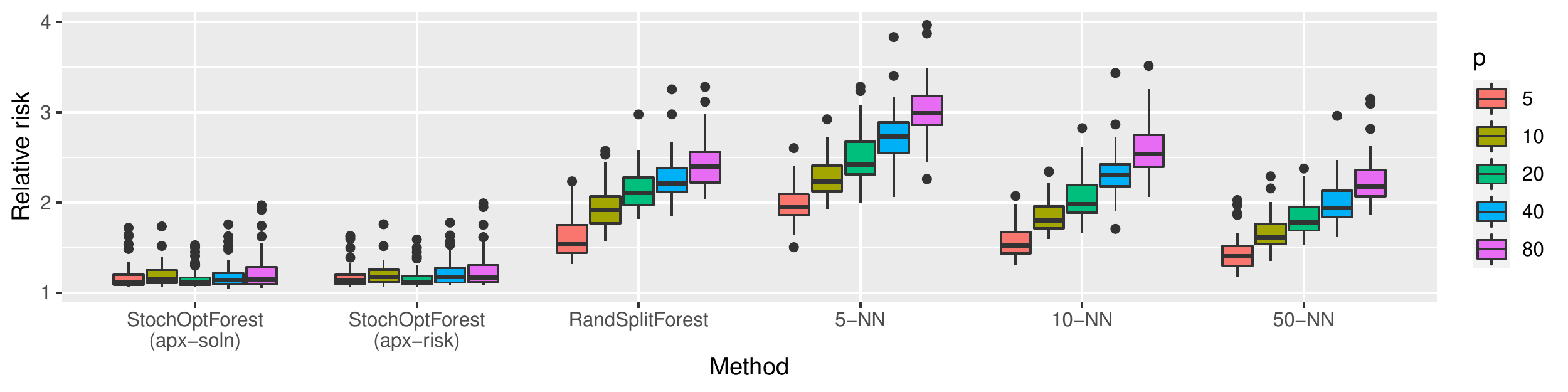}
    \caption{Relative risk for varying $p$ and fixed $n = 800$:  StochOptForest vs. non-adaptive weighting.}\label{fig: nv dim} 
\end{subfigure}\\
\begin{subfigure}{\textwidth}\centering%
   \includegraphics[width=\textwidth]{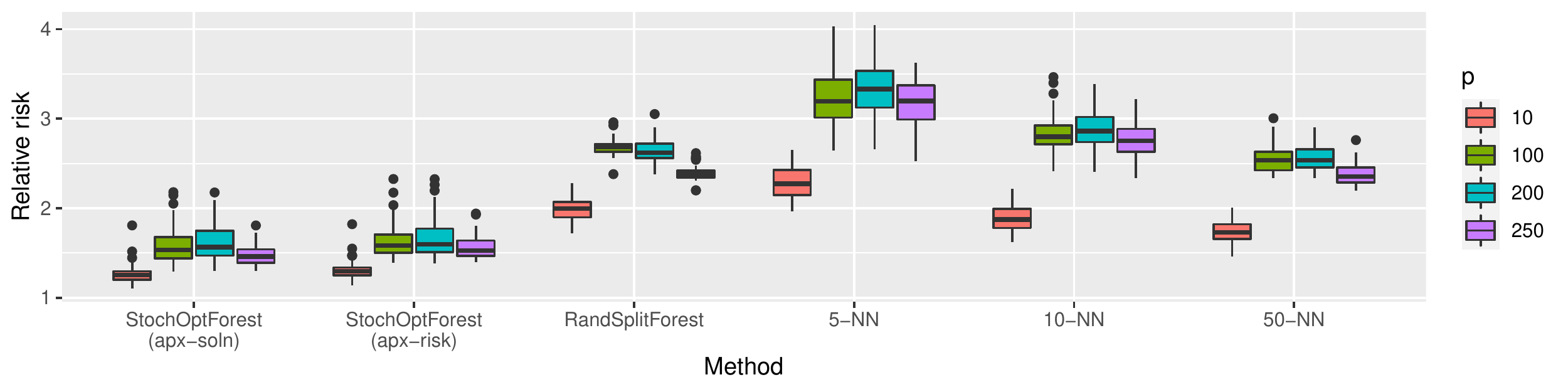}
    \caption{High dimensional setting with fixed $n = 200$.}\label{fig: nv high dim}
\end{subfigure}%
\caption{Results for the multi-item newsvedor CSO problem with fixed $n$ and varying $p$.}%
\end{figure*}

\edit{We here consider an experiment on an unconstrained multi-item newsvendor problem (see \cref{ex: mnv}).}
We consider $d = 2$ products and $p$-dimensional covariates $X$ drawn from a standard Gaussian distribution.
The conditional demand distributions are $Y_1 \mid X \sim \text{TruncNormal}(3, \exp(X_1))$ and $Y_2 \mid X \sim \text{TruncNormal}(3, \exp(X_2))$, where $\text{TruncNormal}(\mu,\sigma)$ is the distribution of $W\mid W\geq 0$ where $W$ is Gaussian with mean $\mu$ and standard deviation $\sigma$. The holding costs are $\alpha_1=5,\,\alpha_2=0.05$ and the backorder costs are $\beta_1=100,\,\beta_2=1$.

We begin by comparing forest policies using different algorithms to construct the forest. We compare our StochOptForest algorithm with either the apx-soln or apx-risk approximate splitting criterion to three benchmarks. All forest-constructing algorithms we consider are identical except for their \emph{splitting criterion}.
One benchmark is StochOptForest with the brute-force oracle splitting criterion, which uses the empirical counterpart to \cref{eq:oraclecrit} (\ie, $\E$ is replaced with $\frac1n\sum_{i=1}^n$) and fully re-optimizes for each candidate split.
A second benchmark is the standard random forest (RandForest) algorithm, which uses the squared error splitting criterion (\cref{ex: prediction}).
Finally, since $z^*(x)$ is the vector of conditional 95\% quantiles of $(Y_1,Y_2)\mid X$, we also consider the GenRandForest algorithm for quantile regression (Example 2 and Section 5 of \citealp{athey2019generalized}; see also \cref{sec: comp to est}). For all forest-constructing algorithms, we use 500 trees, each tree is constructed on bootstrap samples ($\mathcal I_j^\text{tree}=\mathcal I_j^\text{dec}$), candidate splits are all possible splits with at least $20\%$ of observations in each child node, and the minimum node size is $10$. 

To compare these different algorithms, we let $p=10$ and for each $n$ in $\{100,200,400,800\}$ we repeat the following experiment $50$ times. We first draw a training set $\D$ of size $n$ to fit a forest policy, $\hat z(\cdot)$, using each of the above algorithms. Then we sample $200$ query points $x_0$. For each such $x_0$ and for each policy $\hat z(\cdot)$, we compute $\hat z(x_0)$ and then take the average of $c(\hat z(x_0);y)$ over $2000$ values of $y$ drawn from the conditional distribution of $Y\mid X=x_0$. We also compute the average of $c(z^*(x_0);y)$ over these. We average these over the $200$ query points $x_0$. This gives estimates of $\Eb{c(\hat z(X);Y)\mid \D}$ and $\Eb{c(z^*(X);Y)}$. The \emph{relative risk} for each algorithm and each replication is the ratio of these.

In \cref{fig: nv forests risk}, we plot the distribution of relative risk over replications for each forest algorithm and $n$.
The first thing to note is that for $n\geq 400$, the performance of our approximate splitting criteria appear identical to the oracle criterion, as predicted by \cref{thm:apxriskapx,thm:apxsolapx,thm:critconverge}. The second thing to note is that RandForest and GenRandForest have relative risks that are on average roughly 10--16\% worse than our StochOptForest algorithm.

One way to roughly understand these results is to consider how often each algorithm splits on each covariate. Recall there are $p=10$ covariates, the first and the second determine the distribution of the two products, respectively. The first product, however, has higher costs by a factor of $100$. Therefore, to have a well-performing forest policy, we should first and foremost have good forecasts of the demand of the first product, and hence should split very finely on $X_1$. Secondarily, we should consider the second product and $X_2$. 
This is exactly what StochOptForest does. To visualize this, in \cref{fig: nv forests splits}, we consider how often each variable is chosen to be split on in all the nodes of all the trees constructed by each forest algorithm over all replications with $n=400$. We notice that our StochOptForest algorithms indeed split most often on $X_1$, while in contrast algorithms focusing on estimation (RandForest and GenRandForest) split equally often on $X_2$. 
More practically, in CSO problems generally, how important variables are for \edit{estimating optimal decisions} is different than how they impact \edit{decision costs} and the latter is of course most crucial for effective decision making. StochOptForest targets this by directly constructing trees that target their decision risk rather than estimation accuracy.
\edit{In \cref{fig: nv forests importance}, we also plot the impurity-based variable importance measure for each forest algorithm (see \cref{app-sec: var-importance}). 
We do not include the GenRandForest algorithm there since \cite{athey2019generalized} does not provide any variable importance measure. 
Overall the results in \cref{fig: nv forests importance} are consistent with those in \cref{fig: nv forests splits}: our proposed algorithms value $X_1$ the most, while the RandForest algorithm attaches equal importance to both $X_1$ and $X_2$, which again confirms that our proposed method capture signals more relevant to the optimization problem.}

{Finally, we comment on how StochOptForest handles high dimensional features  effectively.} We first consider $n=800$ and vary $p$ in $\{5,10,20,40,80\}$. We compare to non-adaptive weighting methods for CSO, which construct the local decision weights $w_i(x)$ without regard to the data on $Y$ or to the optimization problem \citep{bertsimas2014predictive}. Specifically, we consider two non-adaptive weighting schemes: $k$-nearest neighbors ($k$NN), where $w_i(x)=1/k$ for the $X_i$ that are the $k$ nearest to $x$, and random-splitting forest (RandSplitForest), where trees are constructed by choosing a split uniformly at random from the candidate splits (this is the extreme case for the Extremely Randomized Forests algorithm, \citealp{geurts2006extremely}). We plot the relative risks (computed similarly to the above) for each algorithm and $p$ in \cref{fig: nv dim}. As we can see, non-adaptive methods get worse with dimension due to the curse of dimensionality, while the risk of our StochOptForest algorithms remains stable and low.
\edit{In \cref{fig: nv high dim}, we consider a more challenging setting where the covariate dimension can  be as large as or larger than the sample size: we fix $n = 200$ and increase the covariate dimension from $p = 10$ to $p = 250$. We can observe that the performance of our proposed methods does deteriorate when the covariate dimension is very high, but they still significantly outperform the non-adaptive methods. 
Interestingly, when the dimension grows from $p=200$ to $p=250$, the performance of all methods slightly improve, which is somewhat inconsistent with the conventional wisdom of ``curse of dimensionality.'' We do not have very good explanations for this phenomenon, but we conjecture that this may be related to counter-intuitive behaviors of interpolating estimators in supervised learning \citep{bartlett2020benign,hastie2020surprises,belkin2018overfitting,Belkin15849}. 
For example, it was observed that in linear regression, when the regressor dimension exceeds the sample size, further increasing the dimension may actually improve the out-of-sample prediction performance as long as we focus on the minimum-norm solution. 
Studying this phenomenon in an optimization context is out of the scope of this paper and we leave it for future study.}

\subsection{More details for CVaR Portfolio Optimization}\label{sec: more-cvar}

\subsubsection*{Additional details for \cref{sec: cvar-empirical}.}
\edit{For all algorithms in \cref{sec: cvar-empirical}, the forest specifications are the same as those in \cref{sec: empirical nv}: each forest consists of 500 trees, each tree is constructed on bootstrap samples ($\mathcal I_j^\text{tree}=\mathcal I_j^\text{dec}$), candidate splits are all possible splits with at least $20\%$ of observations in each child node, and the minimum node size is $10$.}

\edit{To evaluate the the relative risks of different forest policies, we follow the testing data generation process in \cref{sec: empirical nv}. We first sample $200$ query points $x_0$ from the marginal distribution of $X$. For each such $x_0$ and for each policy $\hat z(\cdot)$, we compute $\op{CVaR}_{0.2}\prns{Y^\top\hat z(x_0) \mid X = x_0}$ based on $2000$ values of $y$ drawn from the conditional distribution of $Y\mid X=x_0$. 
We also compute $\op{CVaR}_{0.2}\prns{Y^\top z^*(x_0) \mid X = x_0}$ based on the same data. Then we average these over the $200$ query points $x_0$ to estimate $\Eb{\text{CVaR}_{0.2}\prns{ Y^\top \hat z(x) \mid X}\mid \mathcal{D}}$ and $\Eb{\text{CVaR}_{0.2}\prns{ Y^\top  z^*(x) \mid X}}$. The \emph{relative risk} for each algorithm and each replication is the ratio of these, which we plot in \cref{fig: cvar forests risk}.}

\edit{Computing the benchmark splitting criteria that ignore the constraints (our approximate criteria that mistakenly ignore the constraints and the GenRandForest algorithm) requires inverting Hessian estimates. 
But Hessian estimates for the CVaR objective may often not be invertible. 
When this happens, we add $0.001$ times an identity matrix of conformable size to the Hessian estimates so we can invert them and these splitting criteria that ignore the constraints can still run. 
In contrast, our proposed approximate criteria that incorporate constraints require inverting the left hand side coefficient matrices in \cref{eq: linear-equation}.
These matrices are usually invertible thanks to the constraint gradients therein.}

\begin{figure}[t!]
\centering 
\includegraphics[width=\textwidth]{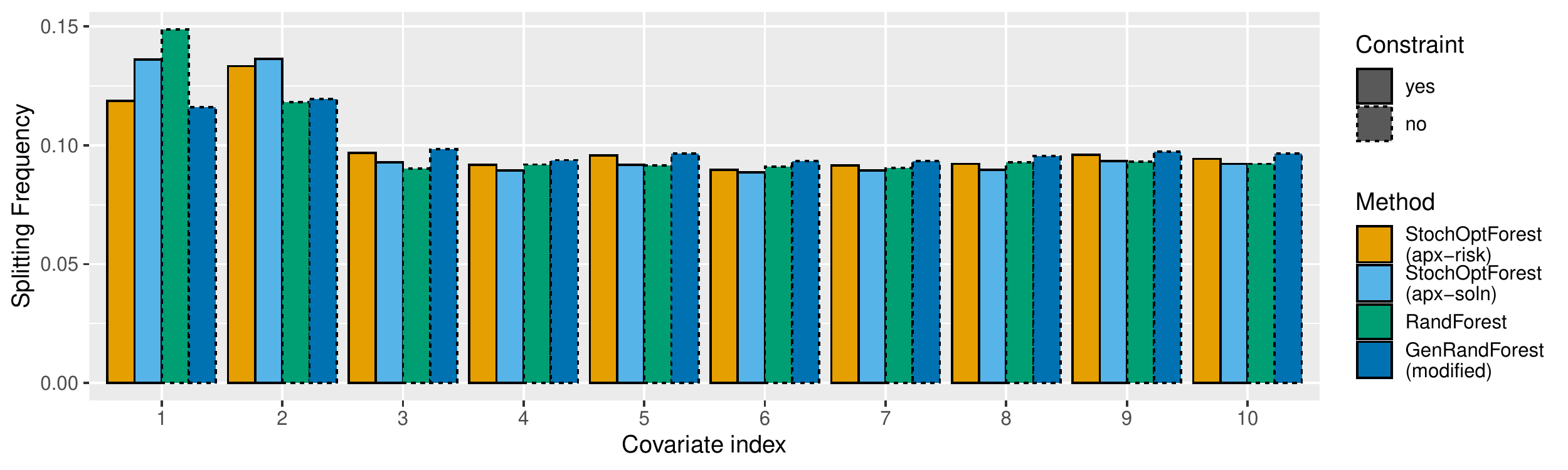}
    \caption{
    The average frequency of splitting on each covariate by different forest policies.}
    \label{fig: cvar-lognormal-split-freq}
\end{figure}
\subsubsection*{Feature Splitting Frequency.} \edit{
    In \cref{fig: cvar-lognormal-split-freq}, we show how often each variable is chosen to be split on in all the nodes of all trees constructed by several forest algorithms over all replications with $n=800$.
    This complements the variable importance measures shown in \cref{fig: cvar forests imp}, offering an alternative way to understand the behaviors of each forest algorithm. 
    We can observe that the RandForest algorithm splits on the first covariate more frequently than any other covariate, as it targets the conditional mean asset returns that are influenced more by the first covariate. This is in line with the observation in \cref{fig: cvar forests imp} that the RandForest algorithm attaches more importance to the first covariate. 
    Moreover, we note that our proposed criteria choose to split on both of the first two covariates very frequently as both of them influence the conditional asset return distributions. At the same time, according to \cref{fig: cvar forests imp}, splits on the second covariate result in  much larger criterion decreases.  
    Finally, we observe that the GenRandForest algorithm also splits on the signal covariates (\ie, the first two covariates) more frequently than any of the noise covariates (\ie, the $3$rd to $10$th covariate), but compared to our proposed methods, the GenRandForest algorithm does still waste more splits on the noise covariates. 
}

\begin{figure}[t!]
\centering 
\includegraphics[width=\textwidth]{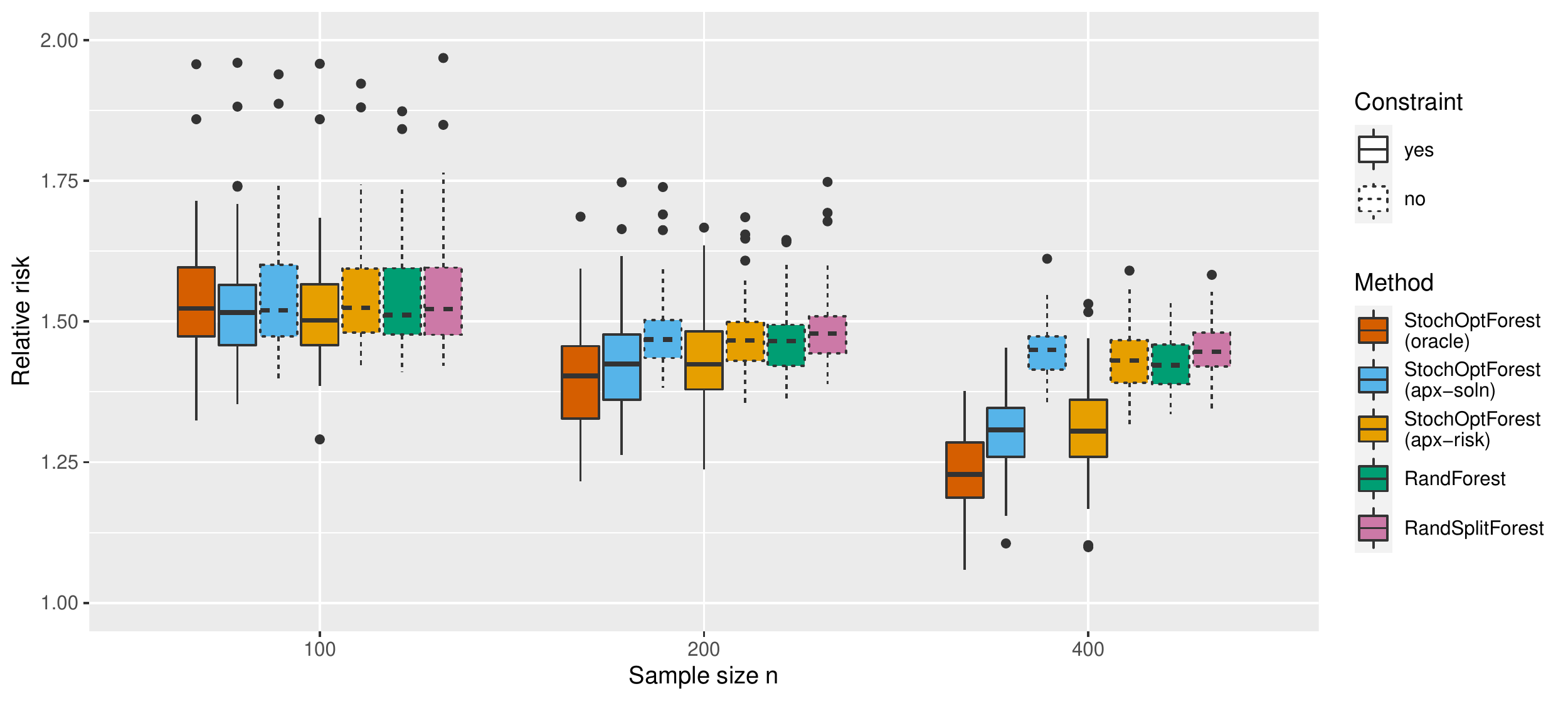}
    \caption{
    Comparing StochOptForest(oracle) with other forest methods in small-scale experiments of 
	CVaR Optimization.}
	\label{fig: cvar-lognormal-oracle}
\end{figure}
\begin{figure}[t!]
\centering 
\includegraphics[width=\textwidth]{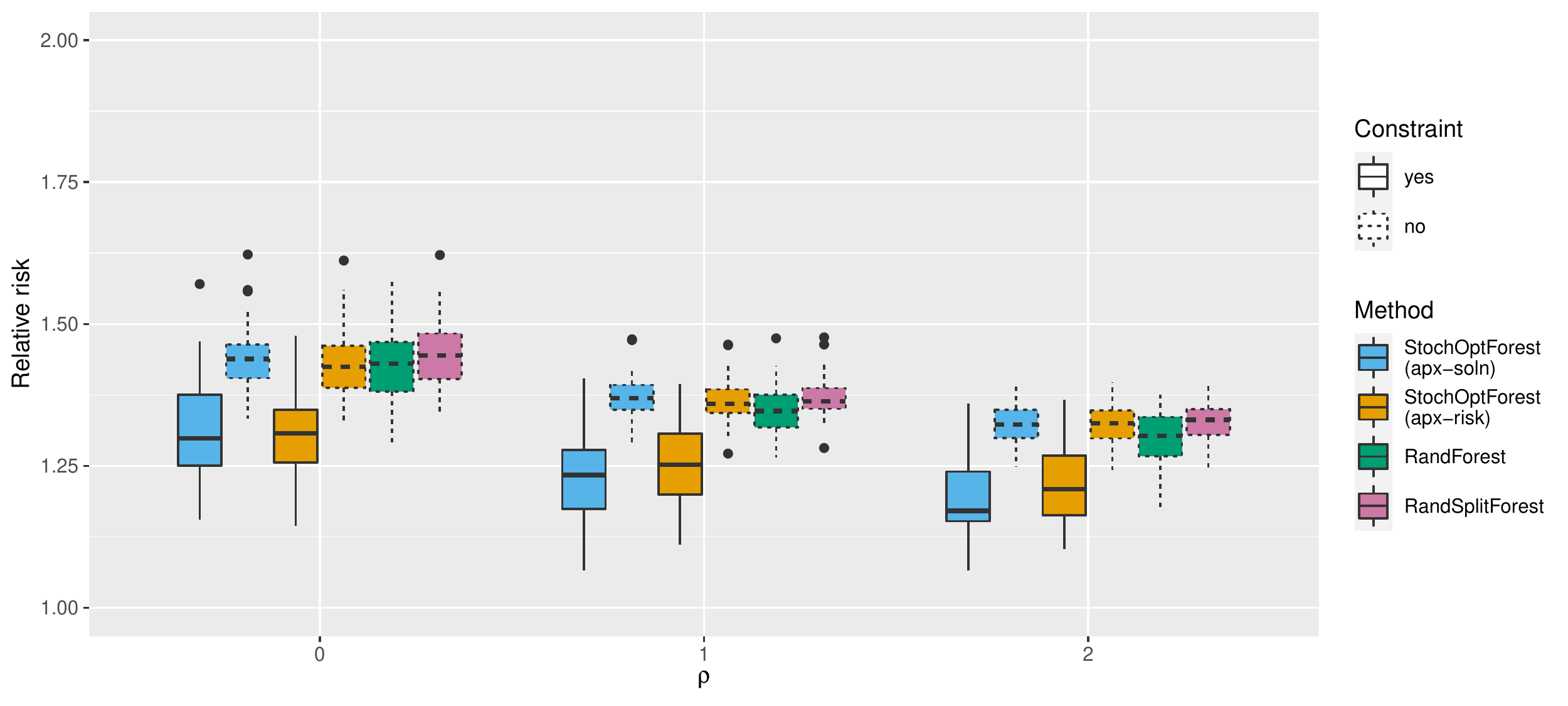}
    \caption{Relative risk of different forest policies for portfolio optimization with a weighted combination of CVaR and mean as objective. The weight of mean is $\rho$.}
    \label{fig: cvar-lognormal-combined}
\end{figure}

\subsubsection*{Performance of StochOptForest (oracle)} We further evaluate the performance of the StochOptForest (oracle) algorithm for CVaR optimization (\cref{sec: cvar-empirical}), but because this algorithm has extremely slow running time (see \cref{table: time-cvar}), we can only do so for a very small-scale experiment. In this experiment, we apply each forest algorithm to construct an ensemble of $50$ trees with the same tree specifications as those in \cref{sec: cvar-empirical}. 
In \cref{fig: cvar-lognormal-oracle}, we show the relative risk of each forest policy over $50$ repetitions for different training data size $n \in \{100, 200, 400\}$.
We can observe that again our StochOptForest algorithms considerably outperform other benchmark methods that do not take the cost structure or constraint structure of CVaR optimization problem into account. 
Moreover, we 
observe that when $n = 400$, the StochOptForest algorithm with the oracle criterion tends to perform better than our approximate criteria. 
However, this observation may be limited to only this small-scale experiment, and we cannot evaluate whether the our approximate criteria and the oracle criterion perform similarly 
 for larger sample size because the StochOptForest algorithm with the oracle criterion is too slow.

\subsubsection*{Linear Combination of CVaR and Mean Return as Objective} In \cref{fig: cvar-lognormal-combined}, we apply the forest algorithms to optimize a linear combination of CVaR and mean returun: $\text{CVaR}_{0.1}(Y^\top z_{1:d} \mid X) - \rho \Eb{Y^\top z_{1:d} \mid X}$ for $\rho \in \{0, 1, 2\}$ and $n = 400$. All other specifications are the same as those in \cref{sec: cvar-empirical}. We observe that across all $\rho$ values, our StochOptForest methods with constraints-aware approximate criteria  perform the best. 

\begin{figure*}[ht]
\begin{subfigure}{0.4\textwidth}
    \includegraphics[width=\textwidth]{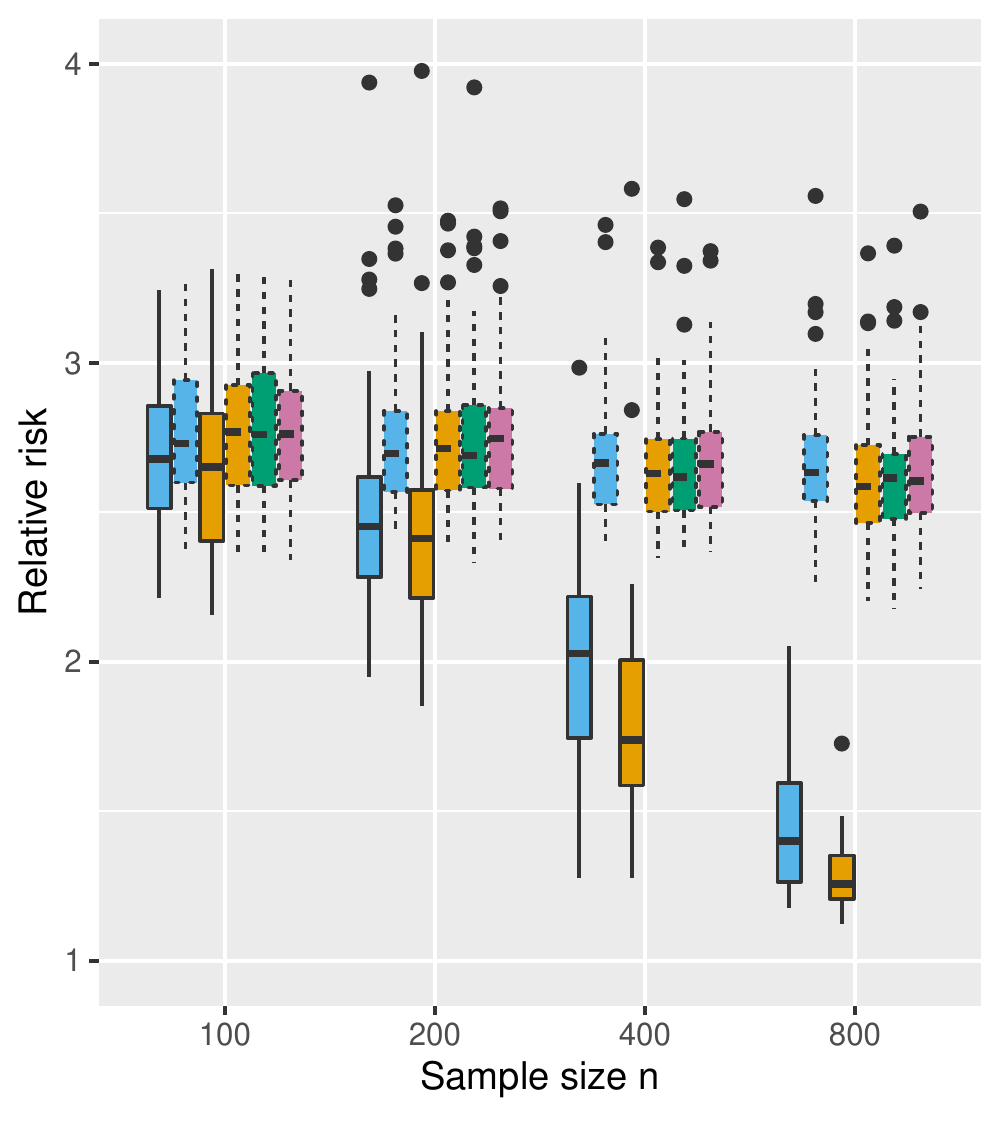}
    \caption{Relative risk of different forest policies.}\label{fig: cvar-normal1}
\end{subfigure}
\begin{subfigure}{0.6\textwidth}
    \includegraphics[width=\textwidth]{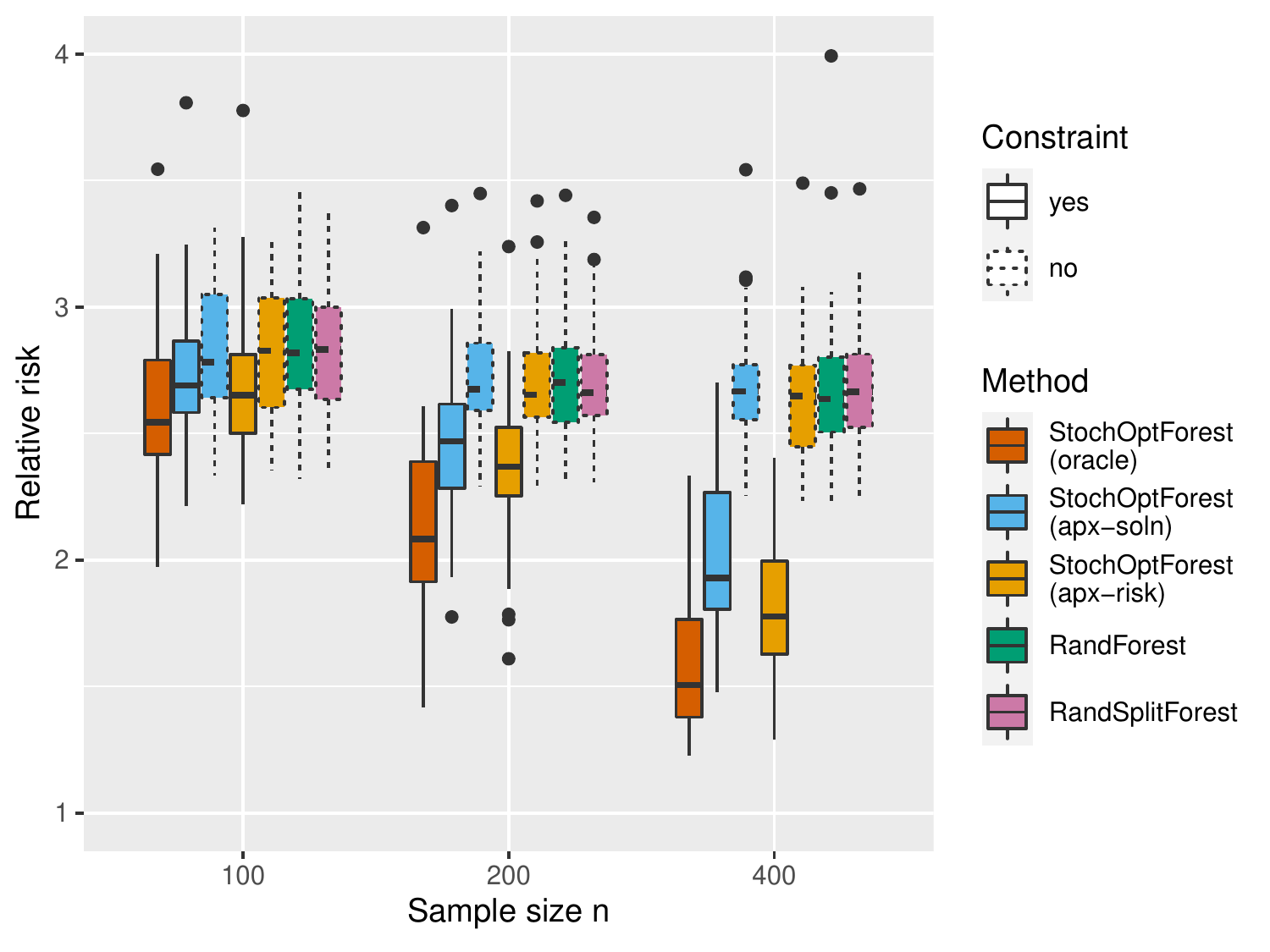}
    \caption{Evaluating StochOptForest(Oracle) policies.}\label{fig: cvar-normal2}
\end{subfigure}%
\caption{CVaR Optimization for asset return data drawn from Gaussian distributions.}\label{fig: cvar-normal}
\end{figure*}

\subsubsection*{Data from Gaussian Distribution} In \cref{fig: cvar-normal}, we present results for CVaR optimization with the asset returns drawn from Gaussian distributions. The experiment setup is the same as that in \cref{sec: cvar-empirical}, except that now the data are drawn from the same Gaussian distributions in \cref{sec: empirical-mean-var}, namely, 
 the covariates $X$ are drawn from a standard Gaussian distribution, and the asset returns are independent and are drawn from the conditional distributions $Y_1 \mid X \sim \text{Normal}\prns{\exp(X_1), 5 - 4\indic{-3 \le X_2 \le -1}}$, $Y_2 \mid X \sim \text{Normal}\prns{-X_1, 5 - 4\indic{-1 \le X_2 \le 1}}$, and $Y_3 \mid X \sim \text{Normal}\prns{|X_1|, 5 - 4\indic{1 \le X_2 \le 3}}$.

 In \cref{fig: cvar-normal1}, we again compare the StochOptForest algorithm with our approximate criteria to other benchmarks, using the same tree and forest specifications as we do in \cref{sec: cvar-empirical}. In \cref{fig: cvar-normal2}, we evaluate the oracle criterion for small forests consisting of $50$ trees. We can observe that the results are qualitatively the same as those in \cref{sec: cvar-empirical} based on asset return data drawn from asymmetric lognormal distributions. 

\subsection{More Details for CVaR Shortest Path Problems}\label{apx-sec: short-path}

\edit{In \cref{sec: shortest-path}, we solve a shortest path problem with a conditional CVaR objective using real data from Uber Movement (\url{https://movement.uber.com/}).
Uber Movement provides historical traveling times from one basic geographical unit to another in many major cities worldwide during five periods in each day (AM Peak, 7am to 10am; Midday, 10am to 4pm; PM Peak, 4pm to 7pm; Evening, 7pm to 12am; Early Morning, 12am to 7am).
These traveling times are estimated from all Uber trips that passed the two basic geographical units during the corresponding time.
The meaning of a basic geographical unit may vary across different cities. 
In \cref{sec: shortest-path}, we focus on Los Angeles where the basic geographical unit  is the census tract. 
In particular, we consider a region in downtown Los Angeles consisting of $45$ census tracts, which is depicted in \cref{fig: uber}. We aim to go from an eastmost census tract (green mark, roughly Aliso Village) to a westmost census tract (red mark, roughly MacArthur Park). 
We collected traveling time observations for $d = 93$ edges during each of the five periods in each day of $2018$ and $2019$, where each edge  represents a path from one census tract to one of its neighbors in the region of interest. 
We denote the corresponding traveling times as $Y \in \R{d}$. 
Our goal is to choose a path between the departure point to the destination, denoted by $z \in \braces{0, 1}^d$, to minimize $\text{CVaR}_{0.8}\prns{Y^\top z \mid X = x_0}$ for each covariate value $x_0$ of interest.}

\edit{
\subsubsection*{Optimization Formulation.} This shortest path problem can be represented by a directed graph consisting of $45$ nodes and $93$ edges. 
We denote the set of nodes as $\mathcal{N}$ with the $1$st node as the departure point and the $45$th node as the destination. If there exists an edge from a node $i \in \mathcal{N}$ to a node $j \in \mathcal{N}$, we denote it as $i \to j$, and denote the set of all $93$ edges as $\mathcal{A}$. 
Then for each decision $z \in \braces{0, 1}^d$, we can index its coordinates by $z_{i \to j}$ for $i, j \in \mathcal{N}$ such that $i\to j\in\mathcal A$. Then, $z_{i \to j} = 1$ means that we decide to travel along the edge $i \to j$ and $z_{i \to j} = 0$ means otherwise.}

\edit{
In terms of the notations above, we can write the CVaR shortest path problem as follows:
\begin{align}
&z^*\prns{x} \in \argmin_{z\in\Z} \op{CVaR}_{0.8}\prns{Y^\top z \mid X = x}, \nonumber \\
&\Z=\braces{z\in\R d~~:~~
\begin{array}{ll}
z_{i \to j} \ge 0 & \text{ for any } i\to j \in \mathcal{A} \\
\sum_{j: i \to j \in \mathcal{A}} z_{i \to j} - \sum_{i: j \to i \in \mathcal{A}} z_{j \to i} = 1 & \text{ if } i = 1 \\
\sum_{j: i \to j \in \mathcal{A}} z_{i \to j} - \sum_{i: j \to i \in \mathcal{A}} z_{j \to i} = -1 & \text{ if } i = 45 \\
\sum_{j: i \to j \in \mathcal{A}} z_{i \to j} - \sum_{i: j \to i \in \mathcal{A}} z_{j \to i} = 0 & \text{ for any } i \in \mathcal{N} \setminus \braces{1, 45}
\end{array}}.
\label{eq: flow-preserve}
\end{align}
Note that we do not enforce integer constraints. 
}

\edit{
\subsubsection*{Data Specifications.} We consider four different sample sizes: half-year data (2019.07.01 to 2019.12.31), one-year data (2019.01.01 to 2019.12.31), one-and-half-year data (2018.07.01 to 2019.12.31), and two-year data (2018.01.01 to 2019.12.31). 
We consider $p = 197$ covariates including weather (Temperature, Wind Speed, Precipitation, Visibility in Miles), period dummy variables (AM Peak, Midday, PM Peak, Evening, Early Morning), weekday dummy variables, month dummy variables, $1$-day-lag traveling times along all edges, and $7$-day-lag traveling times along all edges. 
}

\edit{\subsubsection*{Forest Specifications.} In the experiment in  \cref{sec: shortest-path}, all forests use the same specifications except for the tree splitting criterion. 
In particular, they all consist of $100$ trees, where each tree is constructed on bootstrap samples ($\mathcal I_j^\text{tree}=\mathcal I_j^\text{dec}$) and the minimum node size is $10$.
To reduce computation, in every step of tree construction, we do not consider all possible splits. Instead, we generate candidate splits by first randomly selecting $65$ covariates out of the total $197$ covariates (\ie, around $1/3$ of covariates\footnote{This is the default choice in the ordinary random forest algorithm for regression problems.}), then randomly drawing $365$ cutoff values from all possible ones for each of these selected covariate, and finally restricting to the subset of these splits that results in at least $20\%$ of observations in each child node. 
}

\edit{As in the CVaR portfolio optimization experiment in \cref{sec: cvar-empirical}, whenever we need to invert a numerically singular matrix estimate, we add $0.001$ times an identity matrix of conformable size to the matrices to be inverted, as in \cref{sec: more-cvar} .
Unlike \cref{sec: cvar-empirical}, this becomes an issue also for our criteria that do consider constraints. Indeed, the CVaR shortest path problem has integer-valued optimal solutions so it is not particularly smooth, thus the second order perturbation analysis in \cref{thm:secondorder-const} may not strictly hold. 
Nevertheless, the perturbation analysis still provides a principled way to incorporate optimization problem structure into tree splitting criteria while remaining computationally efficient. 
(Moreover note that the estimated Hessians in these singular matrices are based on probably-misspecified Gaussian assumptions so they are approximations anyways).
For the apx-soln criterion, this may lead to approximate solutions that slightly violate the flow preservation constraints in \cref{eq: flow-preserve} so we project the approximate solutions back onto their affine hull, which is fast operation. 
But we do not modify  the apx-risk criterion any further.}

\subsection{Minimum-variance Portfolio Optimization}\label{sec: empirical-min-var}
\begin{figure*}[ht]
\begin{subfigure}{0.8\textwidth}
    \includegraphics[width=\textwidth]{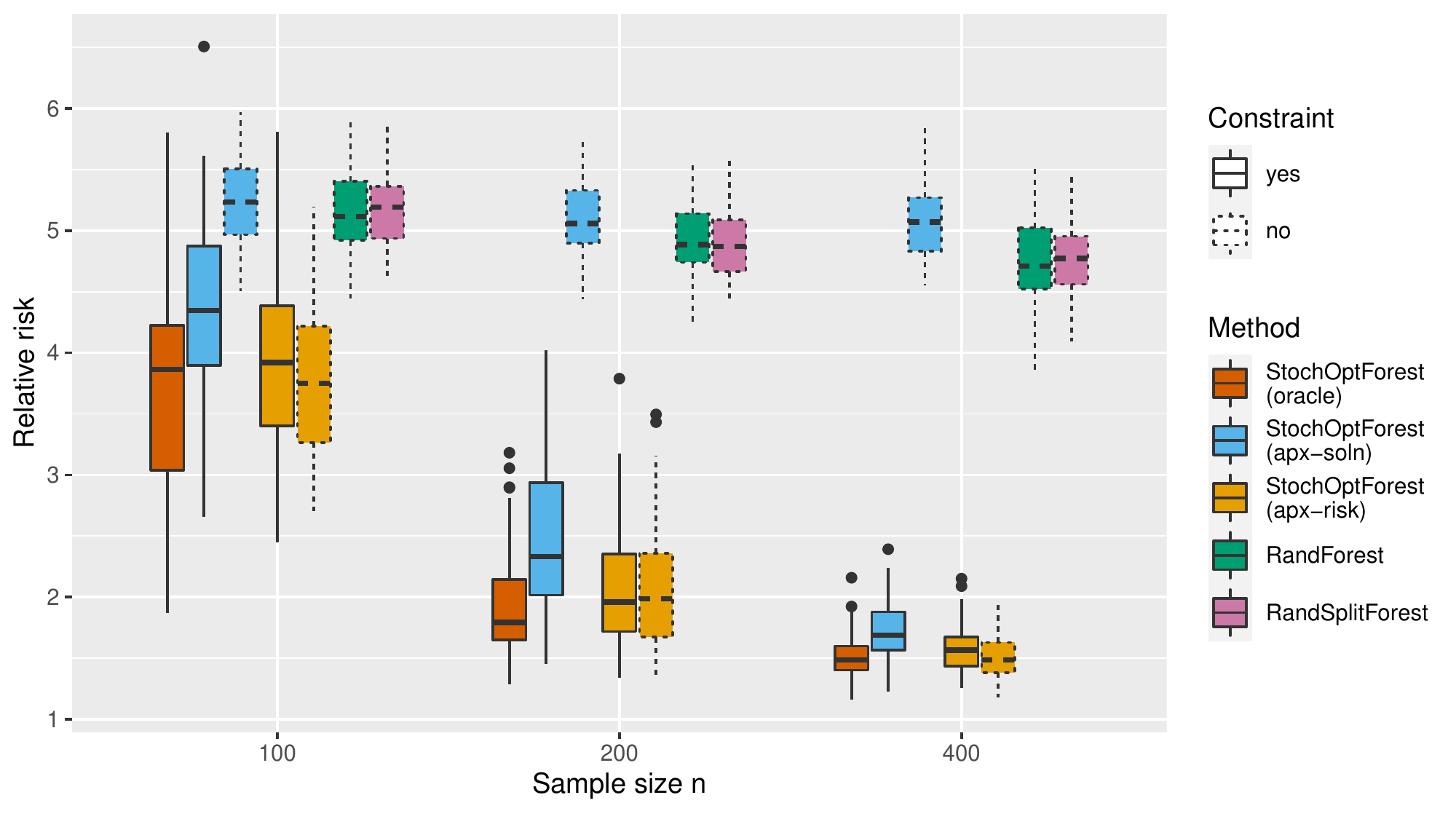}
    \caption{Comparing StochOptForest(oracle) policies with other forest policies in \\small-scale experiments.}\label{fig: minimum-var1}
\end{subfigure}%
\begin{subfigure}{0.2\textwidth}
    \includegraphics[width=\textwidth]{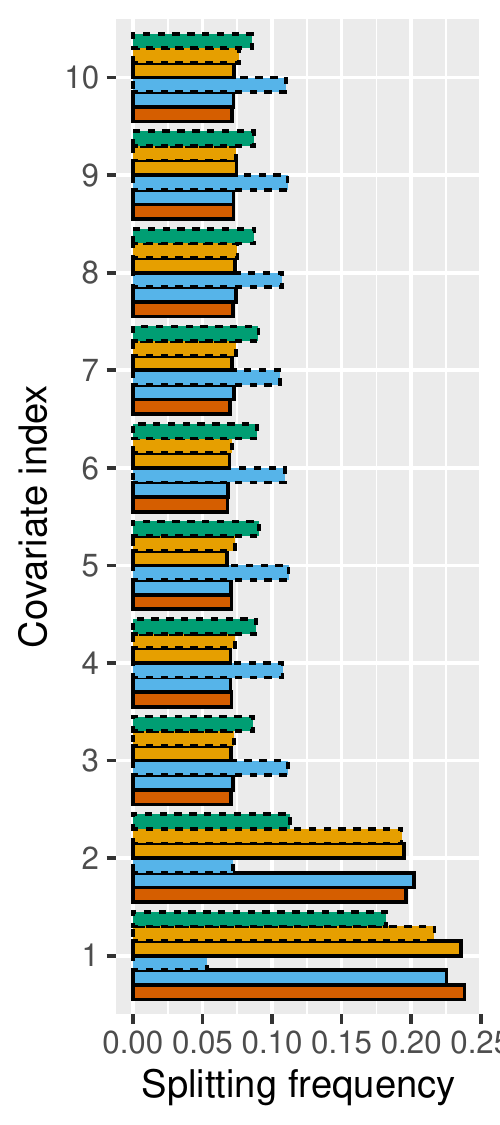}
    \caption{Splitting frequency.}\label{fig: minimum-var2}
\end{subfigure}
\begin{subfigure}{\textwidth}
\includegraphics[width=\textwidth]{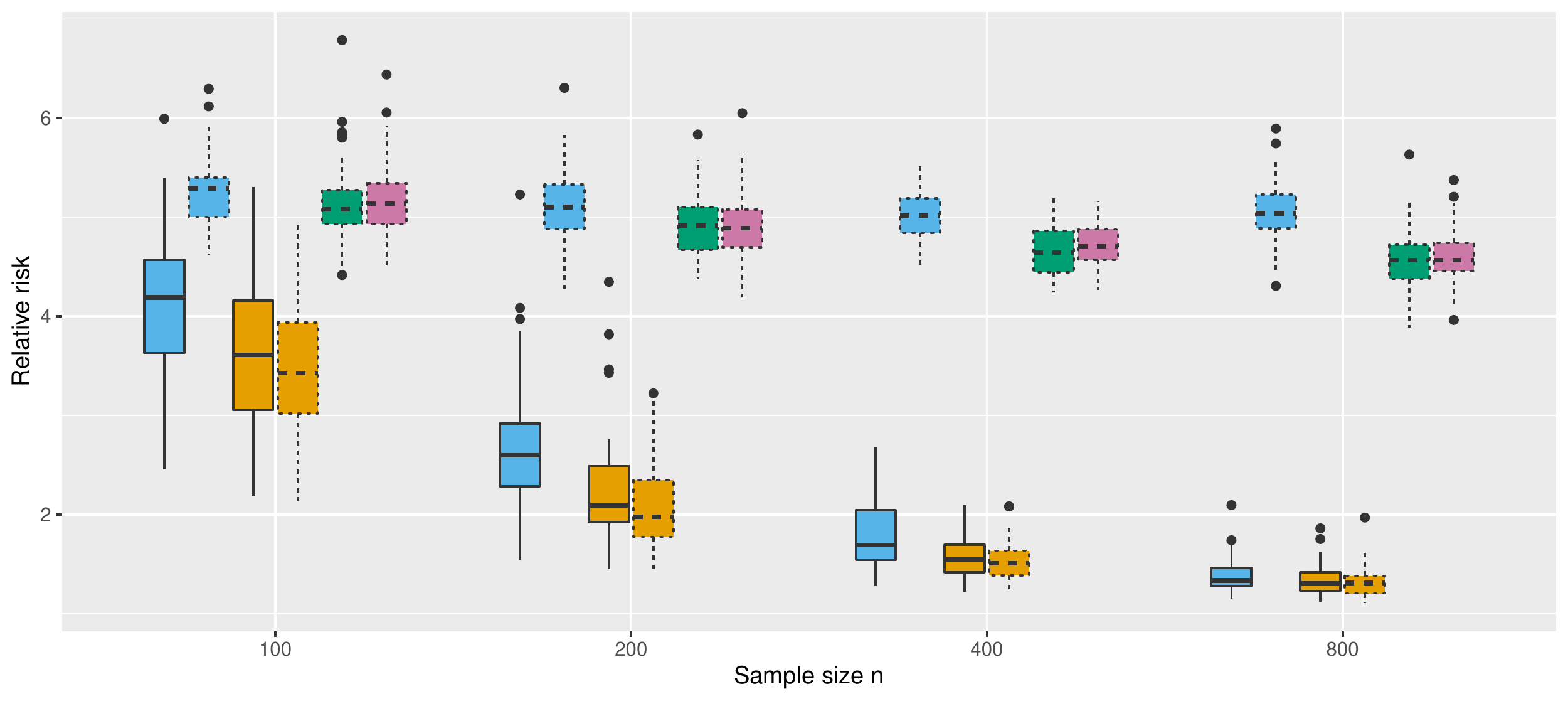}
    \caption{Comparing different forest policies in large-scale experiments.}\label{fig: minimum-var3}
\end{subfigure}
\caption{Results for minimum-variance portfolio optimization without stochastic constraints.}
\label{fig: minimum-var}
\end{figure*}

\begin{table}[ht]
\centering 
\begin{tabular}{|c|c|c|c|}
\hline
Method & $n = 100$ & $n = 200$ & $n = 400$ \\
\hline
StochOptTree (oracle) & $10.24$ ($1.51$) &  $32.60$ ($7.63$) & $90.40$ ($13.67$) \\
\hline 
StochOptTree (apx-risk) & $0.10$ ($0.04$) & $0.23$ ($0.04$) & $0.76$ ($0.09$) \\
\hline 
StochOptTree (apx-soln) & $0.08$ ($0.01$)  & $0.28$ ($0.05$) & $1.48$ ($0.51$)  \\
\hline 
\end{tabular}
\caption{Mean running time (in seconds) of constructing one tree for different algorithms in minimum-variance portfolio optimization over $10$ repetitions. Numbers in parentheses indicate the standard deviation of running time over repetitions.}
\label{table: time-var}
\end{table}

In \cref{fig: minimum-var}, we compare different forest policies for minimizing $\text{Var}(Y^\top z_{1:d} \mid X = x)$ with constraint set $\mathcal Z = \{z \in \mathbb R^{d+1}: z_{1:d} \in \Delta^d\}$. The experiment setup and forest specifications are the same as those in \cref{sec: cvar-empirical}. We only show results for 
return data drawn from Gaussian distributions described in \cref{sec: empirical-mean-var}, and the results for asymmetric lognormal distributions described in \cref{sec: cvar-empirical} are similar so we omit them here.

In \cref{fig: minimum-var1}, we compare the StochOptForest algorithm with the oracle criterion on a small-scale experiment where forests consist of $50$ trees and training data size ranges from $100$ to $400$. We find that the performance of our apx-risk criterion is very close to the oracle criterion, despite that our apx-risk criterion is much faster to compute. All StochOptForest algorithms that account for the optimization structure achieve better performance than the benchmark methods RandForest, RandSplitForest, and StochOptForest with the constraint-ignoring apx-soln criterion.
Interestingly, the StochOptForest algorithm with the constraint-ignoring apx-risk criterion performs quite well, although it fails to incorporate the constraint structure. 
However, we still recommend using approximate criteria that incorporate the constraints, since they consistently perform well across different optimization problems and  ignoring the constraints may undermine the performance. 
For example, in the CVaR optimization experiments in \cref{sec: cvar-empirical}, we find that ignoring the constraints in approximate criteria can considerably hurt their performance. 

Moreover, in 
\cref{fig: minimum-var2} we show the average feature splitting frequencies of different forest algorithms.  We can observe that all well-performing methods frequently split on $X_2$ that determines the conditional variance of asset returns and thus the objective function, while those ill-performing methods typically split on $X_2$ much less often. 
This partly explains the observations in \cref{fig: minimum-var1}. 
In \cref{fig: minimum-var3}, we also evaluate different tree algorithms on larger-scale experiments with forests consisting of $500$ trees and sample size up to $n = 800$, which again shows the superior performance of  our proposed methods.
Finally, we show the running time of each tree algorithm for minimum-variance portfolio optimization in \cref{table: time-var}. We can observe that the StochOptTree algorithm with the apx-risk criterion is more than 100 times faster than the StochOptTree algorithm with the oracle criterion for all sample sizes.

\subsection{Mean-variance Portfolio Optimization}\label{sec: more-mean-var}
\begin{figure*}[ht]
\begin{subfigure}{0.8\textwidth}
    \includegraphics[width=\textwidth]{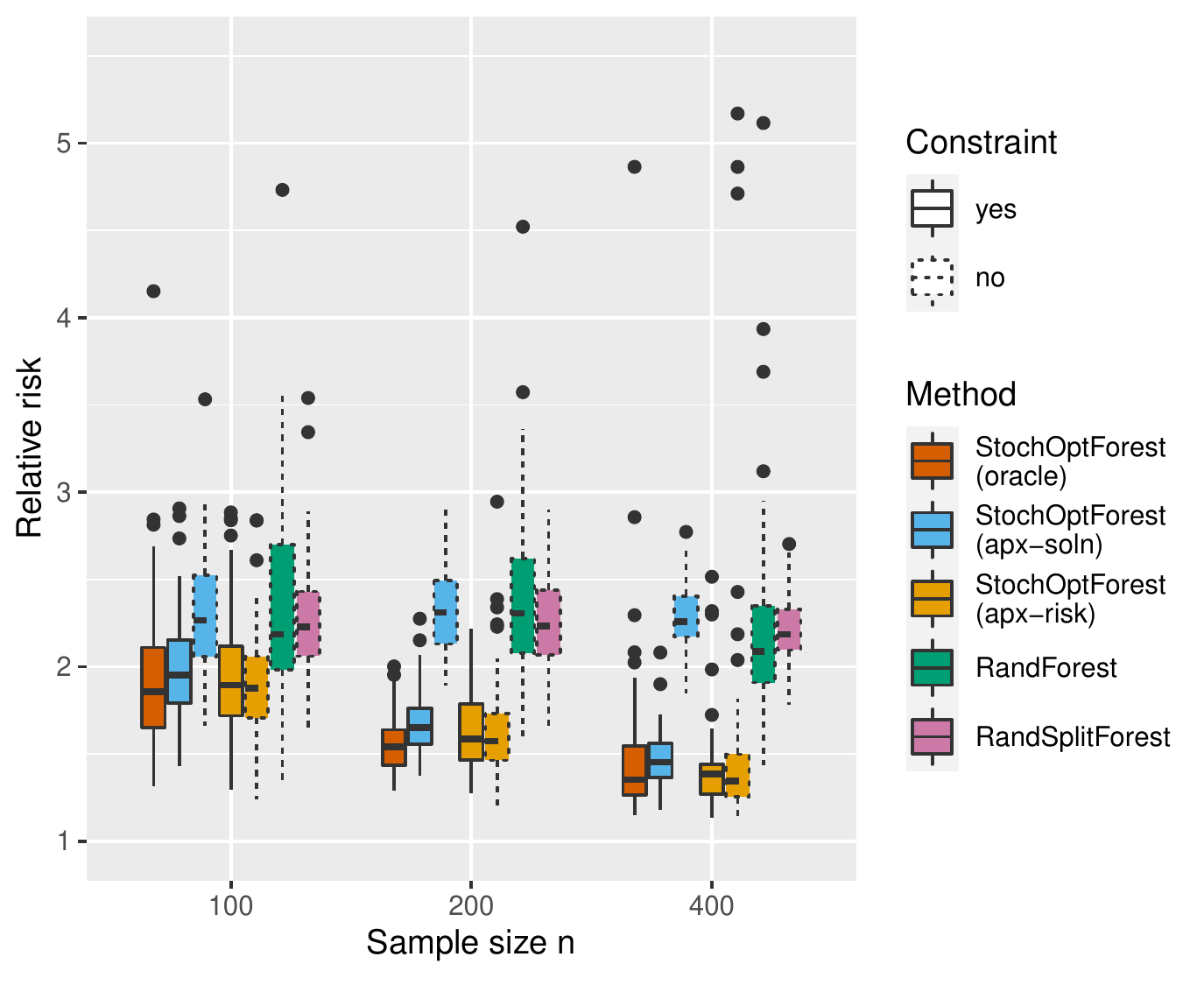}
    \caption{Relative risk of different forest policies.}
\end{subfigure}%
\begin{subfigure}{0.2\textwidth}
    \includegraphics[width=\textwidth]{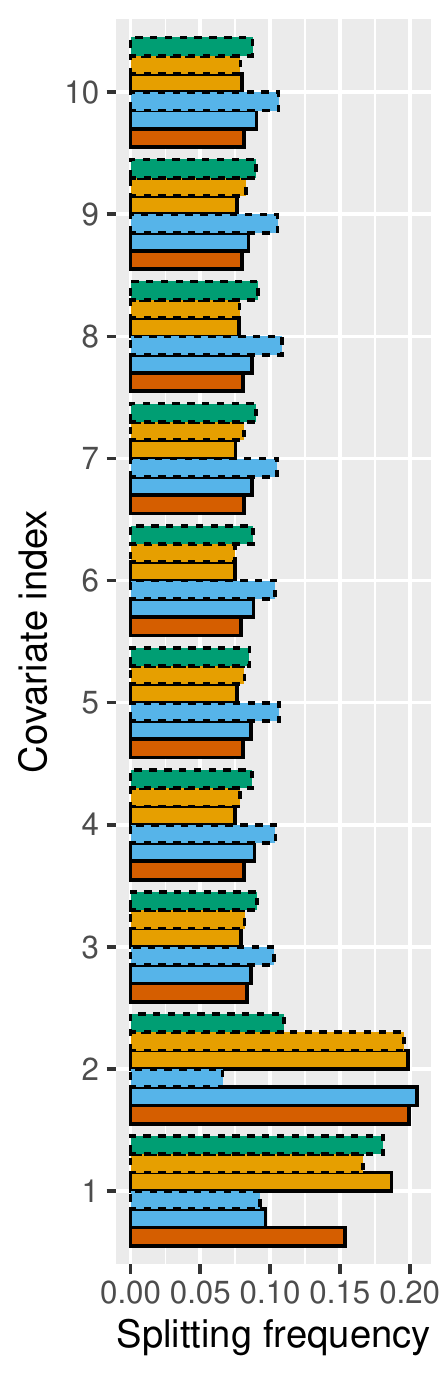}
    \caption{Splitting frequency.}
\end{subfigure}
\caption{Comparing StochOptForest(oracle) with other forest methods in small-scale experiments for mean-variance portfolio optimization.}
\label{fig: mean-var-oracle}
\end{figure*}

\begin{figure*}[ht]
\centering 
\begin{subfigure}{0.5\textwidth}
    \includegraphics[width=\textwidth]{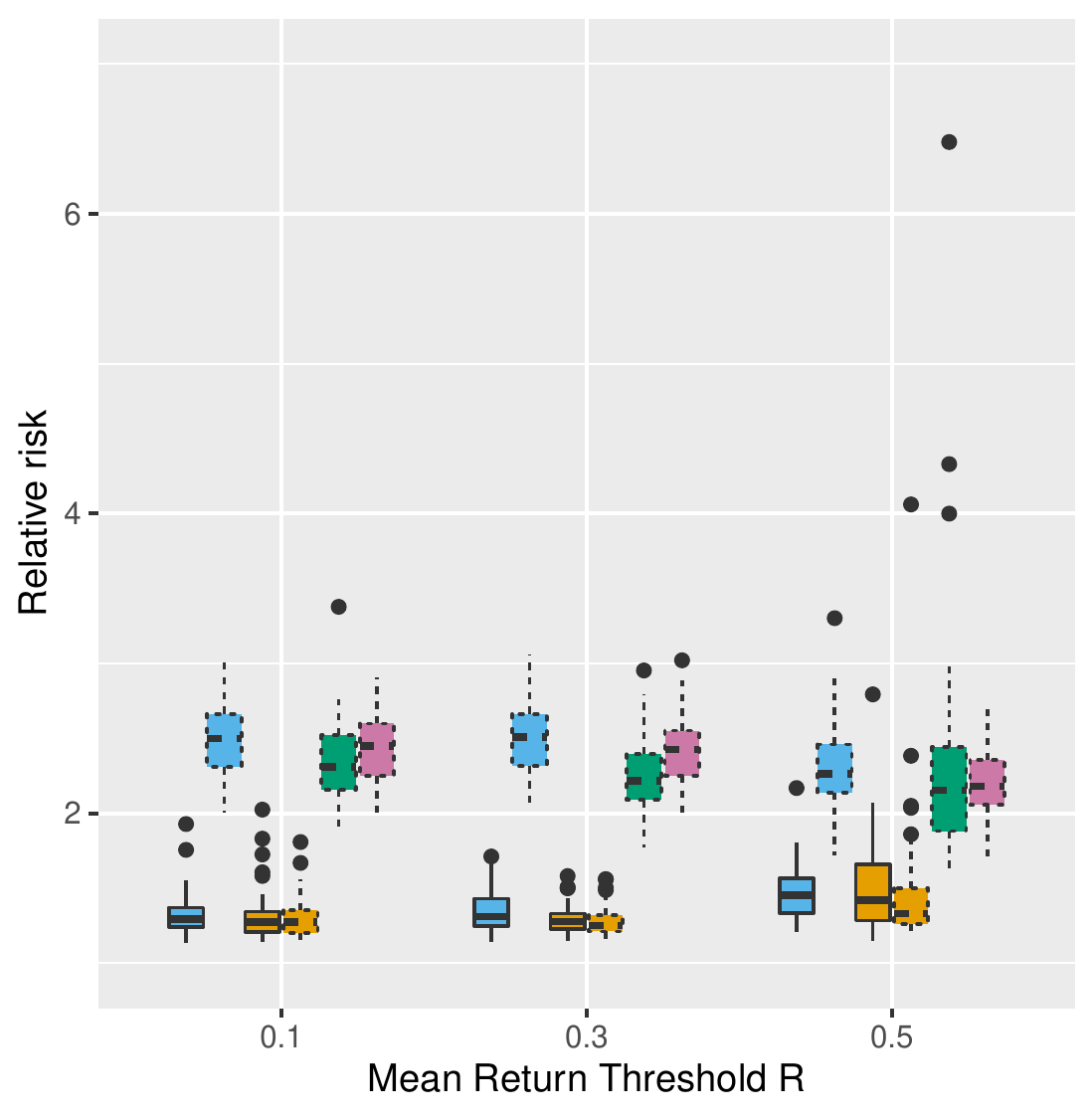}
\end{subfigure}%
\caption{Additional results for mean-variance portfolio optimization experiments in \cref{sec: empirical-mean-var}: relative risks for different return constraint thresholds.}
\label{fig: mean-var-extra}
\end{figure*}

In this section, we provide more experimental results on the mean-variance portfolio optimization in \cref{sec: empirical-mean-var}.

In \cref{fig: mean-var-oracle}, we compare the performance of StochOptForest with oracle criterion with other methods, in particular StochOptForest with our apx-risk and apx-sol approximate criteria. 
Because of the tremendous computational costs of StochOptForest (oracle), here we compare forests consisting of $50$ trees and consider $n$ up to $400$.
We note that the performance of our approximate criteria is very similar to the oracle criterion, and the results for all other methods are similar to those in \cref{fig: mean-var-rel_risk-full}. 

In \cref{fig: mean-var-extra}, we show additional results for the experiments in \cref{sec: empirical-mean-var}. More concretely, 
\cref{fig: rel-risk-R} presents the relative risks of different forest policies when training set size $n = 400$ and the conditional mean return constraint threshold $R$ varies in $\{0.1, 0.3, 0.5\}$. We can see that the performance comparisons are very stable across different thresholds $R$.

\subsection{Honest Forests}\label{sec: honesty}
\begin{figure*}[ht]
\begin{subfigure}{0.44\textwidth}
    \includegraphics[width=\textwidth]{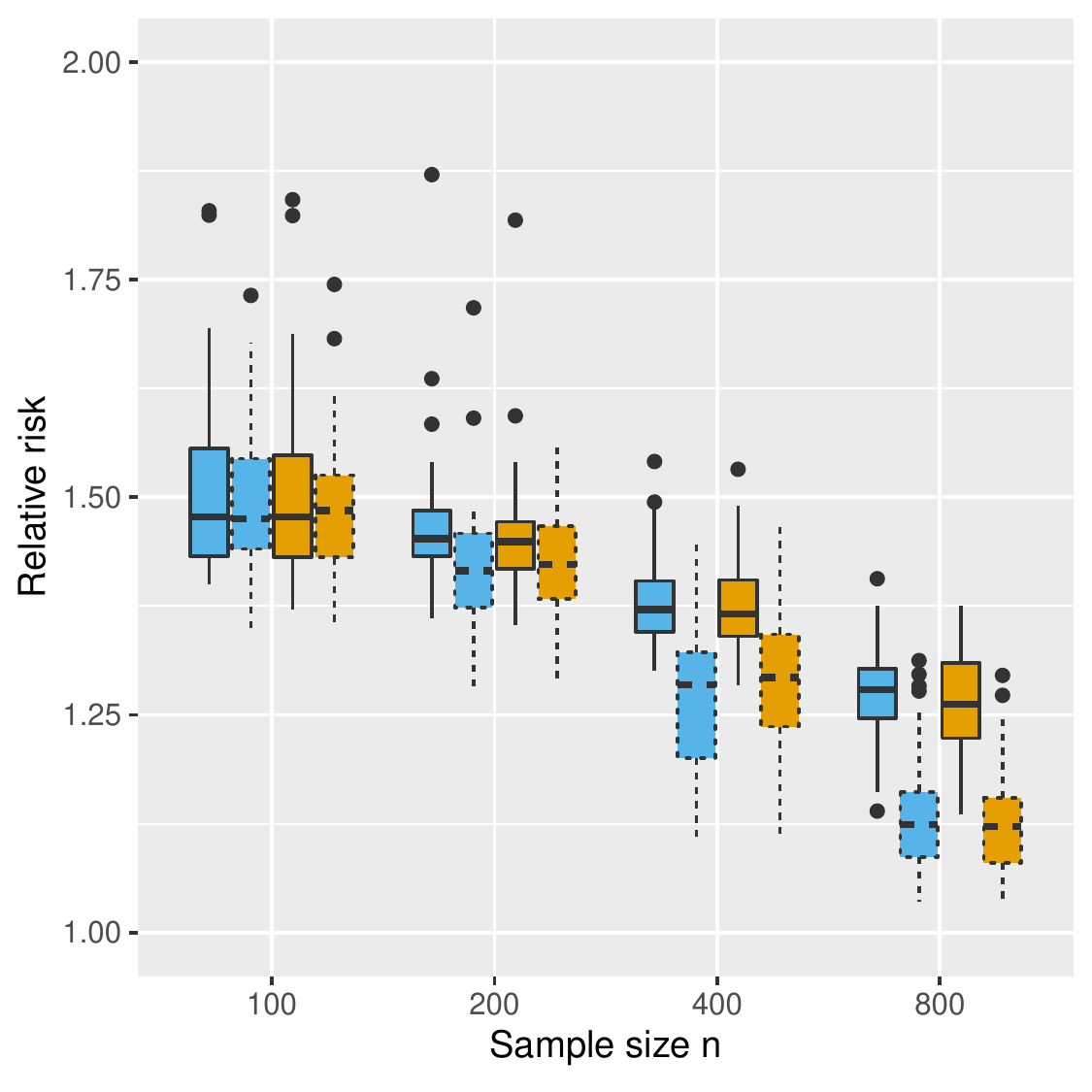}
    \caption{CVaR Optimization.}
    \label{fig: cvar-honesty}
\end{subfigure}%
\begin{subfigure}{0.55\textwidth}
    \includegraphics[width=\textwidth]{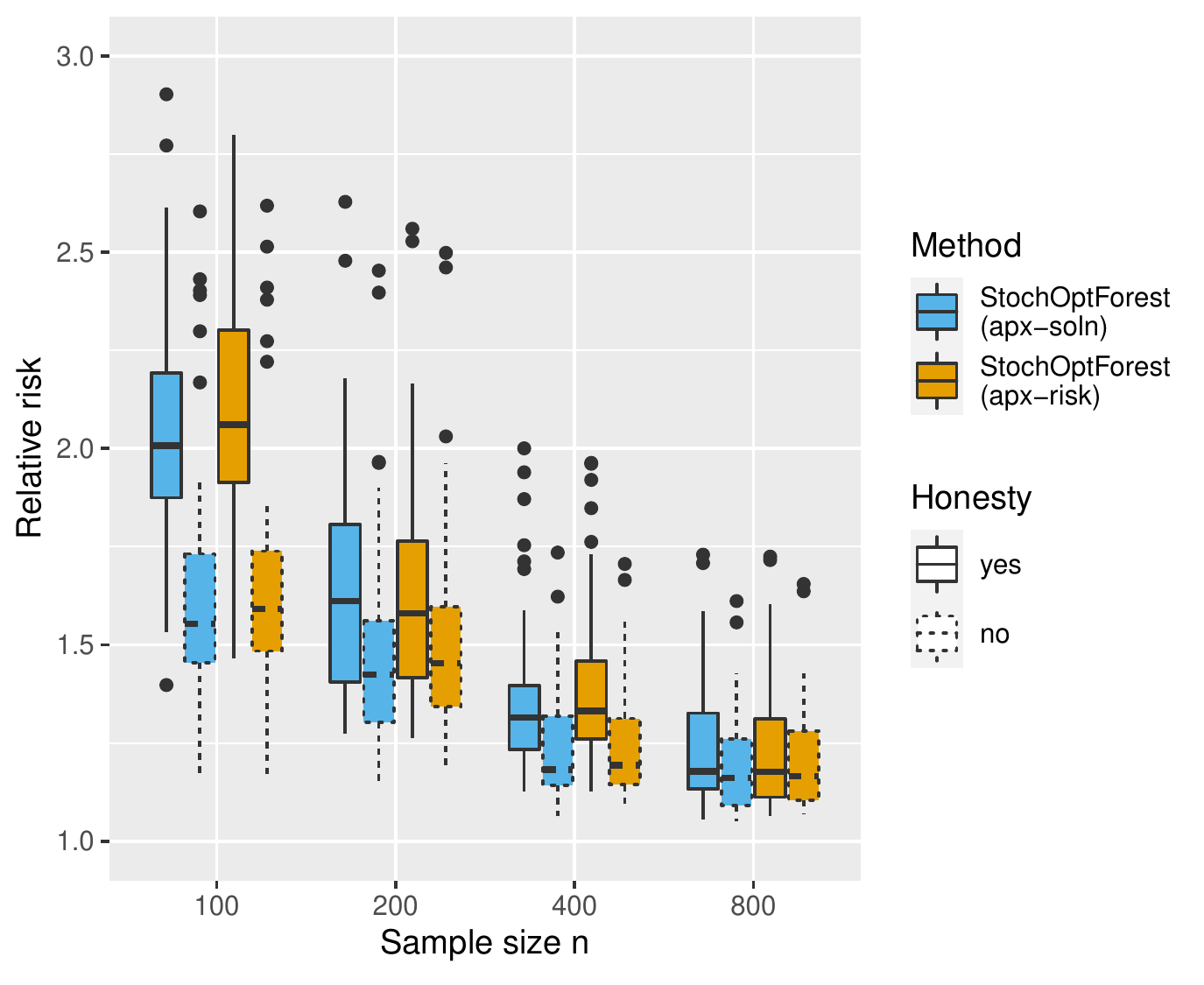}
    \caption{Newsvendor problem.}
    \label{fig: nv-honesty}
\end{subfigure}
\caption{Honest forests vs dishonest forests in CVaR optimization and newsvendor problem.}
\label{fig: honsty}
\end{figure*}
In \cref{fig: honsty}, we evaluate the performance of honest forests that use independent datasets to construct trees and form tree weights respectively (see \cref{assump: tree}), and dishonest forests that use the same datasets to construct trees and tree weights (see \cref{sec: empirical}). 
The specifications of experiments in \cref{fig: cvar-honesty} and \cref{fig: nv-honesty} are the same as those in \cref{sec: cvar-empirical} and \cref{sec: empirical nv} respectively, except that here each tree is constructed from a subsample of size $(1 - \frac{1}{e})n \approx 0.63n$ drawn randomly without replacement from the whole training data. This fraction is the expected size of distinct data points in a bootstrap sample. 
We can observe that honest forests tend to be outperformed by the dishonest counterparts, especially for large $n$ in the CVaR optimization problem and for small $n$ in the newsvendor problem.

\section{Perturbation Analysis}\label{app-sec: perturbation}

In this section we review perturbation analysis of stochastic optimization and use these tools to prove \cref{thm:secondorder,thm:secondorder-const}.

\subsection{A heuristic argument}\label{sec: heuristic}
\edit{We first give a heuristic argument for \cref{thm:secondorder-const} by applying the implicit function theorem to the KKT system. 
This argument does not treat many regularity conditions rigorously, but it is simple and instructive.  
We defer our review of the more general and rigorous analysis developed by \cite{perturbation2000} to \cref{apx-sec: general,sec: additive perturbed problems in finite dims} below.}

\edit{
Consider a constrained version of the perturbation formulation in \cref{eq:v}:
\begin{align*}
&{\min_{z\in\Z}\;f_{0}(z)+t\prns{f_{j}(z) - f_{0}(z)}},
\\\notag
&\text{where}~\Z=\braces{z\in\R d~~:~~
\begin{array}{ll}
h_{k}(z) = 0,~~&~~k=1,\dots,s,\\ h_{k}(z) \le 0,~~&~~k=s + 1,\dots, m
\end{array}}.
\end{align*}
Assume that the problem above corresponding to $t = 0$ has a unique optimal solution $z_0$ with a unique Lagrangian multiplier $v_0$. 
Let $K_h(z_0) = \{k: h_{k}(z_0) = 0, k = s+1, \cdots, m\}$ be the index set of inequality constraints active at $z_0$ and assume the strict complementarity condition, namely, $v_{0, k} > 0$ if and only if $k \in K_h\prns{z_0}$ or $k \le s$. Under the  Mangasarian-Fromovitz constraint qualification condition (condition \ref{cond: constr-MF} in \cref{thm:secondorder-const}), $z_0$ and $v_0$ can be characterized by the following Karush–Kuhn–Tucker (KKT) system:
\begin{align}
&\nabla f_0\prns{z_0} + \sum_{k \in \braces{1, \dots, s}\cup K_h\prns{z_0}} \nu_{0, k}\nabla h_k\prns{z_0} = 0, \label{eq: KKT0-1}\\
&h_k\prns{z_0} = 0, ~~ k \in K_h\prns{z_0}. \label{eq: KKT0-2}
\end{align}
Now consider the problem above with $t \ne 0$. Assume that it has an optimal solution $z_j\prns{t}$ with a Lagrangian multiplier $\nu_j\prns{t}$. 
When $t$ is very close to $0$, we may conjecture that $z_j\prns{t}$ is close to $z_0$ so that the inequalities active at $z_j\prns{t}$ are still given by $K_h\prns{z_0}$, $v_{j, k}\prns{t} > 0$ if and only if $k \in K_h\prns{z_0}$ or $k \in \braces{1, \dots, s}$, and $z_j\prns{t}$ and $\nu_j\prns{t}$ are also characterized by the corresponding KKT system:
\begin{align}
&\Gamma_j\prns{z_j\prns{t}, \nu_j\prns{t}, t} = 0 \label{eq: implicit-fun} \\
&\text{where }
\Gamma_j\prns{z, \nu, t}  = 
\begin{bmatrix}
\nabla f_0\prns{z} + t\prns{\nabla f_j\prns{z}- \nabla f_0\prns{z}} + \sum_{k \in \braces{1, \dots, s}\cup K_h\prns{z_0}} \nu_{k}\nabla h_k\prns{z} \nonumber \\
\mathcal{H}_{K_h}\prns{z} \nonumber
\end{bmatrix}.
\end{align}
Above $\mathcal{H}_{K_h}\prns{z}$ is a column vector whose elements are $h_k\prns{z}$ for $ k \in \braces{1,\dots, s}\cup K_h\prns{z_0}$. 
}

\edit{
Note that \cref{eq: KKT0-1,eq: KKT0-2} imply $\Gamma_j\prns{z_0, \nu_0, 0} = 0$. Moreover, the Jacobian matrix of $\Gamma_j\prns{z, \nu, t}$ at $(z_0, \nu_0, 0)$
is
\begin{align*}
\nabla_{z, \nu}^\top \Gamma_j\prns{z_0, \nu_0, 0} = 
\begin{bmatrix}
~\prns{\nabla^2 f_0(z_0) +  \sum_{k = 1}^{m} \nu_{0,k}  \nabla^2  h_{k}(z_0)}  ~&~ \nabla {\mathcal H^{K_h}}^\top(z_0)~ \\
~\nabla^\top\mathcal H^{K_h}(z_0) ~&~ 0 ~
\end{bmatrix} 
\end{align*}
When this Jacobian matrix is invertible, the implicit function theorem ensures that for $t$ close enough to $0$, there exist unique and continuously differentiable $z_j\prns{t}$ and $\nu_j(t)$ such that $\Gamma_j\prns{z_j\prns{t}, \nu_j\prns{t}, t} = 0$, and $d_z^{j*} = \frac{\partial}{\partial t} z\prns{t} \vert_{t = 0}$ and $\xi_j = \frac{\partial}{\partial t} \nu_j\prns{t} \vert_{t = 0}$ are solutions to the following linear equation system: 
\begin{align}\label{eq: kkt-perturb}
&\begin{bmatrix}
~\prns{\nabla^2 f_0(z_0) +  \sum_{k = 1}^{m} \nu_{0,k}  \nabla^2  h_{k}(z_0)}  ~&~ \nabla {\mathcal H^{K_h}}^\top(z_0)~ \\
~\nabla^\top\mathcal H^{K_h}(z_0) ~&~ 0 ~
\end{bmatrix} 
\begin{bmatrix}
d_z^{j*} \\ \xi_j
\end{bmatrix} 
= \frac{\partial}{\partial t} \Gamma_j\prns{z_0, \nu_0, 0} 
=
\begin{bmatrix}
-\prns{\nabla  f_j(z_0) - \nabla f_0(z_0)}  \\
0 
\end{bmatrix}
.
\end{align}
This implies that $z_j\prns{t} = z_0 + t d_z^{j*} + o(t)$, which is exactly the conclusion in \cref{eq: approx-sol} in \cref{thm:secondorder-const}. 
}

\edit{
Moreover, by the fact that $\prns{z_j\prns{t}, v_j\prns{t}}$ forms a KKT pair, the optimal value $v_j\prns{t}$ has the following formulation:
\begin{align*}
v_j\prns{t} = f_{0}(z_j\prns{t})+t\prns{f_{j}(z_j\prns{t}) - f_{0}(z_j\prns{t})} + \sum_{k \in \braces{1, \dots, s}\cup K_h\prns{z_0}}\nu_{j, k}\prns{t}h_k\prns{z_j\prns{t}}.
\end{align*}
We can then use the chain rule to derive the first and second order derivatives of $v_j\prns{t}$ at $t = 0$ in terms of derivatives of $z_j\prns{t}$ at $t = 0$ and gradients of $f_0, f_j$. 
\begin{proposition}\label{prop: implict-fun-thm}
 Suppose that $z_j\prns{t}, \nu_j\prns{t}$ are twice continuously differentiable at $t = 0$. Then 
 \begin{align*}
 &\frac{\partial}{\partial t} v_j\prns{t}\vert_{t = 0} = f_{j}(z_0) - f_{0}(z_0),   \\
 &\frac{\partial^2}{\partial t^2} v_j\prns{t}\vert_{t = 0} = d_{z}^{j*\top} \prns{\nabla^2 f_0(z_0) +  \sum_{k = 1}^{m} \nu_{0,k}  \nabla^2  h_{k}(z_0)} d_{z}^{j*} + 2d_{z}^{j*\top} \prns{{\nabla  f_j(z_0) - \nabla f_0(z_0)}}.
 \end{align*}
\end{proposition}
Note \cref{prop: implict-fun-thm} agrees with the conclusion of \cref{thm:secondorder-const} in \cref{eq: approx-risk}. 
}

\edit{The above argument is largely heuristic as it makes many assumptions without justifications, like the preservation of the active index set in the perturbed problems, the KKT formulation for the perturbed solutions, and the twice continuous differentiability of the primal and dual solutions to the perturbed problems, \etc. 
In \cref{apx-sec: general,sec: additive perturbed problems in finite dims}, we summarize a more  rigorous and more general perturbation analysis. }

\subsection{General Perturbation Analysis} \label{apx-sec: general}
In this section, we give an overview of the second order perturbation analysis based on results in \cite{perturbation2000}.
Consider the following generic parameterized problem: for $z, u$ in finitely dimensional vector spaces $\mathcal Z, \mathcal U$ respectively,
\begin{align}\label{eq: pert-opt}
\begin{array}{ll}
\min_{z} & f(z, u) \\
\text{s.t.}  &g_k(z, u) = 0, ~ k = 1, \dots, s, \\
&g_k(z, u)  \le 0, ~ k = s+1, \dots,  m,
\end{array}
\end{align}
where both $f(z, u)$ and $g_k(z, u)$ are twice continuously differentiable in both $z$ and $u$. We denote the first and second order derivatives of $f$ w.r.t $(z, u)$ as operators $Df(z, u)$ and $D^2 f(z, u)$ repsectively: 
\begin{align*}
Df(z, u)(d_z, d_u) &= D_zf(z, u)(d_z) + D_uf(z, u)(d_u) = d_z^\top \nabla_z f(z, u) + d_u^\top \nabla_u f(z, u) \\ 
D^2 f(z, u)((d_z, d_u), (d_z, d_u)) &= D_{zz}f(z, u)(d_z, d_z) +  D_{zu}f(z, u)(d_z, d_u) + D_{uz}f(z, u)(d_u, d_z) + D_{uu}f(z, u)(d_u, d_u) \\
&= d_z^\top\nabla_{zz}\edit{f\prns{z, u}}d_z + d_u^\top\nabla_{uu}\edit{f\prns{z, u}}d_u + 2d_z^\top\nabla_{zu}\edit{f\prns{z, u}}d_u.
\end{align*}

 We can similarly denote the partial derivatives of $f$ w.r.t $z$ and $u$ by $D_z f(z, u)$ and $D_u f(z, u)$ respectively. 
Derivatives for $g_k$ can be defined analogously.

Consider the parabolic perturbation path $u(t) = u_0 + t d_u + \frac{1}{2}t^2 r_u + o(t^2)$ for $t > 0$ \edit{and some elements $d_u, r_u$ such that $u(t) \in {\in \mathcal{U}}$}, and  denote the associated optimization problem as $P_{u(t)}$ with optimal value as $V(u(t))$.
We assume that the unperturbed problem $P_{u(0)}$ has a \textit{unique} optimal solution, which we denote as $z^*$, and we also denote $z^*(t)$ as one optimal solution of the perturbed problem $P_{u(t)}$.
We aim to derive the second order taylor expansion of $V(u(t))$, and the first order taylor expansion of $z^*(t)$. 

We first introduce several useful notations. We define the Lagrangian of the parameterized problem as 
\[
	L(z, u; \lambda) = f(z, u) + \sum_{k = 1}^m \lambda_k g_k(z, u),
\]
and the associated Lagrangian multiplier set for any $(z, u)$ as 
\[
	\Lambda(z, u) = \{\lambda: D_z L(z, u; \lambda) = 0, \text{and } \lambda_k \ge 0, \lambda_kg_k(z, u) = 0, k = s + 1, \dots, m\}.
\] 
For any feasible point $z$ for the unperturbed problem \edit{(\ie, $t = 0$)}, we define $K(z, u_0) = \{k: g_k(z, u_0) = 0, k = s + 1, \dots, m\}$ as the index set of inequality constraints that are active at $z$, and further define the index sets for active inequality constraints whose langrangian multipliers are strictly positive or $0$ respectively: 
\begin{align}
K_+(z, u_0, \lambda) = \{k \in K(z, u_0): \lambda_k > 0\}, ~ K_0(z, u_0, \lambda) = \{k \in K(z, u): \lambda_k = 0\}. \label{eq: K+}
\end{align}

Consider a solution path of form $z(t) = z^* + t d_z + \frac{1}{2}t^2 r_z + o(t^2)$ \edit{for some elements $d_z, d_u$ such that $z(t) \in \Z$}. If $z(t)$ is feasible for the perturbed problem $P_{u(t)}$, then we can apply second order taylor expansion to $f(z(t), u(t))$ as follows:
\begin{align}\label{eq: heuristic}
f(z(t), u(t)) = f(z^*, u_0) + tDf(z^*, u_0)(d_z, d_u) + \frac{1}{2}t^2 \left[Df(z^*, u_0)(r_z, r_u) + D^2f(z^*, u_0)((d_z, d_u), (d_z, d_u))\right] + o(t^2).
\end{align}
This heuristic expansion motivates two sets of optimization problems that are useful in approximating the optimal value of the perturbed problems. 

The first set optimization problem is a LP corresponding to the linear approximation term and its dual\footnote{The exact dual problem of problem PL in \cref{eq: PL} actually uses a different constraint for $\lambda$ than that used in \cref{eq: DL}. In the proof of \cref{prop: first-order-simplification}, we show that using these two constraint sets results in the same optimal value, which is stated without proof in \cite{perturbation2000}. So we also call the problem in \cref{eq: DL} as the dual of problem PL in \cref{eq: PL}.}: 
\begin{align}
&V(\PL) = \left\{\begin{array}{ll}
\min_{d_z} & Df(z^*, u_0)(d_z, d_u) \\
\text{s.t} & Dg_k(z^*, u_0)(d_z, d_u) = 0, ~ k = 1, \dots, s  \\
& Dg_k(z^*, u_0)(d_z, d_u) \le 0, ~ k \in K(z^*, u_0)
\end{array}\right., \label{eq: PL}\\
&V({\DL}) = \max_{\lambda \in \Lambda(z^*, u_0)} D_u L(z^*, \lambda, u_0) d_u. \label{eq: DL}
\end{align}

Given a feasible point $d_z$ of the problem PL, we denote the corresponding set of active inequality constraints in the problem PL as 
\[
	K_\PL(z^*, u_0, d_z) = \{k \in K(z^*, u_0): Dg_k(z^*, u_0)(d_z, d_u) = 0\}
\]
We denote the sets of optimal primal and dual solutions \edit{to \cref{eq: PL,eq: DL}} as $S(\PL)$ and $S(\DL)$ respectively. Equation (5.110) in \cite{perturbation2000} shows that $S(\PL)$ has the following form: for any $\lambda \in S(\DL)$, 
\begin{align*}
S(\PL) = \left\{d_z: \begin{array}{l}
Dg_k(z^*, u_0)(d_z, d_u) = 0, k \in \{1, \dots, s\} \cup K_+(z^*, u_0, \lambda), \\
Dg_k(z^*, u_0)(d_z, d_u) \le 0, k \in  K_0(z^*, u_0, \lambda)
\end{array}\right\}.
\end{align*}

The second set of optimization problems is a QP problem corresponding to the second order approximation term and its dual:
\begin{align}
V(\PQ) = \min_{d_z \in S(\PL)}V(\PQ(d_z)), ~ V(\DQ) = \min_{d_z \in S(\PL)} V(\DQ(d_z))
\end{align}
where
\begin{align}
&V(\PQ(d_z)) = \left\{\begin{array}{ll}
\min_{r_z} & Df(z^*, u_0)(r_z, r_u) + D^2 f(z^*, u_0)((d_z, d_u), (d_z, d_u)) \\
\text{s.t} & Dg_k(z^*, u_0)(r_z, r_u) + D^2 g_k(z^*, u_0)((d_z, d_u), (d_z, d_u)) = 0, ~ k = 1, \dots, s \\
& Dg_k(z^*, u_0)(r_z, r_u) + D^2 g_k(z^*, u_0)((d_z, d_u), (d_z, d_u))  \le 0, ~ k \in K_{\PL}(z^*, u_0, d_z)
\end{array}\right., \label{eq: PQ-d} \\
&V(\DQ(d_z)) = \max_{\lambda \in S(\DL)} D_u L(z^*, u_0; \lambda)r_u + D^2 L(z^*, u_0; \lambda)((d_z, d_u), (d_z, d_u)). \label{eq: DQ-d}
\end{align}
The constraints on $d_z$ in problem $\PL$ (\cref{eq: PL}) and constraints on $r_z$ in problem $\PQ(d(z))$ (\cref{eq: PQ-d}) ensure that path of the form $z(t) = z^* + t d_z + \frac{1}{2}t^2 r_z + o(t^2)$ is (approximately) feasible for the perturbed problem problem, so that the expansion in \cref{eq: heuristic} is valid. 

In the following proposition, we  characterize the optimization problems above when assuming the Lagrangian multiplier associated with the optimal solution in the unperturbed problem is unique, \ie, $\Lambda(z^*, u_0)$ is a singleton $\{\lambda^*\}$.
\begin{proposition}\label{prop: first-order-simplification}
If $\Lambda(z^*, u_0) = \{\lambda^*\}$ and $V(\PL), V(\PQ)$ are both finite, then 
\begin{align*}
&V(\PL) = V(\DL) 
	= D_u L(z^*, u_0; \lambda^*)d_u, \\
&V(PQ(d_z)) = V(DQ(d_z)) 
	= D_u L(z^*, u_0; \lambda^*)r_u + D^2 L(z^*, u_0; \lambda^*)((d_z, d_u), (d_z, d_u)).
\end{align*}
\end{proposition}

The following theorem derives the second order expansion of $V(u(t))$ under regularity conditions. 
\begin{theorem}[Theorem 5.53 in \cite{perturbation2000}]\label{thm: perturb}
Suppose  the following conditions hold:
\begin{enumerate}
\item $f(z, u)$ and $g_k(z, u)$ for $k= 1,\dots, m$ are twice continuously differentiable in both $z \in \mathcal Z$ and $u \in \mathcal U$ in a neighborhood around $(z^*, u_0)$;
\item The unperturbed problem corresponding to $t = 0$ (or equivalently problem $P_{u_0}$) has a unique optimal solution $z^*$;
\item Mangasarian-Fromovitz constraint qualification condition is satisfied at $z^*$: 
\begin{align*}
&D_z g_k(z^*, u_0), ~ k = 1, \dots, s \text{ are linearly independent},  \\
&\exists d_z, \text{ s.t. } D_z g_k(z^*, u_0)d_z = 0, ~ k = 1, \dots, s, ~ D_z g_k(z^*, u_0)d_z < 0, ~ k \in K(z^*, u_0);
\end{align*}
\item The set of Lagrangian multipliers  $\Lambda(z^*, u_0)$ for the unperturbed problem is nonempty; 
\item The following strong form of second order sufficient condition is satisfied for the unperturbed problem: 
\begin{align*}
\sup_{\lambda \in S(DL)} D_{zz}L(z^*, \lambda, u_0)(d_z, d_z) > 0, \forall d_z \in C(z^*, u_0; \lambda)\setminus \{0\},  
\end{align*}
where $C(z^*, u_0; \lambda)$ is the critical cone defined as follows:
\begin{align*}
C(z^*, u_0; \lambda)  = \left\{d_z: \begin{array}{l}
D_z g_k(z^*, u_0)(d_z) = 0, k \in \{1, \dots, s\} \cup K_+(z^*, u_0, \lambda), \\
D_z g_k(z^*, u_0)(d_z) \le 0, k \in  K_0(z^*, u_0, \lambda)
\end{array}\right\}.
\end{align*}
\item The inf-compactness condition: \edit{there exist} a constant $\alpha$ and a compact set $\overline C \subseteq \mathcal Z$ such that the sublevel set 
\[
\{z: f(z, u) \le \alpha, ~ g_k(z, u) = 0, ~ k = 1, \dots, s, ~ g_k(z, u)  \le 0, ~ k = s+1, \dots,  m.\}
\]
is nonempty and contained in $\overline C$ for any $u$ within a neighborhood of $u_0$.  
\end{enumerate}
Then the following conclusions hold: 
\begin{enumerate}
\item $V(\PL)$ and $V(\PQ)$ are both finite, and the optimal value function $V(u(t))$ for the perturbed problem in \cref{eq: pert-opt} can be expanded as follows:
\begin{align*}
 V(u(t)) = V(u_0) + t V(\PL) + \frac{1}{2}t^2 V(\PQ) + o(t^2).
\end{align*} 
\item If the problem $\PQ$ has a unique solution $d_z^*$, then any optimal solution $z^*(t)$ of the perturbed problem in \cref{eq: pert-opt} satisfies that 
\[
z^*(t) = z^* + t d_z^* + o(t).
\]
\end{enumerate}
\end{theorem}

We now show that the optimal solutions of problems $\PQ$ and $\DQ$ have a simple formulation under regularity conditions about the dual optimal solution of the unperturbed problem. 
\begin{proposition}\label{prop: second-order-simplify}
Under conditions in \cref{prop: first-order-simplification}, if further 
\edit{$\Lambda(z^*, u_0) = \{\lambda^*\}$, \ie, $z^*$ is associated with a unique Lagrangian multiplier $\lambda^*$, and} the strict complementarity condition holds, \ie, the Lagrangian multipliers associated with all  inequality constraints active at $z^*$ are strictly positive (or equivalently $K_+(z^*, u_0, \lambda^*) = K(z^*, u)$),
then $V(\PQ) = V(\DQ)$ equals the optimal value of the following optimization problem:
\begin{align*}
\begin{array}{ll}
\min_{d_z} ~ & D_u L(z^*, u_0; \lambda^*)r_u + D^2 L(z^*, u_0; \lambda^*)((d_z, d_u), (d_z, d_u)) \\
\text{s.t.} & Dg_k(z^*, u_0)(d_z, d_u) = 0, k \in \{1, \dots, s\} \cup K(z^*, u_0).
\end{array}
\end{align*}
\end{proposition}
\cref{prop: second-order-simplify} shows that under the asserted regularity conditions, the second order approximation term $V(PQ)$ is the optimal value of a simple quadratic programming problem with equality constraints, which can be solved very efficiently provided that $z^*, \lambda^*$ are known.

\edit{
According to \cite{WACHSMUTH201378}, one condition to ensure a unique Lagrangian multiplier is the following linear independence constraint qualification (LICQ) condition:
\begin{align}\label{eq: LICQ}
D_z g_k(z^*, u_0), ~ k \in \{1, \dots, s\} \cup K(z^*, u)  \text{ are linearly independent}.
\end{align}
Actually, this LICQ condition is also stronger than the  Mangasarian-Fromovitz constraint qualification condition in \cref{thm: perturb}.
}

\subsection{Additively perturbed problems in finite-dimensional space and its connection to approximate criteria}\label{sec: additive perturbed problems in finite dims}
\subsubsection*{Additive perturbations.} Consider the following optimization problem: for $z \in \mathbb R^d$ and $t > 0$,
\begin{align}\label{eq: perturb-additive}
v(t) = 
\left\{\begin{array}{ll}
\min_{z \in \mathbb R^d} ~ &f(z) + t\delta_f(z) \\
\text{s.t.} ~ &g_k(z) + t\delta_{g_k}(z) = 0, ~ k = 1, \dots, s, \\
& g_k(z) + t\delta_{g_k}(z) \le 0, ~ k =  s+ 1, \dots, m\\
& h_{k'}(z) = 0, ~ k' = 1, \dots, s', \\
& h_{k'}(z) \le 0, ~ k' = s'+1, \dots, m',
\end{array}\right.
\end{align} 
where $f$, $g_k$, $h_k$ are all twice continuously differentiable in $z \in \mathcal Z$ with gradients and Hessian matrices denoted by $\nabla$ and $\nabla^2$ respectively. Moreover, $\delta_f$ and $\delta_{g_k}$ are all differentiable with gradients denoted by $\nabla$.  
We define the Lagrangian for the problem above as follows:
\[
	L(z, t; \lambda, \nu) = f(z) + t\delta_f(z) + \sum_{k = 1}^m \lambda_k (g_k(z) + t\delta_{g_k}(z)) + \sum_{k' = 1}^{m'} \nu_{k'} h_{k'}(z)
\]
where $\lambda$ and $\nu$ are the Lagrangian multipliers associated with the constraints involving $\{g_k\}_{k = 1}^m$ and $\{h_{k'}\}_{k' = 1}^m$ respectively.
For any feasible point $z$ of the unperturbed problem \edit{(\ie, $t = 0$)}, we denote the index sets of active (unperturbed) inequality constraints $z$ as $K_g(z)$ and $K_{h}(z)$ respectively:
\begin{align*}
K_g(z) = \{k: g_k(z) = 0, k = s + 1, \dots, m\}, ~~ K_h(z) = \{k': h_{k'}(z) = 0, k' = s' + 1, \dots, m'\}.
\end{align*}

Note that the problem above in \cref{eq: perturb-additive} is a special case of the problem in \cref{eq: pert-opt} with perturbation path $u(t) = t$, \ie, $u_0 = 0, d_u = 1, r_u = 0$.
Thus we can apply \cref{thm: perturb} and \cref{prop: second-order-simplify} to prove the following theorem.
\begin{theorem}\label{lemma: additive-perturb}
 Suppose the following conditions hold: 
\begin{enumerate}
\item $f(z), g_k(z), h_{k'}(z)$ for $k= 1,\dots, m$ and $k' = 1, \dots, m'$ are twice continuously differentiable, and $\delta_f, \delta_{g_k}$ for $k = 1, \dots, m$ are continuously differentiable; 
\item The unperturbed problem corresponding to $t = 0$ has a unique optimal primarxy solution $z^*$ that is associated with a unique Lagrangian multiplier $(\lambda^*, \nu^*)$, and $(\lambda^*, \nu^*)$ satisfies the strict complemetarity condition, \ie, $\lambda^*_k > 0$ for $k \in K_g(z^*)$  and $\nu_k^* > 0$ for $k \in K_h(z^*)$;
\item Mangasarian-Fromovitz constraint qualification condition is satisfied at $z^*$: 
\begin{align*}
&\nabla g_k(z^*), ~ k = 1, \dots, s \text{ are linearly independent and } \nabla h_{k'}(z^*), ~ k' = 1, \dots, s' \text{ are linearly independent},  \\
&\exists d_z, \text{ s.t. } \nabla g_k(z^*)d_z = 0, ~ k = 1, \dots, s, ~ \nabla g_k(z^*)d_z < 0, ~ k \in K_g(z^*), \\
&\phantom{\exists d_z, \text{ s.t. }} \nabla h_{k'}(z^*)d_z = 0, ~ k' = 1, \dots, s', ~ \nabla h_{k'}(z^*)d_z < 0, ~ k' \in K_h(z^*);
\end{align*}
\item Second order sufficient condition:
\begin{align*}
d_z^\top \nabla^2 L(z^*, 0; \lambda^*, \nu^*)d_z > 0, \forall d_z \in C(z^*)\setminus \{0\},  
\end{align*}
where $C(z^*)$ is the critical cone defined as follows:
\[
C(z^*) = \left\{d_z: d_z^\top\nabla g_k(z^*)  = 0, ~ k\in \{1, \dots, s\} \cup K_g(x^*), ~ d_z^\top\nabla h_{k'}(z^*) = 0, ~ k' \in \{1, \dots, s'\} \cup K_h(x^*)\right\}.
\]
\item The inf-compactness condition: \edit{there exist} a constant $\alpha$, a positive constant $t_0$, and a compact set $\overline C \subseteq \mathcal Z$ such that the sublevel set 
\begin{align*}
\left\{
z:
\begin{array}{l}
f(z) + t\delta_f(z) \le \alpha, \\
g_k(z) + t\delta_{g_k}(z) = 0, ~ k = 1, \dots, s, ~ g_k(z) + t\delta_{g_k}(z)  \le 0, ~ k = s+1, \dots,  m, \\
h_{k'}(z)  = 0, ~ k' = 1, \dots, s', ~ h_{k'}(z)  \le 0, ~ k' = s'+1, \dots,  m'.
\end{array}
\right\}
\end{align*}
is nonempty and contained in $\overline C$ for any $t \in [0, t_0)$.  
\end{enumerate}
Then 
\begin{align*}
v(t) = v(0) + v'(0)t + \frac{1}{2}t^2 v''(0) + o(t^2)
\end{align*}
where 
\begin{align*}
v'(0) = \delta_f(z^*) + \sum_{k = 1}^ m\lambda_k^* \delta_{g_k}(z^*)
\end{align*}
and 
\begin{align}\label{eq: second-order-problem}
v''(0) = \min_{d_z} ~ 
	& d_z^\top \nabla^2 L(z^*, 0; \lambda^*, \nu^*) d_z + 2d_z^\top\left(\nabla \delta_f(z^*) + \sum_{k = 1}^s \lambda_k^* \nabla \delta_{g_k}(z^*)\right) \\
\text{s.t.}~ & d_z^\top \nabla g_k(z^*) + \delta_{g_k}(z^*) = 0, ~ k\in \{1, \dots, s\} \cup K_g(x^*) \nonumber \\
&d_z^\top \nabla h_k(z^*) = 0, ~ k' \in \{1, \dots, s'\} \cup K_h(x^*). \nonumber 
\end{align}
Moreover, if the optimization problem in \cref{eq: second-order-problem} has a unique optimal solution $d_z^*$, then any optimal solution $z^*(t)$ of the perturbed problem in \cref{eq: perturb-additive} satisfies that 
\begin{align}\label{eq: sol-expansion}
z^*(t) = z^* + td_z^* + o(t).
\end{align}
\end{theorem}

\edit{According to \cref{eq: LICQ},  a sufficient condition for the uniqueness of the Lagrangian multiplier is the following:
\begin{align}\label{eq: strict-cq2}
\begin{aligned}
&\nabla g_k(z^*), ~ k \in \{1, \dots, s\} \cup K_{g}(z^*) \text{ are linearly independent}, \\
&\nabla h_{k'}(z^*), ~ k' = \{1, \dots, s'\}\cup K_h(z^*) \text{ are linearly independent}. 
\end{aligned}
\end{align}
}

\subsubsection{Proving \cref{thm:secondorder-const2}}
In \cref{sec:cso-general}, we aim to approximate $v_j(t)=\min_{z\in\mathcal Z_j(t)}\;(1-t)f_{0}(z)+tf_{j}(z),~z_j(t)\in\argmin_{z\in\mathcal Z_j(t)}\;(1-t)f_{0}(z)+tf_{j}(z)$ for $t\in[0,1],\,j=1,2$, where 
\begin{align*}
\mathcal Z_j(t) =\left\{z:
\begin{array}{l}
(1 - t)g_{0, k}(z) + t g_{j, k}(z) = 0, ~ k = 1, \dots, s, \\
(1 - t)g_{0, k}(z) + t g_{j, k}(z) \le 0, ~ k = s+1, \dots, m,  \\
h_{k'}(z) = 0, ~ k' =1, \dots, s', ~ h_{k'}(z) \le 0, ~ k' =s'+1, \dots, m'
\end{array}
\right\}.
\end{align*}
Note this is a special example of \cref{eq: perturb-additive} with $\delta_f = f_j - f_0$ and $\delta_{g_k} = g_{j, k} - g_{0, k}$ for $k = 1, \dots, m$. Then applying \cref{lemma: additive-perturb} directly gives \cref{thm:secondorder-const2}.

\subsubsection{Proving \cref{thm:secondorder}}
If there are no constraints, \ie, we consider the problem 
\begin{align}\label{eq: unconstr-additive}
v(t) = \min_{z \in \mathbb R^d} f(z) + t\delta_f(z). 
\end{align}
Then \cref{lemma: additive-perturb} reduces to the following corollary. 
\begin{corollary}\label{corollary: unconstr}
Suppose the following conditions hold for $f$:
\begin{enumerate}
\item $f(z)$ is twice continuously differentiable, and $\delta_f(z)$ is continuously differentiable; 
\item there exists a constant $\alpha$, a positive constant $t_0 \in (0, 1]$ and compact set $\overline C \subseteq \mathbb R^d$ such that the sublevel set $\left\{
z \in \mathbb R^d: ~ f(z)+t\delta_f(z) \le \alpha
\right\}$ is nonempty and contained in $\overline C$ for any $t \in [0, t_0)$;
\item $f(z)$ has a unique minimizer over $\mathbb R^d$ (denoted as $z^*$), and $\nabla^2 f(z^*)$ is positive definite.
\end{enumerate}
Then $v(t)$ in \cref{eq: unconstr-additive} satisfies that 
\begin{align}\label{eq: approx-v-unconstr}
v(t) = v(0) + t\delta_f(z^*) - \frac{1}{2}t^2 \nabla \delta_f(z^*)^\top \prns{\nabla^2 f(z^*)}^{-1}\nabla \delta_f(z^*) + o(t^2) 
\end{align}
and any optimal solution $z^*(t)$ of the perturbed problem in \cref{eq: unconstr-additive} satisfies that 
\begin{align*}
z^*(t) = z^* - t\prns{\nabla^2 f(z^*)}^{-1}\nabla \delta_f(z^*) + o(t)
\end{align*}
\end{corollary}

Note applying this corollary with $f = f_0$, $\delta_f = f_j - f_0$, and $\nabla \delta_f(z^*) = \nabla f_j(z^*) - \nabla f_0(z^*) = \nabla f_j(z^*) $ gives \cref{thm:secondorder}. 

\subsection{Stronger Differentiability Results.} All results above are based on directional differentiability, which cannot quantify the magnitude of approximation errors of the second order perturbation analysis. Here we show in the context of unconstrained problems that under stronger regularity conditions, we can also bound the approximation errors by the magnitude of perturbation $\delta_f$. 
We will then use this to prove \cref{thm:apxsolapx,thm:apxriskapx} in \cref{app-sec: sec2}.

Consider the following optimization problem denoted as $P(f)$ for $f: \mathbb R^d \mapsto \mathbb R$, 
\begin{align}\label{eq: perturb-additive-2}
v(f) = \min_{z \in \mathbb R^d} ~ f(z), ~ z^*(f) \in \argmin_{z \in \mathbb R^d} ~ f(z) 
\end{align} 
We restrict $f$ to the twice continuously differentiable function class $\mathcal F$ with norm defined as 
\begin{align}\label{eq: norm}
\|f\|_{\mathcal F} = \max \{\sup_z|f(z)|,\, \sup_z\edit{\|\nabla f(z)\|_2},\, \sup_z\edit{\|\nabla^2 f(z)\|_{\op{F}}}\},
\end{align}
where $\edit{\|\nabla f(z)\|_2}$ is the Euclidean norm of gradient $\nabla f(z)$ and $\edit{\|\nabla^2 f(z)\|_{\op{F}}}$ is the Frobenius norm of the Hessian matrix $\nabla^2 f(z)$.

We consider $P(f_0)$ as the unperturbed problem and $P(f_0 + \delta_f)$ as the target perturbed problem that we hope to approximate. \cref{eq: approx-v-unconstr} gives the first and second order functional directional derivatives of $v$ at $f_0$: 
\begin{align*}
v'(f_0; \delta_f) 
	&\coloneqq \lim_{t \downarrow 0}\frac{v(f_0 + t\delta_f) - v(f_0)}{t} = \delta_f(z^*(f_0)) \\
v''(f_0; \delta_f) 
	&\coloneqq \lim_{t \downarrow 0} \frac{v(f_0 + t\delta_f) - v(f_0) - tv'(f_0; \delta_f)}{\frac{1}{2}t^2} = - \nabla \delta_f(z^*(f_0))^\top \prns{\nabla^2 f_0(z^*(f_0))}^{-1}\nabla \delta_f(z^*(f_0))
\end{align*}
We aim to expand $v(f_0 + \delta_f)$ and $z^*(f_0 + \delta_f)$ in the functional space with approximation errors bounded by the magnitude of $\delta_f$.

\begin{theorem}\label{thm: frechet-diff}
If $f_0(z), \delta_f(z)$ are both twice continuously differentiable, condition 2 in \cref{corollary: unconstr} is satisfied for $t \in [0, 1]$, and for any $t \in [0, 1]$, $f_0(z) + t\delta_f(z)$ has a unique minimizer $z^*(f_0 + t\delta_f)$ and $\nabla^2 \prns{f_0 + t\delta_f}(z^*(f_0 + t\delta_f))$ is positive definite, then 
\begin{align*}
v(f_0 + \delta_f) &= v(f_0) + \delta_f(z^*(f_0)) - \frac{1}{2} \nabla \delta_f(z^*(f_0))^\top \prns{\nabla^2 f(z^*(f_0))}^{-1}\nabla \delta_f(z^*(f_0)) + o(\|\delta_f\|^2_\mathcal F), \\
z^*(f_0 + \delta_f) &= z^*(f_0) - \prns{\nabla^2 f(z^*(f_0))}^{-1}\nabla \delta_f(z^*(f_0)) + o(\|\delta_f\|_\mathcal F).
\end{align*}
\end{theorem}

Based on \cref{thm: frechet-diff}, we can bound the approximation errors of the proposed criteria in \cref{sec: approx-crit} by $o(\|f_j - f_0\|^2_\mathcal F)$ for $j = 1, 2$.
Note that the Lipschitzness
condition (condition \ref{cond: thm-apxriskapx-lipschitz} in \cref{thm:apxriskapx}) implies that $\|f_j - f_0\|_\mathcal F = O(\mathcal D_0^2)$. Therefore, the approximation errors of the proposed criteria are $o(\mathcal D_0^2)$.
 See \cref{thm:apxsolapx,thm:apxriskapx} and their proofs in \cref{app-sec: sec2}.

\section{Optimization with Auxiliary Variables}\label{sec: auxiliary-var}
In \cref{ex: portfolio,ex: portfolio-var}, the cost function involves unconstrained auxiliary variables $z^{\text{aux}}$ in addition to the decision variables $z^{\text{dec}}$. 
In this setting, we can do perturbation analysis in two different ways. These two different approaches lead to different approximate criteria that are equivalent under infinitesimal perturbations but give different extrapolations and differ in terms of computational costs. 
For convenience, we focus on unconstrained problems with a generic cost function $c(z^\text{dec}, z^\text{aux}; y)$. We denote the region to be split as $R_0\subseteq\R d$, and its candidate subpartition as $R_0=R_1\cup R_2$, $R_1\cap R_2=\varnothing$. 

\subsubsection*{Re-optimizing  auxiliary variables.} In the first approach, we acknowledge the auxiliary role of $z^{\text{aux}}$ and define $f_j$ by profiling out $z^{\text{aux}}$ first: $f_j(z^{\text{dec}}) = \min_{z^{\text{aux}}} \Eb{c(z^{\text{dec}}, z^{\text{aux}}; Y) \mid X \in R_j}$ for $j = 0, 1, 2$. We assume that for each fixed $z^{\text{dec}}$ value, $\Eb{c(z^{\text{dec}}, z^{\text{aux}}; Y) \mid X \in R_j}$ has a unique minimizer, which we denote as $z^{\text{aux}}_j(z^{\text{dec}})$. We also denote $z^{\text{dec}}_j$ as the minimizer of $f_j(z^{\text{dec}})$ and $z^{\text{aux}}_j$ as $z^{\text{aux}}_j(z^{\text{dec}}_j)$.
The following proposition derives the gradient and Hessian matrix. 
\begin{proposition}\label{prop: profile-gradient}
Consider $f_j(z^{\text{dec}}) = \min_{z^{\text{aux}}} \Eb{c(z^{\text{dec}}, z^{\text{aux}}; Y) \mid X \in R_j}$ for $j = 0, 1, 2$. Suppose that for each $z^{\text{dec}}$ value and $j = 1, 2$, $\Eb{c(z^{\text{dec}}, z^{\text{aux}}; Y) \mid X \in R_j}$ has a unique minimizer $z^{\text{aux}}_j(z^{\text{dec}})$. Moreover, we assume that $f_0(z^{\text{dec}})$ has a unique minimizer $z^{\text{dec}}_0$, and $f_0, f_1, f_2$ are twice continuously differentiable. Then 
\begin{align*}
&\nabla f_j\prns{z^{\text{dec}}_0} = \frac{\partial}{\partial{z^{\text{dec}}}}\Eb{c(z^{\text{dec}}_0, z^{\text{aux}}_j\prns{z^{\text{dec}}_0}; Y) \mid X \in R_j}, \\
&\nabla^2 f_0\prns{z^{\text{dec}}_0}  = \frac{\partial^2}{\prns{\partial{z^{\text{dec}}}}^\top \partial{z^{\text{dec}}}}\Eb{c(z^{\text{dec}}_0, z^{\text{aux}}_0; Y) \mid X \in R_0} - \frac{\partial^2}{\partial{z^{\text{dec}}}\prns{\partial{z^{\text{aux}}}}^\top}\Eb{c(z^{\text{dec}}_0, z^{\text{aux}}_0; Y) \mid X \in R_0}\\
&\braces{\frac{\partial^2}{\partial{z^{\text{aux}}}\prns{\partial{z^{\text{aux}}}}^\top}\Eb{c(z^{\text{dec}}_0, z^{\text{aux}}_0; Y) \mid X \in R_0}}^{-1}\frac{\partial^2}{\partial{z^{\text{aux}}} \prns{\partial{z^{\text{dec}}}}^\top}\Eb{c(z^{\text{dec}}_0, z^{\text{aux}}_0; Y) \mid X \in R_0}.
\end{align*}
\end{proposition}

It is straightforward to show that \cref{ex: portfolio,ex: portfolio-var} satisfy conditions in \cref{prop: profile-gradient} under regularity conditions. So we can apply the gradients and Hessian matrix in \cref{prop: profile-gradient} to derive the apx-risk criterion and the apx-soln criterion. 
However, in order to estimate the gradients, we need to compute $z^{\text{aux}}_j\prns{z^{\text{dec}}_0}$ for every candidate split by repeatedly minimizing $\Eb{c(z^{\text{dec}}_0, z^{\text{aux}}; Y) \mid X \in R_j}$ with respect to $z^{\text{aux}}$.
This can be too computationally expensive in practice.  

\subsubsection*{Merging auxiliary variables with decision variables.} In the second way, we merge the auxiliary variables with the decision variables, and define $\tilde f_j(z^{\text{dec}}, z^{\text{aux}}) = \Eb{c(z^{\text{dec}}, z^{\text{aux}}; Y) \mid X \in R_j}$ for $j = 0, 1, 2$. 
The following proposition derives the gradients and Hessian matrix with respect to \emph{both} decision variables and auxiliary variables. 

\begin{proposition}\label{prop: no-profile-grad}
Consider $\tilde f_j(z^{\text{dec}}, z^{\text{aux}}) = \Eb{c(z^{\text{dec}}, z^{\text{aux}}; Y) \mid X \in R_j}$ for $j = 0, 1, 2$. Suppose that $\tilde f_0(z^{\text{dec}}, z^{\text{aux}})$ has a unique minimizer $\prns{z^{\text{dec}}_0, z^{\text{aux}}_0}$, and $\tilde f_0, \tilde f_1, \tilde f_2$ are twice continuously differentiable. Then
\begin{align*}
&\nabla \tilde f_j\prns{z^{\text{dec}}_0, z^{\text{aux}}_0} = 
\begin{bmatrix}
\frac{\partial}{\partial{z^{\text{dec}}}}\Eb{c(z^{\text{dec}}_0, z^{\text{aux}}_0; Y) \mid X \in R_j} \\
\frac{\partial}{\partial{z^{\text{aux}}}}\Eb{c(z^{\text{dec}}_0, z^{\text{aux}}_0; Y) \mid X \in R_j}  
\end{bmatrix}, \\
&\nabla^2 \tilde f_0\prns{z^{\text{dec}}_0, z^{\text{aux}}_0}  = 
\begin{bmatrix}
\frac{\partial^2}{\prns{\partial{z^{\text{dec}}}}^\top \partial{z^{\text{dec}}}}\Eb{c(z^{\text{dec}}_0, z^{\text{aux}}_0; Y) \mid X \in R_0}  & \frac{\partial^2}{\prns{\partial{z^{\text{aux}}}}^\top\partial{z^{\text{dec}}}}\Eb{c(z^{\text{dec}}_0, z^{\text{aux}}_0; Y) \mid X \in R_0} \\
\frac{\partial^2}{{\prns{\partial{z^{\text{dec}}}}^\top \partial z^{\text{aux}}}}\Eb{c(z^{\text{dec}}_0, z^{\text{aux}}_0; Y) \mid X \in R_0} & \frac{\partial^2}{\partial{z^{\text{aux}}}\prns{\partial{z^{\text{aux}}}}^\top}\Eb{c(z^{\text{dec}}_0, z^{\text{aux}}_0; Y) \mid X \in R_0}
\end{bmatrix}.
\end{align*}
\end{proposition}
Note that approximate criteria based on the formulations in \cref{prop: profile-gradient} and \cref{prop: no-profile-grad} are both legitimate, and they are equivalent under infinitesimal perturbations according to \cref{thm:secondorder}. Using the apx-risk criterion as an example, the following proposition further investigates the relationship between the approximate criterion based on \cref{prop: profile-gradient} and that based on \cref{prop: no-profile-grad}. 

\begin{proposition}\label{prop: aux-var-relation}
Let ${\crit}^\text{apx-risk}(R_1,R_2)$ and $\tilde{\crit}^\text{apx-risk}(R_1,R_2)$ be the apx-risk criterion based on $\{f_0, f_1, f_2\}$ given in \cref{prop: profile-gradient} and $\{\tilde f_0, \tilde f_1, \tilde f_2\}$ given in \cref{prop: no-profile-grad} respectively. Then 
\begin{align*}
&\crit^\text{apx-risk}(R_1,R_2)=-\sum_{j=1,2}p_j \prns{\frac{\partial}{\partial z^{\text{dec}}} \tilde f_j\prns{z^{\text{dec}}_0, z^{\text{aux}}_j\prns{z^{\text{dec}}_0}}}^\top 
\braces{\prns{\nabla^2 \tilde f_0\prns{z^{\text{dec}}_0, z^{\text{aux}}_0}}^{-1}\bracks{z^{\text{dec}}_0, z^{\text{dec}}_0}} \frac{\partial}{\partial z^{\text{dec}}} \tilde f_j\prns{z^{\text{dec}}_0, z^{\text{aux}}_j\prns{z^{\text{dec}}_0}},  \\
&\tilde \crit^\text{apx-risk}(R_1,R_2)=-\sum_{j=1,2}p_j \prns{\frac{\partial}{\partial z^{\text{dec}}} \tilde f_j\prns{z^{\text{dec}}_0, z^{\text{aux}}_0}}^\top 
\braces{\prns{\nabla^2 \tilde f_0\prns{z^{\text{dec}}_0, z^{\text{aux}}_0}}^{-1}\bracks{z^{\text{dec}}_0, z^{\text{dec}}_0}}\frac{\partial}{\partial z^{\text{dec}}} \tilde f_j\prns{z^{\text{dec}}_0, z^{\text{aux}}_0} + \mathcal R,
\end{align*}
where $\prns{\nabla^2 \tilde f_0\prns{z^{\text{dec}}_0, z^{\text{aux}}_0}}^{-1}\bracks{z^{\text{dec}}_0, z^{\text{dec}}_0}$ is the block of the inverse matrix $\prns{\nabla^2 \tilde f_0\prns{z^{\text{dec}}_0, z^{\text{aux}}_0}}^{-1}$ whose rows and columns both correspond to $z^{\text{dec}}_0$, and $\mathcal R$ is an adjustment term that only depends on $\nabla \tilde f_j\prns{z^{\text{dec}}_0, z^{\text{aux}}_0}$ and $\nabla^2 \tilde f_j\prns{z^{\text{dec}}_0, z^{\text{aux}}_0}$. 
\end{proposition}

\cref{prop: aux-var-relation} shows that
evaluating $\crit^\text{apx-risk}(R_1,R_2)$ requires computing $z^{\text{aux}}_j\prns{z^{\text{dec}}_0}$ repeatedly for all candidate splits, while 
evaluating $\tilde \crit^\text{apx-risk}(R_1,R_2)$ only requires computing $\prns{z^{\text{dec}}_0, z^{\text{aux}}_0}$ once. 
To compensate for the fact that $\tilde \crit^\text{apx-risk}(R_1,R_2)$
does not re-optimize the decision variable for each candidate split, 
$\tilde \crit^\text{apx-risk}(R_1,R_2)$ also has an additional adjustment term $\mathcal R$.
In \cref{ex: portfolio,ex: portfolio-var}, we use $\tilde \crit^\text{apx-risk}(R_1,R_2)$ to reduce computation cost, as this is the main point of using approximate criteria. 

{\blockedit
\section{Other Related Literature}\label{sec: literature}

\subsubsection*{Applications of tree models in other decision making problems.}
In the CSO problem, given a realization $y$, the effect of decisions $z$ on costs is assumed known (\ie, $c(z;y)$). This may not apply if decisions affect uncertain costs in an a priori unknown way, such as the unknown effect of prices on demand or of pharmacological treatments on health indicators. In these applications, data consists only of observations of the realized costs for a single decision and not counterfactual costs for other decisions, known as partial or bandit feedback, which requires additional identification assumptions such as no unobserved confounders \citep{bertsimas2016power}.
\citet{kallus2017recursive,zhou2018offline} apply tree methods to prescribe from a finite set of interventions based on such data using decision quality rather than prediction error as the splitting criterion. 
Since they consider a small finite number of treatments,
the criterion for each candidate split is rapidly computed by enumeration. 
When decisions are continuous, various works use tree ensembles to regress cost on a decision variable (\eg, Sec. 3 of \citealp{bertsimas2014predictive}, \citealp{ferreira2016analytics}, among others) and then search for the input to optimize the output. This is generally a hard optimization problem, to which \citet{mivsic2020optimization} study mixed-integer optimization approaches. 
\citet{elmachtoub2017practical,feraud2016random} similarly use decision trees and random forests for \emph{online} decision-making in contextual bandit problems. 

\edit{\citet{Chen2020Choice} propose to use trees and forests to nonparametrically model irrational customer choices. Under the forest choice model, \citet{Chen2021Choice} further develops mixed-integer optimization algorithms to find the assortment that maximizes expected revenue. \citet{Ciocan2020Interptable} study optimal stopping problems with applications in option pricing, and propose algorithms to construct approximately optimal tree policies that are easy to interpret. }

\subsubsection*{Applications of perturbation analysis in machine learning.} 
Perturbation analysis studies the impact of slight perturbations to the objective and constraint functions of an optimization problem on the optimal value and optimal solutions, which is the foundation of our approximate splitting criteria.  We refer readers to \cite{perturbation2000} for a general treatment of perturbation analysis for smooth optimization, and to \cite{shapiro2014lectures} for its application in statistical inference for stochastic optimization.
In machine learning, perturbation analysis has been successfully applied to approximate cross-validation for model evaluation and tuning parameter selection. Exact cross-validation randomly splits the data into many folds, and then repeatedly solves empirical risk minimization (ERM) using all but one fold data, which can be computationally prohibitive if the number of folds is large. 
In the context of parametric models, recent works propose to solve the ERM problem only once with full data, and then apply a one-step Newton update to the full-data estimate to approximate the estimate when each fold of data is  excluded \citep[\eg,][]{wilson2020approximate,stephenson20a,giordano19a}.  
\cite{koh17a} employ similar ideas to quantify the importance of a data point in model training by approximating the parameter estimate change if the data distribution is infinitesimally perturbed towards the data point of interest. 
All of these works only focus on unconstrained optimization problems.}

\section{Supplementary Lemmas and Propositions}\label{sec: supplement}

\begin{lemma}[Convergence of $\hat z_0$]\label{lemma: consistency-est-sol}
Suppose the following conditions hold:
\begin{enumerate} 
\item $\sup_{z}|\widehat{p_0f_0}(z) - p_0f_0(z)| \overset{a.s.}{\longrightarrow} 0$ as $n \to \infty$. 
\item $f_0$ is a continuous function and $f_0(z)$ has a unique minimizer $z_0$ over $\mathbb R^d$. 
\item For large enough $n$, $\argmin_z \widehat{p_0f_0}(z)$ is almost surely a nonempty and uniformly bounded set.
\end{enumerate}
Then $\hat z_0 \overset{a.s.}{\longrightarrow} z_0$ as $n \to \infty$. 
\end{lemma}

\begin{proposition}[Estimation for Squared Cost Function.]\label{prop: square-cost}
When $c(z; y) = \frac{1}{2}\|z - y\|^2$, 
$$
\frac1{n}\sum_{i=1}^n\indic{X_i\in R_0}c(\hat z_0;Y_i)+\frac12\hat{\crit}^\text{apx-risk}(R_1,R_2)=\hat{\crit}^\text{apx-soln}(R_1,R_2)=\sum_{j=1,2}\frac{n_j}{2n}\sum_{l = 1}^d \op{Var}(\{Y_{i, l}:X_i\in R_j,i\leq n\}),
$$
\end{proposition}

\begin{proposition}[Gradient and Hessian for \cref{ex: portfolio-var}]\label{prop: gradient-var-port}
For the cost function $c(z; y)$ in \cref{eq:portfolio-var} and $f_j(z)= \Eb{c(z;Y) \mid {X\in R_j}}$, we have
\begin{align*}
\nabla f_j(z_0) &= 
2\begin{bmatrix}
\Eb{YY^\top \mid {X \in R_j}}z_{0, 1:d} - \Eb{Y \mid X \in R_j}z_{0, d+1} \\ 
z^\top_{0, 1:d}\prns{\Eb{Y \mid X \in R_0} - \Eb{Y \mid X \in R_j}}
\end{bmatrix}, \\
\nabla^2 f_0(z_0) &=
2\begin{bmatrix}
\Eb{YY^\top \mid {X \in R_0}} & -\Eb{Y \mid X \in R_0} \\
-\Eb{Y^\top \mid X \in R_0} & 1
\end{bmatrix}.
\end{align*}
\end{proposition}

\begin{proposition}[Gradient and Hessian for \cref{ex: portfolio}]\label{prop: gradient-portfolio}
Consider the cost function $c(z; y)$ in \cref{eq:portfolio} and $f_j(z)= \Eb{c(z;Y) \mid {X\in R_j}}$.
If $Y$ has a continuous density function \edit{and $z_0 \ne 0$}, then 
\begin{align*}
&\nabla f_j(z_0) = 
\frac{1}{\alpha}
\begin{bmatrix}
-\Eb{Y\indic{Y^\top z_{0, 1:d} \le q^{\alpha}_0(Y^\top z_{0, 1:d})} \mid X \in R_j} \\
 \Prb{q^{\alpha}_0(Y^\top z_{0, 1:d}) - Y^\top z_{0, 1:d} \ge 0 \mid X \in R_j} - \alpha
\end{bmatrix}, \\
&\nabla^2 f_0(z_0) = 
\frac{\mu_{0}\prns{q^{\alpha}_0(Y^\top z_{0, 1:d})}}{\alpha}
\begin{bmatrix}
\Eb{YY^\top \mid Y^\top z_{0, 1:d} = q^{\alpha}_0(Y^\top z_{0, 1:d}), X \in R_0} &  -\Eb{Y \mid Y^\top z_{0, 1:d} = q^{\alpha}_0(Y^\top z_{0, 1:d}), X \in R_0}  \\
-\Eb{Y^\top \mid Y^\top z_{0, 1:d} = q^{\alpha}_0(Y^\top z_{0, 1:d}), X \in R_0} & 1
\end{bmatrix},
\end{align*} 
where $q^{\alpha}_0(Y^\top z_{0, 1:d})$ is the $\alpha$-quantile of $Y^\top z_{0, 1:d}$ given $X \in R_0$ and $\mu_0$ is the density function of $Y^\top z_{0, 1:d}$ given $X \in R_0$.

If further $Y \mid X \in R_0$ has Gaussian distribution with mean $\mu_0$ and covariance matrix $\Sigma_0$, then 
\begin{align*}
&\Eb{Y \mid Y^\top z_{0, 1:d} = q^{\alpha}_0(Y^\top z_{0, 1:d}), X \in R_0} = m_0 + \Sigma_0 z_{0, 1:d} \prns{z_{0, 1:d}^\top \Sigma_0 z_{0, 1:d}}^{-1}(q^{\alpha}_0(Y^\top z_{0}) - m_0^\top z_{0, 1:d}), \\ 
&\var\prns{Y \mid Y^\top z_{0, 1:d} = q^{\alpha}_0(Y^\top z_{0, 1:d}), X \in R_0} = \Sigma_0 - \Sigma_0 z_{0, 1:d} \prns{z_{0, 1:d}^\top \Sigma_0 z_{0, 1:d}}^{-1}z_{0, 1:d}^\top \Sigma_0.
\end{align*}
\end{proposition}

\begin{lemma}[Sufficient conditions for \cref{thm:critconverge}]\label{lemma: suff-cond-crit-converg}
Under conditions in \cref{lemma: consistency-est-sol}, if we further assume the following conditions:
\begin{enumerate}
\item $\Prb{X \in R_j} > 0$ for $j = 0, 1, 2$; 
\item $\nabla^2 f_0(z)$ is continuous, and $\nabla^2 f_0(z_0)$ is invertible;
\item \edit{there exist} a compact neighborhood $\mathcal N$ around $z_0$ such that $\sup_{z \in \mathcal N}\|\hat H_0(z) - \nabla^2 f_{0}(z)\| = o_p(1)$, $\sup_{z \in \mathcal N}\|\hat h_j(z) - \nabla f_j(z)\| = O_p(n^{-1/2})$, and $\{\indic{X_i \in R_0}c(z; Y_i): z \in \mathcal N\}$ is a Donsker class; 
\item conditions in \cref{lemma: consistency-est-sol} hold; 
\end{enumerate}
Then $\|\hat H^{-1}_0(\hat z_0) - \prns{\nabla^2 f_{0}(z_0)}^{-1}\| = o_p(1)$, $\|\hat h_j(\hat z_0) - \nabla f_j(z_0)\| = O_p(n^{-1/2})$ for $j = 1, 2$.
\end{lemma}

\begin{proposition}[Regularity conditions for Estimators in \cref{ex: mnv}]\label{prop: ex-mnv}
Consider the esstimates $\prns{\hat h_j(\hat z_0)}_l={\frac{1}{n_j}\sum_{i=1}^n\indic{X_i\in R_j,\,Y_l\leq \hat z_{0, l}}}$ and $\prns{\hat H_0(\hat z_0)}_{l}=(\alpha_l+\beta_l)\frac{1}{n_jb}\sum_{i=1}^n\indic{X_i\in R_j}\mathcal K((Y_{i, l}-\hat z_{0, l})/b)$ given in \cref{ex: mnv}. Suppose the following conditions hold:
\begin{enumerate}
\item $\Prb{X \in R_j} > 0$ for $j = 0, 1, 2$;
\item The density function $\mu_l(z)$ is H\"older continuous, \ie, \edit{there exist} a constant $0 < a \le 1$ such that $|\mu_l(z) - \mu_l(z')| \le \|z - z'\|^a$ for $z, z' \in \mathbb R^d$, and  $\mu_l(z) > 0$ for $z$ in a neighborhood around $z_0$; 
\item the bandwidth $b$ satisfies that $b \ge \log n/n$ and $b \to 0$ as $n \to \infty$; 
\item conditions in \cref{lemma: consistency-est-sol} hold. 
\end{enumerate}
Then $\prns{\hat h_j(\hat z_0)}_l$ and $\prns{\hat H_0(\hat z_0)}_{l}$ given in \cref{ex: mnv} satisfy the conditions in \cref{thm:critconverge}.
\end{proposition}

\begin{proposition}[Regularity conditions for Estimators in \cref{ex: portfolio-var}]\label{prop: ex-portfolio-var}
Consider the estimators $\hat h_j(\hat z_0)$ and $\hat H_0(\hat z_0)$ given in \cref{ex: portfolio-var}:
\begin{align*}
&\hat h_j(\hat z_0) = 
2\begin{bmatrix}
\frac{1}{n_j}\sum_{i = 1}^n \indic{X_i \in R_j} Y_iY_i^\top \hat z_{0, 1:d} - \frac{1}{n_j} \sum_{i = 1}^n \indic{X_i \in R_j} Y_i \hat z_{0, d+1}  \\
\hat z^\top_{0, 1:d}\prns{\frac{1}{n_0}\sum_{i = 1}^n \indic{X_i \in R_0} Y_i - \frac{1}{n_j}\sum_{i = 1}^n \indic{X_i \in R_j} Y_i}
\end{bmatrix},  \\
&\hat H_0(\hat z_0) = 
2\begin{bmatrix}
\frac{1}{n_0}\sum_{i = 1}^n \indic{X_i \in R_0}Y_iY_i^\top & -\frac{1}{n_0}\sum_{i = 1}^n \indic{X_i \in R_0}Y_i \\
-\frac{1}{n_0}\sum_{i = 1}^n \indic{X_i \in R_0}Y_i^\top & 1
\end{bmatrix}.
\end{align*}
If conditions in \cref{lemma: consistency-est-sol} and condition 1 in \cref{lemma: suff-cond-crit-converg} hold  and $\var\prns{Y \mid X \in R_0}$ is vertible, then $\hat h_j(\hat z_0)$ and $\hat H_0(\hat z_0)$ satisfy conditions in \cref{thm:critconverge}.
\end{proposition}

\begin{proposition}[Regularity conditions for Estimators in \cref{ex: portfolio}]\label{prop: ex-portfolio}
Consider the estimator $\hat h_j(\hat z_0)$ and $\hat H_0(\hat z_0)$ given in \cref{ex: portfolio}:
\begin{align*}
&\hat h_j(\hat z_0) = 
\begin{bmatrix}
-\frac{1}{n_j}\sum_{i = 1}^n \indic{Y_i^\top \hat z_{0,1:d} \le \hat q^{\alpha}_0(Y^\top \hat z_{0,1:d}), X_i \in R_j}Y_i \\
\frac{1}{n_j}\sum_{i = 1}^n \indic{Y^\top_i \hat z_{0, 1:d} \le q^{\alpha}_0(Y^\top \hat z_{0,1:d}), X_i \in R_j} - \alpha
\end{bmatrix} \\
&\hat H_0(\hat z_0) = \frac{\hat \mu_0(\hat q^{\alpha}_0(Y^\top \hat z_{0}))}{\alpha}
\begin{bmatrix}
\hat M_2 & -\hat M_1 \\
-\hat M_1 & 1
\end{bmatrix}
\end{align*}
where 
\begin{align*}
&\hat M_1 = \hat m_0 + \hat \Sigma_0 \hat z_0 \prns{\hat z_{0,1:d}^\top \hat \Sigma_0 \hat z_{0,1:d}}^{-1}(\hat q^{\alpha}_0(Y^\top \hat z_{0,1:d}) - \hat m_0^\top \hat z_{0,1:d})  \\
&\hat M_2 = \hat M_1\hat M_1^\top + \hat \Sigma_0 - \hat \Sigma_0 \hat z_{0, 1:d} \prns{\hat z_{0,1:d}^\top \hat \Sigma_0 \hat z_{0,1:d}}^{-1}\hat z_{0,1:d}^\top \hat \Sigma_0
\end{align*}

If the conditions in \cref{lemma: consistency-est-sol} and condition 1 in \cref{lemma: suff-cond-crit-converg} holds, the density function of $Y^\top z_0$  is positive at $q^\alpha_0(Y^\top z_0)$ and it also satisfies the H\"older continuity condition, \ie, condition 2 in \cref{prop: ex-mnv}, and also the bandwidth satisfies the condition 3 in \cref{prop: ex-mnv}, then  $\hat h_j(\hat z_0) = \nabla f_j(z_0) + O_p(n^{-1/2})$ and $\|\hat H_0(\hat z_0) - \nabla^2 f_0(z_0)\| \to 0$.
\end{proposition}

\begin{proposition}[Regularity Conditions for Estimators in \cref{ex: smooth}]\label{prop: ex-smooth}
Suppose that $c(z;y)$ is  twice continuously  differentiable in $z$ for every $y$ and conditions in \cref{lemma: consistency-est-sol}   and condition 1 in  \cref{lemma: suff-cond-crit-converg} hold, then the the conditions in \cref{thm:critconverge} are satisfied for 
estimates 
$\hat H_0(\hat z_0)=\frac{1}{n_0}\sum_{i=1}^n\indic{X_i\in R_0}\nabla^2 c\prns{\hat z_0;Y_i}$ and $\hat h_j(\hat z_0)={\frac{1}{n_j}\sum_{i=1}^n\indic{X_i\in R_j}\nabla c\prns{\hat z_0;Y_i}}$ given in \cref{ex: smooth}. 
\end{proposition}

\edit{
    Below we introduce the linear independence constraint qualification condition for deterministic constraints only. See \cref{eq: strict-cq2} for a more complete condition with both deterministic constraints and stochastic constraints. 
}
\edit{\begin{definition}[Linear Independence Constraint Qualification]\label{def: LICQ}
Consider constraints
\begin{align*}
\Z=\braces{z\in\R d~~:~~
\begin{array}{ll}
h_{k}(z) = 0,~~&~~k=1,\dots,s,\\ h_{k}(z) \le 0,~~&~~k=s + 1,\dots, m
\end{array}}
\end{align*}
and the index set of inequality constraints active at a  point $z_0\in\Z$ denoted as $K_h(z_0) = \{k: h_{k}(z_0) = 0, k = s+1, \cdots, m\}$. 
The linear independence constraint qualification condition is satisfied at $z_0 \in \Z$ if $\braces{\nabla h_k\prns{z_0}: k \in \{1, \dots, s\} \cup K_h(z_0)}$ are linearly independent. 
\end{definition}
}

\edit{
According to \cite{WACHSMUTH201378}, the linear independence constraint qualification (LICQ) condition is a sufficient condition for the  Mangasarian-Fromovitz constraint qualification condition (condition \ref{cond: constr-MF} in \cref{thm:secondorder-const}). Moreover, when the LICQ condition is satisfied at a optimal solution $z_0$, then it has a unique Lagrangian multiplier $v_0$ such that $\prns{z_0, v_0}$ satisfy the Karush–Kuhn–Tucker conditions (condition \ref{cond: constr-uniquenss} in \cref{thm:secondorder-const}).  
In the proposition below, we show that the LICQ condition is satisfied for any $z\in\Z$ for the constraints $\Z$ given in \cref{ex: mnv,ex: portfolio-var,ex: portfolio}, so conditions \ref{cond: constr-uniquenss} and \ref{cond: constr-MF} in \cref{thm:secondorder-const}  are satisfied for these examples.  
}
\edit{
\begin{proposition}\label{prop: LICQ}
\cref{ex: mnv} with the constraints $\Z = \braces{z\in \mathbb R^d:  \sum_{l = 1}^d z_l \le C, ~ z_l \ge 0, ~ l = 1, \dots, d}$  and \cref{ex: portfolio-var,ex: portfolio} with the simplex constraint $\Z=\{z\in\R{d+1}: \sum_{l = 1}^d z_l = 1,\,z_l\geq0,\,l=1,\dots,d\}$ all satisfy the linear independence constraint qualification condition in \cref{def: LICQ} at any $z \in \Z$.  
\end{proposition}
}

\edit{
Finally, we point out the splitting criterion considered in \cite{elmachtoub2020decision} is what we termed the oracle criterion in \cref{eq:oraclecrit}.
\begin{proposition}\label{prop: equi-spo-stochopt}
When $c\prns{z; y} = y^\top z$, the Smart Predict-then-Optimize (SPO) criterion in \cite{elmachtoub2020decision} is equivalent the oracle splitting criterion in \cref{eq:oraclecrit}.
\end{proposition}
}

\section{Omitted Proofs}\label{sec: proofs}
\subsection{Proofs for \cref{app-sec: perturbation}}
\begin{proof}{Proof for \cref{prop: implict-fun-thm}}
\edit{
Recall that 
\begin{align*}
v_j\prns{t} = f_{0}(z_j\prns{t})+t\prns{f_{j}(z_j\prns{t}) - f_{0}(z_j\prns{t})} + \sum_{k \in \tilde K_h\prns{z_0}}\nu_{j, k}\prns{t}h_k\prns{z_j\prns{t}},
\end{align*}
where $\tilde K_h\prns{z_0} = \braces{1, \dots, s} \cup K_h\prns{z_0}$.
}

\edit{
Taking the derivatives w.r.t $t$ based on the chain rule, we have 
\begin{align*}
\frac{\partial}{\partial t} v_j\prns{t} 
        &= 
    f_j\prns{z_j\prns{t}} - f_0\prns{z_j\prns{t}} + \sum_{k\in \tilde K_h\prns{z_0}} \prns{\frac{\partial}{\partial t} \nu_{j, k}\prns{t}} h_k\prns{z_j\prns{t}}\\
     &+ \prns{\nabla f_0\prns{z_j\prns{t}} + \sum_{k \in \tilde K_h\prns{z_0}} \nu_{j, k}\prns{t}\nabla h_k\prns{z_j\prns{t}} + t\prns{\nabla f_j\prns{z_j\prns{t}} - \nabla f_0\prns{z_j\prns{t}}}}\frac{\partial}{\partial t} z_j\prns{t}.
\end{align*}
Therefore, 
\begin{align*}
\frac{\partial}{\partial t} v_j\prns{t}\vert_{t=0} 
    &= f_j\prns{z_0} - f_0\prns{z_0} + 
\prns{\nabla f_0\prns{z_0} + \sum_{k \in \tilde K_h\prns{z_0}} \nu_{0, k}\nabla h_k\prns{z_0}}\frac{\partial}{\partial t} z_j\prns{t}\vert_{t=0} +
\sum_{k\in \tilde K_h\prns{z_0}} \prns{\frac{\partial}{\partial t} \nu_{j, k}\prns{t}\vert_{t=0}} h_k\prns{z_0} \\
    &= f_j\prns{z_0} - f_0\prns{z_0},
\end{align*}
where the second equation holds because $\nabla f_0\prns{z_0} + \sum_{k \in \tilde K_h\prns{z_0}} \nu_{0, k}\nabla h_k\prns{z_0} = 0$ according to \cref{eq: KKT0-2}, and $h_k\prns{z_0} = 0$ for any $k \in \tilde K_h\prns{z_0}$ by the definition of $K_h\prns{z_0}$.
}

\edit{
Further taking the second order derivatives, we have 
\begin{align*}
\frac{\partial^2}{\partial t^2} v_j\prns{t}  
    &= 2\prns{\nabla f_j\prns{z_j\prns{t}} - \nabla f_0\prns{z_j\prns{t}}}\frac{\partial}{\partial t} z_j\prns{t} + \sum_{k\in \tilde K_h\prns{z_0}} \prns{\frac{\partial^2}{\partial t^2} \nu_{j, k}\prns{t}} h_k\prns{z_j\prns{t}} \\
    &+ 2\sum_{k\in \tilde K_h\prns{z_0}} \prns{\frac{\partial}{\partial t} \nu_{j, k}\prns{t}} \nabla^\top h_k\prns{z_j\prns{t}} \prns{\frac{\partial}{\partial t} z_{j}\prns{t}} \\
    &+ \prns{\nabla f_0\prns{z_j\prns{t}} + \sum_{k \in \tilde K_h\prns{z_j\prns{t}}} \nu_{j, k}\prns{t}\nabla h_k\prns{z_j\prns{t}} + t\prns{\nabla f_j\prns{z_j\prns{t}} - \nabla f_0\prns{z_j\prns{t}}}}\frac{\partial^2}{\partial t^2} z_j\prns{t} \\
    &+ \prns{\frac{\partial}{\partial t} z_j\prns{t}}^\top \prns{\nabla^2 f_0\prns{z_j\prns{t}} + \sum_{k \in \tilde K_h\prns{z_0}} \nu_{j, k}\prns{t}\nabla^2 h_k\prns{z_j\prns{t}}}\prns{\frac{\partial}{\partial t} z_j\prns{t}}.
\end{align*}
Evaluating the above at $t = 0$ gives 
\begin{align*}
\frac{\partial^2}{\partial t^2} v_j\prns{t}\vert_{t = 0}    &= 2\prns{\nabla f_j\prns{z_0} - \nabla f_0\prns{z_0}}d_z^{j*} + 2\sum_{k\in \tilde K_h\prns{z_0}} \prns{\frac{\partial}{\partial t} \nu_{j, k}\prns{t}\vert_{t=0}} \nabla^\top h_k\prns{z_0}d_z^{j*} \\
    &+\prns{d_z^{j*}}^\top\prns{\nabla^2 f_0\prns{z_j\prns{t}} + \sum_{k \in \tilde K_h\prns{z_0}} \nu_{j, k}\prns{t}\nabla^2 h_k\prns{z_j\prns{t}}}d_z^{j*}.
\end{align*}
According to \cref{eq: kkt-perturb}, we know that 
\begin{align*}
&2\sum_{k\in \tilde K_h\prns{z_0}} \prns{\frac{\partial}{\partial t} \nu_{j, k}\prns{t}\vert_{t=0}} \nabla^\top h_k\prns{z_0}d_z^{j*} 
    = 2\sum_{k\in \tilde K_h\prns{z_0}} \xi_j \nabla^\top h_k\prns{z_0}d_z^{j*} \\
    =& -2 \prns{\nabla  f_j(z_0) - \nabla f_0(z_0)}^\top d_z^{j*}  - 2\prns{d_z^{j*}}^\top\prns{\nabla^2 f_0\prns{z_j\prns{t}} + \sum_{k \in \tilde K_h\prns{z_0}} \nu_{j, k}\prns{t}\nabla^2 h_k\prns{z_j\prns{t}}}d_z^{j*}.
\end{align*}
Moreover, note that by the first equation in \cref{eq: kkt-perturb}, we have 
\begin{align*}
\prns{d_z^{j*}}^\top\prns{\nabla^2 f_0\prns{z_j\prns{t}} + \sum_{k \in \tilde K_h\prns{z_0}} \nu_{j, k}\prns{t}\nabla^2 h_k\prns{z_j\prns{t}}}d_z^{j*} + \prns{d_z^{j*}}^\top\nabla {\mathcal H^{K_h}}^\top(z_0)\xi_j = - \prns{\nabla  f_j(z_0) - \nabla f_0(z_0)}^\top d_z^{j*},
\end{align*}
and by the second equation in \cref{eq: kkt-perturb}, we have 
\begin{align*}
\prns{d_z^{j*}}^\top\nabla {\mathcal H^{K_h}}^\top(z_0)\xi_j = 0.
\end{align*}
Thus 
\begin{align*}
\prns{d_z^{j*}}^\top\prns{\nabla^2 f_0\prns{z_j\prns{t}} + \sum_{k \in \tilde K_h\prns{z_0}} \nu_{j, k}\prns{t}\nabla^2 h_k\prns{z_j\prns{t}}}d_z^{j*}  = - \prns{\nabla  f_j(z_0) - \nabla f_0(z_0)}^\top d_z^{j*}.
\end{align*}
It follows that 
\begin{align*}
\frac{\partial^2}{\partial t^2} v_j\prns{t}\vert_{t = 0}&= - \prns{d_z^{j*}}^\top\prns{\nabla^2 f_0\prns{z_j\prns{t}} + \sum_{k \in \tilde K_h\prns{z_0}} \nu_{j, k}\prns{t}\nabla^2 h_k\prns{z_j\prns{t}}}d_z^{j*} \\
&= \prns{d_z^{j*}}^\top\prns{\nabla^2 f_0\prns{z_j\prns{t}} + \sum_{k \in \tilde K_h\prns{z_0}} \nu_{j, k}\prns{t}\nabla^2 h_k\prns{z_j\prns{t}}}d_z^{j*} + 2 \prns{\nabla  f_j(z_0) - \nabla f_0(z_0)}^\top d_z^{j*}.
\end{align*}
}
\end{proof}

\begin{proof}{Proof for \cref{prop: first-order-simplification}}
Note that the Lagrangian multiplier set $\Lambda(z^*, u_0)$ for the unperturbed problem can be written as follows:
\[
	\Lambda(z^*, u_0) = \left\{\lambda: 
	\begin{array}{l}
	\lambda_k \ge 0 \text{ if } k \in K(z^*, u_0), \lambda_k =0, \text{ if } k \in \{s+1, \dots, m\} \setminus K(z^*, u_0), \\
	\lambda_k \in \mathbb{R} \text{ if } k \in \{1, \dots, s\}, \\
	D_zf(z^*, u_0) + \sum_{k = 1}^m \lambda_k D_z g_k(z^*, u_0) = 0
	\end{array}\right\}.
\]
Similarly, we can define the Lagrangian for the problem PL:
\[
	L_{\PL}(d_z; \lambda) = Df(z^*, u_0)(d_z, d_u) + \sum_{k = 1}^s  \lambda_k Dg_k(z^*, u_0)(d_z, d_u) + \sum_{k \in K(z^*, u_0)} \lambda_k Dg_k(z^*, u_0)(d_z, d_u).
\]
The multiplier set for any $d_z$ feasible for the problem PL as follows:
\[
\Lambda_{\PL}(d_z) = \left\{\lambda: 
	\begin{array}{l}
	\lambda_k \ge 0 \text{ if } k \in K_{\PL}(z^*, u_0, d_z), \lambda_k =0, \text{ if } k \in K(z^*, u_0) \setminus K_{\PL}(z^*, u_0, d_z) \\
	\lambda_k \in \mathbb R \text{ if } k \in \{1, \dots, s\}\cup \left(\{s+1, \dots, m\}\setminus K(z^*, u_0)\right),\\
	D_zf(z^*, u_0) + \sum_{k = 1}^m \lambda_k D_z g_k(z^*, u_0) = 0
	\end{array}\right\}.
\]
Then by duality of linear program, for any $d_z^* \in S(PL)$, 
\[
	V(\PL) = \max_{\lambda} L_{\PL}(d_z^*; \lambda).
\]
We know that any $\lambda^*_{\PL} \in \Lambda_{PL}(d_z^*)$ attains the maximum above. For any $\lambda \in \Lambda_{\PL}(d_z^*)$ or $\lambda \in \Lambda(z^*, u_0)$, by the fact that $D_zf(z^*, u_0) + \sum_{k = 1}^m \lambda_{k} D_z g_k(z^*, u_0) = 0$ , we also have $L_{\PL}(d_z^*; \lambda) = D_u L(z^*, \lambda, u_0) d_u$. Moreover, $ \Lambda_{\PL}(d_z^*)$ differs with $\Lambda(z^*, u_0)$ only in two aspects: (1) for $\lambda \in \Lambda(z^*, u_0)$, $\lambda_k = 0$ for $k \in \{s+1, \dots, m\}\setminus K(z^*, u_0)$, but for $\lambda \in \Lambda_{\PL}(d_z^*)$, $\lambda_k \in \mathbb R$ for $k \in \{s+1, \dots, m\}\setminus K(z^*, u_0)$, which does not matter because $L_{\PL}(d_z; \lambda)$ does not depend on $\lambda_k$ for $k \in \{s+1, \dots, m\}\setminus K(z^*, u_0)$; (2) for $\lambda \in \Lambda(z^*, u_0)$, $\lambda_k \ge 0$ for $k \in K(z^*, u_0)\setminus K_{\PL}(z^*, u_0, d_z^*)$, but for $\lambda \in \Lambda_{\PL}(d_z^*)$, $\lambda_k = 0$ for $k \in K(z^*, u_0)\setminus K_{\PL}(z^*, u_0, d_z^*)$, which does not matter as well because $Dg_k(z^*, u_0)(d_z^*, d_u) = 0$ for $k \in K(z^*, u_0)\setminus K_{\PL}(z^*, u_0, d_z^*)$ so that $\lambda_k$ for $k \in K(z^*, u_0)\setminus K_{\PL}(z^*, u_0, d_z^*)$ do not influence $L_{\PL}(d_z^*; \lambda)$ as well. 
This means that for any $\lambda^*_{\PL} \in \Lambda_{PL}(d_z^*)$, there always exists $\lambda' \in \Lambda(z^*, u_0)$ such that $L_{\PL}(d_z^*; \lambda^*_{\PL}) = \max_\lambda L_{\PL}(d_z^*; \lambda) = L_{\PL}(d_z^*; \lambda')$.

Therefore, for any $d_z^* \in S(\PL)$, 
\[
	V(\PL) = \max_{\lambda} L_{\PL}(d_z^*; \lambda)  = \max_{\lambda \in\Lambda(z^*, u_0)} D_u L(z^*, \lambda, u_0) d_u = V(\DL),
\]
which justifies the dual formulation in \cref{eq: DL}. 

By the definition of Lagrangian multiplier set, $S(\DL) = \Lambda_{\PL}(d_z^*)$ for any $d_z^* \in S(\PL)$. Now we consider the Lagrangian of the problem PQ$(d_z)$: 
\begin{align*}
L_{\PQ}(r_z; \lambda)
	&= Df(z^*, u_0)(r_z, r_u) + D^2 f(z^*, u_0)((d_z, d_u), (d_z, d_u)) \\
	&+ \sum_{k \in K_{\PL}(z^*, u_0, d_z) \cap \{1, \dots, s\}} \lambda_k \prns{Dg_k(z^*, u_0)(r_z, r_u) + D^2 g_k(z^*, u_0)((d_z, d_u), (d_z, d_u))}.
\end{align*}
Note that $\PQ(d_z)$ is a linear program, and by the strong duality, we have that for any $r_z^* \in S(\PQ(d_z))$
\begin{align*}
V(\PQ(d_z)) = V(\DQ(d_z)) = \max_\lambda L_{\PQ}(r_z^*; \lambda).
\end{align*}

The set of Lagrangian multipliers that attain the maximum above is 
\begin{align*}
\Lambda_{\PQ(d_z)}(r_z^*) = 
\left\{\lambda: 
	\begin{array}{l}
	\lambda_k \ge 0 \text{ if } k \in K_{\PQ}(z^*, u_0, r_z), \lambda_k =0, \text{ if } k \in K_{\PL}(z^*, u_0, d_z)\setminus K_{\PQ}(z^*, u_0, r_z) \\
	\lambda_k \in \mathbb R \text{ if } k \in \{1, \dots, s\}\cup \left(\{s+1, \dots, m\}\setminus K_{\PL}(z^*, u_0, d_z)\right),\\
	D_z f(z^*, u_0) + 
	\sum_{k \in K_{\PL}(z^*, u_0, d_z) \cap \{1, \dots, s\}} \lambda_k D_z g_k(z^*, u_0) = 0
\end{array}
\right\},
\end{align*}
where $K_{\PQ}(z^*, u_0, r_z)$ is 
the index set of active inequality constraints in the problem $PQ(d_z)$, \ie,
$$
K_{\PQ}(z^*, u_0, r_z) = \{k \in K_{\PL}(z^*, u_0, d_z): Dg_k(z^*, u_0)(r_z, r_u) + D^2 g_k(z^*, u_0)((d_z, d_u), (d_z, d_u)) = 0\}.
$$
Thus for any $\lambda \in \Lambda_{\PQ(d_z)}$, 
\begin{align*}
V(\PQ(d_z)) = V(\DQ(d_z)) =  L_{\PQ}(r_z^*; \lambda) = D_u L(z^*, u_0; \lambda)r_u + D^2 L(z^*, u_0; \lambda)((d_z, d_u), (d_z, d_u)).
\end{align*}
Again, $\Lambda_{\PQ(d_z)}(r_z^*)$ differs with $S(\DL) = \Lambda_{\PL}(d_z^*)$ only in aspects that do no influence the value of $L_{\PQ}(r_z^*; \lambda)$. So for any $\lambda \in \Lambda_{\PQ(d_z)}$, there always exists $\lambda' \in S(\DL)$, such that $L_{\PQ}(r_z^*; \lambda) = L_{\PQ}(r_z^*; \lambda')$. Therefore, 
\begin{align*}
V(\PQ(d_z)) = V(\DQ(d_z)) =  L_{\PQ}(r_z^*; \lambda) = \sup_{\lambda \in S(\DL)}L_{\PQ}(r_z^*; \lambda).
\end{align*}
This proves the dual formulation in \cref{eq: DQ-d}.

It follows that if the optimal dual solution of the unperturbed problem is unique, \ie, $\Lambda(z^*, u_0) = \{\lambda^*\}$, then 
\[
V(\PL) = V(\DL) = D_u L(z^*, \lambda^*, u_0)d_u.
\]
Since $V(PQ)$ is finite, $S(PL) \ne \emptyset$. By strong duality, we have $\emptyset \ne S(DL) \subseteq \Lambda(z^*, u_0) = \{\lambda^*\}$, thus we must have $S(DL) = \{\lambda^*\}$. Therefore, 
\begin{align*}
V(\PQ(d_z)) = V(\DQ(d_z)) 
	&= \max_{\lambda \in S(DL)} D_u L(z^*, \lambda, u_0)r_u + D^2 L(z^*, \lambda, u_0)((d_z, d_u), (d_z, d_u))  \\
	&= D_u L(z^*, \lambda^*, u_0)r_u + D^2 L(z^*, \lambda^*, u_0)((d_z, d_u), (d_z, d_u)).
\end{align*}
\end{proof}

\begin{proof}{Proof for \cref{prop: second-order-simplify}.}
Under the asserted strict complementarity condition, $K_0(z^*, u_0, \lambda^*)= \emptyset$ and $K_+(z^*, u_0, \lambda^*) = K(z^*, u_0)$, so 
\begin{align*}
S(\PL) = \left\{d_z: \begin{array}{l}
Dg_k(z^*, u_0)(d_z, d_u) = 0, k \in \{1, \dots, s\} \cup K(z^*, u_0)
\end{array}\right\}.
\end{align*}
According to \cref{prop: first-order-simplification}, we have 
\begin{align*}
V(\PQ) = V(\DQ) = \min_{d_z \in S(PL)} D_u L(z^*, \lambda^*, u_0)r_u + D^2 L(z^*, \lambda^*, u_0)((d_z, d_u), (d_z, d_u)).
\end{align*}
The asserted conclusion then follows. 
\end{proof}

\begin{proof}{Proof for \cref{lemma: additive-perturb}.}
Note that the optimization problem  in \cref{eq: perturb-additive} corresponds to perturbation path $u(t) = t$, \ie, $u_0 = 0, d_u = 1, r_u = 0$. By assuming unqiue Lagrangian multipliers $(\lambda^*, \nu^*)$, we have  
\begin{align*}
V(PL) = V(DL) &=
	\nabla_t L(z^*, 0; \lambda^*, \nu^*) d_u \\
 	&= \delta_f(z^*) + \sum_k\lambda_k^* \delta_{g_k}(z^*).
\end{align*}
Moreover, because $d_u = 1, r_u = 0$,
\begin{align*}
V(PQ) = V(DQ) 
	&= \min_{d_z \in S(PL)} d_z^\top \nabla_{zz}^2 L(z^*, 0; \lambda^*, \nu^*) d_z + 2d_z^\top \nabla_{zt}^2 L(z^*, 0; \lambda^*, \nu^*) +  
	\nabla_{tt}^2 L(z^*, 0; \lambda^*, \nu^*) 
\end{align*}
Note that  
\begin{align*}
&\nabla_{tt}^2 L(z^*, 0; \lambda^*, \nu^*) = 0, \\
&\nabla_{tz}^2 L(z^*, 0; \lambda^*, \nu^*) = \nabla \delta_f(z^*) + \sum_{k = 1}^s \lambda_k^* \nabla \delta_{g_k}(z^*).
\end{align*}
Then \cref{eq: second-order-problem} follows from the fact that the constraints in DQ now reduces to the following:
\begin{align*}
&\left\{d_z^\top \nabla_z \left[g_k(z) + t\delta_{g_k}(z)\right]+ d_u \nabla_t \left[g_k(z) + t\delta_{g_k}(z)\right]\right\}\vert_{(z,t) = (z^*, 0)} \\
=& d_z^\top \left[\nabla g_k(z) + t\nabla \delta_{g_k}(z)\right]\vert_{(z,t) = (z^*, 0)}  + \nabla_t \left[g_k(z) + t\delta_{g_k}(z)\right]\vert_{(z,t) = (z^*, 0)}  \\
=& d_z^\top \nabla g_k(z^*) + \delta_{g_k}(z^*), \\
&\left[d_z^\top \nabla_z h_k(z) +  d_u \nabla_t h_k(z)\right]\vert_{(z,t) = (z^*, 0)} = d_z^\top \nabla h_k(z^*).
\end{align*}
\end{proof}

\begin{proof}{Proof for \cref{corollary: unconstr}}
Note that under the asserted conditions, conditions 1, 2, 4, 5 in \cref{lemma: additive-perturb} hold, and condition 3 in \cref{lemma: additive-perturb} degenerates and thus holds trivially. 

Note that $v'(0)$ in \cref{lemma: additive-perturb} now reduces to $\delta_f(z^*)$, and 
$$
v''(0) = \min_{d_z} ~ 
	d_z^\top  \nabla^2 f(z^*) d_z + 2d_z^\top\nabla \delta_f(z^*).
$$
Under the condition that $\nabla^2 f(z^*) $ is positive definite (and thus invertible), we have that the optimization problem in the last display has a unique solution $d_z^* = - \prns{\nabla^2 f(z^*)}^{-1}\nabla \delta_f(z^*)$. Consequently, 
\begin{align*}
v''(0) = - \nabla \delta_f(z^*)^\top \prns{\nabla^2 f(z^*)}^{-1}\nabla \delta_f(z^*).
\end{align*}
\end{proof}

\begin{proof}{Proof for \cref{thm: frechet-diff}.}
Consider the function $\phi(t) = v(f_0+t\delta_f)$. Given the asserted conditions, for any $t \in [0, 1]$, $f_0(z)+t\delta_f(z)$  satisfies the conditions in \cref{corollary: unconstr}, thus results in \cref{corollary: unconstr} imply that $\phi(t)$ is twice differentiable:
\begin{align*}
\phi'(t) 
	&= v'(f_0 + t\delta_f; \delta_f) = \delta_f; \\
\phi''(t) &= v''(f_0 + t\delta_f; \delta_f) = - \nabla \delta_f(z^*(f_0+t\delta_f))^\top \prns{\nabla^2 (f_0+t\delta_f)(z^*(f_0+t\delta_f))}^{-1}\nabla \delta_f(z^*(f_0+t\delta_f)).
\end{align*}
	
We now argue that $\phi''(t)$ is also continuous in $t \in [0, 1]$. 
Since condition 2 in \cref{corollary: unconstr} is satisfied for $t \in [0, 1]$, \edit{there exist} a compact set $\mathcal N$ such that that $z^*(f_0 + t\delta_f(z)) \in \mathcal N$ for $t \in [0, 1]$. 
Note that $\sup_{z \in \mathcal N}|f_0(z) + t\delta_f(z) - f_0(z)| \to 0$ as $t \to 0$ by the fact that $\delta_f(z)$ is bounded over $\mathcal N$. Then according to Theorem 5.3 in \cite{shapiro2014lectures}, $z^*(f_0 + t\delta_f) \to z^*(f_0)$ as $t \to 0$. Similarly, $\sup_{z \in \mathcal N}\edit{\|\nabla^2 (f_0 + t\delta_f)(z) - \nabla^2 f_0(z)\|_{\op{F}}} \to 0$ as $t \to 0$.
This convergence together with the continuity of $\nabla^2 f_0(z)$ and $\nabla^2 \delta_f(z)$ imply that $\edit{\|{\nabla^2 (f_0+t\delta_f)(z^*(f_0+t\delta_f))} - {\nabla^2 f_0(z^*(f_0))}\|_{\op{F}}} \to 0$. It then follows from the invertibility of ${\nabla^2 (f_0+t\delta_f)(z^*(f_0+t\delta_f))}$ for any $t \in [0, 1]$ that  
$\edit{\|\prns{\nabla^2 (f_0+t\delta_f)(z^*(f_0+t\delta_f))}^{-1} - \prns{\nabla^2 f_0(z^*(f_0))}^{-1}\|_{\op{F}}} \to 0$. Moreover, by the continuity of $\nabla \delta_f$, we have that $\edit{\|\nabla \delta_f(z^*(f_0+t\delta_f)) - \nabla \delta_f(z^*(f_0))\|_2} \to 0$ as $t \to 0$. These together show that $\phi''(t)$ is continuous in $t \in [0, 1]$. 

Now that $\phi(t)$ is twice continuously differentiable over $[0, 1]$, there exists $t' \in [0, 1]$ such that 
\[
\phi(1) = \phi(0) + \phi'(0) + \frac{1}{2}\phi''(t'),
\]
where $ \phi'(0) = v'(f; \delta_f)$ and $\phi''(t') = v''(f+t'\delta_f; \delta_f)$. Or equivalently, 
\begin{align*}
v(f_0 + \delta_f) = v(f_0) + v'(f; \delta_f) + v''(f_0; \delta_f) +  v''(f_0 + t'\delta_f; \delta_f) - v''(f_0; \delta_f).
\end{align*}
Denote $R_1 = \nabla \delta_f(z^*(f_0+t\delta_f)) -  \nabla \delta_f(z^*(f_0))$ and $R_2 =\prns{ \nabla^2 \prns{f_0+t\delta_f}(z^*(f_0+t\delta_f))}^{-1} - \prns{\nabla^2 f_0(z^*(f_0))}^{-1}$. It is straightforward to verify that 
\begin{align*}
v''(f_0 + t'\delta_f; \delta_f) - v''(f_0; \delta_f) 
	&=  \nabla \delta_f(z^*(f_0))^\top \prns{\nabla^2 f_0(z^*(f_0))}^{-1}\nabla \delta_f(z^*(f_0)) \\
	&- \nabla \delta_f(z^*(f_0+t\delta_f))^\top \prns{\nabla^2 \prns{f_0+t\delta_f}(z^*(f_0+t\delta_f))}^{-1}\nabla \delta_f(z^*(f_0+t\delta_f)) \\
	&= R_1^\top \prns{\nabla^2 f_0(z^*(f_0))}^{-1}R_1 + 2R_1 \prns{\nabla^2 f_0(z^*(f_0))}^{-1} \nabla \delta_f(z^*(f_0)) \\
	&+  \nabla \delta_f(z^*(f_0+t\delta_f))^\top R_2 \nabla \delta_f(z^*(f_0+t\delta_f)).
\end{align*}
As $\delta_f \to 0$, we have $\sup_{z \in \mathcal N}|\prns{f_0 + t\delta_f}(z) - f_0(z)| \to 0$, so that Theorem 5.3 in \cite{shapiro2014lectures} again implies that $\abs{z^*(f_0 + t\delta_f)  - z^*(f_0)}\to 0$. 
It follows that \edit{there exist} a constant $\beta \in [0, 1]$ such that $R_1 = \nabla^2 \delta_f(\beta z^*(f_0 + t\delta_f) + (1- \beta)z^*(f_0))(z^*(f_0 + t\delta_f) - z^*(f_0)) = o(\nabla^2 \delta_f(\beta z^*(f_0 + t\delta_f) + (1- \beta)z^*(f_0))) = o(\|\delta_f\|_\mathcal F)$.
Similarly, we can also prove that $\edit{\|R_2\|_{\op{F}}} \to 0$ as $\delta_f \to 0$. It follows that as $\delta_f \to 0$, 
\begin{align*}
v''(f_0 + t'\delta_f; \delta_f) - v''(f_0; \delta_f)  = o(\|\delta_f\|^2_\mathcal F).
\end{align*}
Therefore, 
\begin{align*}
v(f_0 + \delta_f) = v(f_0) + \delta_f(z^*) - \frac{1}{2} \nabla \delta_f(z^*)^\top \prns{\nabla^2 f(z^*)}^{-1}\nabla \delta_f(z^*) + o(\|\delta_f\|^2_\mathcal F).
\end{align*}
Similarly, we can prove that 
\[
z^*(f_0 + \delta_f) = z^*(f_0) - \prns{\nabla^2 f(z^*)}^{-1}\nabla \delta_f(z^*) + o(\|\delta_f\|_\mathcal F).
\]
\end{proof}
\subsection{Proofs for \cref{sec: auxiliary-var}}
\begin{proof}{Proof for \cref{prop: profile-gradient}.}
By first order optimality condition, for any $z^{\text{dec}}$, 
\begin{align*}
\frac{\partial}{\partial{z^{\text{aux}}}}\Eb{c(z^{\text{dec}}, z^{\text{aux}}_j\prns{z^{\text{dec}}}; Y) \mid X \in R_j} = 0. 
\end{align*}
It follows that 
\begin{align*}
\nabla f_j\prns{z^{\text{dec}}_0} = \frac{\partial}{\partial{z^{\text{dec}}}}\Eb{c(z^{\text{dec}}_0, z^{\text{aux}}_j\prns{z^{\text{dec}}_0}; Y) \mid X \in R_j}.
\end{align*}
Note that 
\begin{align*}
\nabla^2 f_0\prns{z^{\text{dec}}_0}  
	&= \frac{\partial^2}{\prns{\partial{z^{\text{dec}}}}^\top \partial{z^{\text{dec}}}}\Eb{c(z^{\text{dec}}_0, z^{\text{aux}}_0; Y) \mid X \in R_0} \\
	&+ \frac{\partial^2}{\partial{z^{\text{dec}}}\prns{\partial{z^{\text{aux}}}}^\top}\Eb{c(z^{\text{dec}}_0, z^{\text{aux}}_0; Y) \mid X \in R_0}\frac{\partial}{\prns{\partial z^{\text{dec}}}^\top}z^{\text{aux}}_0\prns{z^{\text{dec}}_0}.
\end{align*}
Moreover, under the asserted smoothness condition and invertibility condition, the implicit function theorem futher implies that 
\begin{align*}
\frac{\partial}{\prns{\partial z^{\text{dec}}}^\top}z^{\text{aux}}_0\prns{z^{\text{dec}}_0} 
	&= -\braces{\frac{\partial^2}{\partial{z^{\text{aux}}}\prns{\partial{z^{\text{aux}}}}^\top}\Eb{c(z^{\text{dec}}_0, z^{\text{aux}}_0; Y) \mid X \in R_0}}^{-1} \\
	&\qquad\qquad \times \frac{\partial^2}{\partial{z^{\text{aux}}} \prns{\partial{z^{\text{dec}}}}^\top}\Eb{c(z^{\text{dec}}_0, z^{\text{aux}}_0; Y)\mid X \in R_0},  
\end{align*}
which in turn proves the formula for $\nabla^2 f_j\prns{z^{\text{dec}}_0}$ in \cref{prop: profile-gradient}.
\end{proof}

\begin{proof}{Proof for \cref{prop: aux-var-relation}}
The conclusion follows directly from the following facts that can be easily verified:
\begin{align*}
&\nabla f_j\prns{z^{\text{dec}}_0} = \frac{\partial}{\partial z^{\text{dec}}} \tilde f_j\prns{z^{\text{dec}}_0, z^{\text{aux}}_j\prns{z^{\text{dec}}_0}}, \\
&\nabla^2 f_0\prns{z^{\text{dec}}_0} = \prns{\nabla^2 \tilde f_0\prns{z^{\text{dec}}_0, z^{\text{aux}}_0}}^{-1}\bracks{z^{\text{dec}}_0, z^{\text{dec}}_0}.
\end{align*}
\end{proof}

\subsection{Proofs for \cref{sec: supplement}}
\begin{proof}{Proof for \cref{lemma: consistency-est-sol}}
The conclusion directly follows from Theorem 5.3 in \cite{shapiro2014lectures}.
\end{proof}

\begin{proof}{Proof for \cref{prop: square-cost}}
Note that 
\begin{align*}
\hat z_0 = \argmin_z \frac{1}{n}\sum_{i = 1}^n \indic{X_i \in R_0} \|z - Y_i\|^2 = \frac{1}{n_0}\sum_i \indic{X_i \in R_0} Y_i.
\end{align*}
Analogously, we define $\hat z_j = \frac{1}{n_j}\sum_i \indic{X_i \in R_j} Y_i$.

Note that $\nabla c(z; y) = z- y$ and $\nabla^2 c(z; y) = I$. Thus the gradient and Hessian estimates are 
\begin{align*}
\hat h_j(\hat z_0) = \frac{1}{n_j}{\sum_i \indic{X_i \in R_j} (\hat z_0 - Y_i)} = \hat z_0 - \hat z_j,  ~~ \hat H_0(\hat z_0) = I.
\end{align*}
It follows that  
\begin{align*}
\ts\hat{\crit}^\text{apx-soln}(R_1,R_2) 
	&= \sum_{j=1,2}\frac1n\sum_{i=1}^n\indic{X_i\in R_j}c\prns{\hat z_0-\hat H^{-1}_0 \hat h_j;\;Y_i} \\
	&= \sum_{j=1,2}\frac{1}{2n}\sum_{i=1}^n\indic{X_i\in R_j}c\prns{\hat z_j;\;Y_i} = \frac{1}{2n}\sum_{j=1,2}\sum_{i=1}^n\indic{X_i\in R_j}
	\edit{\|Y_i - \hat z_j\|^2_2} \\
	&= \frac{1}{2}\sum_{j=1,2}\frac{n_j}{n}\prns{\frac{1}{n_j}\sum_{i=1}^n\indic{X_i\in R_j}
	\edit{\|Y_i - \hat z_j\|^2_2}} = \frac{1}{2}\sum_{j=1,2}\frac{n_j}{n}\sum_{l = 1}^d \op{Var}\prns{\{Y_{i, l}: X_i\in R_j, i \le n\}}
\end{align*}
and 
\begin{align*}
&\ts\frac{1}{2}\hat{\crit}^\text{apx-risk}(R_1,R_2) + \frac{1}{2}\sum_{j =1, 2}\frac{1}{n}\sum_i \indic{X_i \in R_j}\edit{\|Y_i - \hat z_0\|^2_2} \\
=&\frac{1}{2}\sum_{j =1, 2}\frac{n_j}{n}\sum_{l = 1}^d \bracks{\frac{1}{n_j}\sum_i \indic{X_i \in R_j}(Y_{i, l} - \hat z_{0, l})^2 - (\hat z_{0, l} - \hat z_{j, l})^2}  \\
=&\frac{1}{2}\sum_{j =1, 2}\frac{n_j}{n}
\sum_{l = 1}^d\bracks{\frac{1}{n_j}\sum_i \indic{X_i \in R_j}(Y_{i, l} - \hat z_{j, l} + \hat z_{j, l} -  \hat z_{0, l})^2 - (\hat z_{0, l} - \hat z_{j, l})^2} \\
=&\frac{1}{2}\sum_{j =1, 2}\frac{n_j}{n}
\sum_{l = 1}^d\bracks{\frac{1}{n_j}\sum_i \indic{X_i \in R_j}(Y_{i, l} - \hat z_{j, l})^2 + \frac{1}{n_j}\sum_i \indic{X_i \in R_j}(Y_{i, l}- \hat z_{j, l})(\hat z_{j, l} -  \hat z_{0, l})} \\
=& \frac{1}{2}\sum_{j =1, 2}\frac{n_j}{n}\sum_{l = 1}^d \frac{1}{n_j}\sum_{i} \indic{X_i \in R_j}(Y_{i, l} - \hat z_{j, l})^2 = \sum_{j=1,2}\frac{n_j}{2n}\sum_{l = 1}^d \op{Var}\prns{\{Y_{i, l}: X_i\in R_j, i \le n\}}
\end{align*}
\end{proof}

\begin{proof}{Proof for \cref{prop: gradient-var-port}.}
Note that in \cref{prop: gradient-var-port}, 
\begin{align*}
\nabla c(z; y) = 
2\begin{bmatrix}
yy^\top z_{1:d} - z_{d+1}y \\
z_{d+1} - z_{1:d}^\top y
\end{bmatrix}, 
\nabla^2 c(z; y) =
2\begin{bmatrix}
yy^\top & -y \\
-y^\top & 1
\end{bmatrix},
\end{align*}
and also 
$$
z_{0, d+1} = \argmin_{z_{d+1} \in \mathbb R}\left.\Eb{(Y^\top z_{1:d} - z_{d+1})^2 \mid X \in R_0}\right\vert_{z_{1:d} = z_{0, 1:d}}
= z^\top_{0, 1:d}\Eb{Y\mid X \in R_0}.
$$ 
It follows that 
\begin{align*}
\nabla f_j(z_0) = \Eb{\nabla c(z_0; Y) \mid X \in R_j} = 
2\begin{bmatrix}
\Eb{YY^\top \mid {X \in R_j}}z_{0, 1:d} - \Eb{Y \mid X \in R_j}z_{0, d+1} \\ 
z^\top_{0, 1:d}\prns{\Eb{Y \mid X \in R_0} - \Eb{Y \mid X \in R_j}}
\end{bmatrix}
\end{align*}
and 
\begin{align*}
\nabla^2 f_0(z_0) = 
\Eb{\nabla^2 c(z; Y) \mid X \in R_0} = 
2\begin{bmatrix}
\Eb{YY^\top \mid {X \in R_0}} & -\Eb{Y \mid X \in R_0} \\
-\Eb{Y^\top \mid X \in R_0} & 1
\end{bmatrix}
\end{align*}
\end{proof}

\begin{proof}{Proof for \cref{prop: gradient-portfolio}}
Recall that 
\begin{align*}
f_j(z) = \Eb{\frac{1}{\alpha}\prns{z_{d+1} - Y^\top z_{1:d}}\indic{z_{d+1} - Y^\top z_{1:d} \ge 0} - z_{d+1} \mid {X \in R_j}},
\end{align*}
and also $z_{0, d + 1} = q^\alpha_0(Y^\top z_0)$. 

Under the assumption that $Y$ has a continuous density function \edit{and $z \ne 0$}, Lemma 3.1 in \cite{Hong09} implies that 
\begin{align*}
&\frac{\partial}{\partial{z_{1:d}}} f_j(z) = -\frac{1}{\alpha}\Eb{Y\indic{z_{d+1} - Y^\top z_{1:d} \ge 0} \mid X \in R_j}, \\
&\frac{\partial}{\partial{z_{d+1}}} f_j(z) = \frac{1}{\alpha}\Prb{z_{d+1} - Y^\top z_{1:d} \ge 0 \mid  X \in R_j} - 1.
\end{align*}
Before deriving the Hessian, we first denote $\mu_j(u_l, y_l)$ as the joint density of $\sum_{l' \ne l}z_{l'}Y_{l'}$ and $Y_l$ given $X \in R_j$. 

It follows that for $l = 1, \dots, d$, 
\begin{align*}
&\frac{\partial^2}{\partial^2{z_{l}}} f_0(z_0)  \\
	=& -\left.\frac{1}{\alpha}\frac{\partial}{\partial{z_{l}}} \Eb{Y_l\indic{Y^\top z_{1:d} \le z_{d+1}}\mid X \in R_0}\right\vert_{z = z_0} \\
	=& -\left.\frac{1}{\alpha}\frac{\partial}{\partial{z_{l}}}\int\int_{-\infty}^{z_{d+1} - z_ly_l}y_l\mu_l(u_l, y_l)\diff u_l \diff y_l
	\right\vert_{z = z_0} \\
	=& -\left.\frac{1}{\alpha}\int \prns{- y_l} y_l\mu_l(z_{d+1} - z_ly_l, y_l) \diff y_l
	\right\vert_{z = z_0} \\
	=& -\frac{\mu_0\prns{q^\alpha_0(Y^\top z_0)}}{\alpha}\braces{- \Eb{Y_l^2 \mid Y^\top z_0 = q^{\alpha}_0(Y^\top z_0), X \in R_0}} \\
	=& \frac{\mu_0\prns{q^{\alpha}_0(Y^\top z_0)}}{\alpha}\Eb{Y_l^2 \mid Y^\top z_0 = q^{\alpha}_0(Y^\top z_0), X \in R_0}.
\end{align*}

Similarly we can prove that for $l, l' = 1, \dots, d$ 
\begin{align*}
&\frac{\partial^2}{\partial{z_{l}}\partial z_{l'}} f_0(z_0) = \frac{\mu_0\prns{q^{\alpha}_0(Y^\top z_0)}}{\alpha} \Eb{Y_lY_{l'} \mid Y^\top z_0 = q^{\alpha}_0(Y^\top z_0), X \in R_0}.
\end{align*}

In contrast, for $l = 1, \dots, d$, 
\begin{align*}
\frac{\partial^2}{\partial{z_{d+1}}\partial{z_{l}}} f_0(z_0) 
	&= \left.\frac{\partial}{\partial{z_l}}\prns{\frac{1}{\alpha}\Prb{z_{d+1} - Y^\top z_{1:d} \ge 0 \mid X \in R_0} - 1}\right|_{z = z_0} \\
	&= \left.\frac{1}{\alpha} \frac{\partial}{\partial{z_{l}}}\int\int_{-\infty}^{z_{d+1} - z_ly_l}\mu_l(u_l, y_l)\diff u_l \diff y_l \right|_{z = z_0} \\
	&= \left.\frac{1}{\alpha} \int -y_l \mu_l(z_{d+1} - z_ly_l, y_l) \diff y_l \right|_{z = z_0} \\
	&= -\frac{\mu_0\prns{q^{\alpha}_0(Y^\top z_0)}}{\alpha}\Eb{Y_l\mid Y^\top z_0 = q^{\alpha}_0(Y^\top z_0), X \in R_0},
\end{align*}
and 
\begin{align*}
\frac{\partial^2}{\partial^2{z_{d+1}}} f_0(z_0) 
	&=  \left.\frac{\partial}{\partial{z_{d+1}}}\prns{\frac{1}{\alpha}\Prb{z_{d+1} - Y^\top z_{1:d} \ge 0 \mid X \in R_0} - 1}\right|_{z = z_0} \\
	&= \left.\frac{1}{\alpha} \frac{\partial}{\partial{z_{d+1}}}\int\int_{-\infty}^{z_{d+1} - z_ly_l}\mu_l(u_l, y_l)\diff u_l \diff y_l \right|_{z = z_0} \\
	&= \left.\frac{1}{\alpha} \int \mu_l(z_{d+1} - z_l y_l, y_l)\right|_{z = z_0} \\
	&= \frac{\mu_0\prns{q^{\alpha}_0(Y^\top z_0)}}{\alpha}.
\end{align*}

When $Y \mid X \in R_0$ has Gaussian distribution with mean $m_0$ and covariance matrix $\Sigma_0$, then $(Y, Y^\top z_{0, 1:d})$ given $X \in R_0$ is also has a Gaussian distribution 
\begin{align*}
\mathcal N \prns{
	\begin{bmatrix}
	m_0 \\
	m_0^\top z_{0, 1:d}
	\end{bmatrix}
	,
	\begin{bmatrix}
	\Sigma_0 & \Sigma_0 z_{0, 1:d} \\
	z_{0, 1:d}^\top \Sigma_0 & z_{0, 1:d}^\top \Sigma_0 z_{0, 1:d}
	\end{bmatrix}
}
\end{align*}
It follows that $Y \mid Y^\top z_{0, 1:d} = q^{\alpha}_0(Y^\top z_{0, 1:d})$  also has a Gaussian distribution with the following conditional mean and conditional variance:
\begin{align*}
&\Eb{Y \mid Y^\top z_{0, 1:d} = q^{\alpha}_0(Y^\top z_{0, 1:d}), X \in R_0} = m_0 + \Sigma_0 z_{0, 1:d} \prns{z_{0, 1:d}^\top \Sigma_0 z_{0, 1:d}}^{-1}(q^{\alpha}_0(Y^\top z_{0}) - m_0^\top z_{0, 1:d}), \\ 
&\var\prns{Y \mid Y^\top z_{0, 1:d} = q^{\alpha}_0(Y^\top z_{0, 1:d}), X \in R_0} = \Sigma_0 - \Sigma_0 z_{0, 1:d} \prns{z_{0, 1:d}^\top \Sigma_0 z_{0, 1:d}}^{-1}z_{0, 1:d}^\top \Sigma_0.
\end{align*}
\end{proof}

\begin{proof}{Proof for \cref{lemma: suff-cond-crit-converg}}
Under the conditions in \cref{lemma: consistency-est-sol}, we have that $\hat z_0 \to z_0$ almost surely, which implies  that \edit{there exist} a neighborhood $\mathcal N$ around $z_0$ such that $\hat z_0 \in \mathcal N$ almost surely for sufficiently large $n$. 

Since $\sup_{z \in \mathcal N}\edit{\|\hat H_0(z) - \nabla^2 f_{0}(z)\|_{\op{F}}} = o_p(1)$ and $\nabla^2 f_{0}(z)$ is continuous, we have that $\edit{\|\hat H_0(\hat z_0) - \nabla^2 f_{0}(z_0)\|_{\op{F}}} = o_p(1)$ \citep[Proposition 5.1]{shapiro2014lectures}. By the fact that $\nabla^2 f_{0}(z_0)$ is invertible and the  continuous mapping theorem, we also have that  $\hat H_0(\hat z_0)$ is differentiable with high probability and $\edit{\|\hat H^{-1}_0(\hat z_0) - \prns{\nabla^2 f_{0}(z_0)}^{-1}\|_{\op{F}}} = o_p(1)$.

Since $\{(x, y) \mapsto \indic{x \in R_0}c(z; y): z \in \mathcal N\}$ is a Donsker class, $\sqrt{n}\prns{\widehat{p_0f_0}(\cdot) - p_0f_0(\cdot)}$ converges to a Gaussian process \citep[Sec. 19.2]{van2000asymptotic}.
By Slutsky's theorem, this means that $\sqrt{n}\prns{\hat{f_0}(\cdot) - f_0(\cdot)}$ converges to a Gaussian process as well. 
Then according to Theorem 5.8 in \cite{shapiro2014lectures}, if $\sqrt{n}\prns{\hat{f_0}(\cdot) - p_0f_0(\cdot)}$ converges to a Gaussian process as well and $\nabla^2 f_0(z_0)$ is invertible, then $\sqrt{n}(\hat z_0 - z_0)$ also converges to a Gaussian distribution, which implies that $\hat z_0 - z_0 = O_p(n^{-1/2})$. It follows that  we have the following holds almost surely: 
\begin{align}
\edit{\|\hat h_j(\hat z_0) - \nabla f_j(z_0)\|_2}  
	&\le \edit{\|\hat h_j(\hat z_0) - \nabla f_j(\hat z_0)\|_2} + \edit{\|\nabla f_j(\hat z_0) - \nabla f_j(z_0)\|_2} \nonumber \\ 
	&\le \sup_{z \in \mathcal N}\edit{\|\hat h_j(z_0) - \nabla f_j(z_0)\|_2} +\edit{ \|\nabla f_j(\hat z_0) - \nabla f_j(z_0)\|_2} \nonumber \\
	&\le \sup_{z \in \mathcal N}\edit{\|\hat h_j(z_0) - \nabla f_j(z_0)\|_2} + \edit{\|\nabla^2 f_j(z_0)\|_{\op{F}}}\edit{\|\hat z_0 -  z_0\|_2}, \label{eq: square-root-n}
\end{align}
which implies that $\|\hat h_j(\hat z_0) - \nabla f_j(z_0)\| = O_p(n^{-1/2})$.
\end{proof}

\begin{proof}{Proof for \cref{prop: ex-mnv}.}
Note that we only need to verify the conditions in \cref{lemma: suff-cond-crit-converg} and that $\left|\frac{1}{n}\sum_{i = 1}^n \indic{X_i \in R_j}c\prns{\hat z_0-\hat H^{-1}_0(\hat z_0) \hat h_j(\hat z_0);Y_i} - f_j\prns{z_0-\prns{\nabla^2 f_{0}(z_0)}^{-1} {\nabla f_{j}(z_0)}}\right| = O_p(n^{-1/2})$ for $j = 1, 2$.  
Recall that  $(\nabla f_j(z_0))_l=(\alpha_l+\beta_l)\Prb{Y_l\leq z_{0,l} \mid X\in R_j}-\beta_l$ and $(\nabla^2 f_{0}(z_0))_{ll}=(\alpha_l+\beta_l)\mu_{0, l}(z_0)$ for $l=1,\dots,d,\,j=1,2$, and $(\nabla^2 f_{0}(z_0))_{ll'}=0$ for $l \neq l'$, where $\mu_{0,l}$ is the density of $Y_l\mid X\in R_0$.
Note that $\nabla^2 f_0(z)$ is continuous, and $\nabla^2 f_0(z_0)$ is invertible under the asserted conditions. 

Note that the indicator function class $\{(x, y_l) \mapsto \indic{x\in \mathcal R_j, y_l \le z_l}: z_l \in \mathbb R \}$ is a Donsker class \citep[Ex. 19.6]{van2000asymptotic}. Therefore, $\sup_{z \in \mathbb R^d}|\prns{\hat h_j(z)}_l - (\nabla f_j(z_0))_l| = O_p(n^{-1/2})$.

Using the Theorem 2 in \cite{jiang17b} for the uniform convergence of kernel density estimator, we can straightforwardly 
show that under the asserted Holder continuity condition for $\mu_k$ and rate condition for bandwidth $b$, the Hessian estimator satisfies that $\sup_{z}\left|\prns{\hat H_0(z)}_{ll} - (\nabla^2 f_{0}(z_0))_{ll}\right| = o_p(1)$ for $l = 1, \dots, d$.

Moreover, note that 
$$
c(z; y) = \sum_{l = 1}^d \max\{\alpha_l(z_l - y_l), \beta_l(y_l - z_l)\} = \sum_{l = 1}^d \beta_l(y_l - z_l) - (\alpha_l + \beta_l)(y_l - z_l)\indic{y_l \le z_l}.
$$
Here the function classes $\{y_l \mapsto \beta_l(y_l - z_l): z \in \mathbb R^d\}$ and $\{y_l \mapsto (\alpha_l + \beta_l)(y_l - z_l): z \in \mathbb R^d\}$ are linear function classes with fixed dimension, so they are Donsker classes \citep[Ex 19.17]{van2000asymptotic}. Moreover, $\{y_l \mapsto \indic{y_l \le z_l}: z_l \in \mathbb R^d\}$ is also a Donsker class \citep[Ex. 19.6]{van2000asymptotic}. It follows that the function class $\{(x, y) \mapsto \indic{x \in R_j}c(z; y): z \in \mathbb R^d\}$ is also a Donsker class, according to Ex 19.20 of \cite{van2000asymptotic}.
Similar to proving $\|\hat h_j(\hat z_0) - \nabla f_j(\hat z_0)\| = O_p(n^{-1/2})$ in 
\cref{lemma: suff-cond-crit-converg} (see \cref{eq: square-root-n}), we can prove that for $j = 1, 2$, 
$$\left|\frac{1}{n}\sum_{i = 1}^n \indic{X_i \in R_j}c\prns{\hat z_0-\hat H^{-1}_0(\hat z_0) \hat h_j(\hat z_0);Y_i} - f_j\prns{z_0-\prns{\nabla^2 f_{0}(z_0)}^{-1} {\nabla f_{j}(z_0)}}\right| = O_p(n^{-1/2}).$$ 
\end{proof}

\begin{proof}{Proof for \cref{prop: ex-portfolio-var}.}
Since \cref{ex: portfolio-var} is a special example of \cref{ex: smooth}, the conclusions in \cref{prop: ex-portfolio-var} directly follow from 
\cref{prop: ex-smooth}.
\end{proof}

\begin{proof}{Proof for \cref{prop: ex-portfolio}.}
Recall that $\hat q^{\alpha}_0(Y_i^\top \hat z_{0})$ is the empirical quantile of $Y^\top \hat z_0$ based on data in $R_0$. Equivalently, $\hat q^{\alpha}_0(Y_i^\top \hat z_{0})$ is the (approximate) minimizer of the following optimization problem:
\begin{align*}
\min_{\beta \in \mathbb R} \frac{\frac{1}{n}\sum_{i=1}^n \prns{\alpha - \indic{Y_i^\top \hat z_{0} - \beta \le 0}}\prns{Y_i^\top \hat z_{0} - \beta}\indic{X_i \in R_0}}{\frac{1}{n} \sum_{i=1}^n \indic{X_i \in R_0}}
\end{align*}
Since $\braces{y \mapsto {y^\top z - \beta}: z \in \mathbb R^d, \beta \in \mathbb R}$ is a Donsker class \citep[Ex. 19.17]{van2000asymptotic} and so is $\braces{y \mapsto \indic{y^\top z - \beta}: z \in \mathbb R^d, \beta \in \mathbb R}$. This implies that 
\begin{align*}
\sup_{\beta \in \mathbb R, z \in \mathbb R^d}\left|\frac{1}{n}\sum_{i=1}^n \prns{\alpha - \indic{Y_i^\top  z - \beta \le 0}}\prns{Y_i^\top  z - \beta}\indic{X_i \in R_0} - \Eb{\prns{\alpha - \indic{Y_i^\top  z - \beta \le 0}}\prns{Y_i^\top  z - \beta}\indic{X_i \in R_0}}\right| \to 0.
\end{align*}
Together with $\hat z_0 \to z_0$ and the continuity of $\Eb{\prns{\alpha - \indic{Y_i^\top  z - \beta \le 0}}\prns{Y_i^\top  z - \beta}\indic{X_i \in R_0}}$ in $z$, this implies that 
\begin{align*}
\sup_{\beta \in \mathbb R}\left|\frac{1}{n}\sum_{i=1}^n \prns{\alpha - \indic{Y_i^\top  \hat z_0 - \beta \le 0}}\prns{Y_i^\top \hat z_0 - \beta}\indic{X_i \in R_0} - \Eb{\prns{\alpha - \indic{Y_i^\top  z_0 - \beta \le 0}}\prns{Y_i^\top  z_0 - \beta}\indic{X_i \in R_0}}\right| \to 0.
\end{align*}
Moreover, $\frac{1}{n} \sum_{i=1}^n \indic{X_i \in R_0} \to \Prb{X \in R_0}$ by Law of Large Number. 
It follows from Theorem 5.5 in \cite{shapiro2014lectures} that $\hat q^{\alpha}_0(Y_i^\top \hat z_{0})$ converges to the set of minimizers of $\Eb{\prns{\alpha - \indic{Y_i^\top  z - \beta \le 0}}\prns{Y_i^\top  z - \beta} \mid {X_i \in R_0}}$. 
Since the density function of $Y^\top z_0$ at $q^\alpha_0(Y^\top z_0)$ is positive, minimizer of $\Eb{\prns{\alpha - \indic{Y_i^\top  z - \beta \le 0}}\prns{Y_i^\top  z - \beta} \mid {X_i \in R_0}}$ is unique. Therefore, $\hat q^{\alpha}_0(Y_i^\top \hat z_{0})$ converges to $q^{\alpha}_0(Y_i^\top  z_{0})$.

Since $\braces{y \mapsto {y^\top z - \beta}: z \in \mathbb R^d, \beta \in \mathbb R}$ is a Donsker class \citep[Ex. 19.17]{van2000asymptotic},  obviously $\braces{(x, y) \mapsto \indic{y^\top z - \beta, x \in R_j}: z \in \mathbb R^d, \beta \in \mathbb R}$ and thus $\braces{(x, y) \mapsto \indic{y^\top z - \beta \le 0,  x \in R_j}y: z \in \mathbb R^d, \beta \in \mathbb R}$  are also Donsker classes. 
Morever, we already prove that $\hat q^{\alpha}_0(Y_i^\top \hat z_{0})$ converges to $q^{\alpha}_0(Y_i^\top  z_{0})$, and obviously $\frac{n_j}{n} = \frac{1}{n} \sum_{i=1}^n \indic{X_i \in R_j} \to \Prb{X \in R_j}$. Therefore, 
\begin{align*}
-\frac{1}{\alpha n_j}\sum_{i = 1}^n \indic{Y_i^\top \hat z_{0} \le \hat q^{\alpha}_0(Y^\top \hat z_{0}), X_i \in R_j}Y_i 
	&= -\frac{1}{\alpha}\Eb{\indic{Y_i^\top  z_{0} \le  q^{\alpha}_0(Y^\top  z_{0})}Y_i\mid X_i \in R_j} + O_p(n^{-1/2}).
\end{align*}
Similarly, we can show that 
\begin{align*}
&\frac{1}{n_j}\sum_{i = 1}^n \frac{1}{\alpha} \indic{Y^\top_i \hat z_{0, 1:d} \le q^{\alpha}_0(Y^\top \hat z_{0}), X_i \in R_j} - \indic{{X_i \in R_j}} \\
=& \frac{1}{\alpha}\Prb{q^{\alpha}_0(Y^\top z_{0}) - Y^\top z_{0, 1:d} \ge 0 \mid X \in R_j} - 1 + O_p(n^{-1/2}).
\end{align*}
Therefore, $\hat h_j(\hat z_0) = \nabla f_j(z_0) + O_p(n^{-1/2})$.

Under the Gaussian assumption,
\begin{align*}
&\Eb{Y \mid Y^\top z_{0} = q^{\alpha}_0(Y^\top z_{0}), X \in R_0} = m_0 + \Sigma_0 z_0 \prns{z_0^\top \Sigma_0 z_0}^{-1}(q^{\alpha}_0(Y^\top z_{0}) - m_0^\top z_0) \\
&\var\prns{Y \mid Y^\top z_{0} = q^{\alpha}_0(Y^\top z_{0}), X \in R_0} = \Sigma_0 - \Sigma_0 z_0 \prns{z_0^\top \Sigma_0 z_0}^{-1}z_0^\top \Sigma_0.
\end{align*}
Since both are continuous in $z_0, m_0, \Sigma_0, q^{\alpha}_0(Y^\top z_{0})$, so when we plug in the empirical estimators that converge to the true values, the estimator for $\Eb{Y \mid Y^\top z_{0} = q^{\alpha}_0(Y^\top z_{0}), X \in R_0}$, $\var\prns{Y \mid Y^\top z_{0} = q^{\alpha}_0(Y^\top z_{0}), X \in R_0}$, and $\Eb{YY^\top \mid Y^\top z_{0} = q^{\alpha}_0(Y^\top z_{0}), X \in R_0}$ are all consistent. 

Similar to the proof of \cref{prop: ex-mnv}, we can show that under the asserted Holder continuity condition and the rate condition on bandwidth $b$, 
\begin{align*}
\frac{1}{n_0 b}\sum_{i = 1}^n \indic{X_i \in R_0}\mathcal K\prns{\prns{Y_i^\top\hat z_0 - q_{\alpha}(Y^\top \hat z_{0})}/b} = \mu_{0}\prns{q^{\alpha}_0(Y^\top z_{0})} + o_p(1).
\end{align*}
Then by the upper boundedness of $\mu_0$, we have that $\edit{\|\hat H_0(\hat z_0) - \nabla^2 f_{0}(z_0)\|_{\op{F}}} = o_p(1)$. 
\end{proof}

\begin{proof}{Proof for \cref{prop: ex-smooth}.}
Since $\nabla^2 f(z)$ is continuous at $z = z_0$ and $\hat z_0 \to z_0$ almost surely according to \cref{lemma: consistency-est-sol}, 
\edit{there exist} a sufficiently small compact neighborhood $\mathcal N$ around $z_0$ such that the minimum singular value of $\nabla^2 f(z)$, denoted as $\sigma_{\min}(\nabla^2 f(z))$, is at least $\frac{2}{3}\sigma_{\min}\prns{\nabla^2 f(z_0)}$ for any $z \in \mathcal N$, and for $n$ large enough $\hat z_0 \in \mathcal N$ almost surely. Recall that $\nabla^2 f_0(z)=\Eb{\nabla^2 c\prns{z;Y_i} \mid X_i\in R_0}$ and $\nabla f_j(z)=\Eb{\nabla c\prns{z;Y_i} \mid X_i\in R_j}$.

Since $\nabla^2 c(z; y)$ is continuous for all $y$ and $\mathcal N$ is compact, the  class of functions (of $y$) $\{y \mapsto \nabla^2 c(z; y): z \in \mathcal N\}$ is a Glivenko-Cantelli class \citep[Example 19.8]{van2000asymptotic}, which implies the uniform convergence $\sup_{z \in \mathcal N}\|\hat H_0(z) - \nabla^2 f_{0}(z)\| \overset{a.s.}{\to} 0$. Without loss of generality, we can also assume for large enough $n$ that $\hat H_0(z)$ is invertible for $z \in \mathcal N$, and $\sigma_{\min}\prns{\hat H_0(z)} \ge \frac{1}{2} \sigma_{\min}(\nabla^2 f(z)) \ge \frac{1}{3}\sigma_{\min}\prns{\nabla^2 f(z_0)}$ for $z \in \mathcal N$. 

Since $\nabla c(z; y)$ is continuously differentiable, and $\mathcal N$ is compact, $\nabla c(z; y)$ is Lipschitz in $z$ on $\mathcal N$. It follows that $\{y \mapsto \nabla c(z; y): z \in \mathcal N\}$ is a Donsker class \citep[Example 19.7]{van2000asymptotic}. This implies that $n^{1/2}\prns{\hat h_j(\cdot) - \nabla f_j(\cdot)}$ converges to a Gaussian process $\mathbb G(\cdot)$ over $z \in \mathcal N$. 
Therefore, $n^{1/2}\prns{\hat h_j(\cdot) - \nabla f_j(\cdot)} = O_p(1)$, and $\sup_{z \in \mathcal N}\edit{\|\hat h_j(z) - \nabla f_j(z)\|_2} = O_p(n^{-1/2})$. 
Note that the Donsker property of $\{y \mapsto \nabla c(z; y): z \in \mathcal N\}$ also implies that it is a Glivenko-Cantelli class, so that $\sup_{z \in \mathcal N}\edit{\|\hat h_j(z) - \nabla f_j(z)\|_2} \overset{a.s.}{\to} 0$.  

By the fact that for $z \in \mathcal N$, $\sigma_{\min}\prns{\hat H_0(z)} \ge \frac{1}{3}\sigma_{\min}\prns{\nabla^2 f(z_0)}$, $\edit{\|\hat h_j(z) - \nabla f_j(z)\|_2} \overset{a.s.}{\to} 0$ and $\nabla f_j(z)$ is bounded on $\mathcal N$, we have that \edit{there exist} another compact set $\mathcal N'$ such that for sufficiently large $n$, $ z-\hat H^{-1}_0(z) \hat h_j(z) \in \mathcal N'$ for any $z \in \mathcal N$. Since $c(z; y)$ is continuously differnetiable, it is also Lipschitz in $z$ on $\mathcal N'$. Again this means that $\{y \mapsto c(z; y): z \in \mathcal N'\}$ is a Donsker class, so that 
\[
\sup_{z \in \mathcal N'}|\frac{1}{n}\sum_{i = 1}^n \indic{X_i \in \mathcal R_j}c(z; Y_i) - \Eb{\indic{X \in \mathcal R_j}c(z; Y)}| = O_p(n^{-1/2}).
\]
Therefore, $\sup_{z \in \mathcal N}\left|\frac{1}{n_j}\sum_{i = 1}^n \indic{X_i \in R_j}c(\hat r(z); Y_i) - \Eb{\indic{X \in R_j}c(\hat r(z); Y)}\right| = O_p(n^{-1/2})$ for $\hat r(z) \coloneqq z -\hat H^{-1}_0(z) \hat h_j( z)$. 
Since $\hat r(z) \to z_0 - \prns{\nabla^2 f(z_0)}^{-1}\nabla f(z_0)$ almost surely, we have that 
\begin{align*}
\left|\frac{1}{n}\sum_{i = 1}^n \indic{X_i \in \mathcal R_j}c(\hat z_0 -\hat H^{-1}_0(\hat z_0) \hat h_j(\hat z_0); Y_i) - \Eb{\indic{X \in \mathcal R_j}c(z_0 - \prns{\nabla^2 f(z_0)}^{-1}\nabla f(z_0); Y)}\right| = O_p(n^{-1/2}).
\end{align*}
\end{proof}

\edit{
\begin{proof}{Proof of \cref{prop: LICQ}}
The constraints in \cref{ex: mnv} can be rewritten as 
\begin{align*}
\Z = \braces{z \in \R{d}: h_k\prns{z} = -z_k \le 0, h_{d+1}\prns{z} = \sum_{l=1}^d z_l - C \le 0}.
\end{align*}
Note that at any $z \in \Z$, there are at most $d$ active inequality constraints, and their gradients have to be linearly independent. So it satisfies the LICQ condition at any $z\in\Z$. Similarly, we can prove the LICQ condition for \cref{ex: portfolio-var,ex: portfolio} with the simplex constraint. 
\end{proof}
}

\edit{
\begin{proof}{Proof of \cref{prop: equi-spo-stochopt}}
Fix a parent region $R_0$ and a split that partitions it into two subregions $R_1$ and $R_2$ (with sample sizes $n_1, n_2$ respectively). According to Eq. (4) in \cite{elmachtoub2020decision}, the SPO splitting criterion for the given split can be written as follows:
\begin{align*}
\mathcal{C}^{\text{SPO}}\prns{R_1, R_2} 
    &= \sum_{j=1}^2 \frac{n_j}{n}\prns{\min_{z \in \mathcal{Z}}\frac{1}{n_j}\sum_{i: X_i \in R_j} Y_i^\top z - \frac{1}{n_j}\sum_{i: X_i \in R_j} \min_{z\in\mathcal{Z}} Y_i^\top z}  \\
    &= \sum_{j=1}^2 \prns{\min_{z \in \mathcal{Z}}\frac{1}{n}\sum_{i: X_i \in R_j} Y_i^\top z} - \prns{\frac{1}{n}\sum_{i=1}^n \min_{z\in\mathcal{Z}} Y_i^\top z}.
\end{align*}
Note that the second term above does not depend on the split so using the SPO criterion to choose splits is equivalent to using only the first term to choose splits. It is easy to see that the first term is exactly our oracle splitting criterion with all unknown expectations replaced by sample averages: 
\begin{align*}
\hat{\mathcal{C}}^\text{oracle}(R_1,R_2)=\sum_{j=1,2}\min_{z\in\mathcal{Z}}\hat{\mathbb{E}}\bracks{c(z;Y)\indic{X\in R_j}} = \sum_{j=1}^2 \prns{\min_{z \in \mathcal{Z}}\frac{1}{n}\sum_{i: X_i \in R_j} Y_i^\top z}.
\end{align*}
\end{proof}
}

\subsection{Proofs for \cref{sec: unconstr}}\label{app-sec: sec2}
\begin{proof}{Proof for \cref{lemma: consistency-est-sol}}
The conclusion follows from Theorem 5.3 in \cite{shapiro2014lectures} when the population optimization problem has a unique optimal solution.
\end{proof}

\begin{proof}{Proof for \cref{thm:apxriskapx}}
Conditions 1, 2, 3 imply that conditions in \cref{thm: frechet-diff} are satisfied for both $\delta_f = f_1 - f_0$ and $\delta_f = f_2 - f_0$.

Therefore,  \cref{thm: frechet-diff} implies that for $j = 1, 2$,
\[
\min_{z \in \mathbb R^d} f_j(z) = f_{j}(z_0)
-\frac12{\nabla f_{j}(z_0)}^\top \prns{\nabla^2 f_{0}(z_0)}^{-1} {\nabla f_{j}(z_0)} + o(\|f_j - f_0\|_\mathcal F)
\]
where $\|f_j - f_0\|_\mathcal F = \max \{\sup_z|f_j(z) - f_0(z)|, \sup_z\edit{\|\nabla f_j(z) - f_0(z)\|_2}, \sup_z\edit{\|\nabla^2 f_j(z) - f_0(z)\|_{\op{F}}}\}$. By the Lispchitzness condition, we have $\|f_j - f_0\|_\mathcal F = O(\mathcal D_0^2)$. Therefore, 
\begin{align*}
\crit^\text{oracle}(R_1,R_2) = \sum_{j = 1, 2}p_j\min_{z \in \mathbb R^d} f_j(z) 
	&= \sum_{j = 1, 2}p_j\prns{f_{j}(z_0)
-\frac12{\nabla f_{j}(z_0)}^\top \prns{\nabla^2 f_{0}(z_0)}^{-1} {\nabla f_{j}(z_0)}}+ o(\mathcal D_0^2) \\
	&= p_0f_0(z_0) + \frac{1}{2}{\crit^\text{apx-risk}(R_1,R_2)} + o(\mathcal D_0^2).
\end{align*}
\end{proof}

\begin{proof}{Proof for \cref{thm:apxsolapx}}
By \cref{thm: frechet-diff} with $\delta_f = f_j - f_0$ respectively, 
\begin{align*}
z_j(1) = z_0-\prns{\nabla^2 f_{0}(z_0)}^{-1} {\nabla f_{j}(z_0)} + R_j,
\end{align*}
where $R_j = o(\|f_j - f_0\|_\mathcal F)$.

It follows from mean-value theorem that 
\edit{there exist} a diagonal matrix $\Lambda_j$ whose diagonal entries are real numbers within $[0, 1]$ such that 
\begin{align*}
v_j(1) = f_j(z_j(1)) 
	&= f_j\prns{z_0-\prns{\nabla^2 f_{0}(z_0)}^{-1} {\nabla f_{j}(z_0)} + R_j} \\
	&= f_j\prns{z_0-\prns{\nabla^2 f_{0}(z_0)}^{-1} {\nabla f_{j}(z_0)}} + R_j^\top \nabla f_j\prns{z_0-\prns{\nabla^2 f_{0}(z_0)}^{-1} {\nabla f_{j}(z_0)} + \Lambda_j  R_j}.
\end{align*}
We can apply mean-value theorem once again to $\nabla f_j$ to get 
\begin{align*}
 \nabla f_j\prns{z_0-\prns{\nabla^2 f_{0}(z_0)}^{-1} {\nabla f_{j}(z_0)} + \Lambda_j  R_j} = \nabla f_j(z_j(1)) + O(\prns{I - \Lambda_j}^\top R_j) = O(\prns{I - \Lambda_j}^\top R_j),
\end{align*}
where the last equality follows from the first order necessary condition for optimality of $z_j(1)$.

It follows that  
\begin{align*}
v_j(1) 
	&= f_j\prns{z_0-\prns{\nabla^2 f_{0}(z_0)}^{-1} {\nabla f_{j}(z_0)}} + O(R_j^\top (I - \Lambda_j) R_j) \\
	&= f_j\prns{z_0-\prns{\nabla^2 f_{0}(z_0)}^{-1} {\nabla f_{j}(z_0)}} + o(\|f_j - f_0\|^2_\mathcal F) \\
	&= f_j\prns{z_0-\prns{\nabla^2 f_{0}(z_0)}^{-1} {\nabla f_{j}(z_0)}} + o(\mathcal D_0^2) 
\end{align*}

Therefore 
\begin{align*}
\crit^\text{oracle}(R_1,R_2) = \sum_{j = 1, 2}p_j f_j(z_j(1)) 
	&= \sum_{j = 1, 2}p_j f_j\prns{z_0-\prns{\nabla^2 f_{0}(z_0)}^{-1} {\nabla f_{j}(z_0)}} + o(\mathcal D_0^2) \\
	&= \crit^\text{apx-soln}(R_1,R_2) + o(\mathcal D_0^2).
\end{align*}
\end{proof}

\begin{proof}{Proof for \cref{thm:critconverge}.}
Under the asserted condition, $\hat H_0^{-1} - \nabla^2 f^{-1}_{0}(z_0) = o_p(1)$, $\hat h_j = \nabla f_j(z_0) = O_p(n^{-1/2})$ for $j = 1, 2$. It follows that 
\begin{align*}
\ts\hat{\crit}^\text{apx-risk}(R_1,R_2)&\ts
	=-\sum_{j=1,2}\hat h_j^\top \hat H^{-1}_0 \hat h_j  \\
	&= -\sum_{j=1,2}\prns{h_j^\top(z_0) + O_p(n^{-1/2})} \prns{H^{-1}_0(z_0) + o_p(1)} \prns{h_j^\top(z_0) + O_p(n^{-1/2})} \\
	&= -\sum_{j=1,2} h_j^\top( z_0) H^{-1}_0( z_0)  h_j( z_0) + 2 h_j^\top(z_0)H^{-1}_0(z_0)O_p(n^{-1/2}) + o_p(n^{-1/2}) \\
	&= {\crit}^\text{apx-risk}(R_1,R_2) + O_p(n^{-1/2}).
\end{align*}

Note, the condition that $\left|\frac{1}{n}\sum_{i = 1}^n \indic{X_i \in R_j}c\prns{\hat z_0-\hat H^{-1}_0 \hat h_j;Y_i} - f_j(z_0)\right| = O_p(n^{-1/2})$ for $j = 1, 2$ directly ensures that 
\begin{align*}
\ts\hat{\crit}^\text{apx-soln}(R_1,R_2) = {\crit}^\text{apx-soln}(R_1,R_2) + O_p(n^{-1/2}).
\end{align*}

\end{proof}

\subsection{Proofs for \cref{sec: asympt-opt}}\label{sec: opt-proof}

For brevity we define $\overline c(z;x)=\Eb{c(z; Y) \mid X = x}$ and $\hat{\overline c}(z;x)=\sum_{i = 1}^n w_i(x)c(z; Y_i)$. Under \cref{assump: distr}, $\mathcal X$ is compact. Without loss of generality, we assume $\mathcal X\subseteq[0,1]^p$.

\begin{lemma}\label{lemma: exponential-inequality}
Let $x \in \mathcal X$ be fixed and $\braces{w_{ij}(x) = \frac{\indic{i\in\mathcal I^\text{dec}_j,\tau_j(X_i)=\tau_j(x)}}{\sum_{i'=1}^n\indic{i\in\mathcal I^\text{dec}_j,\tau_j(X_{i'})=\tau_j(x)}}, i = 1, \dots, n}$ be the weights derived from the $j$th tree. Under \cref{assump: tree,assump: distr}, 
if further $k_n \to \infty$ and $\log T = o(k_n)$, then for any $z \in \mathcal C$,  as $n \to \infty$,
\[
\sup_{1 \le j \le T} \abs{\sum_{i = 1}^n w_{ij}(x)\prns{c(z; Y_i) - \overline c(z; X_i)}} \overset{p}{\to} 0, ~~~ \sup_{1 \le j \le T} \abs{\sum_{i = 1}^n w_{ij}(x)\prns{b(Y_i) - \Eb{b(Y_i) \mid X_i}} }\overset{p}{\to} 0.
\] 
\end{lemma}
\begin{proof}{Proof for \cref{lemma: exponential-inequality}.}
In this proof, we fix $j = 1, \dots, T$ and implicitly condition on $\mathcal I_j^\text{tree}$. Conditionally on $\{X_i: i \in \mathcal I_j^\text{dec}\}$, $w_{ij}(x)$ for $i = 1, \dots, n$ are all fixed, and  these weights satisfy that $w_{ij}(x) \ge 0$, $\sum_{i = 1}^n w_{ij}(x) = 1$ and $\max_{i} w_{ij}(x) = [\frac{1}{k_n}, \frac{1}{2k_n - 1}]$. According to Lemma 12.1 in \cite{biau2015lectures}, we have that for any $\epsilon \le \frac{\min\{1, 2C\}}{\max\{\eta, \eta'\}}$ and any $z \in \mathcal C$, 
\begin{align*}
&\Prb{\abs{\sum_{i = 1}^n w_{ij}(x)\prns{c(z; Y_i) - \overline c(z; X_i)}}\ge \epsilon \mid X_1, \dots, X_n} \le \exp\prns{-\frac{k_n\epsilon^2\eta^2}{8C}}, \\
&\Prb{\abs{\sum_{i = 1}^n w_{ij}(x)\prns{b(Y_i) - \Eb{b(Y_i) \mid X_i}}}\ge \epsilon \mid X_1, \dots, X_n} \le \exp\prns{-\frac{k_n\epsilon^2\eta'^{2}}{8C}}.
\end{align*}
It follows that 
\begin{align*}
&\Prb{\sup_j \abs{\sum_{i = 1}^n w_{ij}(x)\prns{c(z; Y_i) - \overline c(z; X_i)}}\ge \epsilon \mid X_1, \dots, X_n} \le \exp\prns{\log T -\frac{k_n\epsilon^2\eta^2}{8C}},\\
&\Prb{\sup_j  \abs{\sum_{i = 1}^n w_{ij}(x)\prns{b(Y_i) - \Eb{b(Y_i) \mid X_i}}}\ge \epsilon \mid X_1, \dots, X_n} \le \exp\prns{\log T -\frac{k_n\epsilon^2\eta'^{2}}{8C}}.
\end{align*} 
This means that as $n \to \infty$,  
$$\sup_j \abs{\sum_{i = 1}^n w_{ij}(x)\prns{c(z; Y_i) - \overline c(z; X_i)}} \to 0, ~~ \sup_j\abs{\sum_{i = 1}^n w_{ij}(x)\prns{b(Y_i) - \Eb{b(Y_i) \mid X_i}} }\to 0.$$ 
\end{proof}

\begin{lemma}\label{lemma: uniform-converge}
If the assumptions in \cref{lemma: exponential-inequality} hold, $s_n/k_n \to \infty$ and $T = o(s_n/k_n)$, then as $n \to \infty$,
\[
	\sup_{z \in \mathcal C}\abs{\hat{\overline c}(z; x) - \overline c(z; x)} \overset{p}{\to} 0.  
\]
\end{lemma}
\begin{proof}{Proof for \cref{lemma: uniform-converge}.}
Note that 
\begin{align}
\abs{\hat{\overline c}(z; x) - \overline c(z; x)} 
	&\le \frac{1}{T}\sum_{j = 1}^T \abs{\sum_{i = 1}^n w_{ij}(x)\prns{c(z; Y_i) - \Eb{c(z; Y_i) \mid X_i = x}}} \nonumber \\
	&\le \frac{1}{T}\sum_{j = 1}^T \abs{\sum_{i = 1}^n w_{ij}(x)\prns{c(z; Y_i) - \Eb{c(z; Y_i) \mid X_i}}} \nonumber \\
	&+ \frac{1}{T}\sum_{j = 1}^T \abs{\sum_{i = 1}^n w_{ij}(x)\prns{\Eb{c(z; Y_i) \mid X_i} - \Eb{c(z; Y_i) \mid X_i = x}}}. \label{eq: uniform-consist1}
\end{align}
Denote $R_j(x)$ as the leaf of the $j$th tree that contains $x$.
For a set $S\subseteq\R p$, define $\op{diam}_{\mathcal J}(S)=\sup_{x,x'\in S}\prns{\sum_{j\in\mathcal J}(x_j-x'_j)^2}^{1/2}$.
Following the same arguments as in Lemma 2 of \cite{wager2018estimation}, we have that 
\[
\Prb{\op{diam}_{\mathcal J}(R_j(x)) \ge \epsilon} \le \sqrt{p}\exp\prns{-C_1\log^3\prns{s_n/(2k_n - 1)}}. 
\]
This implies that when $s_n/k_n \to \infty$ and $T = o(s_n/k_n)$, 
\[
\Prb{\sup_{1 \le j \le T}\op{diam}_{\mathcal J}(R_j(x)) \ge \epsilon} \le \sqrt{p}T\exp\prns{-C_1\log^3\prns{s_n/(2k_n - 1)}}  \to 0.
\]

Note that $w_{ij}(x) > 0$ only for $i$ such that $X_i \in R_j(x)$. This and the Lipschitz continuity of $\overline c(z; x)$ in $x$ together imply that  
\begin{align}\label{eq: uniform-consist2}
 \sup_{z \in \mathcal C}\abs{\sum_{i = 1}^n w_{ij}(x)\prns{\Eb{c(z; Y_i) \mid X_i} - \Eb{c(z; Y_i) \mid X_i = x}}} \le L_c \sup_j \op{diam}_{\mathcal J}(R_j(x)) \to 0.
 \end{align} 
 It follows that 
 \begin{align*}
   \sup_{z \in \mathcal C}\frac{1}{T}\sum_{j = 1}^T \abs{\sum_{i = 1}^n w_{ij}(x)\prns{\Eb{c(z; Y_i) \mid X_i} - \Eb{c(z; Y_i) \mid X_i = x}}} \to 0.
 \end{align*}
 Now consider a $\epsilon/\tilde C-$cover of $\mathcal C$, which we denote as $\{z_1, \dots, z_M\}$ with $M \le K\tilde C^d \prns{\frac{\op{diam}(\mathcal C)}{\epsilon}}^d$ for a positive constant $K$.
  This induces brackets of type $[c(z_k; Y) - \frac{\epsilon}{\tilde C}b(Y), c(z_k; Y) + \frac{\epsilon}{\tilde C}b(Y)]$ for the function class $\{y \mapsto c(z; y): z \in \mathcal C\}$. 
Note that for any $z \in \mathcal C$, 
\begin{align*}
\sum_{i = 1}^n w_{ij}(x)\prns{c(z; Y_i) - \Eb{c(z; Y_i) \mid X_i}} 
	&\le \max_{1 \le k \le M}\braces{\sum_{i = 1}^n w_{ij}(x)\prns{c(z_k; Y_i) + \frac{\epsilon}{\tilde C}b(Y_i)} - \sum_{i = 1}^n w_{ij}(x)\Eb{c(z_k; Y_i) + \frac{\epsilon}{\tilde C}b(Y_i)\mid X_i}}\\
	&+ {\frac{\epsilon}{\tilde C}\sum_{i = 1}^n w_{ij}(x)\Eb{b(Y_i)\mid X_i}}
\end{align*}
By \cref{lemma: exponential-inequality}, we have 
\begin{align*}
\sup_j \max_{1 \le k \le M}\braces{\sum_{i = 1}^n w_{ij}(x)\prns{c(z_k; Y_i) + \frac{\epsilon}{\tilde C}b(Y_i)} - \sum_{i = 1}^n w_{ij}(x)\Eb{c(z_k; Y_i) + \frac{\epsilon}{\tilde C}b(Y_i)\mid X_i}} \overset{p}{\to} 0. 
\end{align*}
Moreover, 
\begin{align*}
\sup_j {\frac{\epsilon}{\tilde C}\sum_{i = 1}^n w_{ij}(x)\Eb{b(Y_i)\mid X_i}} 
	&= {\frac{\epsilon}{\tilde C}\sum_{i = 1}^n w_{ij}(x)\Eb{b(Y_i)\mid X_i = x}} + {\frac{\epsilon}{\tilde C}\sum_{i = 1}^n w_{ij}(x)\prns{\Eb{b(Y_i)\mid X_i} - \Eb{b(Y_i)\mid X_i = x}}} \\
	&\le \epsilon + \frac{\epsilon}{\tilde C}L_b \sup_j \op{diam}_{\mathcal J}(R_j(x)) \to \epsilon.
\end{align*}
Thus as $n \to \infty$,
\begin{align*}
\sup_{j, z} \sum_{i = 1}^n w_{ij}(x)\prns{c(z; Y_i) - \Eb{c(z; Y_i) \mid X_i}}  \le \epsilon
\end{align*}
Similarly, we can prove that as $n \to \infty$,
\begin{align*}
\inf_{j, z} \sum_{i = 1}^n w_{ij}(x)\prns{c(z; Y_i) - \Eb{c(z; Y_i) \mid X_i}}  \ge  -\epsilon
\end{align*}
By the arbitrariness of $\epsilon$, we have 
\begin{align}
 \sup_{j, z} \abs{\sum_{i = 1}^n w_{ij}(x)\prns{c(z; Y_i) - \Eb{c(z; Y_i) \mid X_i}}} \to 0. \label{eq: uniform-consist3}
 \end{align}
\cref{eq: uniform-consist1,eq: uniform-consist2,eq: uniform-consist3} together imply that 
\[
	\sup_{z \in \mathcal C}\abs{\hat{\overline c}(z; x) - \overline c(z; x)} \overset{p}{\to} 0.  
\]
\end{proof}

\begin{proof}{Proof for \cref{thm: asymp-opt}.}
By condition \ref{assump: distr inf-compact} of \cref{assump: distr}, we have that for sufficiently large $n$, $\hat z_n \in \mathcal C$ and $\argmin_{z \in \mathcal Z} \overline c(z; x) \subseteq \mathcal C$ almost surely. Let us fix $z_0 \in \argmin_{z \in \mathcal Z} \overline c(z; x)$. 

The conclusion follows from 
\begin{align*}
 \abs{\overline c(\hat z_n; x) - \min_{z \in \mathcal Z} \overline c(z; x)} =  \abs{\overline c(\hat z_n; x) - \overline c(z_0; x)}  
 	&\le \abs{\overline c(\hat z_n; x) - \hat{\overline c}(\hat z_n; x)} + \abs{\hat{\overline c}(\hat z_n; x) - \overline c(z_0; x)} \\
 	&\le 2\sup_{z \in \mathcal C}\abs{z; x) - \overline c(z; x)} \to 0. 
 \end{align*} 
 Here the last inequality follows from the facts that 
 \begin{align*}
 \abs{\overline c(\hat z_n; x) - \hat{\overline c}(\hat z_n; x)} \le \sup_{z \in \mathcal C}\abs{z; x) - \overline c(z; x)}
 \end{align*}
and that 
\begin{align*}
\abs{\hat{\overline c}(\hat z_n; x) - \overline c(z_0; x)} = 
\left\{
\begin{array}{rcl}
 \hat{\overline c}(\hat z_n; x) - \overline c(z_0; x) \le \hat{\overline c}(z_0; x) - \overline c(z_0; x) \le \sup_{z \in \mathcal C}\abs{\hat{\overline c}(z; x) - \overline c(z; x)} & &  \text{if }\hat{\overline c}(\hat z_n; x) \ge \overline c(z_0; x) \\
 \overline c(z_0; x) - \hat{\overline c}(\hat z_n; x) \le  \overline c(\hat z_n; x) - \hat{\overline c}(\hat z_n; x) \le \sup_{z \in \mathcal C}\abs{\hat{\overline c}(z; x) - \overline c(z; x)} & & \text{if } \hat{\overline c}(\hat z_n; x) < \overline c(z_0; x)
\end{array}
\right. .
\end{align*}
\end{proof}

\begin{proof}{Proof for \cref{prop: lipschitz}.}
\textbf{Statement 1.} Consider $c_l(z; y) = \max\{\alpha_l\prns{z_l - y_l}, \beta_l\prns{y_l - z_l}\}$. Note $c_l(z; y) - c_l(z'; y) = \beta_l\prns{z_l - z_l'}$ if $z_l', z_l \le y_l$ and $c_l(z'; y) - c_l(z; y) = \alpha_l\prns{z_l - z_l'}$ if $z_l', z_l \ge y_l$.
Denote $\Delta_{l} = \abs{z_l - z_l'}$, $\Delta_{z_l} = \abs{z_l - y_l}$ and $\Delta_{z_l'} = \abs{z_l' - y_l}$. When $z_l \le y_l, y_l \le z_l'$, obviously $\Delta_{l} = \Delta_{z_l}+\Delta_{z_l'}$, and $\abs{c_l(z; y) - c_l(z'; y)} = \abs{\beta_l\Delta_{z_l} - \alpha_l\Delta_{z_l'}} \le \max\{\alpha_l, \beta_l\}\max\{\Delta_{z_l}, \Delta_{z_l'}\} \le \max\{\alpha_l, \beta_l\}\Delta_l$. Similarly we can show that when $z_l > y_l, y_l > z_l'$, $\abs{c_l(z; y) - c_l(z'; y)} \le \max\{\alpha_l, \beta_l\}\Delta_l$. Therefore, $\abs{c(z; y) - c(z'; y)} \le \sqrt{d}\max\{\alpha_l, \beta_l\}\edit{\|z - z'\|_2}$. 

\textbf{Statement 2.} Letting $C'=\sup_{\tilde z\in\mathcal C}\|\tilde z\|$, note that for any $z, z'\in\mathcal C$,
\begin{align*}
\abs{c(z; y) - c(z'; y)} 
	&= \abs{(y^\top z_{1:d} - z_{d+1})^2 - (y^\top z'_{1:d} - z'_{d+1})^2} \\
	&\le \abs{y^\top \prns{z_{1:d} + z'_{1:d}} - \prns{z_{d+1} + z'_{d+1}}}\abs{y^\top \prns{z_{1:d} - z'_{1:d}} - \prns{z_{d+1} - z'_{d+1}}} \\
	&\le \prns{\edit{\|y\|_2\|z_{d+1} + z'_{d+1}\|_2} + \abs{z_{d+1} + z'_{d+1}}}\edit{\left\|\begin{bmatrix}y \\ 1 \end{bmatrix}\right\|_2}\edit{\|z - z'\|_2} \\
	&\le 2C'\prns{\edit{\|y\|_2} + 1}\sqrt{\edit{\|y\|^2_2} + 1}\edit{\|z - z'\|_2} \\
	&\le 4\sqrt{2}C'\max\{1, \edit{\|y\|^2_2}\}\edit{\|z - z'\|}.
\end{align*}
\textbf{Statement 3. }Note that for any $z, z'\in\mathcal C$,
\begin{align*}
\abs{c(z; y) - c(z'; y)} 	&= \abs{\frac{1}{\alpha}\max\braces{z_{d+1} - y^\top z_{1:d},\, 0} - \frac{1}{\alpha}\max\braces{z'_{d+1} - y^\top z'_{1:d},\, 0} } + \abs{z_{d+1} - z'_{d+1}} \\
	&\le \frac{1}{\alpha}\abs{z_{d+1} - y^\top z_{1:d} - z'_{d+1} + y^\top z'_{1:d}} + \abs{z_{d+1} - z'_{d+1}} \\
	&\le \prns{\edit{\|y\|_2} + 1 + \frac{1}{\alpha}}\edit{\|z - z'\|_2}. 
\end{align*}
\end{proof}

\end{APPENDICES}

\end{document}